\documentclass[psamsfonts, reqno, 11pt]{amsart}

\usepackage{amsmath}
\usepackage{amssymb}
\usepackage{amsthm}
\usepackage[alphabetic]{amsrefs}
\usepackage{amscd}
\usepackage[margin=1in]{geometry}
\usepackage{fancyhdr}
\usepackage{todonotes}
\usepackage{setspace}
\usepackage[all,arc]{xy}
\usepackage{enumerate}
\usepackage{mathrsfs}
\usepackage{color}
\usepackage{tikz-cd}
\usepackage{dsfont} 
\usepackage[colorlinks=true, linkcolor=blue, citecolor=blue, urlcolor=blue]{hyperref}
\usepackage{xcolor}
\usepackage{float}

\newtheorem{thm}{Theorem}[section]
\newtheorem{cor}[thm]{Corollary}
\newtheorem{prop}[thm]{Proposition}
\newtheorem{lem}[thm]{Lemma}
\newtheorem{conj}[thm]{Conjecture}

\newtheorem*{thm*}{Theorem}
\newtheorem*{prop*}{Proposition}
\newtheorem*{lem*}{Lemma}
\newtheorem*{cor*}{Corollary}

\newtheorem{lem'}{Lemma}
\newtheorem{thm'}{Theorem}
\newtheorem{prop'}{Proposition}

\theoremstyle{definition}
\newtheorem{defn}[thm]{Definition}

\newtheorem{exmp}[thm]{Example}
\newtheorem*{exmp*}{Example}

\newtheorem*{exer*}{Exercise}

\theoremstyle{remark}
\newtheorem{rem}[thm]{Remark}

\newtheorem*{claim*}{Claim}
\newtheorem{step}{Step}
\newtheorem*{rem*}{Remark}
\newtheorem{rem'}{Remark}

\newtheorem{qtn}[thm]{Question}
\newtheorem*{qtn'}{Question}
\newtheorem*{soln'}{Solution}

\DeclareMathOperator{\re}{Re}
\DeclareMathOperator{\im}{Im}

\DeclareMathOperator{\sgn}{sgn}

\DeclareMathOperator{\dist}{dist}

\DeclareMathOperator{\Hol}{Hol}

\DeclareMathOperator{\op}{op}
\DeclareMathOperator{\Mat}{Mat}

\DeclareMathOperator{\Error}{Error}

\DeclareMathOperator{\GL}{GL}
\DeclareMathOperator{\SL}{SL}
\DeclareMathOperator{\PSL}{PSL}
\DeclareMathOperator{\PGL}{PGL}
\DeclareMathOperator{\PSO}{PSO}
\DeclareMathOperator{\SO}{SO}
\DeclareMathOperator{\so}{\mathfrak{so}}
\DeclareMathOperator{\Sl}{\mathfrak{sl}}

\DeclareMathOperator{\Stab}{Stab}
\DeclareMathOperator{\Lie}{Lie}
\DeclareMathOperator{\Gr}{Gr}
\DeclareMathOperator{\Ad}{Ad}

\DeclareMathOperator{\ev}{ev}
\DeclareMathOperator{\LHS}{LHS}
\DeclareMathOperator{\RHS}{RHS}
\DeclareMathOperator{\MT}{MT}
\DeclareMathOperator{\RT}{RT}

\newcommand{\true}{\textnormal{true}}
\newcommand{\false}{\textnormal{false}}

\newcommand{\fin}{\textnormal{fin}}

\newcommand{\R}{\mathbf{R}}
\newcommand{\C}{\mathbf{C}}
\newcommand{\Z}{\mathbf{Z}}

\renewcommand{\1}{\mathds{1}}
\newcommand{\g}{\mathfrak{g}}

\renewcommand{\O}{\textnormal{O}}
\renewcommand{\H}{\mathcal{H}}
\newcommand{\A}{\mathcal{A}}
\newcommand{\B}{\mathcal{B}}
\newcommand{\D}{\mathcal{D}}
\newcommand{\E}{\mathcal{E}}
\newcommand{\K}{\mathcal{K}}

\renewcommand{\P}{\mathcal{P}}
\renewcommand{\L}{\mathcal{L}}
\newcommand{\U}{\mathcal{U}}
\newcommand{\V}{\mathcal{V}}

\numberwithin{equation}{section}

\title{A converse theorem for hyperbolic surface spectra and the conformal bootstrap}

\author{Anshul Adve}
\address{Department of Mathematics, Princeton University, Princeton, NJ 08544, USA}
\email{aadve@princeton.edu}

\date{\today}

\begin{document}
	
	\begin{abstract}
		The conformal bootstrap in physics has recently been adapted to prove remarkably sharp estimates on Laplace eigenvalues and triple correlations of automorphic forms on compact hyperbolic surfaces.
		These estimates derive from an infinite family of algebraic equations satisfied by this spectral data.
		The equations encode $G$-equivariance and associativity of multiplication on $\Gamma \backslash G$, for $\Gamma$ a cocompact lattice in $G = \PSL_2(\R)$.
		The effectiveness of the conformal bootstrap suggests that the equations characterize hyperbolic surface spectra, i.e., that every solution to the equations comes from a compact hyperbolic surface.
		This paper proves this rigorously with no analytic assumptions on the solution except discreteness of the spectrum.
		The key intermediate result is an axiomatic characterization of representations of $G$ of the form $L^2(\Gamma \backslash G)$.

	\end{abstract}
	
	\maketitle
	
	\parskip 0em
	\setcounter{tocdepth}{2}
	\tableofcontents
	\parskip 0.5em
	
	
	\section{Introduction}
	
	\subsection{Description of main theorems and motivation from physics} \label{subsec:intro:motivation}
	
	We state three motivating questions, the first two in math and the third in physics.
	We will address the first two directly and the third by analogy.
	Since our rigorous results only concern the first two, we will not be completely precise in our discussion of the third.
	
	\begin{qtn} \label{qtn:spectral_multisets_smooth}
		Let $g \geq 2$ be an integer, and let
		\begin{align} \label{eqn:Laplace_spectrum}
			0 = \lambda_0 < \lambda_1 \leq \lambda_2 \leq \cdots \to \infty
		\end{align}
		be real numbers.
		When is $\{\lambda_r\}_{r \geq 0}$ the spectrum of the Laplacian on a compact hyperbolic surface of genus $g$?
	\end{qtn}
	
	Let $G = \PSL_2(\R)$ be the orientation-preserving isometry group of the hyperbolic plane $\mathbf{H}$.
	Then the following is roughly equivalent to Question~\ref{qtn:spectral_multisets_smooth}.
	
	\begin{qtn} \label{qtn:hyperbolic_rep}
		Let $\H$ be a unitary representation of $G$ with discrete spectrum.
		When is $\H$ isomorphic to $L^2(\Gamma \backslash G)$ for some cocompact lattice $\Gamma$ in $G$?
	\end{qtn}
	
	For $d \geq 3$, let $\widetilde{\SO}(2,d)$ be the identity component of the group of conformal automorphisms of the Lorentzian cylinder $\R \times S^{d-1}$.
	As an abstract Lie group, $\widetilde{\SO}(2,d)$ is the unique infinite cyclic cover of the identity component of $\SO(2,d)$.
	If $t$ is the time-coordinate on the Lorentzian cylinder, then the vector field $\partial_t$ lies in the Lie algebra of $\widetilde{\SO}(2,d)$.
	In conformal field theory (CFT), the element $\frac{1}{i}\partial_t$ of the complexified Lie algebra is called the \emph{dilatation operator} and plays the role of the Hamiltonian.
	A unitary representation of $\widetilde{\SO}(2,d)$ has \emph{positive energy} if the dilatation operator is positive semidefinite in the representation.
	
	Introductions to CFT from the point of view relevant for this paper include \cite{DSD_lectures,Rychkov_lectures,Qualls}.
	
	\begin{qtn} \label{qtn:cft_rep}
		Fix $d \geq 3$.
		Let $\H$ be a unitary, positive energy representation of $\widetilde{\SO}(2,d)$ with discrete spectrum. When is $\H$ the state space of a CFT on the Euclidean space $\R^d$?
		
		Here we use radial quantization, so ``the" state space is the Hilbert space associated to the codimension $1$ submanifold $S^{d-1} \subseteq \R^d$.
		The action of $\widetilde{\SO}(2,d)$ is obtained by using polar coordinates to identify
		$\R^d$ (minus the origin)
		with the Euclidean cylinder $\R \times S^{d-1}$, and then Wick rotating to Lorentzian signature.
	\end{qtn}
	
	There is a standard conjectural answer to Question~\ref{qtn:cft_rep} which is the basis for the \textit{conformal bootstrap}, described below.
	As a model for this conjecture, in this paper we formulate and prove analogous answers to Questions~\ref{qtn:spectral_multisets_smooth} and \ref{qtn:hyperbolic_rep}.
	These are Theorems~\ref{thm:main} and \ref{thm:eq_Gelfand_duality}, respectively, our two main theorems.
	Apart from an analytic technicality (Remark~\ref{rem:smoothness_of_products}), these theorems are essentially equivalent (Proposition~\ref{prop:mult_rep_vs_HB} makes this precise), and we consider them together as one main result.
	This result fits in the context of the analogy between CFTs and hyperbolic manifolds developed in \cite{Bonifacio--Hinterbichler,Bonifacio_22_1,Bonifacio_22_2,Kravchuk--Mazac--Pal,Bonifacio--Mazac--Pal,Gesteau--Pal--Simmons-Duffin--Xu,Adve_et_al}.
	All these papers apply ideas from the conformal bootstrap to place bounds on quantities of interest in the spectral theory of hyperbolic manifolds, such as the first Laplace eigenvalue $\lambda_1$.
	These bounds are often nearly saturated. For example, both \cite{Bonifacio_22_2} and \cite{Kravchuk--Mazac--Pal} show that for any genus $2$ hyperbolic surface, $\lambda_1 \leq 3.838898$; this is nearly saturated by the Bolza surface (the genus 2 surface with the most symmetries), which has $\lambda_1 = 3.838887...$ \cite[Section~5.3]{Strohmaier--Uski}.
	The above is the best known upper bound on $\lambda_1$ for genus $2$ hyperbolic surfaces, better even than what has subsequently been proved using the Selberg trace formula \cite{Bourque--Petri}.
	Theorem~\ref{thm:main} goes some way toward explaining why these papers are able to get such sharp bounds.
	We expand on this in Subsection~\ref{subsec:applications:lambda_1} and in particular Remark~\ref{rem:saturation}.
	
	
	To motivate the statements of Theorems~\ref{thm:eq_Gelfand_duality} and \ref{thm:main}, we begin below by discussing Question~\ref{qtn:cft_rep} and the conformal bootstrap in dimension $d \geq 3$.
	For more details, we refer to \cite{Poland--Rychkov--Vichi,Rychkov--Su} and \cite{Hartman_et_al,Bissi--Sinha--Zhou}, which respectively review the numerical and analytic sides of the conformal bootstrap.
	Mathematicians may find the axiomatic presentations in \cite{Rychkov_20}, \cite{Mazac_talk}, \cite[Section~2.2.1]{Benedetti_et_al}, and \cite[Section~2.3]{Kravchuk_Qiao_Rychkov_II} especially clarifying.
	The introduction \cite{Guillarmou--Kupiainen--Rhodes_review}, written for mathematicians, is also helpful for intuition and context, though it is specific to dimension $2$.
	
	While many 2d CFTs are exactly solvable and algebraic in nature \cite{Belavin--Polyakov--Zamolodchikov,Ribault_exactly_solvable}, CFTs in dimension $d \geq 3$ generally have a more transcendental character.
	This is due to the fact that the Lie algebra of conformal Killing fields on an open ball in $\R^d$ is infinite-dimensional for $d=2$ and finite-dimensional for $d \geq 3$ --- in dimension $2$ any holomorphic or antiholomorphic vector field is a conformal Killing field.
	Holomorphic vector fields span the Virasoro algebra (up to center).
	To illustrate the contrast between $d=2$ and $d \geq 3$, the 2d Ising model is exactly solved, and its spectrum of scaling dimensions (defined below) consists of rational numbers, whereas the spectrum of the 3d Ising model appears to be as complicated as, say, the Laplace spectrum of the $(2,3,13)$-triangle orbifold (the non-arithmetic compact hyperbolic 2-orbifold of minimal area \cite{Takeuchi}).
	Indeed, the spectra of the 3d Ising model and the $(2,3,13)$ orbifold are both expected to consist of transcendental numbers and to obey random matrix statistics at high energy.
	It is because of these considerations that we restrict to $d \geq 3$ in Question~\ref{qtn:cft_rep}.
	
	Let $\H$ be as in Question~\ref{qtn:cft_rep}. If $\H$ is the state space of a CFT on $\R^d$, then by the state-operator correspondence, there is a densely defined multiplication $\H \otimes \H \to \H$ called the \textit{operator product expansion (OPE)}.
	The domain $\D \subseteq \H \otimes \H$ of the OPE can be defined purely representation-theoretically.
	The OPE is $\so(2,d)$-equivariant, and it obeys a form of commutativity/associativity called \textit{crossing symmetry}.
	
	The irreducible, unitary, positive energy representations of $\widetilde{\SO}(2,d)$ are lowest weight representations $\pi_{\Delta,\rho}$ labeled by pairs $(\Delta,\rho)$, where $\Delta$ is a nonnegative real number called the \textit{scaling dimension}, and $\rho$ is an irreducible representation of $\SO(d)$ called the \textit{spin}.
	Consequently, the isomorphism class of $\H$ as a unitary representation of $\widetilde{\SO}(2,d)$ is encoded by the multiset $\{(\Delta_i,\rho_i)\}$ such that $\H \simeq \bigoplus_i \pi_{(\Delta_i,\rho_i)}$.
	The \textit{spectrum} of $\H$ is defined to be $\{\Delta_i\}$.
	There is a standard way to build an orthonormal basis of $\H$: first choose a decomposition of $\H$ into irreducibles, then choose a primary state (i.e., lowest weight vector) generating each irreducible, and finally complete the set of primary states to a basis of $\H$ by including all the descendants of the primary states.
	By definition, descendant states are obtained by acting on primary states with raising operators.
	Given such a basis, the OPE is encoded by a collection of structure constants called \textit{OPE coefficients}.
	Equivariance and crossing symmetry of the OPE are equivalent to a system of quadratic equations in the OPE coefficients. The coefficients in these quadratic equations
	are called \emph{conformal partial waves} or (up to normalization) \emph{conformal blocks}.
	The conformal partial waves are explicit functions of the scaling dimensions $\Delta_i$ and spins $\rho_i$.
	The quadratic equations are called \textit{crossing equations} or \textit{conformal bootstrap equations}.
	
	The conformal bootstrap is the method of studying CFTs by analyzing these equations.
	This idea dates back to the 1970s \cite{Ferrara--Grillo--Gatto,Polyakov,Mack} and developed rapidly after Belavin, Polyakov, and Zamolodchikov \cite{Belavin--Polyakov--Zamolodchikov} used it to solve infinitely many 2d CFTs, namely the minimal models (see \cite{Ribault_exactly_solvable} for an exposition). It was only realized much more recently \cite{Rattazzi--Rychkov--Tonni--Vichi,Rychkov_3D?,El-Showk_et_al_14} that the conformal bootstrap equations are useful in dimension $d \geq 3$ as well, even though they cannot be solved exactly in this case (aside from trivial solutions).
	They are useful because they are quadratic in the OPE coefficients, so they can be analyzed using linear/semidefinite programming.
	This has proven remarkably powerful. For example, for the Ising and $\O(2)$ models in 3d, the conformal bootstrap has given more precise determinations of the lowest few scaling dimensions than either physical experiment or Monte Carlo simulation \cite{Chang_et_al_25,Chester_et_al}.
	For more on the history of the bootstrap, see \cite{Rychkov_25}.
	
	We can now state the expected answer to Question~\ref{qtn:cft_rep}.
	Conjecturally, $\H$ is the state space of a CFT on $\R^d$ if and only if it admits an $\so(2,d)$-equivariant OPE $\D \to \H$ satisfying crossing symmetry.
	In practice, the representation-theoretic definition of the domain $\D$ is unwieldy, and one instead defines an OPE simply as a collection of OPE coefficients.
	In the conformal bootstrap literature, it is common to suppress $\H$ in addition to $\D$. Then the conjecture is that a collection of scaling dimensions, spins, and OPE coefficients comes from a CFT if and only if it satisfies the conformal bootstrap equations.
	
	%
	
	With the above in mind, we turn to Question~\ref{qtn:hyperbolic_rep}.
	Recall that we set $G = \PSL_2(\R)$.
	If $\H \simeq L^2(\Gamma \backslash G)$, then there is a densely defined multiplication $\H \otimes \H \to \H$, namely pointwise multiplication of functions on $\Gamma \backslash G$.
	We will axiomatize the basic algebraic properties of this multiplication (equivariance, commutativity, associativity, etc.).
	In Definition~\ref{def:mult_rep}, we define a \textit{multiplicative representation} to be a unitary representation $\H$ of $G = \PSL_2(\R)$ with discrete spectrum, equipped with a densely defined multiplication obeying the axioms.
	In this setting, the domain of multiplication is easier to define than for CFTs.
	Theorem~\ref{thm:eq_Gelfand_duality} states that every nontrivial multiplicative representation is isomorphic to $L^2(\Gamma \backslash G)$ for some cocompact lattice $\Gamma$.
	This gives an answer to Question~\ref{qtn:hyperbolic_rep}: $\H \simeq L^2(\Gamma \backslash G)$ for some $\Gamma$ if and only if $\H$ admits a multiplication making it into a nontrivial multiplicative representation.
	The proof of Theorem~\ref{thm:eq_Gelfand_duality} occupies the bulk of the paper.
	
	Theorem~\ref{thm:eq_Gelfand_duality} can be viewed as a $G$-equivariant Gelfand duality theorem, in which equivariance serves as a substitute for analytic assumptions.
	A key tool in the proof is the original Gelfand duality theorem, the classification of commutative unital C*-algebras (Theorem~\ref{thm:Gelfand_duality}).
	The proof of Theorem~\ref{thm:eq_Gelfand_duality} consequently splits into two parts: in the first, we build the C*-algebra to which we will apply Theorem~\ref{thm:Gelfand_duality}, and in the second, we show that the compact Hausdorff space given by Theorem~\ref{thm:Gelfand_duality} is of the form $\Gamma \backslash G$.
	The first part is quantitative while the second is qualitative.
	Both are nontrivial.
	The proof is outlined in Section~\ref{sec:outline_existence};
	see Subsection~\ref{subsec:reading} for suggestions for reading.
	
	The practical power of the conformal bootstrap lies in the equations.
	Motivated by this, we encode a multiplicative representation $\H$ by numerical data as follows.
	Multiplicative representations come with a canonical complex conjugation. The subspace $\H_{\R} \subseteq \H$ fixed by complex conjugation is preserved by $G$ (see Remark~\ref{rem:real_subspace}).
	The real unitary representation $\H_{\R}$ is a direct sum of principal series with Casimir eigenvalues $\{\lambda_r\}_{r \geq 0}$ and discrete series with lowest positive weights $\{k_r\}_{r \geq 1}$.
	The multisets of numbers $\{\lambda_r\}_{r \geq 0}$ and $\{k_r\}_{r \geq 1}$ encode the isomorphism class of $\H_{\R}$ as a real unitary representation of $G$.
	We choose an orthonormal basis $\{\psi_i\}_{i \in I}$ of $\H$ by a procedure similar to the one outlined above for CFTs.
	We call such a basis \textit{$(\g,K)$-adapted} (Definition~\ref{defn:(g,K)-adapted}).
	The multiplication on $\H$ is encoded by the structure constants $C_{ij}^{\ell} = \langle \psi_i \psi_j, \psi_{\ell} \rangle_{L^2(\Gamma \backslash G)}$.
	The axioms of a multiplicative representation translate to a system of linear and quadratic equations in the $C_{ij}^{\ell}$ with coefficients depending explicitly on the $\lambda_r$ and $k_r$.
	We call these the \textit{hyperbolic bootstrap equations} (Definition~\ref{defn:spectral_eqns}).
	Like the conformal bootstrap equations in dimension $d \geq 3$, the hyperbolic bootstrap equations are not exactly solvable, but they can be analyzed by linear/semidefinite programming.
	From Theorem~\ref{thm:eq_Gelfand_duality}, we deduce Theorem~\ref{thm:main}, which says that every solution $(\{\lambda_r\}_{r \geq 0}, \{k_r\}_{r \geq 1}, \{C_{ij}^{\ell}\}_{i,j,\ell \in I})$ to the hyperbolic bootstrap equations comes from $L^2(\Gamma \backslash G)$, for some cocompact lattice $\Gamma$, with some choice of $(\g,K)$-adapted basis $\{\psi_i\}_{i \in I}$.
	This is the exact analog for hyperbolic surface spectra of the conjectural answer to Question~\ref{qtn:cft_rep} discussed above.
	The implications of Theorem~\ref{thm:main} for Question~\ref{qtn:spectral_multisets_smooth} are articulated in Remark~\ref{rem:application_to_Q1}.
	
	
	Our formulation of the hyperbolic bootstrap equations is new, but it is closely related to \cite[Theorem~4.12]{Bonifacio--Mazac--Pal} for $G = \PGL_2(\C)$.
	Theorem~\ref{thm:main} resolves the analog for $G = \PSL_2(\R)$ of \cite[Open Problem~8.1]{Bonifacio--Mazac--Pal}.
	
	The idea that equivariance and associativity of multiplication on $\Gamma \backslash G$ impose constraint equations on spectral data is present in the work cited above on bounds for $\lambda_1$ and in the work \cite{Bernstein--Reznikov_10,Reznikov_unfolding,Michel--Venkatesh,Adve_et_al} on subconvexity for $L$-functions.
	In Section~\ref{sec:applications}, we explain how results in \cite{Kravchuk--Mazac--Pal} and \cite{Adve_et_al} regarding $\lambda_1$ and subconvexity can be obtained through the formalism of the hyperbolic bootstrap equations.
	
	
	
	
	We state Theorems~\ref{thm:eq_Gelfand_duality} and \ref{thm:main}, the two main theorems, in Subsections~\ref{subsec:mult_rep} and \ref{subsec:hyp_bootstrap_eqns}, respectively.
	These two subsections are independent of each other.
	Then in Subsection~\ref{subsec:equivalence}, we formalize the sense in which these theorems are equivalent.
	
	\subsection{Multiplicative representations and the first main theorem} \label{subsec:mult_rep}
	
	We begin with some notation and elementary definitions.
	As above, let $G = \PSL_2(\R)$.
	Let $K = \PSO_2(\R)$ be the standard maximal compact subgroup of $G$.
	Let $\g_{\R}$ be the Lie algebra of $G$ and $\g$ its complexification.
	
	Let $\widehat{G}$ be the unitary dual of $G$.
	We say that a subset $\Pi \subseteq \widehat{G}$ is \emph{bounded} if there is a bounded subset $B \subseteq \R$ such that each element of $\Pi$ has Casimir eigenvalue in $B$. A multiset $\Pi \subseteq \widehat{G}$ is \emph{discrete} if its intersection with any bounded subset of $\widehat{G}$ is finite.
	This condition implies that $\Pi$ is countable and has finite multiplicities.
	
	
	If $\H$ is a unitary representation of $G$, then $\H^{\infty}$ denotes the subspace of smooth vectors in $\H$, i.e., vectors $v \in \H$ such that $g \mapsto gv$ is a smooth $\H$-valued function on $G$.
	For example, if $\H = L^2(\Gamma \backslash G)$ for some cocompact lattice $\Gamma$, then $\H^{\infty} = C^{\infty}(\Gamma \backslash G)$.
	
	\begin{defn}[Discrete spectrum and countable spectrum] \label{defn:pure_point_spectrum}
		A unitary representation $\H$ of $G$ has \emph{discrete} (resp. \emph{countable}) \emph{spectrum} if there is a discrete (resp. countable) multiset $\Pi \subseteq \widehat{G}$ such that $\H \simeq \bigoplus_{\pi \in \Pi} \pi$. If so, then $\Pi$ is unique, and it is called the \emph{spectrum} of $\H$.
	\end{defn}
	
	
	We say that a real unitary representation has the above properties if its complexification does (when unspecified, representations and function spaces are over $\C$).
	
	Given an irreducible unitary representation $\pi$ of $G$, let $\pi^{\fin} \subseteq \pi$ denote the dense subspace of $K$-finite vectors. Then $\pi^{\fin} \subseteq \pi^{\infty}$.
	
	\begin{defn}[$\H^{\fin}$] \label{defn:H^fin}
		Suppose $\H$ is a unitary representation of $G$ with discrete spectrum. Let $\Pi$ be the spectrum of $\H$, and fix a decomposition $\H = \bigoplus_{\pi \in \Pi} \pi$. Then define $\H^{\fin} = \bigoplus_{\pi \in \Pi} \pi^{\fin}$, where here we use the algebraic direct sum rather than the Hilbert space direct sum. Since $\H$ has discrete spectrum, $\Pi$ has finite multiplicities, so this definition of $\H^{\fin}$ is independent of the choice of decomposition of $\H$.
	\end{defn}
	
	The subspace $\H^{\fin}$ is $\g$-invariant and $K$-invariant but almost never $G$-invariant.
	It is always dense in $\H$.
	Moreover, $\H^{\fin} \subseteq \H^{\infty}$, and $\H^{\fin}$ is dense in $\H^{\infty}$ with respect to the natural Fr\'echet topology on $\H^{\infty}$.
	
	If $\H = L^2(\Gamma \backslash G)$ for some cocompact lattice $\Gamma$, then $\H$ has discrete spectrum, and $\H^{\fin}$ is the algebraic linear span of automorphic forms on $\Gamma \backslash G$.
	Definition~\ref{def:aut_vector} and Proposition~\ref{prop:Maass_span_H^fin} generalize this characterization of $\H^{\fin}$ to arbitrary $\H$.
	
	Again if $\H = L^2(\Gamma \backslash G)$, then since $\H^{\fin} \subseteq \H^{\infty} = C^{\infty}(\Gamma \backslash G)$, pointwise multiplication gives a well-defined bilinear map $\H^{\fin} \times \H^{\fin} \to \H^{\infty}$.
	We note that the image of this map is not contained in $\H^{\fin}$ (this is not completely obvious; it follows for example from the first remark in \cite{Bernstein--Reznikov_99}).
	The fact that $\H^{\fin}$ is not closed under multiplication is a source of many analytic difficulties.
	
	\begin{defn}[Multiplicative representation] \label{def:mult_rep}
		A \emph{multiplicative representation} $\H$ is a unitary representation of $G$ with discrete spectrum, equipped with a bilinear map $\H^{\fin} \times \H^{\fin} \to \H^{\infty}$ denoted $(\alpha,\beta) \mapsto \alpha\beta$ and called \emph{multiplication}, such that the following properties hold.
		\begin{itemize} \itemsep = 0.5em
			\item \emph{Commutativity}: For all $\alpha,\beta \in \H^{\fin}$, one has $\alpha\beta = \beta\alpha$.
			
			\item \emph{Existence of a unit}: There is an element $\mathbf{1} \in \H^{\fin}$ such that $\mathbf{1}\alpha = \alpha$ for all $\alpha \in \H^{\fin}$. This element is automatically unique (see Remark~\ref{rem:on_mult_rep_def}), and we call it the \emph{unit}.
			
			\item \emph{Normalization}: The unit has norm $\|\mathbf{1}\|_{\H} = 1$.
			
			\item \emph{Ergodicity}: The subspace of $\H$ fixed by $G$ is spanned by the unit, i.e., $\H^G = \mathbf{C1}$.
			
			\item \emph{Equivariance}: For all $X \in\g$ and $\alpha,\beta \in \H^{\fin}$, one has the \emph{product rule}
			\begin{align} \label{eqn:prod_rule}
				X(\alpha\beta)
				= (X\alpha)\beta + \alpha(X\beta).
			\end{align}
			
			\item \emph{Existence of complex conjugates}: For each $\alpha \in \H^{\fin}$, there exists $\overline{\alpha} \in \H^{\fin}$ such that for all $\beta,\gamma \in \H^{\fin}$,
			\begin{align} \label{eqn:3_term_crossing}
				\langle \alpha\beta, \gamma \rangle_{\H}
				= \langle \beta, \overline{\alpha} \gamma \rangle_{\H}.
			\end{align}
			Since $\H^{\fin}$ is dense in $\H$, taking $\gamma = \mathbf{1}$ and ranging over all $\beta \in \H^{\fin}$ shows that $\overline{\alpha}$ is unique. We call $\overline{\alpha}$ the \emph{complex conjugate} of $\alpha$.
			
			\item \emph{Crossing symmetry}: For all $\alpha_1,\alpha_2,\alpha_3,\alpha_4 \in \H^{\fin}$ and all permutations $\sigma \in S_4$,
			\begin{align} \label{eqn:crossing}
				\langle \alpha_1\alpha_2, \overline{\alpha_3} \, \overline{\alpha_4} \rangle_{\H}
				= \langle \alpha_{\sigma(1)}\alpha_{\sigma(2)}, \overline{\alpha_{\sigma(3)}} \, \overline{\alpha_{\sigma(4)}} \rangle_{\H}.
			\end{align}
			Adopting CFT terminology, we call instances of \eqref{eqn:crossing} \textit{crossing equations}.
		\end{itemize}
		Given $\alpha \in \H^{\fin}$, we denote 
		\begin{align*}
			\alpha^2 = \alpha\alpha
			\qquad \text{and} \qquad
			|\alpha|^2 = \alpha\overline{\alpha}.
		\end{align*}
		A multiplicative representation $\H$ is \emph{nontrivial} if $\dim \H > 1$.
		An \emph{isomorphism of multiplicative representations} is an isomorphism of the underlying unitary representations respecting multiplication.
	\end{defn}
	
	Given a lattice $\Gamma$ in $G$, we always equip $\Gamma \backslash G$ with the Haar probability measure.
	
	\begin{exmp}
		Let $\Gamma$ be a cocompact lattice in $G$.
		Then $L^2(\Gamma \backslash G)$ is a multiplicative representation with multiplication given pointwise.
	\end{exmp}
	
	The first of our two main theorems is the following.
	We emphasize that it is for $G = \PSL_2(\R)$.
	
	\begin{thm}[Converse theorem for multiplicative representations] \label{thm:eq_Gelfand_duality}
		Let $\H$ be a nontrivial multiplicative representation. Then $\H \simeq L^2(\Gamma \backslash G)$, as multiplicative representations, for some cocompact lattice $\Gamma$ in $G$.
	\end{thm}
	
	We prove this in Sections~\ref{sec:complex_conj}--\ref{sec:bulk_tail_general}.
	The corresponding uniqueness statement also holds:
	
	\begin{prop} \label{prop:mult_rep_uniqueness}
		Let $\Gamma,\Gamma'$ be cocompact lattices in $G$. Suppose $\Phi \colon L^2(\Gamma \backslash G) \to L^2(\Gamma' \backslash G)$ is an isomorphism of multiplicative representations.
		Then there exists a unique element $g \in G/\Gamma$ such that $\Gamma' = g\Gamma g^{-1}$ and
		\begin{align*}
			\Phi(f)(x)
			= f(g^{-1}x)
		\end{align*}
		for all $f \in L^2(\Gamma \backslash G)$.
	\end{prop}
	
	We prove this in Section~\ref{sec:uniqueness}.
	Proposition~\ref{prop:mult_rep_uniqueness} is much easier than Theorem~\ref{thm:eq_Gelfand_duality}.
	
	\begin{rem}[Motivation for Definition~\ref{def:mult_rep}]
		One could consider other definitions of multiplicative representation, with a different domain of multiplication or with different axioms.
		Definition~\ref{def:mult_rep} is chosen so that when multiplicative representations are encoded by numerical data, Theorem~\ref{thm:eq_Gelfand_duality} translates into a clean numerical statement (Theorem~\ref{thm:main}).
		More specifically, the domain $\H^{\fin}$ is chosen so that Proposition~\ref{prop:basis_spans_H^fin} holds, and the axioms are chosen so that Proposition~\ref{prop:mult_rep_vs_HB} can be proved with relatively little work.
	\end{rem}
	
	\begin{rem}[Technical comments on Definition~\ref{def:mult_rep}]
		\label{rem:on_mult_rep_def}
		\,
		
		\begin{itemize} \itemsep = 0.5em
			\item The unit is unique because if $\mathbf{1}, \mathbf{1}'$ are both units, then $\mathbf{1} = \mathbf{1} \mathbf{1}' = \mathbf{1}'$.
			
			
			\item Since $\H^{\fin}$ is closed under the $\g$-action but not the $G$-action, equivariance of multiplication has to be stated with respect to $\g$ rather than $G$. This is why equivariance amounts to the product rule \eqref{eqn:prod_rule}. Note that the left hand side of \eqref{eqn:prod_rule} makes sense because the multiplication map takes values in $\H^{\infty}$.
			
			\item It is a pleasant feature that existence of complex conjugates is a property rather than part of the data of a multiplicative representation. We will see quite easily in Section~\ref{sec:complex_conj} that complex conjugation on $\H^{\fin}$ obeys all the expected properties.

			\item By commutativity, the crossing equation \eqref{eqn:crossing} is trivial for $\sigma = (12)$ and $\sigma = (34)$. The permutations $(12),(23),(34)$ generate $S_4$, so crossing symmetry for all permutations is equivalent to crossing symmetry for the single permutation $\sigma = (23)$.
			
			\item 
			Crossing symmetry is a substitute for associativity of multiplication.
			Instead of \eqref{eqn:crossing}, we would have preferred to assume $(\alpha\beta)\gamma = \alpha(\beta\gamma)$ for all $\alpha,\beta,\gamma \in \H^{\fin}$. Unfortunately, $\H^{\fin}$ is not closed under multiplication, so $(\alpha\beta)\gamma$ is not defined in general. The standard workaround is to test the equality $(\alpha\beta)\gamma = \alpha(\beta\gamma)$ against a fourth vector $\delta \in \H^{\fin}$ and write
			\begin{align} \label{eqn:associativity_vs_crossing}
				\langle \alpha\beta, \overline{\gamma}\delta \rangle_{\H}
				= ``\langle (\alpha\beta)\gamma, \delta \rangle_{\H}"
				= ``\langle \alpha(\beta\gamma), \delta \rangle_{\H}"
				= \langle \beta\gamma, \overline{\alpha}\delta \rangle_{\H},
			\end{align}
			with the middle equality by associativity.
			After rewriting $\alpha\beta$ as $\beta\alpha$ in the left hand side, the equality of the left and right hand sides becomes \eqref{eqn:crossing} with $\alpha_1=\beta$, $\alpha_2=\alpha$, $\alpha_3=\gamma$, $\alpha_4=\overline{\delta}$, and $\sigma = (23)$.
			
			
			
			\item The normalization and ergodicity axioms are not crucial. If these are dropped from Definition~\ref{def:mult_rep}, then the proof of Theorem~\ref{thm:eq_Gelfand_duality} easily adapts to show that every multiplicative representation is of the form $L^2(X,\mu)$ for some compact $G$-space $X$ and some finite $G$-invariant measure $\mu$, with $X$ a disjoint union of points and $G$-spaces of the form $\Gamma \backslash G$.
		\end{itemize}
	\end{rem}
	
	\begin{rem}[Other groups] \label{rem:other_gps}
		Definition~\ref{def:mult_rep} extends essentially verbatim to other semisimple Lie groups, and one can ask if Theorem~\ref{thm:eq_Gelfand_duality} extends as well.
		Specifically for $G = \PGL_2(\C)$, the effectiveness of \cite{Bonifacio--Mazac--Pal} suggests that Theorem~\ref{thm:eq_Gelfand_duality} holds.
		A proof of this would solve \cite[Open Problem~8.1]{Bonifacio--Mazac--Pal}.
		At the moment, the quantitative part of the proof of Theorem~\ref{thm:eq_Gelfand_duality} makes significant use of the structure of irreducible representations of $G = \PSL_2(\R)$.
		However, the proof does not use anything that would preclude generalization --- notably, it does not use uniqueness of trilinear functionals
		as in \cite{Prasad,Loke}.
	\end{rem}
	
	\subsection{Multiplicative spectra and the second main theorem} \label{subsec:hyp_bootstrap_eqns}
	
	In this subsection we define \emph{multiplicative spectra}, formulate the hyperbolic bootstrap equations, and state our second main theorem.
	Again, we need some preliminary definitions.
	
	
	
	Let $\mathbf{H}$ be the upper half-plane with its hyperbolic metric. The group $G = \PSL_2(\R)$ acts on $\mathbf{H}$ by M\"obius transformations, preserving the metric and preserving orientation.
	The stabilizer of $i \in \mathbf{H}$ is $K = \PSO_2(\R)$, and thus we identify $\mathbf{H} = G/K$.
	In this paper, a \textit{hyperbolic 2-orbifold} is a quotient of $\mathbf{H}$ by a discrete subgroup $\Gamma$ of $G$.
	We will only work with compact hyperbolic 2-orbifolds, so that the Laplace spectrum is discrete.
	A \textit{hyperbolic surface} is a hyperbolic 2-orbifold with no orbifold points.
	The \textit{topological type} of a compact hyperbolic 2-orbifold of genus $g$ with $s$ orbifold points is $[g;m_1,\dots,m_s]$, where $2 \leq m_1 \leq \cdots \leq m_s$ are the orders of the orbifold points.
	
	For $k \in \Z_{\geq 1}$ and $\Gamma$ a cocompact lattice in $G$, let $M_{2k}(\Gamma)$ be the space of holomorphic modular forms of weight $2k$ for $\Gamma$, i.e., the space of holomorphic functions $f \colon \mathbf{H} \to \C$ such that
	\begin{align} \label{eqn:modular_form_def}
		f\Big(\frac{az+b}{cz+d}\Big) = (cz+d)^{2k} f(z)
		\qquad \text{for all} \qquad
		\begin{pmatrix}
			a & b \\
			c & d
		\end{pmatrix}
		\in \Gamma.
	\end{align}
	
	\begin{defn}[Holomorphic spectrum] \label{defn:holomorphic_spectrum}
		The \textit{holomorphic spectrum} of a compact hyperbolic 2-orbifold $\Gamma \backslash \mathbf{H}$ is the multiset of positive integers
		\begin{align} \label{eqn:hol_spectrum_def}
			\{k_1 \leq k_2 \leq \cdots\}
		\end{align}
		in which each $k \in \Z_{\geq 1}$ appears with multiplicity $\dim M_{2k}(\Gamma)$.
	\end{defn}
	
	By Riemann--Roch (in the form \cite[Theorem~4.9]{Milne}), the topological type of $\Gamma \backslash \mathbf{H}$ determines the holomorphic spectrum and vice versa.
	
	\begin{exmp} \label{exmp:hol_spectrum}
		The holomorphic spectrum of a genus 2 surface (with no orbifold points) is
		\begin{align*}
			\{1,1,2,2,2,3,3,3,3,3,4,4,4,4,4,4,4,5,5,5,5,5,5,5,5,5,6,6,6,6,6,6,6,6,6,6,6,\dots\}.
		\end{align*}
		More generally, the holomorphic spectrum of a genus $g$ surface contains each $k \in \Z_{\geq 2}$ with multiplicity $(g-1)(2k-1)$, and contains $k=1$ with multiplicity $g$.
	\end{exmp}
	
	The Laplace and holomorphic spectra of $\Gamma \backslash \mathbf{H}$ together determine the isomorphism class of $L^2(\Gamma \backslash G, \R)$ as a real unitary $G$-representation, and vice versa.
	The precise correspondence is stated in Subsection~\ref{subsec:equivalence}.
	
	\begin{defn}[Multiplication table] \label{defn:mult_table}
		Let $\Gamma \backslash \mathbf{H}$ be a compact hyperbolic 2-orbifold, let $I$ be an indexing set, and let $\{\psi_i\}_{i \in I}$ be an $I$-indexed orthonormal basis of $L^2(\Gamma \backslash G)$. The \textit{multiplication table of $\Gamma \backslash \mathbf{H}$ with respect to $\{\psi_i\}_{i \in I}$} is $\{C_{ij}^{\ell}\}_{i,j,\ell \in I}$, where
		\begin{align} \label{eqn:C_ij^l_def}
			C_{ij}^{\ell}
			= \langle \psi_i \psi_j, \psi_{\ell} \rangle_{L^2(\Gamma \backslash G)}
		\end{align}
		are the unique complex numbers such that
		\begin{align*}
			\psi_i \psi_j = \sum_{\ell \in I} C_{ij}^{\ell} \psi_{\ell}
		\end{align*}
		for all $i,j \in I$
		(assume the $\psi_i$ are all in $L^4$, so that $\psi_i\psi_j \in L^2$ and the inner product in \eqref{eqn:C_ij^l_def} makes sense; soon we will assume much more regularity than this).
	\end{defn}
	
	Under the standard identification of $\Gamma \backslash G$ with the unit tangent bundle of $\Gamma \backslash \mathbf{H}$, Haar measure on $\Gamma \backslash G$ corresponds to Liouville measure on $T^1(\Gamma \backslash \mathbf{H})$. Thus the multiplication table can be viewed as an invariant of a compact hyperbolic 2-orbifold $\Sigma$ equipped with an orthonormal basis of $L^2(T^1\Sigma$), where the unit tangent bundle $T^1\Sigma$ has the Liouville probability measure.
	This perspective shows that the multiplication table is independent of the realization of $\Sigma$ as a quotient of $\mathbf{H}$.
	
	
	We combine the Laplace spectrum, holomorphic spectrum, and multiplication table into one master invariant.
	
	\begin{defn}[Multiplicative spectrum] \label{defn:mult_spectrum}
		With notation as in Definition~\ref{defn:mult_table}, the \textit{multiplicative spectrum of $\Gamma \backslash \mathbf{H}$ with respect to $\{\psi_i\}_{i \in I}$} is the triple
		\begin{align*}
			(\{\lambda_r\}_{r \geq 0}, \{k_r\}_{r \geq 1}, \{C_{ij}^{\ell}\}_{i,j,\ell \in I}),
		\end{align*}
		where $\{\lambda_r\}_{r \geq 0}$ is the spectrum of the Laplacian on $\Gamma \backslash \mathbf{H}$ (indexed in increasing order as in \eqref{eqn:Laplace_spectrum}), $\{k_r\}_{r \geq 1}$ is the holomorphic spectrum of $\Gamma \backslash \mathbf{H}$ (indexed in increasing order as in \eqref{eqn:hol_spectrum_def}), and $\{C_{ij}^{\ell}\}_{i,j,\ell \in I}$ is the multiplication table of $\Gamma \backslash \mathbf{H}$ with respect to $\{\psi_i\}_{i \in I}$.
	\end{defn}
	
	Among all orthonormal bases of $L^2(\Gamma \backslash G)$, there is a well-behaved class which we call \textit{$(\g,K)$-adapted}, because bases in this class interact via simple formulas with the actions of $\g$ and $K$ on $C^{\infty}(\Gamma \backslash G)$ (see Proposition~\ref{prop:(g,K)-adapted_key_properties} for these formulas).
	The construction of such bases is very similar to the construction of bases of primary and descendant states for a CFT.
	We defer the precise definition of $(\g,K)$-adapted bases to Section~\ref{sec:(g,K)-adapted}.
	For now, the important features of $(\g,K)$-adapted bases $\{\psi_i\}_{i \in I}$ are the following two bullet points, as well as Proposition~\ref{prop:basis_exist_unique} below.
	\begin{itemize} \itemsep = 0.5em
		\item The basis elements $\psi_i$ are automorphic forms on $\Gamma \backslash G$ (i.e., simultaneous eigenfunctions for the Casimir operator and for translation by elements of $K$).
		
		\item The indexing set $I$ is a subset of $\Z^2$ depending only on the topological type of $\Gamma \backslash \mathbf{H}$.
		Explicitly, $I$ is defined in terms of the holomorphic spectrum $\{k_r\}_{r \geq 1}$ by \eqref{eqn:I_hol_def} below.
		We take $I$ to be a subset of $\Z^2$ so that each index $i$ has two components $i_1,i_2$; the first component $i_1$ indexes the Casimir eigenvalue of $\psi_i$, and the second component $i_2$ is equal to the weight of $\psi_i$.
		In our normalization, the weight can be any integer, not just any even integer (c.f. Remark~\ref{rem:wts_PSL_2_vs_SL_2}).
	\end{itemize}
	Although the precise definition of the indexing set $I$ will not be motivated until Section~\ref{sec:(g,K)-adapted}, we state it here for full transparency:
	\begin{align} \label{eqn:I_hol_def}
		I =
		\{i = (i_1,i_2) \in \Z^2 : i_1 > 0 \text{ or } i = (0,0) \text{ or } (i_1 < 0 \text{ and } |i_2| \geq k_{-i_1})\}.
	\end{align}
	Since the $k_r$ are positive integers,
	\begin{align} \label{eqn:i2=0_implies_i1geq0}
		i \in I \text{ and } i_2 = 0
		\quad\implies\quad
		i_1 \geq 0.
	\end{align}
	
	\begin{prop}[Existence and uniqueness of $(\g,K)$-adapted bases] \label{prop:basis_exist_unique}
		Let $\Gamma \backslash \mathbf{H}$ be a compact hyperbolic 2-orbifold, let $\{k_r\}_{r \geq 1}$ be its holomorphic spectrum, and let $I$ be as in \eqref{eqn:I_hol_def}. Then there exists a $(\g,K)$-adapted basis $\{\psi_i\}_{i \in I}$ of $L^2(\Gamma \backslash G)$. If $\{\psi_i'\}_{i \in I}$ is another such basis, then there is a unique automorphism of $L^2(\Gamma \backslash G, \R)$, as a real unitary $G$-representation, which (after complexifying) takes $\psi_i$ to $\psi_i'$ for all $i \in I$.
	\end{prop}
	
	Proposition~\ref{prop:basis_exist_unique} is a special case of Proposition~\ref{prop:(g,K)-adapted_exist_unique}.
	We note that the automorphism in Proposition~\ref{prop:basis_exist_unique} need not be induced by an automorphism of the $G$-space $\Gamma \backslash G$.
	Thus from a geometric point of view, $(\g,K)$-adapted bases are \emph{not} canonical.
	
	\begin{defn}[Candidate spectrum] \label{defn:potential_spectrum}
		A \emph{candidate spectrum} $\mathcal{S}$ is a tuple
		\begin{align*}
			\mathcal{S} = (\{\lambda_r\}_{r \geq 0}, \{k_r\}_{r \geq 1}, \{C_{ij}^{\ell}\}_{i,j,\ell \in I}),
		\end{align*}
		where
		\begin{align} \label{eqn:lambda_r_to_infty}
			0 = \lambda_0 < \lambda_1 \leq \lambda_2 \leq \cdots \to \infty
		\end{align}
		are real numbers,
		\begin{align} \label{eqn:k_r_to_infty}
			k_1 \leq k_2 \leq \cdots \to \infty
		\end{align}
		are positive integers, $I$ is defined in terms of the $k_r$ by \eqref{eqn:I_hol_def}, and the $C_{ij}^{\ell}$ are complex numbers.
	\end{defn}
	
	The conditions \eqref{eqn:lambda_r_to_infty} and \eqref{eqn:k_r_to_infty} imply that $\{\lambda_r\}_{r \geq 0}$ is discrete and $\{k_r\}_{r \geq 1}$ has finite multiplicities. The significance of these two conditions is discussed in Remark~\ref{rem:analytic_hypotheses}.
	
	The second of our two main theorems is
	
	\begin{thm}[Converse theorem for multiplicative spectra] \label{thm:main}
		A candidate spectrum $\mathcal{S}$ is the multiplicative spectrum of a compact hyperbolic 2-orbifold with respect to a $(\g,K)$-adapted basis if and only if $\mathcal{S}$ solves the hyperbolic bootstrap equations \eqref{eqn:SE1}--\eqref{eqn:SE6} in Definition~\ref{defn:spectral_eqns} below.
	\end{thm}
	
	
	Remark~\ref{rem:convergence_assumptions} gives a slight strengthening of Theorem~\ref{thm:main}.
	This remark is of a technical nature, and should not be read before Definition~\ref{defn:spectral_eqns}.
	
	Again, there is a corresponding uniqueness statement:
	
	\begin{prop} \label{prop:mult_spectra_uniqueness}
		Let $(\Gamma \backslash \mathbf{H}, \{\psi_i\}_{i \in I})$ and $(\Gamma' \backslash \mathbf{H}, \{\psi_i'\}_{i \in I'})$ be compact hyperbolic 2-orbifolds equipped with $(\g,K)$-adapted bases.
		If both have the same multiplicative spectrum, then $I = I'$, and there exists a unique element $g \in G/\Gamma$ such that $\Gamma' = g\Gamma g^{-1}$ and
		\begin{align*}
			\psi_i'(x)
			= \psi_i(g^{-1}x)
		\end{align*}
		for all $i \in I$.
	\end{prop}
	
	In Section~\ref{sec:uniqueness}, we will deduce Proposition~\ref{prop:mult_spectra_uniqueness} from Proposition~\ref{prop:mult_rep_uniqueness}.
	A very similar uniqueness result, for general closed Riemannian manifolds, is \cite{Schaefer}.
	Both Proposition~\ref{prop:mult_spectra_uniqueness} and \cite{Schaefer} use Gelfand duality (Theorem~\ref{thm:Gelfand_duality}).
	
	%
	
	\begin{defn}[Hyperbolic bootstrap equations] \label{defn:spectral_eqns}
		Let
		\begin{align*}
			\mathcal{S}
			= (\{\lambda_r\}_{r \geq 0}, \{k_r\}_{r \geq 1}, \{C_{ij}^{\ell}\}_{i,j,\ell \in I})
		\end{align*}
		be a candidate spectrum.
		Before stating the equations, let us fix some additional notation.
		Given $i,j,\ell \in \Z^2$ not all in $I$, put $C_{ij}^{\ell}
		= 0$.
		For $r$ a negative integer, let
		\begin{align} \label{eqn:lambda_{-r}_def}
			\lambda_r = -k_{-r}(k_{-r}-1)
		\end{align}
		(this is the Casimir eigenvalue of the discrete series representation of $G$ with lowest weight $k_{-r}$; see Proposition~\ref{prop:lowest_wt_wt} and Remark~\ref{rem:wts_PSL_2_vs_SL_2}).
		Now $\lambda_r$ is defined for all $r \in \Z$. It follows from \eqref{eqn:I_hol_def}, \eqref{eqn:lambda_r_to_infty}, and \eqref{eqn:lambda_{-r}_def} that for all $i = (i_1,i_2) \in I$,
		\begin{align} \label{eqn:sqrt_nonnegativity}
			\lambda_{i_1} + i_2(i_2+1) \geq 0
			\qquad \text{and} \qquad
			\lambda_{i_1} + i_2(i_2-1) \geq 0.
		\end{align}
		This implies that the square roots in \eqref{eqn:SE5} below are well-defined.
		For $i = (i_1,i_2) \in \Z^2$, denote
		\begin{align*} 
			\overline{i} = (i_1,-i_2)
			\qquad \text{and} \qquad
			i^+ = (i_1,i_2+1)
			\qquad \text{and} \qquad
			i^- = (i_1,i_2-1).
		\end{align*}
		Note that $I$ is preserved by $i \mapsto \overline{i}$, but not by $i \mapsto i^+$ or $i \mapsto i^-$.
		With the above notation, the six hyperbolic bootstrap equations are below. After each equation, we give a brief description of where it comes from.
		\begin{itemize} \itemsep = 0.5em
			\item \textbf{Hyperbolic Bootstrap Equation 1:} For all $i,j \in I$,
			\begin{align} \label{eqn:SE1}
				C_{ij}^{\ell}
				= C_{ji}^{\ell}.
				\tag{HB1}
			\end{align}
			\textit{Description}: This encodes that multiplication (of functions on $\Gamma \backslash G$) is commutative.
			
			\item \textbf{Hyperbolic Bootstrap Equation 2:} For all $i,j,\ell \in I$ with $i_2+j_2 \neq \ell_2$,
			\begin{align} \label{eqn:SE2}
				C_{ij}^{\ell}
				= 0.
				\tag{HB2}
			\end{align}
			\textit{Description}: This encodes that multiplication is $K$-equivariant.
			
			\item \textbf{Hyperbolic Bootstrap Equation 3:} For all $i,j,\ell \in I$,
			\begin{align} \label{eqn:SE3}
				\overline{C_{ij}^{\ell}}
				= C_{\overline{i} \, \overline{j}}^{\overline{\ell}}.
				\tag{HB3}
			\end{align}
			\textit{Description}: This encodes that multiplication commutes with complex conjugation.
			
			\item \textbf{Hyperbolic Bootstrap Equation 4:} For $i=(0,0)$ and all $j,\ell \in I$,
			\begin{align} \label{eqn:SE4}
				C_{ij}^{\ell} = \1_{j=\ell}.
				\tag{HB4}
			\end{align}
			\textit{Description}: This encodes that the constant function $1$ is a multiplicative identity.
			
			\item \textbf{Hyperbolic Bootstrap Equation 5:} For all $i,j,\ell \in I$,
			\begin{align} \label{eqn:SE5}
				\sqrt{\lambda_{\ell_1} + \ell_2(\ell_2-1)} \, C_{ij}^{\ell^-}
				=
				\sqrt{\lambda_{i_1} + i_2(i_2+1)} \, C_{i^+j}^{\ell}
				+ \sqrt{\lambda_{j_1} + j_2(j_2+1)} \, C_{ij^+}^{\ell}.
				\tag{HB5}
			\end{align}
			\textit{Description}: This encodes the product rule
			\begin{align*}
				E(\psi_i \psi_j) = (E\psi_i) \psi_j + \psi_i (E\psi_j),
			\end{align*}
			where $E \in \g$ is the raising operator (defined in \eqref{eqn:H,E_def}) acting on $C^{\infty}(\Gamma \backslash G)$.
			
			\item \textbf{Hyperbolic Bootstrap Equation 6:} For all $i,j,i',j' \in I$, the following holds with absolute convergence on both sides:
			\begin{align} \label{eqn:SE6}
				\sum_{\ell \in I} (-1)^{\ell_2} C_{ij}^{\ell} C_{i'j'}^{\overline{\ell}}
				= \sum_{\ell \in I} (-1)^{\ell_2} C_{ii'}^{\ell} C_{jj'}^{\overline{\ell}}.
				\tag{HB6}
			\end{align}
			
			\noindent
			\textit{Description}: This encodes the crossing equation
			\begin{align*}
				\int_{\Gamma \backslash G} (\psi_i\psi_j) (\psi_{i'} \psi_{j'})
				= \int_{\Gamma \backslash G} (\psi_i \psi_{i'}) (\psi_j \psi_{j'})
			\end{align*}
			(c.f. \eqref{eqn:crossing}). The sign $(-1)^{\ell_2}$ is due to \eqref{eqn:psi_i_bar_vs_bar_psi_i}.
			We could remove the sign by using a different definition of $(\g,K)$-adapted basis, but this would introduce signs in the other equations.
		\end{itemize}
	\end{defn}
	
	\begin{rem}[Application to Question~\ref{qtn:spectral_multisets_smooth}] \label{rem:application_to_Q1}
		Let $g$ and $\{\lambda_r\}_{r \geq 0}$ be as in Question~\ref{qtn:spectral_multisets_smooth}.
		Let $\{k_r\}_{r \geq 1}$ be the holomorphic spectrum of a genus $g$ surface, described concretely in Example~\ref{exmp:hol_spectrum}.
		Let $I$ be the indexing set \eqref{eqn:I_hol_def}.
		Then by Theorem~\ref{thm:main} and the fact that the holomorphic spectrum determines the topological type of a compact hyperbolic 2-orbifold, $\{\lambda_r\}_{r \geq 0}$ is the Laplace spectrum of a compact hyperbolic surface of genus $g$ if and only if there exist $\{C_{ij}^{\ell}\}_{i,j,\ell \in I}$ solving the hyperbolic bootstrap equations.
		Since each of these equations is either linear or quadratic in the $C_{ij}^{\ell}$, the condition that $\{\lambda_r\}_{r \geq 0}$ is a Laplace spectrum of the desired type is equivalent to feasibility of a quadratically constrained quadratic program (QCQP).
		Physicists working on the conformal bootstrap have well-developed software to determine feasibility of such QCQPs, principally the semidefinite programming solver SDPB \cite{DSD_SDPB,DSD--Landry,Rychkov--Su}.
		The SDPB was used to prove the upper bound on $\lambda_1$ for genus $2$ surfaces mentioned in Subsection~\ref{subsec:intro:motivation}.
		For details on how the hyperbolic bootstrap equations can be used to bound $\lambda_1$, see Subsection~\ref{subsec:applications:lambda_1}.
	\end{rem}
	
	\begin{rem}[Analogy with the conformal bootstrap]
		The conformal bootstrap equations are all quadratic in the OPE coefficients. In contrast, the first five hyperbolic bootstrap equations are linear in the $C_{ij}^{\ell}$.  Using these linear equations to eliminate variables, the hyperbolic bootstrap equations can be reduced to a system of purely quadratic equations. This system has fewer variables but more complicated coefficients (c.f. \cite[Theorem~4.12]{Bonifacio--Mazac--Pal} for hyperbolic 3-manifolds). It is this system of purely quadratic equations which is most closely analogous to the conformal bootstrap equations in CFT. This raises the question of how to interpret the linear equations \eqref{eqn:SE1}--\eqref{eqn:SE5} from the point of view of CFT.
		The first four are trivial and need no interpretation, but for each $i,j,\ell \in I$ we have a nontrivial linear relation \eqref{eqn:SE5}.
		The CFT analogs of these relations are the Casimir recursion relations \cite{Hogervorst--Rychkov,Costa--Hansen--Penedones--Trevisani} used to compute conformal blocks (see \cite[Section~III.F.3]{Poland--Rychkov--Vichi} for a brief exposition).
	\end{rem}

	A consequence of Theorem~\ref{thm:main}, which can be stated without mention of hyperbolic surfaces, is
	
	\begin{cor} \label{cor:unique_bootstrap_soln}
		There exists a unique candidate spectrum which solves the hyperbolic bootstrap equations and satisfies $\lambda_1 \geq 30$.
	\end{cor}
	
	\begin{proof}
		This follows from Theorem~\ref{thm:main} together with the fact \cite{Kravchuk--Mazac--Pal} that there is a unique hyperbolic 2-orbifold with $\lambda_1 \geq 30$, namely the $(2,3,7)$-triangle orbifold.
	\end{proof}

	By Corollary~\ref{cor:unique_bootstrap_soln}, the space of solutions to the hyperbolic bootstrap equations contains an isolated point.
	It is known that the space of solutions to the 2d conformal bootstrap equations (with Virasoro symmetry) contains isolated points, e.g., the 2d Ising model.
	This is expected but unknown in 3d: again, the 3d Ising model should be isolated --- the ``island" in \cite{Chang_et_al_25} should shrink to a point.
	This statement can be formulated rigorously without defining what a 3d CFT is.
	To prove it would likely require a result analogous to Theorem~\ref{thm:main} for the 3d conformal bootstrap equations.
	It seems difficult even to make a precise conjecture for such a result.

	\subsection{Equivalence of main theorems} \label{subsec:equivalence}
	
	To state the equivalence, we need a few more definitions.
	
	\begin{defn}[Laplace and holomorphic spectra of $G$-representations] \label{defn:Laplace_hol_spectra_rep}
		Let $\H_{\R}$ be a real unitary representation of $G$ with discrete spectrum.
		Then by the classification of real irreducible unitary representations of $G$ (Subsection~\ref{subsec:irrep_classification}), there exist unique increasing sequences
		\begin{align*}
			\lambda_0 \leq \lambda_1 \leq \cdots
			\qquad \text{and} \qquad
			k_1 \leq k_2 \leq \cdots
		\end{align*}
		such that $\H_{\R}$ is the direct sum of principal series with Casimir eigenvalues $\lambda_r$ and discrete series with lowest positive weights $k_r$.
		Here $\{\lambda_r\}_{r \geq 0}$ and $\{k_r\}_{r \geq 1}$ could be finite or infinite.
		The \textit{Laplace spectrum} of $\H_{\R}$ is $\{\lambda_r\}_{r \geq 0}$, and the \textit{holomorphic spectrum} of $\H_{\R}$ is $\{k_r\}_{r \geq 1}$.
	\end{defn}
	
	For $\Gamma$ a cocompact lattice in $G$, the Laplace and holomorphic spectra of $L^2(\Gamma \backslash G, \R)$ coincide with the Laplace and holomorphic spectra of $\Gamma \backslash \mathbf{H}$ (see, e.g., \cite[Chapter~2]{Bump}).
	
	\begin{defn}[Bi-infinite spectrum]
		Let $\H_{\R}$ be a real unitary representation of $G$ with discrete spectrum.
		We say that $\H_{\R}$ has \textit{bi-infinite spectrum} when both the Laplace and holomorphic spectra of $\H_{\R}$ are infinite.
	\end{defn}
	
	\begin{rem}[Real subspace of a multiplicative representation] \label{rem:real_subspace}
		Let $\H$ be a multiplicative representation.
		Then the following facts are established in Section~\ref{sec:complex_conj}.
		Complex conjugation on $\H^{\fin}$ extends by continuity to a $G$-equivariant antilinear involution on $\H$ commuting with both multiplication and the inner product.
		Denote $\H_{\R} = \{v \in \H : \overline{v} = v\}$.
		Then $\H_{\R}$ is a real subrepresentation with complexification $\H$.
		The unit $\mathbf{1}$ is in $\H_{\R}$.
	\end{rem}
	
	We define the Laplace and holomorphic spectra of a multiplicative representation to be the Laplace and holomorphic spectra of $\H_{\R}$, where $\H_{\R}$ is as in Remark~\ref{rem:real_subspace}.
	Then given an orthonormal basis $\{\psi_i\}_{i \in I}$ of a multiplicative representation $\H$, with $\psi_i \in \H^{\fin}$ for all $i \in I$, we define the multiplication table and multiplicative spectrum of $\H$ in exactly the same way as in Definitions~\ref{defn:mult_table} and \ref{defn:mult_spectrum}, setting $C_{ij}^{\ell} = \langle \psi_i \psi_j, \psi_{\ell} \rangle_{\H}$.
	
	The notion of $(\g,K)$-adapted basis defined in Section~\ref{sec:(g,K)-adapted} makes sense for multiplicative representations with bi-infinite spectrum;
	we require bi-infinite spectrum only for notational convenience.
	Proposition~\ref{prop:basis_exist_unique} generalizes to this setting (the generalization is Proposition~\ref{prop:(g,K)-adapted_exist_unique}).
	
	The reason we choose $\H^{\fin}$ to be the domain of multiplication in Definition~\ref{def:mult_rep} is
	
	\begin{prop} \label{prop:basis_spans_H^fin}
		Let $\H$ be a unitary representation of $G$ as in Section~\ref{sec:(g,K)-adapted} (e.g., a multiplicative representation with bi-infinite spectrum). Then any $(\g,K)$-adapted basis $\{\psi_i\}_{i \in I}$ of $\H$ is a basis for $\H^{\fin}$ as an abstract vector space.
	\end{prop}
	
	Proposition~\ref{prop:basis_spans_H^fin} will be obvious once we have defined $(\g,K)$-adapted bases. We omit the proof.
	
	
	\begin{prop}[Equivalence of main theorems] \label{prop:mult_rep_vs_HB}
		A candidate spectrum $\mathcal{S}$ is the multiplicative spectrum of a multiplicative representation with bi-infinite spectrum, with respect to a $(\g,K)$-adapted basis, if and only if $\mathcal{S}$ solves the hyperbolic bootstrap equations.
	\end{prop}
	
	We prove Proposition~\ref{prop:mult_rep_vs_HB} in Sections~\ref{sec:spectral_eqns_pf} and \ref{sec:classification_implies_main_thm}.
	Theorem~\ref{thm:main} is a direct consequence of Proposition~\ref{prop:mult_rep_vs_HB} and Theorem~\ref{thm:eq_Gelfand_duality}.
	In the converse direction, Proposition~\ref{prop:mult_rep_vs_HB} and Theorem~\ref{thm:main} almost imply Theorem~\ref{thm:eq_Gelfand_duality}, though not completely. Fortunately, we do not need this direction.
	
	\subsection{Technical remarks}
	
	\begin{rem}[Discrete spectrum is necessary] \label{rem:discrete_vs_pure_pt}
		Suppose that instead of requiring $\H$ to have discrete spectrum in Definition~\ref{def:mult_rep}, we only require it to have countable spectrum in the sense of Definition~\ref{defn:pure_point_spectrum}. Then the spectrum of $\H$ may have infinite multiplicities, so the definition of $\H^{\fin}$ in Definition~\ref{defn:H^fin} may depend on the choice of irreducible decomposition of $\H$, but this is not a serious problem: just fix a choice. Then Definition~\ref{def:mult_rep} still makes sense, and one can ask if Theorem~\ref{thm:eq_Gelfand_duality} remains true. The following counterexample shows that it does not.
		Let $\Gamma_0 \supsetneq \Gamma_1 \supsetneq \Gamma_2 \supsetneq \cdots$ be a decreasing chain of cocompact lattices in $G$. As usual, equip $\Gamma_n \backslash G$ with the Haar probability measure, and view $L^2(\Gamma_n \backslash G)$ as a multiplicative representation. Then the natural maps
		\begin{align*}
			L^2(\Gamma_0 \backslash G)
			\to L^2(\Gamma_1 \backslash G)
			\to L^2(\Gamma_2 \backslash G)
			\to \cdots
		\end{align*}
		are $G$-equivariant isometric inclusions compatible with multiplication. Let $\H$ be the inductive limit of these Hilbert spaces (i.e., the completion of the increasing union). Then $\H$ inherits a $G$-action, making $\H$ into a unitary representation with countable spectrum. Choose an irreducible decomposition of $\H$ such that each irreducible appears in $L^2(\Gamma_n \backslash G)$ for some $n$. Then $\H^{\fin}$ (defined with respect to the chosen decomposition) is the increasing union of the $L^2(\Gamma_n \backslash G)^{\fin}$. The multiplication maps $L^2(\Gamma_n \backslash G)^{\fin} \times L^2(\Gamma_n \backslash G)^{\fin} \to L^2(\Gamma_n \backslash G)^{\infty}$ assemble to give a multiplication $\H^{\fin} \times \H^{\fin} \to \H^{\infty}$. This makes $\H$ into a nontrivial multiplicative representation, except with countable rather than discrete spectrum. Since $\H$ has neither discrete nor continuous spectrum, $\H$ is not of the form $L^2(\Gamma \backslash G)$ for any lattice $\Gamma$.
	\end{rem}
	
	\begin{rem}[Use of discrete spectrum] \label{rem:using_lambda_r_to_infty}
		Because the spectrum of a multiplicative representation $\H$ is discrete, band-limited subspaces of $\H^K$ are finite-dimensional. Certain pairs of norms in the proof of Theorem~\ref{thm:eq_Gelfand_duality} will consequently be defined on a finite-dimensional space, and will thus be equivalent. Equivalence of two norms means that each is bounded by a finite positive constant $C$ times the other.
		Since discrete spectrum is only a qualitative assumption, we will have no quantitative control \textit{a priori} on $C$.
		However, we will be able to prove certain ``self-improving" estimates for $C$ which when combined with $C<\infty$ imply that $C$ is quantitatively bounded, giving us the control we need.
		To illustrate this mechanism, suppose given some $C \in [0,\infty]$, and suppose one can show that $C \leq 1+\frac{1}{2}C$. Then $C < \infty$ automatically implies $C \leq 2$.
		Remark~\ref{rem:discrete_spectrum_first_use} and the paragraph above it explain where such constants $C$ come up in the proof of Theorem~\ref{thm:eq_Gelfand_duality}.
	\end{rem}
	
	\begin{rem}[Analytic hypotheses in Theorem~\ref{thm:main}] \label{rem:analytic_hypotheses}
		We emphasize that Definition~\ref{defn:potential_spectrum} imposes no quantitative growth rate on $\lambda_r$ and $k_r$.
		By Remark~\ref{rem:discrete_vs_pure_pt}, at least one of the qualitative ``discrete spectrum" conditions \eqref{eqn:lambda_r_to_infty} and \eqref{eqn:k_r_to_infty} is necessary for Theorem~\ref{thm:main} to hold.
		We do not use \eqref{eqn:k_r_to_infty} in any way except to simplify notation, but Remark~\ref{rem:using_lambda_r_to_infty} demonstrates that we do use \eqref{eqn:lambda_r_to_infty} in a critical way.
		As for the $C_{ij}^{\ell}$, Theorem~\ref{thm:main} only makes the mild analytic assumption that both sides of \eqref{eqn:SE6} converge absolutely. In fact, Remark~\ref{rem:convergence_assumptions} explains that even this mild assumption can be removed, so no analytic assumptions on $C_{ij}^{\ell}$ are necessary at all.
		To sum up, as advertised in the abstract, the only analytic hypothesis we really use in the proof of Theorem~\ref{thm:main} is that $\{\lambda_r\}_{r \geq 0}$ is discrete. This is philosophically important because linear/semidefinite programming as in \cite{Bonifacio_22_2,Kravchuk--Mazac--Pal} does not make use of any analytic information about $\lambda_r,k_r,C_{ij}^{\ell}$.
	\end{rem}
	
	A portion of the proof of Theorem~\ref{thm:eq_Gelfand_duality} can be simplified if $\H$ obeys a \emph{polynomial Weyl law} (Definition~\ref{def:poly_Weyl_law} below).
	We present this simplification in Section~\ref{sec:bulk_tail_poly}.
	Remark~\ref{rem:why_not_assume_Weyl?} justifies going the extra mile in Section~\ref{sec:bulk_tail_general} to prove Theorem~\ref{thm:eq_Gelfand_duality} for $\H$ with arbitrary discrete spectrum.
	
	\begin{defn}[Polynomial Weyl law] \label{def:poly_Weyl_law}
		Let $\{\lambda_r\}_{r \geq 0}$ be nonnegative real numbers.
		We say that $\{\lambda_r\}_{r \geq 0}$ obeys a \textit{polynomial Weyl law} if for $\Lambda \geq 1$,
		\begin{align} \label{eqn:poly_Weyl_law}
			\#\{\lambda \in [0,\Lambda] : \lambda = \lambda_r \text{ for some } r\}
			\lesssim \Lambda^{O(1)}.
		\end{align}
		We say a multiplicative representation obeys a polynomial Weyl law if its Laplace spectrum does.
	\end{defn}
	
	Note that $\LHS\eqref{eqn:poly_Weyl_law}$ does not count with multiplicity.
	Thus the arguments in Section~\ref{sec:bulk_tail_poly} work even if the Laplace spectrum of $\H$ has rapidly growing multiplicities, as long as \eqref{eqn:poly_Weyl_law} holds.
	This indicates that these arguments can be extended to the general case.
	This is done in Section~\ref{sec:bulk_tail_general}.
	
	\begin{rem}[Why not assume a polynomial Weyl law?] \label{rem:why_not_assume_Weyl?}
		We give two reasons in addition to the fact that linear/semidefinite programming does not take into account the growth rate of the $\lambda_r$.
		\begin{itemize} \itemsep = 0.5em
			\item The spectrum of a CFT in $d$ dimensions does not satisfy a polynomial Weyl law. Instead, the size of the spectrum in the interval $[0,\Lambda]$ is expected to be equal to $\exp(c\Lambda^{\frac{d-1}{d}})$ for some $c = c(\Lambda)$ bounded above and below as $\Lambda \to \infty$ \cite{Shaghoulian}.
			This is the generalized Cardy formula, first derived by Cardy for $d=2$ using the modular bootstrap \cite{Cardy}.
			Thus from the point of view of the analogy with CFT, it is preferable not to assume a polynomial Weyl law in Theorem~\ref{thm:eq_Gelfand_duality}.
			
			\item The extra work needed to extend Theorem~\ref{thm:eq_Gelfand_duality} from the polynomial Weyl law case to the general case is, in the language of the conformal bootstrap, to generalize certain estimates for single correlators to mixed correlators. This manifests in the proof as extending bounds for Casimir eigenvectors to approximate Casimir eigenvectors. Subsection~\ref{subsec:outline:bulk_tail_general} explains why we need to work with approximate eigenvectors when we allow arbitrary discrete spectrum.
			This is notable from a technical standpoint because it illustrates a common theme in the bootstrap, namely that results proved by linear programming with single correlators can often be improved by semidefinite programming with mixed correlators (as in, e.g., \cite{Kos_Poland_DSD,Chester_et_al,Chester_et_al_2,Erramilli_et_al,Chang_et_al_25}).
			Unusually, the improvement in this paper is qualitative (removing the polynomial Weyl law assumption) rather than numerical.
		\end{itemize}
	\end{rem}
	
	\begin{rem}[Smoothness of products] \label{rem:smoothness_of_products}
		Let $\H$ be a multiplicative representation.
		By definition, the product of two elements of $\H^{\fin}$ is in $\H^{\infty}$.
		This is a nontrivial analytic condition.
		Given a candidate spectrum $\mathcal{S}$ solving the hyperbolic bootstrap equations, to prove Proposition~\ref{prop:mult_rep_vs_HB} we have to build a multiplicative representation $\H$ with multiplicative spectrum $\mathcal{S}$.
		This construction is entirely formal, except for checking that multiplication lands in $\H^{\infty}$.
		In order to check this, we will need a certain polynomial decay estimate on the $C_{ij}^{\ell}$.
		This estimate, Proposition~\ref{prop:poly_decay_C_ij^l}, will imply the desired smoothness similarly to how polynomial decay of Fourier coefficients implies smoothness for functions on $\R/\Z$.
		The proof of Proposition~\ref{prop:poly_decay_C_ij^l} can be thought of as linear programming --- indeed, it uses positivity in a crucial way.
	\end{rem}
	
	\begin{rem}[Convergence assumptions in \eqref{eqn:SE6}] \label{rem:convergence_assumptions}
		Linear programming as in \cite{Bonifacio_22_2,Kravchuk--Mazac--Pal} or as in the proof of Proposition~\ref{prop:poly_decay_C_ij^l} never uses absolute convergence in \eqref{eqn:SE6}.
		Instead, \eqref{eqn:SE6} is only used when both sides are \textit{already known to converge} in $\overline{\R}+i\overline{\R}$ in the sense of Definition~\ref{defn:weak_convergence} below.
		Then \eqref{eqn:SE6} is interpreted as an equality in $\overline{\R}+i\overline{\R}$.
		With this in mind, let us say that a candidate spectrum $\mathcal{S}$ is a \textit{weak solution} to the hyperbolic bootstrap equations if it obeys \eqref{eqn:SE1}--\eqref{eqn:SE5}, and obeys \eqref{eqn:SE6} for the subset of $i,j,i',j' \in I$ for which both sides of \eqref{eqn:SE6} converge in $\overline{\R}+i\overline{\R}$.
		Then Proposition~\ref{prop:poly_decay_C_ij^l} implies that a weak solution is automatically a solution in the original sense, i.e., \eqref{eqn:SE6} in fact holds with absolute convergence for all $i,j,i',j' \in I$.
		This implication is explained in Remark~\ref{rem:weak_to_strong}.
		Thus the proof of Theorem~\ref{thm:main} shows that every weak solution comes from a compact hyperbolic 2-orbifold.
		This justifies the claim in Remark~\ref{rem:analytic_hypotheses} that we need not make any analytic assumptions on the $C_{ij}^{\ell}$ in Theorem~\ref{thm:main}.
	\end{rem}
	
	\begin{defn}[Convergence in $\overline{\R} + i\overline{\R}$] \label{defn:weak_convergence}
		Denote $\overline{\R} = [-\infty,+\infty]$.
		We say that an unordered sum $\sum_k x_k$ of real numbers \textit{converges in $\overline{\R}$} if at least one of
		\begin{align*}
			\sum_{k \, : \, x_k > 0} x_k
			\qquad \text{and} \qquad
			\sum_{k \, : \, x_k < 0} x_k
		\end{align*}
		is finite (note that absolute convergence is equivalent to both being finite).
		More generally, we say that an unordered sum $\sum_k z_k$ of complex numbers \textit{converges in $\overline{\R} + i\overline{\R}$} if both $\sum_k \re z_k$ and $\sum_k \im z_k$ converge in $\overline{\R}$.
		Then $\sum_k z_k$ has an unambiguous value in $\overline{\R} + i\overline{\R}$.
	\end{defn}
	
	Convergence in $\overline{\R}+i\overline{\R}$ is a useful notion because it can often be verified algebraically (in particular it can often be verified by a computer), as in the following prototypical example.
	
	\begin{exmp}[Automatic convergence in $\overline{\R}+i\overline{\R}$]
		Let $z_k \in \C$. Assume that for all but finitely many $k$, one has a formula $z_k = a_k|c_k|^2$ with $a_k \geq 0$ and $c_k \in \C$.
		Then $\sum_k z_k$ converges in $\overline{\R}+i\overline{\R}$.
	\end{exmp}
	
	\subsection{Organization}
	
	Section~\ref{sec:sl_2} recalls standard facts about the group $G = \PSL_2(\R)$ and establishes the notation we will need to do calculations with representations of $G$.
	Section~\ref{sec:(g,K)-adapted} defines $(\g,K)$-adapted bases and proves their main properties.
	Sections~\ref{sec:spectral_eqns_pf} and \ref{sec:classification_implies_main_thm} prove the ``only if" and ``if" parts of Proposition~\ref{prop:mult_rep_vs_HB}, respectively.
	Section~\ref{sec:applications} explains how \cite{Kravchuk--Mazac--Pal} and \cite{Adve_et_al} can be seen as applications of the hyperbolic bootstrap equations.
	Section~\ref{sec:C*} recalls the definition and classification of commutative unital C*-algebras.
	Section~\ref{sec:uniqueness} proves the two uniqueness results, Propositions~\ref{prop:mult_rep_uniqueness} and \ref{prop:mult_spectra_uniqueness}.
	Sections~\ref{sec:complex_conj}--\ref{sec:bulk_tail_general} prove Theorem~\ref{thm:eq_Gelfand_duality}.
	Let $\H$ be a multiplicative representation.
	Then Section~\ref{sec:complex_conj} proves the expected properties of complex conjugation on $\H^{\fin}$. Section~\ref{sec:L^infty_L^4} introduces $L^{\infty}$ and $L^4$ norms on $\H^{\fin}$ which agree with the usual $L^{\infty}$ and $L^4$ norms when $\H = L^2(\Gamma \backslash G)$. With the preliminaries in Sections~\ref{sec:complex_conj} and \ref{sec:L^infty_L^4} out of the way, Section~\ref{sec:outline_existence} outlines the proof of Theorem~\ref{thm:eq_Gelfand_duality}. The logical structure of the proof is:
	\begin{align} \label{eqn:pf_structure}
		\emptyset
		\xrightarrow{\text{Sections~\ref{sec:bulk_tail_poly}, \ref{sec:bulk_tail_general}}}
		\left\{
		\begin{matrix}
			\text{Thm~\ref{thm:L^4_quasi-Sobolev}} \\
			\text{Thm~\ref{thm:exp_decay_form}} \\
			\text{Thm~\ref{thm:exp_decay_quasimode}}
		\end{matrix}
		\right\}
		\xrightarrow{\text{Section~\ref{sec:reduce_to_bulk_tail}}}
		\left\{
		\begin{matrix}
			\text{Thm~\ref{thm:high_deriv_L^infty_bd}} \\
			\text{Thm~\ref{thm:L^infty_quasi-Sobolev}}
		\end{matrix}
		\right\}
		\xrightarrow{\text{Section~\ref{sec:reduce_to_L^infty}}}
		\{\text{Thm~\ref{thm:weak_structure_thm}}\}
		\xrightarrow{\text{Section~\ref{sec:embedding_homog}}}
		\{\text{Thm~\ref{thm:eq_Gelfand_duality}}\}.
	\end{align}
	Arrows denote implications.
	One can see that the proof is presented backwards, as a series of reductions.
	As a warmup for Section~\ref{sec:bulk_tail_general}, Section~\ref{sec:bulk_tail_poly} proves Theorems~\ref{thm:L^4_quasi-Sobolev}, \ref{thm:exp_decay_form}, and \ref{thm:exp_decay_quasimode} in the case where $\H$ obeys a polynomial Weyl law.

	
	
	\subsection{Suggestions for reading} \label{subsec:reading}
	
	Sections~\ref{sec:sl_2} and \ref{sec:C*} review standard material.
	Sections~\ref{sec:(g,K)-adapted}--\ref{sec:applications}, Section~\ref{sec:uniqueness}, and Sections~\ref{sec:complex_conj}--\ref{sec:bulk_tail_general} can be read in any order.
	
	The quickest route to the core ideas in the paper is to read the proof of Theorem~\ref{thm:eq_Gelfand_duality} assuming a polynomial Weyl law.
	This consists of Sections~\ref{sec:complex_conj}--\ref{sec:bulk_tail_poly}.
	The proof begins in earnest in Section~\ref{sec:embedding_homog}, following the elementary Sections~\ref{sec:complex_conj} and \ref{sec:L^infty_L^4} and the outline in Section~\ref{sec:outline_existence}.
	
	\subsection{Notation for estimates}
	
	Given a nonnegative number $X$, the expression $O(X)$ denotes a quantity which in absolute value is $\leq CX$ for some absolute constant $C \geq 0$. If we want to allow $C$ to depend on parameters $a_1,\dots,a_n$, then we indicate this in the notation by writing $O_{a_1,\dots,a_n}(X)$. We refer to $C$ as the \emph{$O$-constant}. The notation $X \lesssim Y$ or $Y \gtrsim X$ means that $Y$ is nonnegative and $X \leq O(Y)$. Almost always, $X$ will be nonnegative as well. We write $X \sim Y$ when $X \lesssim Y \lesssim X$.
	For nonnegative $X,Y$, the notation $X \ll Y$ or $Y \gg X$ means that $X \leq cY$ for some constant $c>0$ which is sufficiently small for the relevant context, but not too small in the sense that $c \gtrsim 1$. Two sample uses of the symbol $\gg$ are:
	\begin{enumerate} \itemsep = 0.5em
		\item[1.] ``Let $X \gg 1$. Then $X^{1/100} \geq \log X$."
		
		\item[2.] ``Let $X,Y \geq 0$. Then $X \lesssim Y$ or $X \gg Y$ (or both). In the former case, we do [something], and in the latter case, we do [something else]."
	\end{enumerate}
	We refer to the $O$-constant implicit in the symbols $\lesssim,\gtrsim,\sim,\ll,\gg$ as the \emph{implicit constant}. If the implicit constant depends on parameters, then we again indicate this with a subscript.
	From Section~\ref{sec:complex_conj} onward, all $O$-constants, implicit constants, and absolute constants may depend on the multiplicative representation $\H$ under consideration.
	
	
	\subsection{Further notation and conventions} \label{subsec:intro:conventions}
	
	\begin{itemize} \itemsep = 0.5em
		\item Unless stated otherwise, $G = \PSL_2(\R)$ and $K = \PSO_2(\R)$. The Lie algebra of $G$ is $\g_{\R}$, and $\g$ is the complexification of $\g_{\R}$.
		
		\item For $\Gamma$ a lattice in $G$, we always give $\Gamma \backslash G$ the Haar probability measure.
		
		\item By default, representations and function spaces are over $\C$. When we work over $\R$, we say so explicitly.
		
		\item Inner products are linear in the first argument and antilinear in the second.
		
		\item Isomorphisms of unitary representations preserve the inner product.
		
		\item For $\E$ a subset of a representation of $G$ and $H \subseteq G$ a subgroup, $\E^H$ denotes the subset of vectors in $\E$ fixed by $H$.
		
		\item We use boldface $\mathbf{1}$ to denote the unit in a multiplicative representation, and double-struck $\1$ to denote an indicator function.
		
		\item Whenever we discuss the hyperbolic bootstrap equations, we use the notation from Definition~\ref{defn:spectral_eqns}.
		
		\item We abbreviate ``left hand side" and ``right hand side" by ``LHS" and ``RHS" respectively.
		
		\item For a Hilbert space $\H$, we write $\B(\H)$ for the C*-algebra of bounded linear operators on $\H$.
		
		\item From Section~\ref{sec:complex_conj} onward, $\H$ denotes a multiplicative representation, and the notation in Subsection~\ref{subsec:hyp_bootstrap_eqns} ceases to be relevant.
		
		
		\item A \emph{$G$-space} is a topological space $X$ with a $G$-action, such that the action map $G \times X \to X$ is continuous.
		
		\item All measures are Borel measures on topological spaces.
		
		\item Our convention for order of operations is that differentiation comes before multiplication. So for example if $\H$ is a multiplicative representation, $\alpha,\beta \in \H^{\fin}$, and $X,Y \in \g$, then $X\alpha Y\beta$ means $(X\alpha)(Y\beta)$ rather than $X(\alpha Y\beta)$. 
		
		\item We interpret $0^0$ as being equal to $1$. Thus we may write, e.g., $n^n \geq 1$ for all $n \in \Z_{\geq 0}$.
		
		\item We define
		\begin{align} \label{eqn:log_+_def}
			\log_+(x) = \log\max\{x,2\},
		\end{align}
		so that $\log_+$ is defined on all of $\R$, uniformly bounded below, and equal to $\log$ for $x \geq 2$.
	\end{itemize}
	
	\subsection*{Acknowledgments}
	
	This project benefited from conversations with a large number of people, among whom I highlight three: Peter Sarnak (my advisor), Dalimil Maz\'a\v{c}, and Sridip Pal.
	I was introduced to the conformal bootstrap by Peter, who pointed me to \cite{Kravchuk--Mazac--Pal}.
	Soon after, I met Dalimil at a conference, where in his talk he raised the possibility of a converse theorem for \cite{Kravchuk--Mazac--Pal}.
	We began talking then, and much of what I know about the conformal bootstrap comes from conversations with him.
	Both Peter and Dalimil have given feedback on several drafts of this paper.
	I thank them both for their guidance, encouragement, and support.
	I also thank Sridip and Dalimil for inviting me to visit Caltech and IPhT Saclay, respectively, and for their hospitality while I was there.
	Before meeting Sridip, I had proved Theorem~\ref{thm:eq_Gelfand_duality} assuming a polynomial Weyl law.
	It was at Sridip's insistence that I worked out how to remove this assumption.
	More recently, my collaboration with Dalimil, Sridip, and the other authors of \cite{Adve_et_al} has been a source of joy.
	
	Figure~\ref{fig:torus} was created with ChatGPT.
	
	I am supported by the US National Science Foundation Graduate Research Fellowship Program under Grant No. DGE-2039656.

	\section{Preliminaries on $\PSL_2(\R)$} \label{sec:sl_2}
	
	In this section, we state all the facts we need about the representation theory of $G = \PSL_2(\R)$.
	Standard facts are often stated without proof.
	For details, see \cite{Knapp,Bump}.
	This section also sets up much of the notation used later on.

	
	\subsection{The Lie algebra}
	\label{subsec:lie_alg}
	
	We have $\g_{\R} = \Sl_2(\R)$ and $\g = \Sl_2(\C)$.
	Let
	\begin{align} \label{eqn:H,E_def}
		H = \frac{1}{2i}
		\begin{pmatrix}
			0 & 1 \\
			-1 & 0
		\end{pmatrix}
		\qquad \text{and} \qquad
		E = \frac{1}{2}
		\begin{pmatrix}
			1 & i \\
			i & -1
		\end{pmatrix},
	\end{align}
	viewed as elements of $\g$.
	Then
	\begin{align*}
		\g = \C H \oplus \C E \oplus \C\overline{E}
		\qquad \text{and} \qquad
		\Lie(K) = i\R H.
	\end{align*}
	The basis elements $H,E,\overline{E}$ satisfy the commutation relations
	\begin{align} \label{eqn:comm_rlns}
		[H,E] = E,
		\qquad
		[H,\overline{E}] = -\overline{E},
		\qquad
		[E,\overline{E}] = 2H.
	\end{align}
	
	Any unitary representation of $G$ has a weight space decomposition given by diagonalizing the action of $K = \PSO_2(\R)$. Weights are canonically elements of $\widehat{K}$, i.e., unitary characters of $K$.
	We identify $\widehat{K} \simeq \Z$ by defining the $n$th character of $K$ to be $e^{i\theta H} \mapsto e^{in\theta}$ for $\theta \in \R/2\pi\Z$; this is well-defined because $\theta \mapsto e^{i\theta H}$ is an isomorphism from $\R/2\pi\Z$ to $K$.
	With this convention, a vector $v$ in a representation of $G$ has weight $n$ when $e^{i\theta H}v = e^{in\theta}v$, or equivalently when $Hv = nv$. If $v$ has weight $n$ in the complexification of a real representation, then its complex conjugate $\overline{v}$ has weight $-n$. This is because $\overline{H} = -H$. It follows from \eqref{eqn:comm_rlns} that if $v$ has weight $n$, then
	$Ev$ has weight $n+1$ and $\overline{E}v$ has weight $n-1$. So $E,\overline{E}$ are raising and lowering operators, respectively. If $\H$ is a multiplicative representation and $\alpha,\beta \in \H^{\fin}$ have weights $n,m$, respectively, then
	\begin{align} \label{eqn:H_product_rule}
		H(\alpha\beta)
		= (H\alpha)\beta + \alpha(H\beta)
		= (n+m)\alpha\beta
	\end{align}
	by the product rule, so $\alpha\beta$ has weight $n+m$.
	In particular, if $\alpha \in \H^{\fin}$ is a weight vector, then by Remark~\ref{rem:real_subspace} (more specifically Proposition~\ref{prop:bar_equivariant}) and \eqref{eqn:H_product_rule}, $|\alpha|^2$ has weight $0$ and hence is in $\H^K$.
	
	\begin{rem}[Normalization of weights for $\PSL_2(\R)$ vs $\SL_2(\R)$] \label{rem:wts_PSL_2_vs_SL_2}
		Suppose $f$ is a holomorphic modular form of weight $2k$ for some lattice $\Gamma$ in $G$, as in \eqref{eqn:modular_form_def}.
		Then $f \, dz^k$ is $\Gamma$-invariant on $\mathbf{H}$, so it defines a function on the unit tangent bundle $T^1(\Gamma \backslash \mathbf{H}) \simeq \Gamma \backslash G$.
		In our convention, this function has weight $k$ rather than $2k$ as an element of the representation $L^2(\Gamma \backslash G)$.
		This convention is the natural one for representations of $\PSL_2(\R)$ as opposed to $\SL_2(\R)$.
		Related to this, $E$ and $\overline{E}$ raise and lower weights by $1$ rather than $2$.
	\end{rem}
	
	Normalize the Casimir operator $\Delta$ by
	\begin{align} \label{eqn:Casimir_def}
		\Delta
		= -H^2 - \frac{1}{2}(E\overline{E} + \overline{E}E).
	\end{align}
	This is in the center $\mathfrak{Z}(\g)$ of the universal enveloping algebra $\mathfrak{U}(\g)$. It is in fact in $\mathfrak{Z}(\g_{\R})$ because $\overline{\Delta} = \Delta$. Rearranging \eqref{eqn:Casimir_def} using the commutation relations \eqref{eqn:comm_rlns},
	\begin{align} \label{eqn:EEbar}
		E\overline{E} = -(\Delta + H^2 - H)
		\qquad \text{and} \qquad
		\overline{E}E = -(\Delta + H^2 + H).
	\end{align}
	In particular, in a unitary representation $\H$ of $G$, the Casimir acts on $K$-invariant vectors by
	\begin{align} \label{eqn:Delta|_H^K}
		\Delta|_{\mathcal{H}^K}
		= -E\overline{E}|_{\mathcal{H}^K}
		= -\overline{E}E|_{\H^K}.
	\end{align}
	If $v,w$ are smooth vectors in a unitary representation $\H$ of $G$, and if $X \in \g$, then
	\begin{align} \label{eqn:Lie_alg_unitarity}
		\langle Xv,w \rangle_{\H}
		= -\langle v, \overline{X}w \rangle_{\H}.
	\end{align}
	In other words, the formal adjoint of $X \in \g$ is $X^{\star} = -\overline{X}$. Using the notation $X^{\star}$, we can rewrite \eqref{eqn:Casimir_def} as
	\begin{align*}
		\Delta = -H^{\star}H + \frac{1}{2}(\overline{E}^{\star}\overline{E} + E^{\star}E).
	\end{align*}
	From this and \eqref{eqn:Delta|_H^K}, it follows that $\Delta$ is formally self-adjoint on $\H$ and formally positive semidefinite on $\H^K$, i.e.,
	\begin{align*}
		\langle \Delta v, w \rangle
		= \langle v, \Delta w \rangle
		\qquad \text{and} \qquad
		\langle \Delta u,u \rangle
		\geq 0
	\end{align*}
	for all $v,w \in \H^{\infty}$ and $u \in \H^K \cap \H^{\infty}$. Let
	\begin{align} \label{eqn:P_def}
		\P
		= \Delta + 2H^2
		= H^{\star}H + \frac{1}{2}(\overline{E}^{\star} \overline{E} + E^{\star}E).
	\end{align}
	Then $\P$ is formally positive semidefinite on all of $\H$, i.e., $\langle \P v,v \rangle \geq 0$ for all $v \in \H^{\infty}$. When $\P$ is viewed as a left-invariant differential operator on functions on $G$, it is elliptic.
	Indeed, this differential operator is the Laplacian for a certain left-invariant Riemannian metric on $G$.
	
	From now on, we will often use the above equations without comment.
	
	\begin{defn}[Automorphic vector] \label{def:aut_vector}
		An \emph{automorphic vector} in a unitary representation of $G$ is a smooth vector which is both a weight vector and a Casimir eigenvector, i.e., an eigenvector for both $H$ and $\Delta$.
	\end{defn}
	
	By definition, for $\Gamma$ a cocompact lattice in $G$, automorphic vectors in $L^2(\Gamma \backslash G)$ are the same as automorphic forms on $\Gamma \backslash G$.
	
	We will talk about lowest/highest weight vectors in not-necessarily-irreducible representations of $G$.
	When we do so, we mean the following.
	
	\begin{defn}[Lowest/highest weight vector] \label{def:lo_hi}
		A \emph{lowest} (resp. \emph{highest}) \emph{weight vector} in a unitary representation of $G$ is a nonzero smooth weight vector $f$ with nonzero weight such that $\overline{E}f = 0$ (resp. $Ef = 0$).
	\end{defn}
	
	The nonzero weight condition excludes $G$-invariant vectors from being lowest or highest weight vectors.
	A vector $f$ in the complexification of a real unitary representation of $G$ is a lowest weight vector if and only if $\overline{f}$ is a highest weight vector.
	
	\begin{prop} \label{prop:lowest_wt_wt}
		Let $f$ be a lowest (resp. highest) weight vector of weight $k$ (resp. $-k$) in a unitary representation $\H$ of $G$. Then $k \in \Z_{\geq 1}$, and $f$ is an automorphic vector in $\H$ with Casimir eigenvalue $-k(k-1)$.
	\end{prop}
	
	We give the proof of this well-known fact in the case of lowest weight vectors, because it is a short computation which serves as a warmup for more involved computations in Sections~\ref{sec:reduce_to_bulk_tail}, \ref{sec:bulk_tail_poly}, and \ref{sec:bulk_tail_general}.
	The case of highest weight vectors is almost identical.
	
	\begin{proof}[Proof for lowest weight vectors]
		Let $f \in \H$ be a lowest weight vector of weight $k$. Then using \eqref{eqn:Lie_alg_unitarity} and \eqref{eqn:comm_rlns}, write
		\begin{align*}
			0
			\leq \|Ef\|_{\H}^2
			= -\langle \overline{E}Ef,f \rangle_{\H}
			= 2\langle Hf,f \rangle_{\H} - \langle E\overline{E}f,f \rangle_{\H} 
			= 2k\|f\|_{\H}^2,
		\end{align*}
		where the term $\langle E\overline{E}f,f \rangle_{\H}$ vanishes because $f$ is lowest weight. This forces $k \geq 0$. By definition, $k \neq 0$, so we must have $k \geq 1$. Next, by \eqref{eqn:EEbar},
		\begin{align*}
			0 = -E\overline{E}f
			= (\Delta+H^2-H)f
			= \Delta f + k(k-1)f,
		\end{align*}
		where the first equality is because $f$ is lowest weight.
		Rearranging, $\Delta f = -k(k-1)f$.
	\end{proof}
	
	
	\subsection{Irreducible unitary representations}
	\label{subsec:irrep_classification}
	
	Let $\pi_{\R}$ be a real irreducible unitary representation of $G$, and let $\pi$ be its complexification. By a suitable version of Schur's lemma, the Casimir acts on $\pi^{\infty}$ by multiplication by a real scalar $\lambda$. Depending on whether or not $\pi_{\R}$ (or equivalently $\pi$) has a nonzero $K$-invariant vector, $\pi_{\R}$ is classified into one of two types: \textit{principal series} or \textit{discrete series} (we do not distinguish between principal and complementary series).
	
	\begin{itemize} \itemsep = 0.5em
		\item (Principal series) $\pi_{\R}$ is \emph{principal series} if $\pi_{\R}^K \neq 0$. In this case, as a complex unitary representation, $\pi$ is also irreducible. Since the Casimir is formally positive semidefinite on $\pi^K$, we have $\lambda \geq 0$. The Casimir eigenvalue $\lambda$ determines $\pi_{\R}$ up to isomorphism. If $\lambda = 0$, then $\pi_{\R}$ is the trivial representation. If $\lambda > 0$, let $\varphi \in \pi^K$ be nonzero. Then $\varphi$ is a smooth vector, and
		\begin{align*}
			\dots,
			\overline{E}^2\varphi,
			\overline{E}\varphi,
			\varphi,
			E\varphi,
			E^2\varphi,
			\dots
		\end{align*}
		are nonzero automorphic vectors forming an orthogonal basis of $\pi$.
		There are no lowest or highest weight vectors in $\pi$ in the sense of Definition~\ref{def:lo_hi} (this is trivial when $\lambda=0$ and follows from Proposition~\ref{prop:lowest_wt_wt} when $\lambda > 0$).
		
		\item (Discrete series) $\pi_{\R}$ is \emph{discrete series} if $\pi_{\R}^K = 0$. In this case, as a complex unitary representation, $\pi$ splits into two irreducible pieces as $\pi = \pi^+ \oplus \pi^-$, where $\pi^+$ (resp. $\pi^-$) is the closed subspace of $\pi$ generated by positive (resp. negative) weight vectors. One has $\pi^- = \overline{\pi^+}$. Let $k \in \Z_{\geq 1}$ be the lowest positive weight in $\pi$. Then $\pi_{\R}$ is determined up to isomorphism by $k$.
		Let $f \in \pi^+$ be nonzero and have weight $k$. Then $f$ is a lowest weight vector, so by Proposition~\ref{prop:lowest_wt_wt}, we have $\lambda = -k(k-1)$. The vectors
		\begin{align*}
			f,Ef,E^2f,\dots
		\end{align*}
		are nonzero automorphic vectors forming an orthogonal basis of $\pi^+$.
		It follows from this, complex conjugation symmetry, and Proposition~\ref{prop:lowest_wt_wt} that all lowest (resp. highest) weight vectors in $\pi$ in the sense of Definition~\ref{def:lo_hi} are scalar multiples of $f$ (resp. $\overline{f}$).
	\end{itemize}
	
	All complex irreducible unitary representations of $G$ arise as above from real representations, i.e., either as the complexification of a real principal series representation, or as $\pi^+$ or $\pi^-$ for some real discrete series representation $\pi_{\R}$.
	
	The following two propositions are immediate consequences of the above classification.
	
	\begin{prop} \label{prop:Maass_span_H^fin}
		Let $\H$ be a unitary representation of $G$ with discrete spectrum. Then $\H^{\fin}$ is the algebraic linear span of the automorphic vectors in $\H$, and as an abstract vector space, $\H^{\fin}$ has countable dimension.
	\end{prop}
	
	%
	
	\begin{prop} \label{prop:Maass_is_raised}
		Let $\alpha$ be an automorphic vector of nonnegative weight in a unitary representation of $G$ with discrete spectrum. Then there is either a $K$-invariant automorphic vector $\varphi$ or a lowest weight vector $f$ such that $\alpha = E^m\varphi$ or $\alpha = E^mf$ for some $m \in \Z_{\geq 0}$.
	\end{prop}

	\subsection{Spectral decomposition and functional calculus for the Casimir} \label{subsec:sl_2:func_calc}
	
	Given a unitary representation $\H$ of $G$ with discrete spectrum, let $\H_{\Delta=\lambda}$ denote the $\lambda$-eigenspace of the Casimir (i.e., the closure in $\H$ of the $\lambda$-eigenspace of the Casimir acting on $\H^{\infty}$). Because $\H$ has discrete spectrum, each $\H_{\Delta=\lambda}$ is a direct sum of finitely many irreducible subrepresentations, and
	\begin{align*}
		\H = \bigoplus_{\lambda \in \R} \H_{\Delta=\lambda},
	\end{align*}
	where the set $\{\lambda : \H_{\Delta=\lambda} \neq 0\}$ has finite intersection with each bounded subset of $\R$. Let $\1_{\Delta=\lambda}$ denote the orthogonal projection onto $\H_{\Delta=\lambda}$. More generally, given a logical statement $S$ depending on a real number, i.e., a function $S \colon \R \to \{\true,\false\}$, let
	\begin{align*}
		\H_{S(\Delta)}
		= \bigoplus_{\lambda \colon S(\lambda) = \true} \H_{\Delta=\lambda}
	\end{align*}
	and $\1_{S(\Delta)}$ the orthogonal projection onto $\H_{S(\Delta)}$. For example, for any $R \geq 0$, one could consider $\H_{|\Delta|\leq R}$, the direct sum of Casimir eigenspaces with eigenvalues bounded by $R$.
	For $f \colon \R \to \C$ and $v \in \H$, define
	\begin{align} \label{eqn:norm_f(Delta)v_def}
		\|f(\Delta)v\|_{\H}
		= \Big(\sum_{\lambda \in \R} |f(\lambda)|^2 \|\1_{\Delta = \lambda}v\|_{\H}^2\Big)^{\frac{1}{2}}.
	\end{align}
	If this is finite, let
	\begin{align*}
		f(\Delta)v
		= \sum_{\lambda \in \R} f(\lambda) \1_{\Delta=\lambda} v.
	\end{align*}
	By finiteness of \eqref{eqn:norm_f(Delta)v_def}, this sum converges unconditionally to a well-defined element of $\H$ with norm given by \eqref{eqn:norm_f(Delta)v_def}. The operator $f(\Delta)$ obtained in this way is densely defined with domain containing $\H^{\fin}$, and is bounded if $f$ is bounded.
	For example, if $f(\mu) = \1_{S(\mu)}$ with $S$ as above, then $f(\Delta) = \1_{S(\Delta)}$.
	Observe that $\|f(\Delta)v\|_{\H}$ is monotonically increasing with $|f|$. This allows us to use $O$-notation in the functional calculus. An example we will often use is $\|\exp(O(\log_+^2\Delta))v\|_{\H}$, which denotes a quantity equal to $\|f(\Delta)v\|_{\H}$ for some $f \colon \R \to \C$ with $|f(x)| = \exp(O(\log_+^2 x))$.
	Recall from \eqref{eqn:log_+_def} that $\log_+(x) = \log\max\{x,2\}$.
	The $\exp \log^2$ type asymptotic arises naturally in the proof of Theorem~\ref{thm:L^4_quasi-Sobolev}, and appears for this reason in many subsequent estimates.
	
	Since $\H_{\Delta=\lambda}^K = 0$ for $\lambda < 0$, all of the above makes sense with $\H$ replaced by $\H^K$ and $\R$ replaced by $[0,\infty)$. For example, for any $R \geq 0$, we can consider $\H_{\Delta \leq R}^K$. Since $\H$ has discrete spectrum, and since each irreducible unitary representation of $G$ has a finite-dimensional (in fact $\leq 1$-dimensional) subspace of $K$-invariants, $\H_{\Delta \leq R}^K$ is finite-dimensional and $\H_{\Delta \leq R}^K \subseteq \H^{\fin}$. These two facts will be used repeatedly.
	
	\subsection{Smooth and super-smooth vectors} As above, let $\H$ be a unitary representation of $G$ with discrete spectrum.
	
	\begin{prop} \label{prop:smooth_iff_spectral_decay}
		Let $v \in \H$ be a weight vector. Then $v \in \H^{\infty}$ if and only if $\|\Delta^Nv\|_{\H} < \infty$ for all $N \geq 0$.
	\end{prop}
	
	Here $\|\Delta^Nv\|_{\H}$ is well-defined by \eqref{eqn:norm_f(Delta)v_def} even if $v$ is not smooth.
	Proposition~\ref{prop:smooth_iff_spectral_decay} is a standard fact with multiple proofs.
	One way to see it is as a consequence of elliptic regularity:
	
	\begin{proof}[Sketch of proof of Proposition~\ref{prop:smooth_iff_spectral_decay}]
		If $v \in \H^{\infty}$, then it is trivial that $\|\Delta^Nv\|_{\H} < \infty$. Conversely, if $\|\Delta^Nv\|_{\H} < \infty$ for all $N \geq 0$, then since $v$ is a weight vector, $\|\P^Nv\|_{\H} < \infty$ for all $N \in \Z_{\geq 0}$, where $\P$ is defined by \eqref{eqn:P_def}. As mentioned after \eqref{eqn:P_def}, $\P$ is elliptic, so by elliptic regularity, $g \mapsto gv$ must be a smooth $\H$-valued function on $G$. This means that $v \in \H^{\infty}$.
	\end{proof}
	
	By Proposition~\ref{prop:smooth_iff_spectral_decay}, the following property is stronger than smoothness.
	
	\begin{defn}[Super-smooth vector] 
		A vector $v \in \H$ is \emph{super-smooth} if it is $K$-invariant and
		\begin{align} \label{eqn:super-smooth}
			\|\exp(A\log_+^2\Delta) v\|_{\H} < \infty
		\end{align}
		for every $A \geq 0$.
	\end{defn}
	
	The technical lemma below
	is used to prove Corollary~\ref{cor:qual_Sob_emb}.
	The defining property \eqref{eqn:super-smooth} of super-smoothness is chosen to be the weakest one which makes the proof of Corollary~\ref{cor:qual_Sob_emb} valid.
	
	\begin{lem} \label{lem:positive_super-smooth}
		There exists a nonnegative continuous function $\varphi \in L^1(G)$, not identically zero, such that for all $v \in \H$, the convolution
		\begin{align*}
			\varphi \ast v := \int_G \varphi(g) gv \, dg
		\end{align*}
		is super-smooth.
	\end{lem}
	
	There are more sophisticated results, such as \cite{Nelson_59}, about smoothing vectors in representations of general Lie groups.
	Indeed, one can prove Lemma~\ref{lem:positive_super-smooth} using the heat kernel technique of \cite[Section~8]{Nelson_59}.
	Since Lemma~\ref{lem:positive_super-smooth} is for $G = \PSL_2(\R)$ specifically, one can alternatively use the Selberg--Harish-Chandra transform \cite[Proposition~11.2.9]{Deitmar--Echterhoff} to give a short proof of Lemma~\ref{lem:positive_super-smooth} with an explicit construction of $\varphi$: take $g$ in \cite[Proposition~11.2.9]{Deitmar--Echterhoff} to be a Gaussian, and then take $\varphi$ to be the function $f$ in \cite[Proposition~11.2.9]{Deitmar--Echterhoff}.
	Both the heat kernel argument and the Selberg--Harish-Chandra argument give the stronger conclusion that $\varphi \ast v$ is an analytic vector for $\Delta$ (the notion of an analytic vector for an unbounded operator is defined at the beginning of \cite{Nelson_59}).
	Since we only need $\varphi \ast v$ to be super-smooth, a more elementary proof is available, using nothing specific about $\PSL_2(\R)$.
	We sketch this elementary proof below.
	
	\begin{proof}[Sketch of proof of Lemma~\ref{lem:positive_super-smooth}]
		Fix any $\sigma > 1$.
		By the Denjoy--Carleman theorem \cite[Theorem~4.1.15]{Krantz--Parks} and \cite[Remark~4.1.18]{Krantz--Parks}, there exists $f \in C_c^{\infty}(\R)$, not identically zero, satisfying the derivative bounds $\|f^{(n)}\|_{L^{\infty}} \lesssim O(n)^{\sigma n}$ for all $n \in \Z_{\geq 0}$ (i.e., $f$ is of Gevrey class $\sigma$).
		Using $f$ as a building block, it is easy to construct a function $\varphi \in C_c^{\infty}(G)$, not identically zero, such that for any compact subset $\K \subseteq \g$ and any $X_1,\dots,X_n \in \K$,
		\begin{align} \label{eqn:phi_deriv_bds}
			\|X_1 \cdots X_n \varphi\|_{L^{\infty}(G)}
			\lesssim_{\K} O_{\K}(n)^{\sigma n}.
		\end{align}
		By replacing $\varphi$ with its square, we may assume $\varphi$ is nonnegative.
		Then by averaging $\varphi$ by $K$, we may assume $\varphi$ is left $K$-invariant.
		An immediate consequence of \eqref{eqn:phi_deriv_bds} is that
		\begin{align} \label{eqn:phi_Casimir_deriv_bds}
			\|\Delta^n \varphi\|_{L^{\infty}(G)}
			\lesssim O(n)^{2\sigma n}
		\end{align}
		for all $n \in \Z_{\geq 0}$.
		Now let $v \in \H$.
		Since $\varphi$ is left $K$-invariant, $\varphi \ast v$ is $K$-invariant. By \eqref{eqn:phi_Casimir_deriv_bds},
		\begin{align*} 
			\|\Delta^n(\varphi \ast v)\|_{\H}
			\lesssim O(n)^{2\sigma n} \|v\|_{\H}
		\end{align*}
		for all $n \in \Z_{\geq 0}$.
		From this, a simple dyadic decomposition argument shows that for every $A \geq 0$,
		\begin{align} \label{eqn:phi_star_v_super-smooth}
			\|\exp(A\log_+^2\Delta)(\varphi \ast v)\|_{\H}
			< \infty
		\end{align}
		with room to spare (for a similar dyadic decomposition argument, see the proof of Theorem~\ref{thm:exp_decay_form} at the end of Subsection~\ref{subsec:bulk:first_tail_bd}).
		Since $\varphi \ast v$ is $K$-invariant and satisfies \eqref{eqn:phi_star_v_super-smooth}, it is super-smooth.
	\end{proof}
	
	This proof shows that in Lemma~\ref{lem:positive_super-smooth}, one can take $\varphi \in C_c^{\infty}(G)$. This would not be possible if we required $\varphi \ast v$ to be an analytic vector for $\Delta$.
	These subtleties are unimportant for us.
	
	\section{$(\g,K)$-Adapted bases} \label{sec:(g,K)-adapted}
	
	We first give a definition-by-construction of $(\g,K)$-adapted bases, and then we prove their most important properties (Propositions~\ref{prop:(g,K)-adapted_exist_unique} and \ref{prop:(g,K)-adapted_key_properties}).
	Once we prove these properties, we can forget about the construction entirely.
	
	Fix once and for all a choice of representative of each isomorphism class of real irreducible unitary representations of $G$ (normally we do not distinguish representations in the same isomorphism class, but in this section it is helpful for clarity).
	In addition, for each representative $\pi_{\R}$ with complexification $\pi$, fix a unit vector $\varphi_{\pi} \in \pi_{\R}^K$ if $\pi_{\R}$ is principal series, or fix a unit vector $f_{\pi}$ of lowest positive weight in $\pi$ if $\pi_{\R}$ is discrete series.
	These unit vectors are analogous to primary states in a CFT.
	Let $\{\xi_{\pi,n}\}_n$ be the following orthonormal basis of $\pi$.
	\begin{itemize} \itemsep = 0.5em
		\item If $\pi_{\R}$ is trivial, define $\xi_{\pi,n}$ for $n=0$ by
		\begin{align*}
			\xi_{\pi,n} = \varphi_{\pi}.
		\end{align*}
		
		\item If $\pi_{\R}$ is principal series and nontrivial, define $\xi_{\pi,n}$ for $n \in \Z$ by
		\begin{align} \label{eqn:xi_def_princ}
			\xi_{\pi,n} = \frac{E^n \varphi_{\pi}}{\|E^n \varphi_{\pi}\|}
			\text{ for } n \geq 0
			\qquad \text{and} \qquad
			\xi_{\pi,n} = (-1)^n \overline{\xi_{\pi,-n}} \text{ for } n \leq 0
		\end{align}
		(note that these are consistent for $n=0$ because $\varphi_{\pi}$ is real).
		
		\item If $\pi_{\R}$ is discrete series, let $k=k_{\pi}$ be the weight of $f_{\pi}$. Define $\xi_{\pi,n}$ for integers $|n| \geq k$ by
		\begin{align} \label{eqn:xi_def_disc}
			\xi_{\pi,n} = \frac{E^{n-k} f_{\pi}}{\|E^{n-k} f_{\pi}\|} \text{ for } n \geq k
			\qquad \text{and} \qquad
			\xi_{\pi,n} = (-1)^n \overline{\xi_{\pi,-n}} \text{ for } n \leq -k.
		\end{align}
	\end{itemize}
	In all cases, $\xi_{\pi,n}$ is an automorphic vector of weight $n$.
	By the structure theory in Subsection~\ref{subsec:irrep_classification}, $\{\xi_{\pi,n}\}_n$ is indeed a well-defined orthonormal basis of $\pi$.
	This can be thought of as a basis of primary and descendant states.
	
	Now let $\H_{\R}$ be a real unitary representation of $G$ with discrete bi-infinite spectrum. Let $\H$ be its complexification.
	Assume the trivial representation appears exactly once in $\H_{\R}$, i.e., $\dim \H_{\R}^G = 1$.
	Then the Laplace and holomorphic spectra of $\H_{\R}$ satisfy \eqref{eqn:lambda_r_to_infty} and \eqref{eqn:k_r_to_infty}, respectively.
	Assume further that $\H_{\R}^G$ is generated by a distinguished unit vector $\mathbf{1}$.
	The main example to keep in mind is $\H$ a multiplicative representation with bi-infinite spectrum, $\H_{\R}$ as in Remark~\ref{rem:real_subspace}, and $\mathbf{1}$ the unit in $\H$.
	%
	%
	
	Let $\{\lambda_r\}_{r \geq 0}$ and $\{k_r\}_{r \geq 1}$ be the Laplace and holomorphic spectra of $\H_{\R}$.
	Then
	\begin{align*}
		\H_{\R} \simeq \bigoplus_{r \in \Z} \pi_{r,\R},
	\end{align*}
	where $\pi_{r,\R}$ is the chosen representative of the isomorphism class of
	\begin{align*}
		\begin{cases}
			\text{real principal series representations with Casimir eigenvalue } \lambda_r &\text{for } r \geq 0, \\
			\text{real discrete series representations with lowest positive weight } k_{-r} &\text{for } r < 0.
		\end{cases}
	\end{align*}
	Let $\pi_r$ be the complexification of $\pi_{r,\R}$.
	Define $I \subseteq \Z^2$ in terms of the $k_r$ by \eqref{eqn:I_hol_def}.
	The motivation for \eqref{eqn:I_hol_def} is that for $(r,n) \in \Z^2$, one has $(r,n) \in I$ if and only if $n$ is in the range in which $\xi_{\pi_r,n}$ is defined.
	Note that $I$ depends only on the isomorphism class of $\H_{\R}$.
	For $i = (i_1,i_2) \in I$, let
	\begin{align*}
		\eta_i = \xi_{\pi_{i_1},i_2}.
	\end{align*}
	Then $\{\eta_i\}_{i \in I}$ is an orthonormal basis of $\bigoplus_{r \in \Z} \pi_r$ consisting of automorphic vectors.
	Each $\eta_i$ has Casimir eigenvalue $\lambda_{i_1}$ and weight $i_2$ (here we use the notation \eqref{eqn:lambda_{-r}_def} when $i_1 < 0$).
	
	\begin{defn}[$(\g,K)$-Adapted basis] \label{defn:(g,K)-adapted}
		A \textit{$(\g,K)$-adapted basis} of $\H$ is an orthonormal basis $\{\psi_i\}_{i \in I}$ of $\H$, such that $\psi_{(0,0)} = \mathbf{1}$, and such that there exists an isomorphism $\H_{\R} \simeq \bigoplus_{r \in \Z} \pi_{r,\R}$ which (after complexifying) takes $\psi_i$ to $\eta_i$ for all $i \in I$.
		
		When we talk about $(\g,K)$-adapted bases of multiplicative representations, we understand that $\H_{\R}$ is as in Remark~\ref{rem:real_subspace} and $\mathbf{1}$ is the unit.
	\end{defn}
	
	
	The following proposition is obvious from Definition~\ref{defn:(g,K)-adapted}.
	
	\begin{prop}[Existence and uniqueness of $(\g,K)$-adapted bases] \label{prop:(g,K)-adapted_exist_unique}
		There exists a $(\g,K)$-adapted basis $\{\psi_i\}_{i \in I}$ of $\H$.
		Furthermore, if $\H_{\R}',\H',\mathbf{1}'$ are also as above, with $\H_{\R} \simeq \H_{\R}'$ as real unitary representations, and if $\{\psi_i'\}_{i \in I}$ is a $(\g,K)$-adapted basis of $\H'$, then there exists a unique isomorphism $\H_{\R} \simeq \H_{\R}'$ which (after complexifying) takes $\psi_i$ to $\psi_i'$.
	\end{prop}
	
	The properties of $(\g,K)$-adapted bases which we will use to make calculations are the following. We use the notation
	from Definition~\ref{defn:spectral_eqns}.
	
	\begin{prop}[Key properties of $(\g,K)$-adapted bases] \label{prop:(g,K)-adapted_key_properties}
		Let $\{\psi_i\}_{i \in I}$ be a $(\g,K)$-adapted basis of $\H$.
		Then $\psi_{(0,0)} = \mathbf{1}$.
		For $i \in \Z^2 \setminus I$, denote $\psi_i = 0$.
		Then the following hold for all $i \in I$.
		\begin{itemize} \itemsep = 0.5em
			\item $\psi_i$ has Casimir eigenvalue $\lambda_{i_1}$ and weight $i_2$.
			
			\item The complex conjugate of $\psi_i$ is
			\begin{align} \label{eqn:psi_i_bar_vs_bar_psi_i}
				\overline{\psi_i} = (-1)^{i_2} \psi_{\overline{i}}.
			\end{align}
			
			\item The raising operator $E$ acts on $\psi_i$ by
			\begin{align} \label{eqn:E_psi_i_formula}
				E\psi_i
				= \sqrt{\lambda_{i_1} + i_2(i_2+1)} \, \psi_{i^+}.
			\end{align}
			
			\item The lowering operator $\overline{E}$ acts on $\psi_i$ by
			\begin{align} \label{eqn:Ebar_psi_i_formula}
				\overline{E}\psi_i
				= -\sqrt{\lambda_{i_1} + i_2(i_2-1)} \, \psi_{i^-}.
			\end{align}
			
		\end{itemize}
	\end{prop}
	
	\begin{proof}
		By definition, $\psi_{(0,0)} = \mathbf{1}$.
		The first two bullet points are also clear from the definition, but the third and fourth require arguments.
		To prove \eqref{eqn:E_psi_i_formula}, we first show that
		\begin{align} \label{eqn:Epsi_i_some_c_i}
			E\psi_i = c_i \psi_{i^+}
			\qquad \text{for some} \qquad
			c_i \geq 0;
		\end{align}
		we will compute $c_i$ afterward by equating norms.
		We check \eqref{eqn:Epsi_i_some_c_i} in different ways in different cases.
		
		\noindent
		\emph{Case 1}: $i = (0,0)$.
		
		Then $\psi_i = \mathbf{1}$ is $G$-invariant, so $E\psi_i = 0$, and we can take $c_i = 0$ in \eqref{eqn:Epsi_i_some_c_i}.
		
		\noindent
		\emph{Case 2}: $i_1 < 0$ and $i_2 = -k_{-i_1}$.
		
		Then $\psi_i$ is a highest weight vector, so $E\psi_i = 0$, and we can again take $c_i = 0$.
		
		\noindent
		\emph{Case 3}: $i \neq (0,0)$ and $i_2 \geq 0$.
		
		Then \eqref{eqn:Epsi_i_some_c_i} is immediate from \eqref{eqn:xi_def_princ} when $i_1>0$ and from \eqref{eqn:xi_def_disc} when $i_1<0$. Moreover, it is implicit in \eqref{eqn:xi_def_princ} and \eqref{eqn:xi_def_disc} that $E\psi_i \neq 0$, so $c_i > 0$.
		
		\noindent
		\emph{Case 4}: ($i_1 > 0$ and $i_2 < 0$) or ($i_1 < 0$ and $i_2 < -k_{-i_1}$).
		
		The conditions on $i$ imply that $i^+ \in I$ and $(i^+)_2 = i_2+1 \leq 0$.
		Since $I$ is preserved by $i \mapsto \overline{i}$, we have $\overline{i^+} \in I$. We know \eqref{eqn:Epsi_i_some_c_i} with $\overline{i^+}$ in place of $i$ by Case 3, and moreover the proof in Case 3 tells us that $c_{\overline{i^+}} > 0$.
		Combining this with \eqref{eqn:psi_i_bar_vs_bar_psi_i}, and noting that $\overline{i} = (\overline{i^+})^+$, we can write
		\begin{align*}
			\overline{E\psi_i}
			= (-1)^{i_2} \overline{E} \psi_{\overline{i}}
			= (-1)^{i_2} \overline{E} \psi_{(\overline{i^+})^+}
			= \frac{(-1)^{i_2}}{c_{\overline{i^+}}} \overline{E} E\psi_{\overline{i^+}}.
		\end{align*}
		By the first bullet point in Proposition~\ref{prop:(g,K)-adapted_key_properties}, $\psi_{\overline{i^+}}$ has Casimir eigenvalue $\lambda_{i_1}$ and weight $-i_2-1$.
		Thus by \eqref{eqn:EEbar},
		\begin{align*}
			\overline{E\psi_i}
			= \frac{(-1)^{i_2+1}}{c_{\overline{i^+}}} (\Delta+H^2+H) \psi_{\overline{i^+}}
			= \frac{(-1)^{i_2+1}}{c_{\overline{i^+}}} (\lambda_{i_1} + i_2(i_2+1)) \psi_{\overline{i^+}}.
		\end{align*}
		Taking complex conjugates using \eqref{eqn:psi_i_bar_vs_bar_psi_i}, we get
		\begin{align*}
			E\psi_i = \frac{\lambda_{i_1} + i_2(i_2+1)}{c_{\overline{i^+}}} \psi_{i^+}.
		\end{align*}
		The numerator is nonnegative by \eqref{eqn:sqrt_nonnegativity}, so \eqref{eqn:Epsi_i_some_c_i} holds.
		
		Cases 1--4 cover all possibilities, so \eqref{eqn:Epsi_i_some_c_i} holds for all $i \in I$.
		To prove \eqref{eqn:E_psi_i_formula}, it remains to show that we can take
		\begin{align} \label{eqn:c_i_formula}
			c_i = \sqrt{\lambda_{i_1} + i_2(i_2+1)}.
		\end{align}
		If $i^+ \not\in I$, then $\psi_{i^+} = 0$, so $E\psi_i = 0$ by \eqref{eqn:Epsi_i_some_c_i}, and hence \eqref{eqn:Epsi_i_some_c_i} holds with any choice of $c_i$. Therefore we may assume $i^+ \in I$. Then $\|\psi_{i^+}\|_{\H} = 1$, so taking norms in \eqref{eqn:Epsi_i_some_c_i} gives $\|E\psi_i\|_{\H} = c_i$.
		We compute
		\begin{align} \label{eqn:argument_is_norm_squared}
			c_i^2
			= \|E\psi_i\|_{\H}^2
			= -\langle \overline{E}E\psi_i, \psi_i \rangle_{\H}
			= \langle (\Delta+H^2+H) \psi_i, \psi_i \rangle_{\H}
			= (\lambda_{i_1} + i_2(i_2+1)) \|\psi_i\|_{\H}^2.
		\end{align}
		Here we have used \eqref{eqn:Lie_alg_unitarity}, \eqref{eqn:EEbar}, and the first bullet point in Proposition~\ref{prop:(g,K)-adapted_key_properties}.
		Since $i \in I$, we have $\|\psi_i\|_{\H} = 1$.
		Thus we obtain \eqref{eqn:c_i_formula} upon taking square roots in \eqref{eqn:argument_is_norm_squared}.
		Hence \eqref{eqn:E_psi_i_formula} is established.
		
		%
		%
		
		We finally come to the last bullet point.
		By \eqref{eqn:psi_i_bar_vs_bar_psi_i} and \eqref{eqn:E_psi_i_formula},
		\begin{align*}
			\overline{\overline{E}\psi_i}
			= E\overline{\psi_i}
			= (-1)^{i_2} E\psi_{\overline{i}}
			= (-1)^{i_2} \sqrt{\lambda_{i_1} + (-i_2)(-i_2+1)} \, \psi_{\overline{i}^+}.
		\end{align*}
		Noting that $(-i_2)(-i_2+1) = i_2(i_2-1)$ and $\overline{i}^+ = \overline{i^-}$, taking complex conjugates using \eqref{eqn:psi_i_bar_vs_bar_psi_i} yields \eqref{eqn:Ebar_psi_i_formula}.
		This completes the proof of Proposition~\ref{prop:(g,K)-adapted_key_properties}.
	\end{proof}
	
	\section{Proof of the hyperbolic bootstrap equations} \label{sec:spectral_eqns_pf}
	
	In this section we prove the ``only if" direction in Proposition~\ref{prop:mult_rep_vs_HB}, namely that if
	\begin{align*}
		\mathcal{S} = (\{\lambda_r\}_{r \geq 0}, \{k_r\}_{r \geq 1}, \{C_{ij}^{\ell}\}_{i,j,\ell \in I})
	\end{align*}
	is the multiplicative spectrum of a multiplicative representation $\H$ with bi-infinite spectrum, with respect to a $(\g,K)$-adapted basis $\{\psi_i\}_{i \in I}$, then $\mathcal{S}$ solves the hyperbolic bootstrap equations.
	
	\begin{proof}[Proof of \eqref{eqn:SE1}]
		Since $\psi_i \psi_j = \psi_j \psi_i$, \eqref{eqn:SE1} is trivial.
	\end{proof}
	
	\begin{proof}[Proof of \eqref{eqn:SE2}]
		By the first bullet point in Proposition~\ref{prop:(g,K)-adapted_key_properties}, $\psi_i \psi_j$ has weight $i_2+j_2$, and $\psi_{\ell}$ has weight $\ell_2$. Thus $\psi_i \psi_j$ and $\psi_{\ell}$ are orthogonal unless $i_2+j_2 = \ell_2$.
		In other words, \eqref{eqn:SE2} holds.
	\end{proof}
	
	\begin{proof}[Proof of \eqref{eqn:SE3}]
		If $i_2+j_2 \neq \ell_2$, then by \eqref{eqn:SE2}, both sides of \eqref{eqn:SE3} vanish.
		Therefore, we may assume $i_2+j_2 = \ell_2$.
		By Remark~\ref{rem:real_subspace} (more specifically Propositions~\ref{prop:bar_extends} and \ref{prop:conj_commutes_w_mult}) and \eqref{eqn:psi_i_bar_vs_bar_psi_i},
		\begin{align*}
			\overline{C_{ij}^{\ell}}
			= \langle \overline{\psi_i} \, \overline{\psi_j}, \overline{\psi_{\ell}} \rangle_{\H}
			= (-1)^{i_2+j_2+\ell_2} \langle \psi_{\overline{i}} \psi_{\overline{j}}, \psi_{\overline{\ell}} \rangle_{\H}
			= (-1)^{i_2+j_2+\ell_2} C_{\overline{i} \, \overline{j}}^{\overline{\ell}}.
		\end{align*}
		Since $i_2+j_2 = \ell_2$, the sign is $+1$.
		Thus \eqref{eqn:SE3} holds.
	\end{proof}
	
	\begin{proof}[Proof of \eqref{eqn:SE4}]
		Since $i = (0,0)$, we have $\psi_i = \mathbf{1}$. Therefore
		\begin{align*}
			C_{ij}^{\ell}
			= \langle \psi_i \psi_j, \psi_{\ell} \rangle_{\H}
			= \langle \psi_j, \psi_{\ell} \rangle_{\H}
			= \1_{j=\ell},
		\end{align*}
		with the last equality by orthonormality. This is \eqref{eqn:SE4}.
	\end{proof}
	
	\begin{proof}[Proof of \eqref{eqn:SE5}]
		By \eqref{eqn:Ebar_psi_i_formula},
		\begin{align*}
			\LHS\eqref{eqn:SE5}
			= -\langle \psi_i\psi_j, \overline{E} \psi_{\ell} \rangle_{\H}.
		\end{align*}
		Moving $\overline{E}$ to the other side via \eqref{eqn:Lie_alg_unitarity}, applying the product rule, and then evaluating the result using \eqref{eqn:E_psi_i_formula},
		\begin{align*}
			&\LHS\eqref{eqn:SE5}
			= \langle E(\psi_i\psi_j), \psi_{\ell} \rangle_{\H}
			= \langle (E\psi_i)\psi_j, \psi_{\ell} \rangle_{\H}
			+ \langle \psi_i(E\psi_j), \psi_{\ell} \rangle_{\H}
			= \RHS\eqref{eqn:SE5}.
			\qedhere
		\end{align*}
		%
		%
	\end{proof}
	
	\begin{proof}[Proof of \eqref{eqn:SE6}]
		By \eqref{eqn:psi_i_bar_vs_bar_psi_i}, spectral expansion, \eqref{eqn:SE3}, and \eqref{eqn:SE2},
		\begin{align} \label{eqn:crossing_to_C's}
			\langle \psi_i \psi_j, \overline{\psi_{i'}} \, \overline{\psi_{j'}} \rangle_{\H}
			&= (-1)^{i_2'+j_2'} \langle \psi_i \psi_j, \psi_{\overline{i'}} \psi_{\overline{j'}} \rangle_{\H}
			\notag
			\\&= (-1)^{i_2'+j_2'} \sum_{\ell \in I} C_{ij}^{\ell} \overline{C_{\overline{i'} \, \overline{j'}}^{\ell}}
			= (-1)^{i_2'+j_2'} \sum_{\ell \in I} C_{ij}^{\ell} C_{i'j'}^{\overline{\ell}}
			= \sum_{\ell \in I} (-1)^{\ell_2} C_{ij}^{\ell} C_{i'j'}^{\overline{\ell}},
		\end{align}
		with absolute convergence on the right hand side by Bessel's inequality.
		By crossing symmetry in Definition~\ref{def:mult_rep}, the left hand side is unchanged by permutation of the indices, so the right hand side must be as well.
		In particular, switching $i'$ and $j$ preserves the right hand side.
		This is precisely the content of \eqref{eqn:SE6}.
	\end{proof}
	
	
	\section{Building a multiplicative representation from a solution}
	\label{sec:classification_implies_main_thm}
	
	In this section we prove the ``if" direction in Proposition~\ref{prop:mult_rep_vs_HB}, namely that if
	\begin{align*}
		\mathcal{S}
		= (\{\lambda_r\}_{r \geq 0}, \{k_r\}_{r \geq 1}, \{C_{ij}^{\ell}\}_{i,j,\ell \in I})
	\end{align*}
	is a candidate spectrum solving the hyperbolic bootstrap equations,
	then $\mathcal{S}$ is the multiplicative spectrum of a multiplicative representation with bi-infinite spectrum, with respect to a $(\g,K)$-adapted basis.
	
	\begin{rem}[Convergence assumptions on $\mathcal{S}$] \label{rem:convergence_for_poly_decay}
		We only assume $\mathcal{S}$ is a weak solution as defined in Remark~\ref{rem:convergence_assumptions}.
		This means that instead of assuming \eqref{eqn:SE6} with absolute convergence for all $i,j,i',j' \in I$, we only assume the following: if $i,j,i',j' \in I$ are such that both sides of \eqref{eqn:SE6} converge in $\overline{\R}+i\overline{\R}$ in the sense of Definition~\ref{defn:weak_convergence}, then \eqref{eqn:SE6} holds as an equality in $\overline{\R}+i\overline{\R}$.
		%
	\end{rem}
	
	\subsection{Polynomial decay of $C_{ij}^{\ell}$ as $\ell \to \infty$}
	\label{subsec:poly_decay}
	
	The main result of this subsection is
	
	\begin{prop}
		\label{prop:poly_decay_C_ij^l}
		Let $i,j \in I$ and $N \geq 0$. Then
		\begin{align} \label{eqn:weak_to_strong}
			\sum_{\ell \in I} |\lambda_{\ell_1}|^N |C_{ij}^{\ell}|^2
			< \infty.
		\end{align}
	\end{prop}
	
	The proof of this proposition is by elementary analysis of the hyperbolic bootstrap equations.
	
	\begin{rem}[Weak solutions are solutions] \label{rem:weak_to_strong}
		By Cauchy--Schwarz, the $N=0$ case of Proposition~\ref{prop:poly_decay_C_ij^l} implies that both sides of \eqref{eqn:SE6} converge absolutely for all $i,j,i',j' \in I$.
		In view of Remark~\ref{rem:convergence_for_poly_decay}, this is not a circular argument.
	\end{rem}
	
	
	We first prove an analog of \eqref{eqn:SE4} in which $\ell = (0,0)$ rather than $i = (0,0)$.
	
	\begin{lem} \label{lem:SE4_l=0}
		For $\ell = (0,0)$ and all $i,j \in I$,
		\begin{align*}
			C_{ij}^{\ell}
			= (-1)^{i_2} \1_{i=\overline{j}}.
		\end{align*}
	\end{lem}
	
	\begin{proof}
		Define $i' = j' = (0,0)$. Then by \eqref{eqn:SE4}, followed by \eqref{eqn:SE6}, followed by \eqref{eqn:SE1} and \eqref{eqn:SE4},
		\begin{align*}
			C_{ij}^{\ell}
			= \sum_{m \in I} (-1)^{m_2} C_{ij}^m C_{i'j'}^{\overline{m}}
			= \sum_{m \in I} (-1)^{m_2} C_{ii'}^m C_{jj'}^{\overline{m}}
			= \sum_{m \in I} (-1)^{m_2} \1_{i=m} \1_{j=\overline{m}}
			= (-1)^{i_2} \1_{i=\overline{j}}.
		\end{align*}
		Here \eqref{eqn:SE6} is applicable according to Remark~\ref{rem:convergence_for_poly_decay} since at most one term in each sum is nonzero.
	\end{proof}
	
	We next prove an analog of \eqref{eqn:SE5} with $\ell^+,i^-,j^-$ in place of $\ell^-,i^+,j^+$.
	
	\begin{lem} \label{lem:SE5_inverted}
		For all $i,j,\ell \in I$,
		\begin{align} \label{eqn:SE5_inverted}
			\sqrt{\lambda_{\ell_1} + \ell_2(\ell_2+1)} \, C_{ij}^{\ell^+}
			=
			\sqrt{\lambda_{i_1} + i_2(i_2-1)} \, C_{i^-j}^{\ell}
			+ \sqrt{\lambda_{j_1} + j_2(j_2-1)} \, C_{ij^-}^{\ell}.
		\end{align}
	\end{lem}
	
	\begin{proof}
		Since $I$ is closed under $i \mapsto \overline{i}$, we can apply \eqref{eqn:SE5} with $\overline{i},\overline{j},\overline{\ell}$ in place of $i,j,\ell$ to get
		\begin{align*}
			\sqrt{\lambda_{\ell_1} + \ell_2(\ell_2+1)} \, C_{\overline{i}\,\overline{j}}^{\overline{\ell}^-}
			=
			\sqrt{\lambda_{i_1} + i_2(i_2-1)} \, C_{\overline{i}^+\overline{j}}^{\overline{\ell}}
			+ \sqrt{\lambda_{j_1} + j_2(j_2-1)} \, C_{\overline{i}\,\overline{j}^+}^{\overline{\ell}}.
		\end{align*}
		Taking complex conjugates using \eqref{eqn:SE3} yields \eqref{eqn:SE5_inverted}.
	\end{proof}
	
	The following lemma is obtained by combining \eqref{eqn:SE5} with Lemma~\ref{lem:SE5_inverted}.
	
	\begin{lem} \label{lem:num_recursion}
		For all $i,j,\ell \in I$,
		\begin{align}
			(\lambda_{\ell_1} - \lambda_{i_1} - \lambda_{j_1} + 2i_2j_2) C_{ij}^{\ell}
			= \, &\sqrt{(\lambda_{i_1} + i_2(i_2+1))(\lambda_{j_1} + j_2(j_2-1))} \, C_{i^+j^-}^{\ell}
			\notag
			\\&+ \sqrt{(\lambda_{i_1} + i_2(i_2-1))(\lambda_{j_1} + j_2(j_2+1))} \, C_{i^-j^+}^{\ell}.
			\label{eqn:num_recursion}
		\end{align}
	\end{lem}
	
	\begin{proof}
		Let us first verify this when $\ell = (0,0)$.
		Then by Lemma~\ref{lem:SE4_l=0}, both sides vanish unless $i=\overline{j}$, in which case \eqref{eqn:num_recursion} reduces to
		\begin{align*}
			(-2\lambda_{i_1} - 2i_2^2) (-1)^{i_2}
			= (\lambda_{i_1} + i_2(i_2+1)) (-1)^{i_2+1}
			+ (\lambda_{i_1} + i_2(i_2-1)) (-1)^{i_2-1}.
		\end{align*}
		This is true by elementary algebra.
		
		Suppose now that $\ell \neq (0,0)$.
		Then we see from the definition of $I$ that at least one of $\ell^+ \in I$ or $\ell^- \in I$ must hold.
		We will treat the case where $\ell^- \in I$.
		The case $\ell^+ \in I$ can be dealt with by an almost identical computation (writing $\ell = (\ell^+)^-$ instead of $\ell = (\ell^-)^+$), or can be deduced from the case $\ell^- \in I$ by a ``complex conjugation trick" similar to the proof of Lemma~\ref{lem:SE5_inverted}.
		So assume $\ell^- \in I$.
		Then writing $\ell = (\ell^-)^+$ and applying Lemma~\ref{lem:SE5_inverted},
		\begin{align*}
			\sqrt{\lambda_{\ell_1} + \ell_2(\ell_2-1)} \, C_{ij}^{\ell}
			= \sqrt{\lambda_{\ell_1} + \ell_2(\ell_2-1)} \, C_{ij}^{(\ell^-)^+}
			= \sqrt{\lambda_{i_1} + i_2(i_2-1)} \, C_{i^-j}^{\ell^-}
			+ \sqrt{\lambda_{j_1} + j_2(j_2-1)} \, C_{ij^-}^{\ell^-}.
		\end{align*}
		Multiplying both sides by $\sqrt{\lambda_{\ell_1} + \ell_2(\ell_2-1)}$,
		\begin{align}
			(\lambda_{\ell_1} + \ell_2(\ell_2-1)) C_{ij}^{\ell}
			&= \sqrt{\lambda_{i_1} + i_2(i_2-1)} \sqrt{\lambda_{\ell_1} + \ell_2(\ell_2-1)} \, C_{i^-j}^{\ell^-}
			\notag
			\\&+ \sqrt{\lambda_{j_1} + j_2(j_2-1)} \sqrt{\lambda_{\ell_1} + \ell_2(\ell_2-1)} \, C_{ij^-}^{\ell^-}.
			\label{eqn:to_rewrite_by_SE5}
		\end{align}
		We claim that the first term on the right hand side can be rewritten as
		\begin{align}
			\sqrt{\lambda_{i_1} + i_2(i_2-1)} \sqrt{\lambda_{\ell_1} + \ell_2(\ell_2-1)} \, C_{i^-j}^{\ell^-}
			= \sqrt{\lambda_{i_1} + i_2(i_2-1)} \, &\Big(\sqrt{\lambda_{i_1} + i_2(i_2-1)} \, C_{ij}^{\ell}
			\notag
			\\&+ \sqrt{\lambda_{j_1} + j_2(j_2+1)} \, C_{i^-j^+}^{\ell}\Big).
			\label{eqn:SE5_with_prefactor}
		\end{align}
		If $i^- \in I$, then \eqref{eqn:SE5} with $i^-$ in place of $i$ implies \eqref{eqn:SE5_with_prefactor}.
		If $i^- \not\in I$, then it follows from the definition of $I$ that either $i_1 = i_2 = 0$, or $i_1 < 0$ and $i_2 = k_{-i_1}$.
		In the former case, $\lambda_{i_1} = 0$, while in the latter case, $\lambda_{i_1} = -k_{-i_1}(k_{-i_1}-1)$ by \eqref{eqn:lambda_{-r}_def}.
		We see that in both cases,
		\begin{align*}
			\lambda_{i_1} + i_2(i_2-1) = 0,
		\end{align*}
		and hence \eqref{eqn:SE5_with_prefactor} holds with both sides equal to $0$.
		Thus \eqref{eqn:SE5_with_prefactor} holds in all cases.
		Since $i,j$ are arbitrary in \eqref{eqn:SE5_with_prefactor}, we can switch $i,j$ to get
		\begin{align}
			\sqrt{\lambda_{j_1} + j_2(j_2-1)} \sqrt{\lambda_{\ell_1} + \ell_2(\ell_2-1)} \, C_{ij^-}^{\ell^-}
			= \sqrt{\lambda_{j_1} + j_2(j_2-1)} \, &\Big(\sqrt{\lambda_{j_1} + j_2(j_2-1)} \, C_{ij}^{\ell}
			\notag
			\\&+ \sqrt{\lambda_{i_1} + i_2(i_2+1)} \, C_{i^+j^-}^{\ell}\Big).
			\label{eqn:SE5_with_prefactor_switched}
		\end{align}
		Expanding the right hand sides of \eqref{eqn:SE5_with_prefactor} and \eqref{eqn:SE5_with_prefactor_switched}, and then inserting \eqref{eqn:SE5_with_prefactor} and \eqref{eqn:SE5_with_prefactor_switched} into \eqref{eqn:to_rewrite_by_SE5},
		\begin{align*}
			(\lambda_{\ell_1} + \ell_2(\ell_2-1)) C_{ij}^{\ell}
			&= (\lambda_{i_1} + i_2(i_2-1)) C_{ij}^{\ell}
			+ \sqrt{(\lambda_{i_1} + i_2(i_2-1)) (\lambda_{j_1} + j_2(j_2+1))} \, C_{i^-j^+}^{\ell}
			\\&+ (\lambda_{j_1} + j_2(j_2-1)) C_{ij}^{\ell}
			+ \sqrt{(\lambda_{i_1} + i_2(i_2+1)) (\lambda_{j_1} + j_2(j_2-1))} \, C_{i^+j^-}^{\ell}.
		\end{align*}
		Rearranging,
		\begin{align*}
			(\lambda_{\ell_1} - \lambda_{i_1} - \lambda_{j_1} + \ell_2(\ell_2-1) - i_2(i_2-1) - j_2(j_2-1)) C_{ij}^{\ell}
			&= \sqrt{(\lambda_{i_1} + i_2(i_2+1)) (\lambda_{j_1} + j_2(j_2-1))} \, C_{i^+j^-}^{\ell}
			\\&+ \sqrt{(\lambda_{i_1} + i_2(i_2-1)) (\lambda_{j_1} + j_2(j_2+1))} \, C_{i^-j^+}^{\ell}.
		\end{align*}
		By \eqref{eqn:SE2}, we may assume $i_2+j_2 = \ell_2$, or else both sides of \eqref{eqn:num_recursion} will be zero.
		Then
		\begin{align*}
			\ell_2(\ell_2-1) - i_2(i_2-1) - j_2(j_2-1)
			= 2i_2j_2.
		\end{align*}
		Inserting this into the equation above yields \eqref{eqn:num_recursion}.
	\end{proof}
	
	\begin{cor} \label{cor:num_recursion_specialized}
		For all $i,\ell \in I$,
		\begin{align} \label{eqn:num_recursion_specialized}
			(\lambda_{i_1} + i_2(i_2+1)) C_{i^+\overline{i^+}}^{\ell}
			= (\lambda_{\ell_1} - 2\lambda_{i_1} - 2i_2^2) C_{i\overline{i}}^{\ell}
			- (\lambda_{i_1} + i_2(i_2-1)) C_{i^-\overline{i^-}}^{\ell}.
		\end{align}
	\end{cor}
	
	\begin{proof}
		Take $j = \overline{i}$ in Lemma \ref{lem:num_recursion} and rearrange.
	\end{proof}
	
	We will prove the proposition below by iterating Corollary~\ref{cor:num_recursion_specialized}.
	Given $i = (i_1,i_2) \in \Z^2$ and $n \in \Z$, denote $i^{+n} = (i_1,i_2+n)$.
	In the remainder of this subsection, allow all implicit constants to depend on $\mathcal{S}$.
	
	\begin{prop} \label{prop:C_+n_-n_l_size}
		Let $r \in \Z_{\geq 1}$ and $n \in \Z_{\geq 0}$. Let $i = (r,0)$ and $j = (-r,k_r)$. Let $\ell \in I$ with $\ell_1 > 0$ and $\lambda_{\ell_1} \gg_{r,n} 1$. Then there are positive quantities
		\begin{align} \label{eqn:A,B_size}
			A \sim_{r,n} \lambda_{\ell_1}^n
			\qquad \text{and} \qquad
			B \sim_{r,n} \lambda_{\ell_1}^n
		\end{align}
		such that
		\begin{align} \label{eqn:A,B_def}
			C_{i^{+n} \overline{i^{+n}}}^{\ell}
			= A C_{ii}^{\ell}
			\qquad \text{and} \qquad
			C_{j^{+n} \overline{j^{+n}}}^{\ell}
			= B C_{j\overline{j}}^{\ell}
		\end{align}
		(of course, if $C_{ii}^{\ell} \neq 0$ (resp. $C_{j\overline{j}}^{\ell} \neq 0$), then $A$ (resp. $B$) is determined by division).
	\end{prop}
	
	\begin{proof}
		Let us note at the outset that $i^{+n}, j^{+n} \in I$ for all $n \in \Z_{\geq 0}$.
		We will induct on $n$, with base cases $n=0$ and $n=1$. First, when $n=0$ we can take $A = B = 1$. Next, suppose $n=1$. Since $i_2 = 0$, we have by \eqref{eqn:SE1} that $C_{i^+\overline{i^+}}^{\ell} = C_{i^-\overline{i^-}}^{\ell}$. Thus by Corollary \ref{cor:num_recursion_specialized} with the given choices of $i,\ell$,
		\begin{align*}
			2\lambda_{i_1} C_{i^+\overline{i^+}}^{\ell}
			= (\lambda_{\ell_1} - 2\lambda_{i_1}) C_{ii}^{\ell},
		\end{align*}
		where we have moved the second term on the right hand side of \eqref{eqn:num_recursion_specialized} to the left, and used that $\overline{i} = i$ because $i_2 = 0$. Since $i_1 = r \geq 1$, we have $\lambda_{i_1} > 0$, so we can divide by $\lambda_{i_1}$ and take
		\begin{align*}
			A = \frac{\lambda_{\ell_1} - 2\lambda_{i_1}}{2\lambda_{i_1}}.
		\end{align*}
		This will be positive and $\sim_r \lambda_{\ell_1}$ when $\lambda_{\ell_1} \gg_r 1$. This establishes the first half of \eqref{eqn:A,B_size} and \eqref{eqn:A,B_def} when $n=1$. For the second half (still with $n=1$), Corollary \ref{cor:num_recursion_specialized} with $j$ in place of $i$ says that
		\begin{align} \label{eqn:num_recursion_with_j}
			(\lambda_{j_1} + j_2(j_2+1)) C_{j^+\overline{j^+}}^{\ell}
			= (\lambda_{\ell_1} - 2\lambda_{j_1} - 2j_2^2) C_{j\overline{j}}^{\ell} - (\lambda_{j_1} + j_2(j_2-1)) C_{j^-\overline{j^-}}^{\ell}.
		\end{align}
		By \eqref{eqn:lambda_{-r}_def}, we have $\lambda_{j_1} = -k_r(k_r-1)$. By the definition of $j$, we also have $j_2 = k_r$ and $j^- \not\in I$. The latter implies $C_{j^-\overline{j^-}}^{\ell} = 0$ by definition. Therefore \eqref{eqn:num_recursion_with_j} reduces to
		\begin{align*}
			2k_r C_{j^+\overline{j^+}}^{\ell}
			= (\lambda_{\ell_1} - 2k_r) C_{j\overline{j}}^{\ell}.
		\end{align*}
		Thus we can take
		\begin{align} \label{eqn:B_formula_n=1}
			B = \frac{\lambda_{\ell_1} - 2k_r}{2k_r}.
		\end{align}
		This will be positive and $\sim_r \lambda_{\ell_1}$ when $\lambda_{\ell_1} \gg_r 1$. The base cases $n=0$ and $n=1$ are now complete.
		
		For the induction step, let $n \in \Z_{\geq 1}$, and let us prove the proposition for $n+1$. By Corollary \ref{cor:num_recursion_specialized} with $i^{+n} = (r,n)$ in place of $i$,
		\begin{align*}
			(\lambda_{i_1} + n(n+1)) C_{i^{+(n+1)} \overline{i^{+(n+1)}}}^{\ell}
			= (\lambda_{\ell_1} - 2\lambda_{i_1} - 2n^2) C_{i^{+n} \overline{i^{+n}}}^{\ell}
			- (\lambda_{i_1} + n(n-1)) C_{i^{+(n-1)} \overline{i^{+(n-1)}}}^{\ell}.
		\end{align*}
		By induction, for $\lambda_{\ell_1} \gg_{r,n} 1$, we have
		\begin{align*}
			C_{i^{+n} \overline{i^{+n}}}^{\ell}
			= A' C_{ii}^{\ell}
			\qquad \text{and} \qquad
			C_{i^{+(n-1)} \overline{i^{+(n-1)}}}^{\ell}
			= A'' C_{ii}^{\ell}
		\end{align*}
		for some positive $A',A''$ with $A' \sim_{r,n} \lambda_{\ell_1}^n$ and $A'' \sim_{r,n} \lambda_{\ell_1}^{n-1}$. Thus
		\begin{align*}
			C_{i^{+(n+1)} \overline{i^{+(n+1)}}}^{\ell}
			= A C_{ii}^{\ell}
			\qquad \text{with} \qquad
			A = \frac{(\lambda_{\ell_1} - 2\lambda_{i_1} - 2n^2)A' - (\lambda_{i_1} + n(n-1))A''}{\lambda_{i_1} + n(n+1)}.
		\end{align*}
		For $\lambda_{\ell_1} \gg_{r,n} 1$, this $A$ is positive and $\sim_{r,n} \lambda_{\ell_1}^{n+1}$. This completes the induction for the first half of the proposition. The induction step for the second half can be carried out in almost exactly the same way, with $j$ replacing $i$ and $B$ replacing $A$, and with minor differences in algebra.
	\end{proof}
	
	The proof of Proposition~\ref{prop:C_+n_-n_l_size} shows that $A,B$ obey second-order recurrences in $n$.
	We will not need these recurrences in this section, but we will use the recurrence for $B$ in Section~\ref{sec:applications}, so we record it here.
	
	\begin{cor} \label{cor:C_+n_-n_formula_explicit}
		Let $r \in \Z_{\geq 1}$ and $j = (-r,k_r)$. Let $\ell \in I$. Then for $n \in \Z_{\geq 0}$,
		\begin{align} \label{eqn:C_+n_-n_formula_explicit}
			C_{j^{+n} \overline{j^{+n}}}^{\ell}
			= B_{k_r,n}(\lambda_{\ell_1}) C_{j\overline{j}}^{\ell},
		\end{align}
		where $B_{k,n}$ is given by the second-order recurrence
		\begin{align} \label{eqn:B_k,n_recurrence}
			B_{k,n+1}(\lambda)
			= \frac{(\lambda + 2k(k-1) - 2(n+k)^2) B_{k,n}(\lambda) - ((n+k)(n+k-1) - k(k-1)) B_{k,n-1}(\lambda)}{(n+k)(n+k+1) - k(k-1)}
		\end{align}
		for $n \geq 1$, with initial conditions
		\begin{align*}
			B_{k,0}(\lambda)
			= 1
			\qquad \text{and} \qquad
			B_{k,1}(\lambda)
			= \frac{\lambda}{2k}-1.
		\end{align*}
	\end{cor}
	
	\begin{proof}
		The case $n=0$ is trivial. The case $n=1$ is \eqref{eqn:B_formula_n=1}. If \eqref{eqn:C_+n_-n_formula_explicit} holds for $n$ and for $n-1$, then it holds for $n+1$ because of Corollary~\ref{cor:num_recursion_specialized} and the fact that $\lambda_{j_1} = -k_r(k_r-1)$.
	\end{proof}
	
	We can now prove the following special case of our goal, Proposition~\ref{prop:poly_decay_C_ij^l}.
	
	\begin{prop} \label{prop:weak_to_strong_primitive}
		Let $r \in \Z_{\geq 1}$. Let $i = (r,0)$ and $j = (-r,k_r)$. Let $N \geq 0$. Then
		\begin{align*}
			\sum_{\ell \in I} |\lambda_{\ell_1}|^N |C_{ii}^{\ell}|^2 < \infty
			\qquad \text{and} \qquad
			\sum_{\ell \in I} |\lambda_{\ell_1}|^N |C_{j\overline{j}}^{\ell}|^2 < \infty.
		\end{align*}
	\end{prop}
	
	
	
	\begin{proof}
		Let $n \in \Z_{\geq 0}$. By \eqref{eqn:SE6}, with assumptions as in Remark~\ref{rem:convergence_for_poly_decay}, we will have
		\begin{align} \label{eqn:+n_+n+1_crossing}
			\sum_{\ell \in I} (-1)^{\ell_2} C_{i^{+n} i^{+(n+1)}}^{\ell} C_{\overline{i^{+n}} \, \overline{i^{+(n+1)}}}^{\overline{\ell}}
			= \sum_{\ell \in I} (-1)^{\ell_2} C_{i^{+n} \overline{i^{+n}}}^{\ell} C_{i^{+(n+1)} \, \overline{i^{+(n+1)}}}^{\overline{\ell}}
		\end{align}
		as soon as we can show that both sides converge in $\overline{\R} + i\overline{\R}$.
		By \eqref{eqn:SE2}, the summand on the left hand side of \eqref{eqn:+n_+n+1_crossing} is zero unless $\ell_2 = 2n+1$ (an odd number), in which case the summand is nonpositive by \eqref{eqn:SE3}.
		Therefore $\LHS\eqref{eqn:+n_+n+1_crossing}$ converges in $\overline{\R}$ to a nonpositive value.
		By \eqref{eqn:SE2}, the summand on the right is zero unless $\ell_2 = 0$.
		By \eqref{eqn:i2=0_implies_i1geq0} and \eqref{eqn:lambda_r_to_infty}, for all but finitely many $\ell \in I$ with $\ell_2 = 0$, we have $\lambda_{\ell_1} \gg_{r,n} 1$.
		For such $\ell$, Proposition~\ref{prop:C_+n_-n_l_size} tells us that
		\begin{align} \label{eqn:+n_+n+1_asymp}
			C_{i^{+n} \overline{i^{+n}}}^{\ell} C_{i^{+(n+1)} \, \overline{i^{+(n+1)}}}^{\overline{\ell}}
			\sim_{r,n} \lambda_{\ell_1}^{2n+1} (C_{ii}^{\ell})^2
		\end{align}
		(with positive implicit constant).
		Since $i_2 = \ell_2 = 0$, it follows from \eqref{eqn:SE3} that $C_{ii}^{\ell} \in \R$.
		Therefore $\RHS\eqref{eqn:+n_+n+1_asymp} \geq 0$, and hence $\LHS\eqref{eqn:+n_+n+1_asymp} \geq 0$.
		Thus all but finitely many terms in $\RHS\eqref{eqn:+n_+n+1_crossing}$ are nonnegative, and consequently $\RHS\eqref{eqn:+n_+n+1_crossing}$ converges in $\overline{\R}+i\overline{\R}$ to a value other than $-\infty$.
		We now know that the left and right hand sides of \eqref{eqn:+n_+n+1_crossing} both converge in $\overline{\R}+i\overline{\R}$, so we can apply \eqref{eqn:SE6} to get that they have the same value in $\overline{\R}+i\overline{\R}$. We further know that their common value must be in the interval $(-\infty,0]$; in particular, it must be finite.
		Since all but finitely many terms on each side of \eqref{eqn:+n_+n+1_crossing} have the same sign, we conclude that both sides converge absolutely.
		Absolute convergence on the right hand side means
		\begin{align} \label{eqn:RHS_+n_+n+1_crossing_abs}
			\sum_{\substack{\ell \in I \\ \ell_2 = 0}} |C_{i^{+n} \overline{i^{+n}}}^{\ell} C_{i^{+(n+1)} \, \overline{i^{+(n+1)}}}^{\overline{\ell}}|
			< \infty
		\end{align}
		(we have included the condition $\ell_2 = 0$ in the sum, but this makes no difference because the summand vanishes when $\ell_2 \neq 0$).
		Plugging \eqref{eqn:+n_+n+1_asymp} into $\eqref{eqn:RHS_+n_+n+1_crossing_abs}$ for all but finitely many $\ell$, we obtain
		\begin{align*}
			\sum_{\substack{\ell \in I \\ \ell_2 = 0}} |\lambda_{\ell_1}|^{2n+1} |C_{ii}^{\ell}|^2
			< \infty
		\end{align*}
		(the absolute values are redundant).
		We can remove the condition $\ell_2 = 0$ because again, by \eqref{eqn:SE2}, the summand vanishes when $\ell_2 \neq 0$.
		Since $n$ may be taken arbitrarily large, the first half of the proposition follows.
		The second half can be proved by a very similar argument.
	\end{proof}
	
	The next proposition is a slight generalization of Proposition~\ref{prop:weak_to_strong_primitive}.
	
	\begin{prop} \label{prop:weak_to_strong_K-inv}
		Let $m \in I$ and $N \geq 0$. Then
		\begin{align} \label{eqn:weak_to_strong_K-inv}
			\sum_{\ell \in I} |\lambda_{\ell_1}|^N |C_{m\overline{m}}^{\ell}|^2 < \infty.
		\end{align}
	\end{prop}
	
	\begin{proof}
		We may assume $m_2 \geq 0$ by replacing $m$ with $\overline{m}$ if necessary; by \eqref{eqn:SE1}, this will not change \eqref{eqn:weak_to_strong_K-inv}.
		The case $m = (0,0)$ is trivial by \eqref{eqn:SE4}, so assume furthermore that $m \neq (0,0)$.
		Then we see from the definition of $I$ that either $m = i^{+n}$ or $m = j^{+n}$ for some $n \in \Z_{\geq 0}$, with $i,j$ of the form in Proposition \ref{prop:weak_to_strong_primitive}. We claim that
		\begin{align} \label{eqn:C_mml_bd}
			|C_{m\overline{m}}^{\ell}|
			\lesssim_m |\lambda_{\ell_1}|^n |C_{ii}^{\ell}|
			\qquad \text{or} \qquad
			|C_{m\overline{m}}^{\ell}|
			\lesssim_m |\lambda_{\ell_1}|^n |C_{j\overline{j}}^{\ell}|
		\end{align}
		for all but finitely many $\ell \in I$. If $\ell_2 \neq 0$, then by \eqref{eqn:SE2}, we have $C_{m\overline{m}}^{\ell} = 0$, and \eqref{eqn:C_mml_bd} holds trivially. By \eqref{eqn:i2=0_implies_i1geq0} and \eqref{eqn:lambda_r_to_infty}, we have $\lambda_{\ell_1} \gg_m 1$ for all but finitely many $\ell \in I$ with $\ell_2 = 0$. For such $\ell$, Proposition \ref{prop:C_+n_-n_l_size} implies \eqref{eqn:C_mml_bd}. Thus the claim holds. Inserting \eqref{eqn:C_mml_bd} into \eqref{eqn:weak_to_strong_K-inv} for all but finitely many $\ell$, we obtain \eqref{eqn:weak_to_strong_K-inv} from Proposition \ref{prop:weak_to_strong_primitive}.
	\end{proof}
	
	We are finally ready to prove the main result of this subsection, Proposition~\ref{prop:poly_decay_C_ij^l}.
	
	\begin{proof}[Proof of Proposition~\ref{prop:poly_decay_C_ij^l}]
		We will prove \eqref{eqn:weak_to_strong} for nonnegative even integers $N$ by induction. We begin with the base case $N=0$. By \eqref{eqn:SE6}, with assumptions as in Remark~\ref{rem:convergence_for_poly_decay}, we will have
		\begin{align*}
			\sum_{\ell \in I} (-1)^{\ell_2} C_{ij}^{\ell} C_{\overline{i} \, \overline{j}}^{\overline{\ell}}
			= \sum_{\ell \in I} (-1)^{\ell_2} C_{i\overline{i}}^{\ell} C_{j \overline{j}}^{\overline{\ell}}
		\end{align*}
		as soon as we can show that both sides converge in $\overline{\R}+i\overline{\R}$.
		By \eqref{eqn:SE2} and \eqref{eqn:SE3}, the terms on the left hand side are either zero or have sign $(-1)^{i_2+j_2}$. In particular, they are either all nonnegative or all nonpositive, so the left hand side converges in $\overline{\R}$. By Cauchy--Schwarz and the $N=0$ case of Proposition \ref{prop:weak_to_strong_K-inv}, the right hand side converges absolutely in $\C$.
		Thus \eqref{eqn:SE6} applies and tells us in particular that the left hand side converges in $\overline{\R}$ to an element of $\overline{\R} \cap \C = \R$.
		Therefore the left hand side is finite, and since the terms on the left all have the same sign, the left hand side converges absolutely.
		By \eqref{eqn:SE3}, this absolute convergence is exactly the same as the $N=0$ case of \eqref{eqn:weak_to_strong}. Hence the base case holds.
		
		For the induction step, let $N \in \Z_{\geq 2}$, assume Proposition~\ref{prop:poly_decay_C_ij^l} holds for $N-2$, and let us prove the proposition for $N$. By Lemma \ref{lem:num_recursion},
		\begin{align*}
			|\lambda_{\ell_1} C_{ij}^{\ell}|
			\lesssim_{i,j} |C_{ij}^{\ell}| + |C_{i^+j^-}^{\ell}| + |C_{i^-j^+}^{\ell}|.
		\end{align*}
		Squaring this and inserting it in the left hand side of \eqref{eqn:weak_to_strong},
		\begin{align*}
			\sum_{\ell \in I} |\lambda_{\ell_1}|^N |C_{ij}^{\ell}|^2
			\lesssim_{i,j} \sum_{\ell \in I} |\lambda_{\ell_1}|^{N-2} (|C_{ij}^{\ell}|^2 + |C_{i^+j^-}^{\ell}|^2 + |C_{i^-j^+}^{\ell}|^2).
		\end{align*}
		The right hand side is finite by induction, so the proof is complete.
	\end{proof}
	
	\subsection{Building the multiplicative representation}
	
	Let $\H_{\R}$ be a real unitary representation of $G$ with Laplace spectrum $\{\lambda_r\}_{r \geq 0}$ and holomorphic spectrum $\{k_r\}_{r \geq 1}$.
	By \eqref{eqn:lambda_r_to_infty} and \eqref{eqn:k_r_to_infty}, $\H_{\R}$ has discrete bi-infinite spectrum, and the trivial representation appears exactly once in $\H_{\R}$.
	Fix a choice of unit vector $\mathbf{1} \in \H_{\R}^G$.
	Let $\H$ be the complexification of $\H_{\R}$.
	Let $\{\psi_i\}_{i \in I}$ be a $(\g,K)$-adapted basis of $\H$ with $\psi_{(0,0)} = \mathbf{1}$.
	By Proposition~\ref{prop:basis_spans_H^fin}, $\{\psi_i\}_{i \in I}$ is a basis for $\H^{\fin}$ as an abstract vector space.
	Define a bilinear multiplication map $\H^{\fin} \times \H^{\fin} \to \H^{\infty}$ by setting
	\begin{align*} 
		\psi_i \psi_j
		= \sum_{\ell \in I} C_{ij}^{\ell} \psi_{\ell}
	\end{align*}
	and extending by linearity. The right hand side is in $\H^{\infty}$ by Propositions~\ref{prop:smooth_iff_spectral_decay} and \ref{prop:poly_decay_C_ij^l}.
	
	\begin{prop} \label{prop:H_is_mult_rep}
		With the above multiplication, $\H$ is a multiplicative representation.
		Moreover, for all $\alpha \in \H^{\fin}$, the complex conjugate of $\alpha$ in the sense of Definition~\ref{def:mult_rep} coincides with the complex conjugate of $\alpha$ determined by viewing $\H$ as the complexification of $\H_{\R}$.
	\end{prop}
	
	Assuming this proposition, the multiplicative spectrum of $\H$ with respect to $\{\psi_i\}_{i \in I}$ is $\mathcal{S}$ by construction, and the ``if" direction in Proposition~\ref{prop:mult_rep_vs_HB} holds.
	
	To prove Proposition~\ref{prop:H_is_mult_rep}, we need to verify the axioms in Definition~\ref{def:mult_rep}: commutativity, existence of a unit, normalization, ergodicity, equivariance, existence of complex conjugates (compatibly with $\H = \H_{\R} \otimes_{\R} \C$), and crossing symmetry.
	
	Recall from Proposition~\ref{prop:(g,K)-adapted_key_properties} that $\psi_i$ has Casimir eigenvalue $\lambda_{i_1}$ and weight $i_2$. We shall use this without comment from now on.
	
	\begin{proof}[Proof of commutativity]
		By \eqref{eqn:SE1}, we have $\psi_i \psi_j = \psi_j \psi_i$ for all $i,j \in I$. Commutativity follows by linearity.
	\end{proof}
	
	\begin{proof}[Proof of existence of a unit]
		We claim that $\mathbf{1}$ is a unit. To show this, it suffices by linearity to check that $\mathbf{1} \psi_j = \psi_j$ for all $j \in I$. Denote $i=(0,0)$. Then by \eqref{eqn:SE4},
		\begin{align*}
			&\mathbf{1} \psi_j
			= \psi_i \psi_j
			= \sum_{\ell \in I}  C_{ij}^{\ell} \psi_{\ell}
			= \psi_j.
			\qedhere
		\end{align*}
	\end{proof}
	
	\begin{proof}[Proof of normalization]
		Above, we chose $\mathbf{1}$ to be a unit vector.
	\end{proof}
	
	\begin{proof}[Proof of ergodicity]
		We already saw that the trivial representation appears exactly once in $\H_{\R}$. By definition, $\mathbf{1} \in \H_{\R}^G$ is a unit vector. It follows that $\H^G = \mathbf{\C} \mathbf{1}$.
	\end{proof}
	
	\begin{proof}[Proof of equivariance]
		By linearity, it suffices to prove the product rule
		\begin{align*}
			X(\psi_i\psi_j)
			= (X\psi_i) \psi_j + \psi_i(X\psi_j)
		\end{align*}
		for all $i,j \in I$ and $X \in \{H,E,\overline{E}\}$.
		Testing this against $\psi_{\ell}$, it in fact suffices to prove
		\begin{align} \label{eqn:weak_prod_rule}
			-\langle \psi_i\psi_j, \overline{X}\psi_{\ell} \rangle_{\H}
			= \langle (X\psi_i) \psi_j, \psi_{\ell} \rangle_{\H}
			+ \langle \psi_i (X\psi_j), \psi_{\ell} \rangle_{\H}
		\end{align}
		for all $i,j,\ell \in I$ and $X \in \{H,E,\overline{E}\}$.
		
		First suppose $X = H$. Then since $-\overline{H} = H$,
		\begin{align*}
			\LHS\eqref{eqn:weak_prod_rule}
			= \ell_2 C_{ij}^{\ell}
			\qquad \text{and} \qquad
			\RHS\eqref{eqn:weak_prod_rule}
			= (i_2+j_2) C_{ij}^{\ell}.
		\end{align*}
		By \eqref{eqn:SE2}, both sides are zero unless $i_2+j_2 = \ell_2$, in which case both sides are equal.
		
		Next suppose $X = E$. Then by \eqref{eqn:E_psi_i_formula} and \eqref{eqn:Ebar_psi_i_formula},
		\begin{align*}
			\LHS\eqref{eqn:weak_prod_rule}
			= \sqrt{\lambda_{\ell_1} + \ell_2(\ell_2-1)} \, C_{ij}^{\ell^-}
		\end{align*}
		and
		\begin{align*}
			\RHS\eqref{eqn:weak_prod_rule}
			= \sqrt{\lambda_{i_1} + i_2(i_2+1)} \, C_{i^+j}^{\ell}
			+ \sqrt{\lambda_{j_1} + j_2(j_2+1)} \, C_{ij^+}^{\ell}.
		\end{align*}
		These are equal by \eqref{eqn:SE5}.
		
		Finally, suppose $X = \overline{E}$.
		Then again by \eqref{eqn:E_psi_i_formula} and \eqref{eqn:Ebar_psi_i_formula},
		\begin{align*}
			\LHS\eqref{eqn:weak_prod_rule}
			= -\sqrt{\lambda_{\ell_1} + \ell_2(\ell_2+1)} \, C_{ij}^{\ell^+}
		\end{align*}
		and
		\begin{align*}
			\RHS\eqref{eqn:weak_prod_rule}
			= -\sqrt{\lambda_{i_1} + i_2(i_2-1)} \, C_{i^-j}^{\ell}
			- \sqrt{\lambda_{j_1} + j_2(j_2-1)} \, C_{ij^-}^{\ell}.
		\end{align*}
		These are equal by Lemma~\ref{lem:SE5_inverted}.
	\end{proof}
	
	\begin{proof}[Proof of existence of complex conjugates (compatibly with $\H = \H_{\R} \otimes_{\R} \C$)]
		For $\alpha \in \H^{\fin}$, let $\overline{\alpha}$ be the complex conjugate of $\alpha$ obtained by viewing $\H$ as the complexification of $\H_{\R}$. We must show that $\overline{\alpha}$ is the complex conjugate of $\alpha$ in the sense of Definition~\ref{def:mult_rep}.
		By linearity, it suffices to prove this for $\alpha = \psi_i$ for $i \in I$. By linearity again, this is equivalent to
		\begin{align} \label{eqn:psi_i_adjoint}
			\langle \psi_i\psi_j, \psi_{\ell} \rangle_{\H}
			= \langle \psi_j, \overline{\psi_i} \psi_{\ell} \rangle_{\H}
		\end{align}
		for all $j,\ell \in I$.
		In summary, we are reduced to proving \eqref{eqn:psi_i_adjoint} for all $i,j,\ell \in I$.
		We have
		\begin{align*}
			\LHS\eqref{eqn:psi_i_adjoint}
			= C_{ij}^{\ell}
			\quad \text{and} \quad
			\RHS\eqref{eqn:psi_i_adjoint}
			= \overline{\langle \overline{\psi_i}\psi_{\ell}, \psi_j \rangle_{\H}}
			= (-1)^{i_2} \overline{\langle \psi_{\overline{i}} \psi_{\ell}, \psi_j \rangle_{\H}}
			= (-1)^{i_2} \overline{C_{\overline{i} \ell}^j}
			= (-1)^{i_2} C_{i\overline{\ell}}^{\overline{j}},
		\end{align*}
		where the last equality is by \eqref{eqn:SE3}.
		Thus it remains to show that
		\begin{align} \label{eqn:ijl_permutable_with_bars}
			C_{ij}^{\ell} = (-1)^{i_2} C_{i\overline{\ell}}^{\overline{j}}
		\end{align}
		for all $i,j,\ell \in I$.
		Define $m = (0,0)$.
		Then by \eqref{eqn:SE1} and \eqref{eqn:SE4}, followed by \eqref{eqn:SE6}, followed by \eqref{eqn:SE1} and \eqref{eqn:SE4} again,
		\begin{align} \label{eqn:permutability_pf_by_SE6}
			(-1)^{\ell_2} C_{ij}^{\ell}
			= \sum_{n \in I} (-1)^{n_2} C_{ij}^n C_{\overline{\ell}m}^{\overline{n}}
			= \sum_{n \in I} (-1)^{n_2} C_{i\overline{\ell}}^n C_{jm}^{\overline{n}}
			= (-1)^{j_2} C_{i\overline{\ell}}^{\overline{j}}.
		\end{align}
		By \eqref{eqn:SE2}, both sides of \eqref{eqn:ijl_permutable_with_bars} vanish unless $i_2+j_2 = \ell_2$, so we may assume $i_2+j_2=\ell_2$.
		Then \eqref{eqn:permutability_pf_by_SE6} implies \eqref{eqn:ijl_permutable_with_bars}.
	\end{proof}
	
	\begin{proof}[Proof of crossing symmetry]
		As explained in Remark~\ref{rem:on_mult_rep_def}, it suffices to check the crossing equation \eqref{eqn:crossing} for the permutation $\sigma = (23)$.
		Therefore, by linearity, it suffices to show that
		\begin{align} \label{eqn:psi_crossing}
			\langle \psi_i\psi_j, \overline{\psi_{i'}} \, \overline{\psi_{j'}} \rangle_{\H}
			= \langle \psi_i \psi_{i'}, \overline{\psi_j} \, \overline{\psi_{j'}} \rangle_{\H}
		\end{align}
		for all $i,j,i',j' \in I$.
		By the same computation as in the proof of \eqref{eqn:SE6} in Section~\ref{sec:spectral_eqns_pf}, namely \eqref{eqn:crossing_to_C's},
		\begin{align*}
			\langle \psi_i \psi_j, \overline{\psi_{i'}} \, \overline{\psi_{j'}} \rangle_{\H}
			= \sum_{\ell \in I} (-1)^{\ell_2} C_{ij}^{\ell} C_{i'j'}^{\overline{\ell}}.
		\end{align*}
		By \eqref{eqn:SE6}, the right hand side is preserved by switching $i'$ and $j$, so the left hand side must be preserved by the same switch.
		This means that \eqref{eqn:psi_crossing} holds.
	\end{proof}
	
	The proof of Proposition~\ref{prop:H_is_mult_rep} is complete. Between Section~\ref{sec:spectral_eqns_pf} and this section, we have now proved both the ``if" and ``only if" directions in Proposition~\ref{prop:mult_rep_vs_HB}.
	
	\section{Applications of the hyperbolic bootstrap equations} \label{sec:applications}
	
	There have been two strands of applications of the conformal bootstrap in the context of hyperbolic surface spectra. The first, in the low energy regime, is the work \cite{Bonifacio_22_2,Kravchuk--Mazac--Pal,Gesteau--Pal--Simmons-Duffin--Xu} on bounds for small eigenvalues. The second, in the high energy regime, is the work \cite{Bernstein--Reznikov_10,Adve_et_al} on subconvexity for triple product $L$-functions.
	To show that the hyperbolic bootstrap equations capture the information used in the conformal bootstrap,
	in this section we illustrate how the proofs in \cite{Kravchuk--Mazac--Pal} and \cite{Adve_et_al} can be phrased in terms of these equations.
	The proofs in \cite{Bonifacio_22_2,Gesteau--Pal--Simmons-Duffin--Xu} and \cite{Bernstein--Reznikov_10} can be phrased similarly (though \cite{Gesteau--Pal--Simmons-Duffin--Xu} works with $\SL_2(\R)$ rather than $\PSL_2(\R)$ and thus needs a slight generalization of the hyperbolic bootstrap equations).
	
	\subsection{Bounds on $\lambda_1$} \label{subsec:applications:lambda_1}
	
	The \emph{bass note spectrum} of compact hyperbolic 2-orbifolds, defined by Sarnak in \cite[Lecture~1]{Sarnak_Chern_lectures}, is the set of positive real numbers
	\begin{align*} 
		\{\lambda_1(\Gamma \backslash \mathbf{H}) : \Gamma \text{ is a cocompact lattice in } G = \PSL_2(\R)\},
	\end{align*}
	i.e., the set of numbers which arise as the first Laplace eigenvalue of a compact hyperbolic 2-orbifold.
	In \cite[Figure~2 and Conjecture~4.2]{Kravchuk--Mazac--Pal}, Kravchuk, Maz\'a\v{c}, and Pal give a conjectural description of the bass note spectrum, and in \cite[Theorem~4.3]{Kravchuk--Mazac--Pal} they prove almost all of the conjecture.
	The following is a summary of their conjecture.
	
	\begin{conj} \label{conj:bass}
		The bass note spectrum of compact hyperbolic 2-orbifolds is
		\begin{align} \label{eqn:bass_note_spectrum_conj}
			(0,15.7902...] \cup \{23.0785...\} \cup \{28.0798...\} \cup \{44.8883...\},
		\end{align}
		where these four numbers are the first eigenvalues of specific triangle orbifolds given in \cite[Conjecture~4.2]{Kravchuk--Mazac--Pal}.
	\end{conj}
	
	Qualitatively, \eqref{eqn:bass_note_spectrum_conj} consists of an interval together with three singleton outliers.
	None of the triangle orbifolds which conjecturally give rise to the three outliers carry holomorphic modular forms of weight 4.
	Therefore, a sample consequence of Conjecture~\ref{conj:bass} is that if $\Gamma \backslash \mathbf{H}$ carries a holomorphic form of weight 4, then $\lambda_1(\Gamma \backslash \mathbf{H})$ lies in the interval.
	Kravchuk, Maz\'a\v{c}, and Pal almost prove this: they show that $\lambda_1(\Gamma \backslash \mathbf{H}) \leq 15.79144$ (see \cite[Table~2]{Kravchuk--Mazac--Pal}).
	To prove this bound they used linear programming on a computer. An easier estimate, which can be obtained by hand with the same technique, is the theorem below.
	The case $k=2$ of the theorem gives $\lambda_1(\Gamma \backslash \mathbf{H}) \leq 16$ for $\Gamma \backslash \mathbf{H}$ as above.
	
	\begin{thm} \label{thm:KMP_basic_bd}
		Let $\lambda_1$ be the first eigenvalue of a compact hyperbolic 2-orbifold admitting a nonzero holomorphic modular form of weight $2k$.
		Then
		\begin{align} \label{eqn:KMP_basic_bd}
			\lambda_1
			\leq \frac{\sqrt{33k^2+18k+1} + 9k+1}{2}.
		\end{align}
	\end{thm}
	
	The inequality \eqref{eqn:KMP_basic_bd} is \cite[(2.48)]{Kravchuk--Mazac--Pal}.
	Although it is far from the sharpest estimate in \cite{Kravchuk--Mazac--Pal}, it is still remarkably accurate given that it has such a simple form.
	The quality of \eqref{eqn:KMP_basic_bd} is illustrated in the table below, which displays the largest known $\lambda_1$ of a compact hyperbolic 2-orbifold with a holomorphic form of weight $2k$.
	The last two columns in the table come from \cite[Table~2]{Kravchuk--Mazac--Pal}.
	\begin{center}
		\begin{tabular}{|c|c|c|c|}
			\hline
			$k$ & $\RHS\eqref{eqn:KMP_basic_bd}$ & Largest known $\lambda_1$ & Topological type of maximizer \\
			\hline
			1 & 8.6055... & 8.4677... & [1;2] \\
			2 & 16 & 15.7902... & [0;2,2,2,3] \\
			3 & 23.3808... & 23.0785... & [0;3,3,4] \\
			4 & 30.7577... & 28.0798... & [0;2,4,5] \\
			6 & 45.5069... & 44.8883... & [0;2,3,7] \\
			\hline
		\end{tabular}
	\end{center}
	We note, as explained below Table~2 in \cite{Kravchuk--Mazac--Pal}, that every 2-orbifold admits a holomorphic form of weight $2k$ for some $k \in \{1,2,3,4,6\}$.
	
	
	We now outline the proof of Theorem~\ref{thm:KMP_basic_bd} in the language of this paper.
	This is not a direct translation from \cite{Kravchuk--Mazac--Pal}, but the general mechanism is the same.
	Let $\Gamma \backslash \mathbf{H}$ be a compact hyperbolic 2-orbifold as in Theorem~\ref{thm:KMP_basic_bd}, and let
	\begin{align*}
		\mathcal{S}
		= (\{\lambda_r\}_{r \geq 0}, \{k_r\}_{r \geq 1}, \{C_{ij}^{\ell}\}_{i,j,\ell \in I})
	\end{align*}
	be the multiplicative spectrum of $\Gamma \backslash \mathbf{H}$ with respect to some choice of $(\g,K)$-adapted basis.
	The only information needed to prove Theorem~\ref{thm:KMP_basic_bd} is that $\mathcal{S}$ solves the hyperbolic bootstrap equations.
	Since $\RHS\eqref{eqn:KMP_basic_bd}$ increases with $k$, we may assume $k = k_1$.
	Then in the remainder of this subsection, let $i = (-1,k) \in I$.
	Following \cite{Kravchuk--Mazac--Pal}, we will deduce Theorem~\ref{thm:KMP_basic_bd} from the two results below.
	
	\begin{prop} \label{prop:KMP_(2.47)}
		One has
		\begin{align} \label{eqn:KMP_(2.47)}
			\sum_{r=1}^{\infty} \lambda_r(\lambda_r^2 - (9k+1)\lambda_r + 12k^2) |C_{i\overline{i}}^{(r,0)}|^2
			= 0.
		\end{align}
	\end{prop}
	
	
	\begin{lem} \label{lem:nonvanishing_C_exists}
		There exists $r \in \Z_{\geq 1}$ such that $C_{i\overline{i}}^{(r,0)} \neq 0$.
	\end{lem}
	
	\begin{proof}[Proof of Theorem~\ref{thm:KMP_basic_bd} assuming Proposition~\ref{prop:KMP_(2.47)} and Lemma~\ref{lem:nonvanishing_C_exists}]
		The right hand side of \eqref{eqn:KMP_basic_bd} is the larger of the two roots of the quadratic polynomial $\lambda \mapsto \lambda^2 - (9k+1)\lambda + 12k^2$. Thus if \eqref{eqn:KMP_basic_bd} fails, then $\LHS\eqref{eqn:KMP_(2.47)}$ is nonnegative, and in fact by Lemma~\ref{lem:nonvanishing_C_exists} it is strictly positive. This is a contradiction.
	\end{proof}
	
	It remains to prove Proposition~\ref{prop:KMP_(2.47)} and Lemma~\ref{lem:nonvanishing_C_exists}.
	As in Section~\ref{sec:classification_implies_main_thm}, denote
	\begin{align*}
		i^{+n}
		= (i_1,i_2+n)
		= (-1,k+n).
	\end{align*}
	We will use the following instances of \eqref{eqn:SE6}: for $n \in \Z_{\geq 0}$,
	\begin{align} \label{eqn:s-t_crossing_=}
		\sum_{\ell \in I} (-1)^{\ell_2} C_{i^{+n} \overline{i^{+n}}}^{\ell} C_{i^{+n} \overline{i^{+n}}}^{\overline{\ell}}
		= \sum_{\ell \in I} (-1)^{\ell_2} C_{i^{+n} i^{+n}}^{\ell} C_{\overline{i^{+n}} \, \overline{i^{+n}}}^{\overline{\ell}}
	\end{align}
	and
	\begin{align} \label{eqn:s-t_crossing_+1}
		\sum_{\ell \in I} (-1)^{\ell_2} C_{i^{+n} \overline{i^{+n}}}^{\ell} C_{i^{+(n+1)} \overline{i^{+(n+1)}}}^{\overline{\ell}}
		= \sum_{\ell \in I} (-1)^{\ell_2} C_{i^{+n} i^{+(n+1)}}^{\ell} C_{\overline{i^{+n}} \, \overline{i^{+(n+1)}}}^{\overline{\ell}}.
	\end{align}
	Actually, for Proposition~\ref{prop:KMP_(2.47)} we will only use the cases $n=0$ and $n=1$, and for Lemma~\ref{lem:nonvanishing_C_exists}, we will only use the case $n=0$.
	Given $N \in \Z_{\geq 0}$ and coefficients $a_0,\dots,a_N,b_0,\dots,b_N \in \R$, form the linear combination
	\begin{align} \label{eqn:s-t_crossing_combo}
		&\sum_{\ell \in I} (-1)^{\ell_2} \sum_{n=0}^{N} (a_n C_{i^{+n} \overline{i^{+n}}}^{\ell} C_{i^{+n} \overline{i^{+n}}}^{\overline{\ell}}
		+ b_n C_{i^{+n} \overline{i^{+n}}}^{\ell} C_{i^{+(n+1)} \overline{i^{+(n+1)}}}^{\overline{\ell}})
		\notag
		\\= &\sum_{\ell \in I} (-1)^{\ell_2} \sum_{n=0}^{N} (a_n C_{i^{+n} i^{+n}}^{\ell} C_{\overline{i^{+n}} \, \overline{i^{+n}}}^{\overline{\ell}}
		+ b_n C_{i^{+n} i^{+(n+1)}}^{\ell} C_{\overline{i^{+n}} \, \overline{i^{+(n+1)}}}^{\overline{\ell}})
	\end{align}
	of \eqref{eqn:s-t_crossing_=} and \eqref{eqn:s-t_crossing_+1}.
	
	\begin{lem} \label{lem:cancel_discrete_contribution}
		Let $N \in \Z_{\geq 0}$ and $a_0,\dots,a_N \in \R$. Then there exist $b_0,\dots,b_N \in \R$, depending only on $k,N,a_0,\dots,a_N$, such that $\RHS\eqref{eqn:s-t_crossing_combo} = 0$.
	\end{lem}
	
	The proof is elementary but lengthy, so we only give a sketch.
	
	\begin{proof}[Sketch of proof]
		For $\ell \in I$ and $n \in \Z_{\geq 0}$, we have
		\begin{align} \label{eqn:HB5_same_terms}
			\sqrt{\lambda_{\ell_1} + \ell_2(\ell_2-1)} \, C_{i^{+n} i^{+n}}^{\ell^-}
			= 2\sqrt{\lambda_{i_1} + (i_2+n)(i_2+n+1)} \, C_{i^{+n}i^{+(n+1)}}^{\ell}
		\end{align}
		by \eqref{eqn:SE5} followed by \eqref{eqn:SE1}, and
		\begin{align} \label{eqn:HB5_inverted_same_terms}
			\sqrt{\lambda_{\ell_1} + \ell_2(\ell_2+1)} \, C_{i^{+n} i^{+n}}^{\ell^+}
			= 2\sqrt{\lambda_{i_1} + (i_2+n)(i_2+n-1)} \, C_{i^{+(n-1)}i^{+n}}^{\ell}
		\end{align}
		by Lemma~\ref{lem:SE5_inverted} followed by \eqref{eqn:SE1}.
		Since $\lambda_{i_1} = -k(k-1)$ by \eqref{eqn:lambda_{-r}_def} and $i_2 = k$, the square roots on the right hand sides of \eqref{eqn:HB5_same_terms} and \eqref{eqn:HB5_inverted_same_terms} depend only on $k$ and $n$.
		
		Using \eqref{eqn:HB5_same_terms}, \eqref{eqn:HB5_inverted_same_terms}, \eqref{eqn:SE2}, and \eqref{eqn:SE3}, all of which are linear relations between the $C$'s, one can eliminate variables in $\RHS\eqref{eqn:s-t_crossing_combo}$ to the point where $\RHS\eqref{eqn:s-t_crossing_combo}$ is expressed as a linear combination of the $N+1$ quantities $D_0,\dots,D_N$ defined by
		\begin{align*}
			D_n = \sum_{\ell \in I_{2k+2n}} |C_{i^{+n} i^{+n}}^{\ell}|^2,
		\end{align*}
		where
		\begin{align*}
			I_m = \{\ell \in I : \ell_1 < 0 \text{ and } \ell_2 = k_{-\ell_1} = m\}
		\end{align*}
		is the set of indices $\ell$ for which $\psi_{\ell}$ is a lowest weight vector in $L^2(\Gamma \backslash G)$ of weight $m$.
		
		Now, since we have $N+1$ degrees of freedom in the choice of $b_0,\dots,b_N$, it is plausible that we can choose the $b_n$ such that no matter what values the $N+1$ variables $D_n$ take, $\RHS\eqref{eqn:s-t_crossing_combo}$ always vanishes. Indeed, performing the elimination of variables explicitly yields an explicit formula for $\RHS\eqref{eqn:s-t_crossing_combo}$ in terms of the $D_n$, and from this formula one can see that such a choice of $b$'s exists.
	\end{proof}
	
	
	For use in the proofs of Proposition~\ref{prop:KMP_(2.47)} and Lemma~\ref{lem:nonvanishing_C_exists}, we record two instances of Lemma~\ref{lem:cancel_discrete_contribution}.
	
	\begin{lem} \label{lem:cancel_discrete_explicit}
		If $N = 0$ and $a_0 = 1$, then taking $b_0 = 2$ makes $\RHS\eqref{eqn:s-t_crossing_combo}$ vanish.
		If $N = 1$, $a_0 = -1$, and $a_1 = 1$, then taking
		\begin{align} \label{eqn:b0_b1_def}
			b_0 = -\frac{4k+2}{k+1}
			\qquad \text{and} \qquad
			b_1 = \frac{4k+2}{k+1}
		\end{align}
		makes $\RHS\eqref{eqn:s-t_crossing_combo}$ vanish.
	\end{lem}
	
	\begin{proof}[Sketch of proof]
		Carry out the proof of Lemma~\ref{lem:cancel_discrete_contribution} in these two cases.
	\end{proof}
	
	We can now prove Lemma~\ref{lem:nonvanishing_C_exists}.
	
	\begin{proof}[Proof of Lemma~\ref{lem:nonvanishing_C_exists}]
		Let $N=0$, $a_0=1$, and $b_0=2$.
		Then by Lemma~\ref{lem:cancel_discrete_explicit}, the equation \eqref{eqn:s-t_crossing_combo} reduces to
		\begin{align*}
			\sum_{\ell \in I} (-1)^{\ell_2} (C_{i\overline{i}}^{\ell} C_{i\overline{i}}^{\overline{\ell}} + 2C_{i\overline{i}}^{\ell} C_{i^+\overline{i^+}}^{\overline{\ell}})
			= 0.
		\end{align*}
		By the case $n=1$ of Corollary~\ref{cor:C_+n_-n_formula_explicit}, this further simplifies to
		\begin{align*}
			\sum_{\ell \in I} (-1)^{\ell_2} \Big[1 + 2\Big(\frac{\lambda_{\ell_1}}{2k}-1\Big)\Big] C_{i\overline{i}}^{\ell} C_{i\overline{i}}^{\overline{\ell}}
			= 0.
		\end{align*}
		When $\ell = (0,0)$, the expression in brackets is nonzero, so by Lemma~\ref{lem:SE4_l=0}, the summand is nonzero. On the other hand, the sum of all the terms is equal to zero, so at least one term besides $\ell = (0,0)$ must be nonzero.
		This forces $C_{i\overline{i}}^{\ell} \neq 0$ for some $\ell \neq (0,0)$.
		By \eqref{eqn:SE2} and \eqref{eqn:i2=0_implies_i1geq0}, such an $\ell$ must be of the form $\ell = (r,0)$ for some $r \geq 1$.
	\end{proof}
	
	We finally prove Proposition~\ref{prop:KMP_(2.47)}, completing the proof of Theorem~\ref{thm:KMP_basic_bd}.
	
	\begin{proof}[Proof of Proposition~\ref{prop:KMP_(2.47)}]
		Let $N=1$, $a_0=-1$, $a_1=1$, and $b_0,b_1$ as in \eqref{eqn:b0_b1_def}.
		Then by Lemma~\ref{lem:cancel_discrete_explicit}, the equation \eqref{eqn:s-t_crossing_combo} reduces to
		\begin{align} \label{eqn:s-t_crossing_combo_N=1}
			\sum_{\ell \in I} (-1)^{\ell_2} \Big[-C_{i\overline{i}}^{\ell} C_{i\overline{i}}^{\overline{\ell}}
			- \Big(\frac{4k+2}{k+1}\Big) C_{i\overline{i}}^{\ell} C_{i^+\overline{i^+}}^{\overline{\ell}}
			+ C_{i^+\overline{i^+}}^{\ell} C_{i^+\overline{i^+}}^{\overline{\ell}}
			+ \Big(\frac{4k+2}{k+1}\Big) C_{i^+\overline{i^+}}^{\ell} C_{i^{+2}\overline{i^{+2}}}^{\overline{\ell}} \Big]
			= 0.
		\end{align}
		By Corollary~\ref{cor:C_+n_-n_formula_explicit}, we have that for all $\ell \in I$,
		\begin{align*}
			C_{i^+\overline{i^+}}^{\ell}
			= \Big(\frac{\lambda_{\ell_1}}{2k} - 1\Big) C_{i\overline{i}}^{\ell}
			\qquad \text{and} \qquad
			C_{i^{+2}\overline{i^{+2}}}^{\ell}
			= \frac{1}{4k+2}\Big[(\lambda_{\ell_1}-6k-2) \Big(\frac{\lambda_{\ell_1}}{2k} - 1\Big) - 2k\Big] C_{i\overline{i}}^{\ell}.
		\end{align*}
		Inserting these into \eqref{eqn:s-t_crossing_combo_N=1},
		\begin{align} \label{eqn:s-t_crossing_R}
			\sum_{\ell \in I} (-1)^{\ell_2} R_k(\lambda_{\ell_1}) C_{i\overline{i}}^{\ell} C_{i\overline{i}}^{\overline{\ell}}
			= 0,
		\end{align}
		where
		\begin{align*}
			R_k(\lambda)
			= -1
			- \Big(\frac{4k+2}{k+1}\Big) \Big(\frac{\lambda}{2k}-1\Big)
			+ \Big(\frac{\lambda}{2k}-1\Big)^2
			+ \Big(\frac{1}{k+1}\Big) \Big(\frac{\lambda}{2k}-1\Big) \Big[(\lambda-6k-2) \Big(\frac{\lambda}{2k} - 1\Big) - 2k\Big].
		\end{align*}
		Expanding this formula,
		\begin{align} \label{eqn:R_k_formula_clean}
			R_k(\lambda)
			= \frac{\lambda(\lambda^2 - (9k+1)\lambda + 12 k^2)}{4 k^2 (k+1)}.
		\end{align}
		By \eqref{eqn:SE1} and \eqref{eqn:SE3}, we have $C_{i\overline{i}}^{\ell} C_{i\overline{i}}^{\overline{\ell}} = |C_{i\overline{i}}^{\ell}|^2$, and by \eqref{eqn:SE2} and \eqref{eqn:i2=0_implies_i1geq0}, we have $C_{i\overline{i}}^{\ell} = 0$ unless $\ell = (r,0)$ for some $r \geq 0$.
		Thus multiplying \eqref{eqn:s-t_crossing_R} by $4k^2(k+1)$ yields
		\begin{align} \label{eqn:KMP_(2.47)_r=0}
			\sum_{r=0}^{\infty} \lambda_r(\lambda_r^2 - (9k+1)\lambda_r + 12k^2) |C_{i\overline{i}}^{(r,0)}|^2
			= 0.
		\end{align}
		The term where $r=0$ vanishes, so this is equivalent to \eqref{eqn:KMP_(2.47)}.
	\end{proof}
	
	The motivation for choosing $a_0=-1$ and $a_1=1$ above is to make the $r=0$ term in \eqref{eqn:KMP_(2.47)_r=0} vanish; this condition determines $(a_0,a_1)$ uniquely up to scalars.
	
	\begin{rem} \label{rem:saturation}
		The above proof of Theorem~\ref{thm:KMP_basic_bd} used \eqref{eqn:s-t_crossing_combo} with $N=1$. The bounds in \cite[Table~2]{Kravchuk--Mazac--Pal} can be proved using \eqref{eqn:s-t_crossing_combo} with $N$ larger, optimizing the choice of parameters $a_0,\dots,a_N$ and $b_0,\dots,b_N$ (c.f. \cite[Section~3.9]{Kravchuk--Mazac--Pal}).
		For each $N$, there is some optimal upper bound on $\lambda_1$ that can be obtained from \eqref{eqn:s-t_crossing_combo} in this way, and one can ask if this bound converges as $N \to \infty$ to the largest known $\lambda_1$ in the table above.
		Although \cite{Kravchuk--Mazac--Pal} found that this bound comes close, \cite{Radcliffe} gives strong numerical evidence that it does not converge all the way to the largest known $\lambda_1$.
		Theorem~\ref{thm:main} suggests that by using more instances of \eqref{eqn:SE6} than just \eqref{eqn:s-t_crossing_=} and \eqref{eqn:s-t_crossing_+1}, it may be possible to prove sharper bounds which do converge to the truth in the limit as more and more hyperbolic bootstrap equations are used.
		Inspired by the work \cite{Viazovska,Cohn_et_al} and \cite{Mazac_17,Mazac--Paulos_I} on sphere packing and the 1d conformal bootstrap (which by \cite{Hartman--Mazac--Rastelli} are closely related),
		one could even speculate that there is an infinite linear combination of instances of \eqref{eqn:SE6}, playing the role of \eqref{eqn:s-t_crossing_combo}, which gives an upper bound exactly saturated by the $\lambda_1$'s in the table above.
		Such a linear combination is called an \emph{extremal functional} in the conformal bootstrap literature.
		It would be striking if an extremal functional could be found explicitly.
		Because of the inexplicit nature of Laplace eigenvalues, this seems much harder than the analogous problem of finding a magic function for sphere packing (``magic function" is defined for example in \cite{Cohn_24}).
		For further discussion, we refer to \cite[Section~5]{Kravchuk--Mazac--Pal} and \cite{Radcliffe}.
	\end{rem}
	
	\subsection{Subconvexity for triple product $L$-functions} \label{subsec:applications:triple}
	
	Let $\Gamma$ be a cocompact lattice in $G$, and let $\psi$ be an automorphic form on $\Gamma \backslash G$.
	Then $|\psi|^2$ has weight $0$, so $|\psi|^2$ is a function on $\Gamma \backslash G/K = \Gamma \backslash \mathbf{H}$, and thus can be expanded in a Laplace eigenbasis.
	To express this in our notation, let $\{\psi_i\}_{i \in I}$ be a $(\g,K)$-adapted basis of $L^2(\Gamma \backslash G)$, and let
	\begin{align*}
		\mathcal{S}
		= (\{\lambda_r\}_{r \geq 0}, \{k_r\}_{r \geq 1}, \{C_{ij}^{\ell}\}_{i,j,\ell \in I})
	\end{align*}
	be the multiplicative spectrum of $\Gamma \backslash \mathbf{H}$ with respect to $\{\psi_i\}_{i \in I}$.
	Then $\{\psi_{(r,0)}\}_{r \geq 0}$ is a Laplace eigenbasis of $\Gamma \backslash G/K = \Gamma \backslash \mathbf{H}$.
	Suppose $\psi = \psi_i$ for some $i \in I$.
	Then the expansion of $|\psi|^2$ in the above eigenbasis is
	\begin{align*}
		|\psi|^2
		= \psi_i \overline{\psi_i}
		= (-1)^{i_2} \psi_i \psi_{\overline{i}}
		= (-1)^{i_2} \sum_{\ell \in I} C_{i\overline{i}}^{\ell} \psi_{\ell}
		= (-1)^{i_2} \sum_{r=0}^{\infty} C_{i\overline{i}}^{(r,0)} \psi_{(r,0)}.
	\end{align*}
	Here we have used \eqref{eqn:psi_i_bar_vs_bar_psi_i} for the second equality and \eqref{eqn:SE2} and \eqref{eqn:i2=0_implies_i1geq0} for the last equality.
	
	The rate of decay of the coefficients $C_{i\overline{i}}^{(r,0)}$ as $r \to \infty$ has been well studied due to the connection to $L$-functions explained below. For more discussion and references, see \cite{Adve_et_al}.
	The sharpest
	estimates on the rate of decay
	can be recast as applications of the hyperbolic bootstrap equations.
	These estimates are summarized in the following theorem.
	In the rest of this subsection, implicit constants may depend on $\Gamma$.
	
	\begin{thm} \label{thm:triple_bds}
		Let $r,\rho \in \Z_{\geq 1}$ (think of $\rho$ as fixed, and think of $r$ as going to infinity).
		\begin{itemize} \itemsep = 0.5em
			\item \textnormal{(Maass case; \cite{Bernstein--Reznikov_10})}
			Let $i = (\rho,0)$. Then for every $\varepsilon > 0$,
			\begin{align} \label{eqn:triple_bd_Maass}
				|C_{i\overline{i}}^{(r,0)}|
				\lesssim_{\rho,\varepsilon} \lambda_r^{-\frac{1}{12}+\varepsilon} e^{-\frac{\pi}{2}\sqrt{\lambda_r}}.
			\end{align}
			
			\item \textnormal{(Holomorphic case; \cite{Adve_et_al})}
			Let $k = k_{\rho}$ and $i = (-\rho,k)$. Then for every $\varepsilon > 0$,
			\begin{align} \label{eqn:triple_bd_holo}
				|C_{i\overline{i}}^{(r,0)}|
				\lesssim_{\rho,\varepsilon} \lambda_r^{k-\frac{1}{6}+\varepsilon} e^{-\frac{\pi}{2}\sqrt{\lambda_r}}.
			\end{align}
		\end{itemize}
	\end{thm}
	
	In the former case, $\psi_i$ is a Maass form, and in the latter case, $\psi_i$ is a holomorphic modular form (in both cases viewed as functions on $\Gamma \backslash G$).
	
	The exponential decay in $\sqrt{\lambda_r}$ is a general phenomenon independent of the constant curvature of $\Gamma \backslash \mathbf{H}$ (see Subsection~\ref{subsec:outline:bds_G/Gamma} for a soft argument which gives such exponential decay).
	However, the fact that $\frac{\pi}{2}$ is the correct constant in the exponent does use constant curvature. This was first established in the generality of Theorem~\ref{thm:triple_bds} by Sarnak \cite{Sarnak_94}, who was motivated by applications to number theory as suggested by Selberg \cite{Selberg}.
	The achievement of \cite{Bernstein--Reznikov_10,Adve_et_al} is to give the best known polynomial factor in front of the exponential.
	One can formulate an analog of Theorem~\ref{thm:triple_bds} for non-cocompact $\Gamma$ as well (including Eisenstein series in the Laplace eigenbasis).
	Then specifically for $\Gamma = \PSL_2(\Z)$, Blomer, Jana, and Nelson \cite{Blomer--Jana--Nelson} showed that $\RHS\eqref{eqn:triple_bd_Maass}$ can be improved to
	\begin{align} \label{eqn:BJN_bd}
		\lambda_r^{-\frac{1}{6}+\varepsilon} e^{-\frac{\pi}{2}\sqrt{\lambda_r}}.
	\end{align}
	Their method should generalize to congruence subgroups of $\PSL_2(\Z)$.
	Aside from \cite{Blomer--Jana--Nelson}, no other improvement to \eqref{eqn:triple_bd_Maass} or \eqref{eqn:triple_bd_holo} is known even for any single $\Gamma$.
	
	Theorem~\ref{thm:triple_bds}
	is of particular interest when $\Gamma$ is arithmetic in a suitable sense, and when the basis elements $\psi_j$ are eigenfunctions for the Hecke operators.
	Then whenever $C_{i\overline{i}}^{(r,0)}$ is nonzero, the Watson--Ichino formula \cite{Watson,Ichino} expresses $|C_{i\overline{i}}^{(r,0)}|^2$ as the central value of a triple product $L$-function multiplied by an explicit constant of proportionality.
	The estimates \eqref{eqn:triple_bd_Maass} and \eqref{eqn:triple_bd_holo} translate to subconvex bounds for these $L$-functions in the spectral aspect, and in the holomorphic case, \eqref{eqn:triple_bd_holo} is a bound of \emph{Weyl quality} (see \cite{Michel} for an introduction to subconvexity).
	For those triple product $L$-functions related to one of the $C$'s coming from $\Gamma = \PSL_2(\Z)$, one also has the Weyl bound in the Maass case by \cite{Blomer--Jana--Nelson}.
	Experience with subconvexity for lower degree $L$-functions suggests that to improve upon the Weyl bound would require major new ideas.
	Thus $\frac{1}{6}$ in \eqref{eqn:triple_bd_holo} and \eqref{eqn:BJN_bd} is a significant threshold.
	
	At a high level, the proofs of \eqref{eqn:triple_bd_Maass} and \eqref{eqn:triple_bd_holo} are similar (see \cite[Section~6]{Adve_et_al} for a comparison).
	Here, we focus on \eqref{eqn:triple_bd_holo}.
	So let $k$ and $i$ be as in the holomorphic case of Theorem~\ref{thm:triple_bds}.
	Then the only information about $C_{i\overline{i}}^{(r,0)}$ used in the proof of \eqref{eqn:triple_bd_holo} is \cite[Corollary~3.8]{Adve_et_al}, which is equivalent to the following proposition.
	In the statement, ${}_2F_1$ is the Gauss hypergeometric function, and $s_r$ denotes one of the two solutions to the quadratic equation $s_r(1-s_r) = \lambda_r$ (it does not matter which one).
	
	\begin{prop} \label{prop:KMP_crossing}
		Let $z \in \C \setminus [1,\infty)$. Then
		\begin{align} \label{eqn:KMP_crossing}
			\sum_{r=0}^{\infty} {}_2F_1(s_r,1-s_r,1,z) |C_{i\overline{i}}^{(r,0)}|^2
			= \Big(\frac{1}{1-z}\Big)^{2k} \sum_{r=0}^{\infty} {}_2F_1\Big(s_r,1-s_r,1, \frac{z}{z-1}\Big) |C_{i\overline{i}}^{(r,0)}|^2,
		\end{align}
		with both sides converging absolutely and locally uniformly in $z$.
	\end{prop}
	
	From the point of view of the conformal bootstrap, \eqref{eqn:KMP_crossing} can be interpreted as a crossing equation, while from the point of view of the analytic theory of automorphic forms, \eqref{eqn:KMP_crossing} can be interpreted as a spectral reciprocity formula.
	
	The convergence assertion in Proposition~\ref{prop:KMP_crossing} is powerful by itself --- we will see when we prove \eqref{eqn:pi_correct_exponent} that convergence of $\LHS\eqref{eqn:KMP_crossing}$ implies
	\begin{align*}
		|C_{i\overline{i}}^{(r,0)}|
		\lesssim e^{-(\frac{\pi}{2}-o(1)) \sqrt{\lambda_r}}
	\end{align*}
	as $r \to \infty$.
	
	The plan for the remainder of this subsection is to give an alternative proof of Proposition~\ref{prop:KMP_crossing} based on the hyperbolic bootstrap equations.
	This proof is quite different from the one in \cite{Adve_et_al} (which in turn follows \cite{Kravchuk--Mazac--Pal}).
	We will derive \eqref{eqn:KMP_crossing} as a formal identity by taking linear combinations of crossing equations as in the proof of Proposition~\ref{prop:KMP_(2.47)}. This formal derivation uses only that $\mathcal{S}$ solves the hyperbolic bootstrap equations.
	Unlike in the proof of Proposition~\ref{prop:KMP_(2.47)}, we will take an infinite linear combination, and to prove convergence, we will need an additional estimate on the $C$'s (``additional" in the sense that we do not assume it in Theorem~\ref{thm:main}).
	This estimate is the lemma below, and its proof uses that $\mathcal{S}$ comes from a 2-orbifold.
	Specifically, the proof uses a Sobolev embedding inequality for functions on $\Gamma \backslash G$.
	
	\begin{lem} \label{lem:bulk_bd_for_Weyl}
		Let $j,j' \in I$. Then
		\begin{align} \label{eqn:bulk_bd_for_Weyl}
			\sum_{\ell \in I} |C_{jj'}^{\ell}|^2
			\lesssim_{j_1,j_1'} (|j_2|+|j_2'|+1)^{O(1)}.
		\end{align}
	\end{lem}
	
	\begin{proof}
		Write
		\begin{align*}
			\sum_{\ell \in I} |C_{jj'}^{\ell}|^2
			= \|\psi_j \psi_{j'}\|_{L^2(\Gamma \backslash G)}^2
			\leq \|\psi_j\|_{L^4(\Gamma \backslash G)}^2 \|\psi_{j'}\|_{L^4(\Gamma \backslash G)}^2.
		\end{align*}
		This reduces us to showing that
		\begin{align*}
			\|\psi_j\|_{L^4(\Gamma \backslash G)}
			\lesssim_{j_1} (|j_2|+1)^{O(1)}.
		\end{align*}
		By Sobolev embedding, it further suffices to show that
		\begin{align*}
			\|\mathcal{D}\psi_j\|_{L^2(\Gamma \backslash G)}
			\lesssim_{\mathcal{D},j_1} (|j_2|+1)^{O_{\mathcal{D}}(1)}
		\end{align*}
		for any differential operator $\mathcal{D}$ in the universal enveloping algebra $\mathfrak{U}(\g)$.
		Since $\g$ is spanned by $H,E,\overline{E}$, this follows from Proposition~\ref{prop:(g,K)-adapted_key_properties} by induction on the degree of $\mathcal{D}$.
	\end{proof}
	
	
	To prove Proposition~\ref{prop:KMP_crossing}, we will use the following instances of \eqref{eqn:SE6}, one for each $n \in \Z_{\geq 0}$:
	\begin{align} \label{eqn:u-t_crossing_n}
		\sum_{\ell \in I} (-1)^{\ell_2} C_{\overline{i}i^{+n}}^{\ell} C_{i\overline{i^{+n}}}^{\overline{\ell}}
		= \sum_{\ell \in I} (-1)^{\ell_2} C_{\overline{i}i}^{\ell} C_{i^{+n}\overline{i^{+n}}}^{\overline{\ell}}
	\end{align}
	(as usual, $i^{+n} = (i_1,i_2+n)$).
	For $z \in \C$ with $|z| < 1$, consider the infinite linear combination
	\begin{align} \label{eqn:u-t_crossing_combo}
		\sum_{n=0}^{\infty} (-1)^n \frac{(2k)_n}{n!} z^n \sum_{\ell \in I} (-1)^{\ell_2} C_{\overline{i}i^{+n}}^{\ell} C_{i\overline{i^{+n}}}^{\overline{\ell}}
		= \sum_{n=0}^{\infty} (-1)^n \frac{(2k)_n}{n!} z^n \sum_{\ell \in I} (-1)^{\ell_2} C_{\overline{i}i}^{\ell} C_{i^{+n}\overline{i^{+n}}}^{\overline{\ell}}
	\end{align}
	of \eqref{eqn:u-t_crossing_n}, where
	\begin{align*}
		(q)_n = q(q+1) \cdots (q+n-1)
	\end{align*}
	denotes the (rising) Pochhammer symbol (and $(q)_0 = 1$).
	By Cauchy--Schwarz and Lemma~\ref{lem:bulk_bd_for_Weyl}, both sides of \eqref{eqn:u-t_crossing_combo} converge absolutely and locally uniformly in $z$ for $|z| < 1$.
	Thus, to prove Proposition~\ref{prop:KMP_crossing} in the case $|z| < 1$, it suffices to prove the two lemmas below.
	
	\begin{lem} \label{lem:KMP_crossing=combo_LHS}
		Let $z \in \C$ with $|z| < 1$. Then $\LHS\eqref{eqn:u-t_crossing_combo}
		= \LHS\eqref{eqn:KMP_crossing}$.
	\end{lem}
	
	\begin{lem} \label{lem:KMP_crossing=combo_RHS}
		Let $z \in \C$ with $|z| < 1$. Then $\RHS\eqref{eqn:u-t_crossing_combo}
		= \RHS\eqref{eqn:KMP_crossing}$.
	\end{lem}
	
	Before we prove these two lemmas, let us see how the full Proposition~\ref{prop:KMP_crossing} can be deduced once we know it for $|z| < 1$.
	The key is the following hypergeometric asymptotic.
	
	\begin{lem} \label{lem:hyp_asymptotic}
		Fix $z \in \C \setminus [1,\infty)$. Then there exists $\delta = \delta(z) > 0$, such that if $s \in \C$ satisfies $s(1-s) = \lambda$ with $\lambda \gg_z 1$, then
		\begin{align} \label{eqn:2F1_upper_bd}
			|{}_2F_1(s,1-s,1,z)|
			\leq e^{(\pi-\delta)\sqrt{\lambda}}.
		\end{align}
		Furthermore, $\delta(z)$ may be taken to be locally uniformly bounded below as a function of $z$.
		
		Conversely, for each $\delta > 0$, if $z \in (0,1)$ is sufficiently close to $1$ depending on $\delta$, and if $s \in \C$ satisfies $s(1-s) = \lambda$ with $\lambda \gg_{\delta} 1$, then
		\begin{align} \label{eqn:2F1_lower_bd}
			{}_2F_1(s,1-s,1,z) \geq e^{(\pi-\delta) \sqrt{\lambda}}.
		\end{align}
	\end{lem}
	
	\begin{proof}[Idea of proof]
		Apply the method of steepest descent to Euler's integral representation for ${}_2F_1$.
	\end{proof}
	
	We do not give a complete proof of Lemma~\ref{lem:hyp_asymptotic} because it is tangential to our main story.
	
	\begin{proof}[Proof of Proposition~\ref{prop:KMP_crossing} for general $z$ assuming Proposition~\ref{prop:KMP_crossing} for $|z| < 1$]
		By assumption, $\LHS\eqref{eqn:KMP_crossing}$ converges absolutely for $|z| < 1$. It follows from \eqref{eqn:2F1_lower_bd} that for every $\delta>0$,
		\begin{align} \label{eqn:pi_correct_exponent}
			\sum_{r=0}^{\infty} |C_{i\overline{i}}^{(r,0)}|^2 e^{(\pi-\delta)\sqrt{\lambda_r}} < \infty.
		\end{align}
		Thus by \eqref{eqn:2F1_upper_bd}, both sides of \eqref{eqn:KMP_crossing} converge absolutely for all $z \in \C \setminus [1,\infty)$.
		Since $\delta = \delta(z)$ in \eqref{eqn:2F1_upper_bd} may be taken locally uniformly bounded below, both sides of \eqref{eqn:KMP_crossing} also converge locally uniformly.
		Hence both sides of \eqref{eqn:KMP_crossing} are well-defined holomorphic functions on all of $\C \setminus [1,\infty)$.
		By assumption, \eqref{eqn:KMP_crossing} holds for $|z| < 1$, so it holds for all $z \in \C \setminus [1,\infty)$ by analytic continuation.
	\end{proof}
	
	It remains to prove Lemmas~\ref{lem:KMP_crossing=combo_LHS} and \ref{lem:KMP_crossing=combo_RHS}.
	To do so, we need one final lemma, which will allow us to compute $\LHS\eqref{eqn:u-t_crossing_combo}$.
	
	\begin{lem} \label{lem:C_recurrence_for_Weyl}
		Let $\ell \in I$ and $n \in \Z_{\geq 0}$. Suppose $C_{\overline{i}i^{+n}}^{\ell} \neq 0$. Then $\ell = (r,n)$ for some $r \geq 0$, and
		\begin{align} \label{eqn:C_recurrence_for_Weyl}
			|C_{\overline{i}i^{+n}}^{\ell}|^2
			= \frac{(s_r)_n (1-s_r)_n}{n!(2k)_n} |C_{i\overline{i}}^{(r,0)}|^2.
		\end{align}
	\end{lem}
	
	\begin{proof}
		We induct on $n$.
		For the base case, suppose $n=0$.
		Then by \eqref{eqn:SE2} and \eqref{eqn:i2=0_implies_i1geq0}, nonvanishing of $C_{\overline{i}i^{+n}}^{\ell}$ forces $\ell = (r,0)$ for some $r \geq 0$.
		Since $n=0$, the quotient in $\RHS\eqref{eqn:C_recurrence_for_Weyl}$ is equal to $1$. Thus \eqref{eqn:C_recurrence_for_Weyl} holds by \eqref{eqn:SE1}.
		
		For the induction step, suppose $n>0$.
		Then by \eqref{eqn:SE5} and the fact that $\overline{i}^+ \not\in I$,
		\begin{align*}
			\sqrt{\lambda_{\ell_1} + \ell_2(\ell_2-1)} \, C_{\overline{i}i^{+(n-1)}}^{\ell^-}
			= \sqrt{\lambda_{i_1} + (i_2+n-1)(i_2+n)} \, C_{\overline{i} i^{+n}}^{\ell}.
		\end{align*}
		Inserting $\lambda_{i_1} = -k(k-1)$ and $i_2 = k$, and then squaring both sides,
		\begin{align*}
			(\lambda_{\ell_1} + \ell_2(\ell_2-1)) |C_{\overline{i}i^{+(n-1)}}^{\ell^-}|^2
			= (-k(k-1) + (k+n-1)(k+n)) |C_{\overline{i}i^{+n}}^{\ell}|^2.
		\end{align*}
		Simplifying the coefficient on the right hand side, and then dividing both sides by this coefficient,
		\begin{align} \label{eqn:C_recurrence_before_induction}
			|C_{\overline{i}i^{+n}}^{\ell}|^2
			= \frac{\lambda_{\ell_1} + \ell_2(\ell_2-1)}{n(2k+n-1)} |C_{\overline{i}i^{+(n-1)}}^{\ell^-}|^2.
		\end{align}
		By assumption, the left hand side is nonzero, so the right hand side must be nonzero. Then by induction, $\ell^- = (r,n-1)$ for some $r \geq 0$, so $\ell = (r,n)$.
		Furthermore, by induction, we can rewrite the right hand side using \eqref{eqn:C_recurrence_for_Weyl}, to get
		\begin{align} \label{eqn:C_recurrence_after_induction}
			|C_{\overline{i}i^{+n}}^{\ell}|^2
			= \frac{\lambda_r + n(n-1)}{n(2k+n-1)}
			\,
			\frac{(s_r)_{n-1} (1-s_r)_{n-1}}{(n-1)!(2k)_{n-1}} |C_{i\overline{i}}^{(r,0)}|^2.
		\end{align}
		We note that the numerator in the first quotient in $\RHS\eqref{eqn:C_recurrence_after_induction}$ agrees with the numerator in $\RHS\eqref{eqn:C_recurrence_before_induction}$ because $\ell_1 = r$ and $\ell_2 = n$.
		This numerator factors as
		\begin{align*} 
			\lambda_r + n(n-1)
			= s_r(1-s_r) + n(n-1)
			= (s_r+n-1) (1-s_r + (n-1)).
		\end{align*}
		Inserting this into \eqref{eqn:C_recurrence_after_induction} yields \eqref{eqn:C_recurrence_for_Weyl}.
	\end{proof}
	
	We are finally ready to prove Lemmas~\ref{lem:KMP_crossing=combo_LHS} and \ref{lem:KMP_crossing=combo_RHS}.
	
	\begin{proof}[Proof of Lemma~\ref{lem:KMP_crossing=combo_LHS}]
		By \eqref{eqn:SE3},
		\begin{align*}
			\LHS\eqref{eqn:u-t_crossing_combo}
			= \sum_{n=0}^{\infty} (-1)^n \frac{(2k)_n}{n!} z^n \sum_{\ell \in I} (-1)^{\ell_2} |C_{\overline{i}i^{+n}}^{\ell}|^2.
		\end{align*}
		By Lemma~\ref{lem:C_recurrence_for_Weyl}, this simplifies to
		\begin{align*}
			\LHS\eqref{eqn:u-t_crossing_combo}
			= \sum_{r=0}^{\infty} \Big[\sum_{n=0}^{\infty} \frac{(s_r)_n(1-s_r)_n}{n!^2} z^n \Big] |C_{i\overline{i}}^{(r,0)}|^2.
		\end{align*}
		The expression in brackets is the standard power series expansion of the hypergeometric which appears in $\LHS\eqref{eqn:KMP_crossing}$.
	\end{proof}
	
	\begin{proof}[Proof of Lemma~\ref{lem:KMP_crossing=combo_RHS}]
		Let $B_{k,n}$ be as in Corollary~\ref{cor:C_+n_-n_formula_explicit}.
		Then
		\begin{align*}
			\RHS\eqref{eqn:u-t_crossing_combo}
			= \sum_{\ell \in I} (-1)^{\ell_2} \Big[\sum_{n=0}^{\infty} (-1)^n \frac{(2k)_n}{n!} B_{k,n}(\lambda_{\ell_1}) z^n \Big] C_{\overline{i}i}^{\ell} C_{i\overline{i}}^{\overline{\ell}}.
		\end{align*}
		By \eqref{eqn:SE1} and \eqref{eqn:SE3}, we have $C_{\overline{i}i}^{\ell} C_{i\overline{i}}^{\overline{\ell}} = |C_{i\overline{i}}^{\ell}|^2$. By \eqref{eqn:SE2} and \eqref{eqn:i2=0_implies_i1geq0}, we get
		\begin{align*}
			\RHS\eqref{eqn:u-t_crossing_combo}
			= \sum_{r=0}^{\infty} \Big[\sum_{n=0}^{\infty} (-1)^n \frac{(2k)_n}{n!} B_{k,n}(\lambda_r) z^n \Big] |C_{i\overline{i}}^{(r,0)}|^2.
		\end{align*}
		Thus to prove Lemma~\ref{lem:KMP_crossing=combo_RHS}, we just need to show that
		\begin{align} \label{eqn:t-block_power_series}
			\sum_{n=0}^{\infty} (-1)^n \frac{(2k)_n}{n!} B_{k,n}(\lambda_r) z^n
			= \Big(\frac{1}{1-z}\Big)^{2k} {}_2F_1\Big(s_r,1-s_r,1, \frac{z}{z-1}\Big).
		\end{align}
		This can be done by using the recurrence \eqref{eqn:B_k,n_recurrence} and the hypergeometric ODE to check that both sides of \eqref{eqn:t-block_power_series} obey the same ODE, with the same initial conditions at $z=0$.
		We omit the details.
	\end{proof}
	
	The proof of Proposition~\ref{prop:KMP_crossing} is now complete.
	
	\section{Preliminaries on C*-algebras} \label{sec:C*}
	
	A \emph{C*-algebra} is a complex Banach algebra $\A$ with an antilinear involution $A \mapsto A^*$ such that
	\begin{align*}
		(AB)^* = B^*A^*
		\qquad \text{and} \qquad
		\|A^*A\| = \|A\|^2
	\end{align*}
	for all $A,B \in \A$. A morphism of C*-algebras (or \emph{C*-homomorphism}) is a norm-decreasing map which respects both the algebra structure and the involution.
	For us, the two most important examples of C*-algebras are:
	\begin{itemize} \itemsep = 0.5em
		\item Let $X$ be a compact Hausdorff space. Then $C(X)$ is a commutative unital C*-algebra with the uniform norm and with complex conjugation as the involution.
		
		\item Let $\H$ be a Hilbert space. Then $\B(\H)$ is a C*-algebra with the operator norm and with the adjoint as the involution. If $\A \subseteq \B(\H)$ is a subalgebra which is closed in the operator norm topology and closed under taking adjoints, then $\A$ is a C*-subalgebra.
	\end{itemize}
	
	Commutative unital C*-algebras were classified by Gelfand.
	There are many textbook accounts of the classification, such as \cite[Chapter~1]{Folland}.
	
	\begin{thm}[Gelfand duality] \label{thm:Gelfand_duality}
		The functor $X \mapsto C(X)$ is an equivalence from the opposite category of compact Hausdorff spaces to the category of commutative unital C*-algebras.
	\end{thm}
	
	Concretely, this means the following two things. The first says that $X \mapsto C(X)$ is fully faithful, and the second says that $X \mapsto C(X)$ is essentially surjective.
	\begin{enumerate} \itemsep = 0.5em
		\item[(1)] If $X,Y$ are compact Hausdorff spaces and $\Phi \colon C(X) \to C(Y)$ is a C*-homomorphism, then there is a unique continuous map $T \colon Y \to X$ such that $\Phi(f) = f \circ T$ for all $f \in C(X)$.
		
		\item[(2)] Every commutative unital C*-algebra $\A$ is isomorphic to $C(X)$ for some compact Hausdorff space $X$. It follows from (1) that $X$ is unique up to homeomorphism. The space $X$ is called the \emph{spectrum} of $\A$.
	\end{enumerate}
	The book \cite{Folland} cited above only proves essential surjectivity, but full faithfulness is easy to see from the proof of essential surjectivity.
	
	
	A compact Hausdorff space $X$ is metrizable if and only if $C(X)$ is separable (i.e., has a countable dense subset). Therefore, the spectrum of a commutative unital C*-algebra $\A$ is metrizable if and only if $\A$ is separable.
	
	A \emph{state} on a C*-algebra $\A$ is a positive linear functional of norm $1$. Positivity means that the functional takes a nonnegative value on $A^*A$ for all $A \in \A$. Given a Hilbert space $\H$, a unital C*-subalgebra $\A \subseteq \B(\H)$, and a vector $v \in \H$ of norm $1$, the \emph{Gelfand--Naimark--Segal (GNS) state} on $\A$ associated to $v$ is the functional $A \mapsto \langle Av,v \rangle_{\H}$.
	
	\section{Proofs of uniqueness results}
	\label{sec:uniqueness}
	
	In this section we prove Propositions~\ref{prop:mult_rep_uniqueness} and \ref{prop:mult_spectra_uniqueness}.
	Proposition~\ref{prop:mult_spectra_uniqueness} is essentially a corollary of Proposition~\ref{prop:mult_rep_uniqueness}, so we start with Proposition~\ref{prop:mult_rep_uniqueness}.
	Recall the statement:
	
	\begin{prop*}[Restatement of Proposition~\ref{prop:mult_rep_uniqueness}]
		Let $\Gamma,\Gamma'$ be cocompact lattices in $G$. Suppose $\Phi \colon L^2(\Gamma \backslash G) \to L^2(\Gamma' \backslash G)$ is an isomorphism of multiplicative representations.
		Then there exists a unique element $g \in G/\Gamma$ such that $\Gamma' = g\Gamma g^{-1}$ and
		\begin{align} \label{eqn:Phi_given_by_translation}
			\Phi(f)(x)
			= f(g^{-1}x)
		\end{align}
		for all $f \in L^2(\Gamma \backslash G)$.
	\end{prop*}
	
	We first prove a sequence of lemmas with notation as in Proposition~\ref{prop:mult_rep_uniqueness}.
	Recall that $L^2(\Gamma \backslash G)^{\fin}$ is dense in $C^{\infty}(\Gamma \backslash G)$, and in particular dense in $C(\Gamma \backslash G)$ with respect to the uniform norm.
	
	\begin{lem} \label{lem:Phi_preserves_L^infty}
		Let $f \in L^2(\Gamma \backslash G)^{\fin}$. Then $\|\Phi(f)\|_{L^{\infty}}
		= \|f\|_{L^{\infty}}$.
	\end{lem}
	
	
	\begin{proof}
		Using that $L^2(\Gamma' \backslash G)^{\fin}$ is dense in $L^2(\Gamma' \backslash G)$, that $\Phi$ is an isomorphism, and that $L^2(\Gamma \backslash G)^{\fin}$ is dense in $L^2(\Gamma \backslash G)$,
		\begin{align*}
			\|\Phi(f)\|_{L^{\infty}}
			= \sup_v \|\Phi(f)v\|_{L^2}
			= \sup_w \|\Phi(f) \Phi(w)\|_{L^2}
			= \sup_w \|fw\|_{L^2}
			= \|f\|_{L^{\infty}},
		\end{align*}
		with suprema over $v \in L^2(\Gamma' \backslash G)^{\fin}$ with $\|v\|_{L^2} \leq 1$ and $w \in L^2(\Gamma \backslash G)^{\fin}$ with $\|w\|_{L^2} \leq 1$.
	\end{proof}
	
	\begin{lem} \label{lem:Phi_iso_on_C}
		Let $f \in C(\Gamma \backslash G)$. Then $\Phi(f) \in C(\Gamma' \backslash G)$, and $\|\Phi(f)\|_{L^{\infty}} = \|f\|_{L^{\infty}}$.
	\end{lem}
	
	\begin{proof}
		Since $L^2(\Gamma \backslash G)^{\fin}$ is dense in $C(\Gamma \backslash G)$, there is a sequence $f_n \in L^2(\Gamma \backslash G)^{\fin}$ converging uniformly to $f$.
		Since $\Phi$ is an isometry of $L^2$ spaces, $\Phi(f_n) \to \Phi(f)$ in $L^2$.
		By Lemma~\ref{lem:Phi_preserves_L^infty}, the sequence $\Phi(f_n)$ is Cauchy with respect to the uniform norm, so in fact $\Phi(f_n) \to \Phi(f)$ uniformly. A uniform limit of continuous functions is continuous, so $\Phi(f)$ is continuous.
		Lemma~\ref{lem:Phi_preserves_L^infty} together with the uniform convergence of $\Phi(f_n)$ to $\Phi(f)$ implies that $\|\Phi(f)\|_{L^{\infty}} = \|f\|_{L^{\infty}}$.
	\end{proof}
	
	\begin{lem} \label{lem:Phi_algebra_iso}
		Let $f,h \in C(\Gamma \backslash G)$.
		Then $\Phi(fh) = \Phi(f) \Phi(h)$.
	\end{lem}
	
	\begin{proof}
		Let $f_n,h_n \in L^2(\Gamma \backslash G)^{\fin}$ with $f_n \to f$ and $h_n \to h$ uniformly. Then $f_nh_n \to fh$ uniformly. It follows from Lemma~\ref{lem:Phi_iso_on_C} that $\Phi(f_n)$, $\Phi(h_n)$, and $\Phi(f_nh_n)$ converge uniformly to $f$, $h$, and $fh$, respectively. Therefore
		\begin{align*}
			\Phi(fh)
			= \lim_{n \to \infty} \Phi(f_nh_n)
			= \lim_{n \to \infty} \Phi(f_n) \Phi(h_n)
			= \Phi(f) \Phi(h).
		\end{align*}
		Here the second equality is because $\Phi$ is an isomorphism of multiplicative representations.
	\end{proof}
	
	\begin{lem} \label{lem:Phi_preserves_conjugation}
		Let $f \in L^2(\Gamma \backslash G)$. Then $\Phi(\overline{f}) = \overline{\Phi(f)}$.
	\end{lem}
	
	\begin{proof}
		By continuity, it suffices to prove this for $f$ in $L^2(\Gamma \backslash G)^{\fin}$.
		Then since $\Phi$ is an isomorphism of multiplicative representations, $\Phi(\overline{f})$ is the complex conjugate of $\Phi(f)$ in the sense of Definition~\ref{def:mult_rep}. Complex conjugation on $L^2(\Gamma \backslash G)$ as a multiplicative representation coincides with complex conjugation of functions, so $\Phi(\overline{f}) = \overline{\Phi(f)}$.
	\end{proof}
	
	We are now ready to prove Proposition~\ref{prop:mult_rep_uniqueness}.
	
	\begin{proof}[Proof of Proposition~\ref{prop:mult_rep_uniqueness}]
		By Lemmas~\ref{lem:Phi_iso_on_C}, \ref{lem:Phi_algebra_iso}, and \ref{lem:Phi_preserves_conjugation}, $\Phi$ is a C*-homomorphism from $C(\Gamma \backslash G)$ to $C(\Gamma' \backslash G)$. Applying the same argument to $\Phi^{-1}$ shows that $\Phi$ is a C*-isomorphism.
		Thus by Gelfand duality (Theorem~\ref{thm:Gelfand_duality}), there is a unique homeomorphism $T \colon \Gamma' \backslash G \to \Gamma \backslash G$ such that $\Phi(f) = f \circ T$ for all $f \in C(\Gamma \backslash G)$.
		Since $\Phi$ is $G$-equivariant, Gelfand duality implies that $T$ is $G$-equivariant.
		Let $e$ denote the identity coset in $\Gamma' \backslash G$, and set $g = T(e)^{-1} \in G/\Gamma$.
		Then for $x \in G$,
		\begin{align} \label{eqn:T_formula}
			T(x)
			= T(ex)
			= T(e)x
			= g^{-1}x,
		\end{align}
		where the second equality is because $T$ is equivariant.
		In words, \eqref{eqn:T_formula} says that $T$ is given by left-translation by $g^{-1}$.
		Thus we have the following equality in $\Gamma \backslash G$:
		\begin{align*}
			g^{-1}\Gamma'
			= T(\Gamma')
			= T(e)
			= g^{-1}.
		\end{align*}
		The fact that this holds in $\Gamma \backslash G$ means that $g^{-1}\Gamma' \subseteq \Gamma g^{-1}$ as subsets of $G$.
		Rearranging, $\Gamma' \subseteq g\Gamma g^{-1}$.
		Since $T$, given by \eqref{eqn:T_formula}, is a homeomorphism, we must have $\Gamma' = g\Gamma g^{-1}$.
		The formula \eqref{eqn:T_formula} implies \eqref{eqn:Phi_given_by_translation} for all $f \in C(\Gamma \backslash G)$.
		Since $\Phi$ is an isometry of $L^2$ spaces, \eqref{eqn:Phi_given_by_translation} extends to all $f \in L^2(\Gamma \backslash G)$ by continuity.
		Since $T$ is unique, the uniqueness of $g$ in Proposition~\ref{prop:mult_rep_uniqueness} is clear.
	\end{proof}
	
	Proposition~\ref{prop:mult_spectra_uniqueness} follows easily from Proposition~\ref{prop:mult_rep_uniqueness}.
	Recall the statement of Proposition~\ref{prop:mult_spectra_uniqueness}:
	
	\begin{prop*}[Restatement of Proposition~\ref{prop:mult_spectra_uniqueness}]
		Let $(\Gamma \backslash \mathbf{H}, \{\psi_i\}_{i \in I})$ and $(\Gamma' \backslash \mathbf{H}, \{\psi_i'\}_{i \in I'})$ be compact hyperbolic 2-orbifolds equipped with $(\g,K)$-adapted bases.
		If both have the same multiplicative spectrum, then $I = I'$, and there exists a unique element $g \in G/\Gamma$ such that $\Gamma' = g\Gamma g^{-1}$ and
		\begin{align*}
			\psi_i'(x)
			= \psi_i(g^{-1}x)
		\end{align*}
		for all $i \in I$.
	\end{prop*}
	
	\begin{proof}
		Since $\Gamma \backslash \mathbf{H}$ and $\Gamma' \backslash \mathbf{H}$ have the same Laplace and holomorphic spectra, $L^2(\Gamma \backslash G, \R)$ and $L^2(\Gamma' \backslash G, \R)$ are isomorphic as real unitary representations, and $I = I'$.
		Thus by Proposition~\ref{prop:(g,K)-adapted_exist_unique}, there exists a unique isomorphism $\Phi_{\R} \colon L^2(\Gamma \backslash G, \R) \to L^2(\Gamma' \backslash G, \R)$ whose complexification $\Phi$ satisfies $\Phi(\psi_i) = \psi_i'$ for all $i \in I$.
		Since $(\Gamma \backslash \mathbf{H}, \{\psi_i\}_{i \in I})$ and $(\Gamma' \backslash \mathbf{H}, \{\psi_i'\}_{i \in I})$ have the same multiplication table, $\Phi$ is an isomorphism of multiplicative representations.
		Therefore Proposition~\ref{prop:mult_rep_uniqueness} implies that there exists an element $g \in G/\Gamma$ satisfying the desired properties.
		Since $\Phi_{\R}$ is unique, $g$ is unique.
	\end{proof}
	
	\section{Elementary properties of complex conjugation on $\H^{\fin}$} \label{sec:complex_conj}
	
	The remainder of the paper is dedicated to the proof of Theorem~\ref{thm:eq_Gelfand_duality}.
	So from now on, let $\H$ be a multiplicative representation.
	
	We show in this section that complex conjugation on $\H^{\fin}$, as defined by the ``existence of complex conjugates" axiom in Definition~\ref{def:mult_rep}, obeys the expected properties.
	In particular, we verify all the assertions in Remark~\ref{rem:real_subspace}.
	These properties will be used freely without comment later on.
	
	\begin{prop} \label{prop:1_bar=1}
		One has $\overline{\mathbf{1}} = \mathbf{1}$.
	\end{prop}
	
	\begin{proof}
		For all $\alpha,\beta \in \H^{\fin}$,
		\begin{align*}
			\langle \alpha,\beta \rangle_{\H}
			= \langle \mathbf{1}\alpha, \beta \rangle_{\H}
			= \langle \alpha, \overline{\mathbf{1}}\beta \rangle_{\H}.
		\end{align*}
		Since $\H^{\fin}$ is dense in $\H$, it follows that $\beta = \overline{\mathbf{1}}\beta$. Thus $\overline{\mathbf{1}}$ is a unit, and hence $\overline{\mathbf{1}} = \mathbf{1}$ because the unit is unique.
	\end{proof}
	
	\begin{prop} \label{prop:alpha_bar_bar=alpha}
		Let $\alpha \in \H^{\fin}$. Then $\overline{\overline{\alpha}} = \alpha$.
	\end{prop}
	
	\begin{proof}
		For all $\beta,\gamma \in \H^{\fin}$,
		\begin{align*}
			\langle \overline{\alpha} \beta, \gamma \rangle_{\H}
			= \overline{\langle \gamma, \overline{\alpha} \beta \rangle_{\H}}
			= \overline{\langle \alpha\gamma, \beta \rangle_{\H}}
			= \langle \beta, \alpha\gamma \rangle_{\H}.
		\end{align*}
		This shows that $\alpha$ satisfies the adjointness property characterizing $\overline{\overline{\alpha}}$.
	\end{proof}
	
	\begin{prop} \label{prop:inner_product_bar}
		Let $\alpha,\beta \in \H^{\fin}$. Then
		\begin{align*}
			\langle \overline{\alpha},\overline{\beta} \rangle_{\H}
			= \overline{\langle \alpha,\beta \rangle_{\H}}.
		\end{align*}
	\end{prop}
	
	\begin{proof}
		Using that $\overline{\overline{\alpha}} = \alpha$ by Proposition~\ref{prop:alpha_bar_bar=alpha},
		\begin{align*}
			&\langle \overline{\alpha},\overline{\beta} \rangle_{\H}
			= \langle \overline{\alpha},\overline{\beta} \mathbf{1} \rangle_{\H}
			= \langle \overline{\alpha} \beta, \mathbf{1} \rangle_{\H}
			= \langle \beta,\alpha \rangle_{\H}
			= \overline{\langle \alpha,\beta \rangle_{\H}}.
			\qedhere
		\end{align*}
	\end{proof}
	
	\begin{prop} \label{prop:bar_antilinear}
		Complex conjugation is antilinear, i.e., for all $\alpha,\beta \in \H^{\fin}$ and $c \in \C$,
		\begin{align*}
			\overline{c\alpha+\beta}
			= \overline{c} \, \overline{\alpha} + \overline{\beta}.
		\end{align*}
	\end{prop}
	
	\begin{proof}
		Clearly $\overline{c} \, \overline{\alpha} + \overline{\beta}$ obeys the adjointness property characterizing $\overline{c\alpha+\beta}$.
	\end{proof}
	
	Putting Propositions~\ref{prop:alpha_bar_bar=alpha}, \ref{prop:inner_product_bar}, and \ref{prop:bar_antilinear} together, we obtain
	
	\begin{prop} \label{prop:bar_extends}
		Complex conjugation extends by continuity to an antilinear involution on $\H$, denoted $v \mapsto \overline{v}$, with
		\begin{align} \label{eqn:bar_extended_preserves_norm}
			\langle \overline{v},\overline{w} \rangle_{\H}
			= \overline{\langle v,w \rangle_{\H}}
			\qquad \text{and in particular} \qquad
			\|\overline{v}\|_{\H}
			= \|v\|_{\H}
		\end{align}
		for all $v,w \in \H$.
	\end{prop}
	
	\begin{proof}
		By Propositions~\ref{prop:inner_product_bar} and \ref{prop:bar_antilinear}, complex conjugation on $\H^{\fin}$ is norm-preserving and antilinear, so it extends by continuity to $\H$. The stated properties of complex conjugation on $\H$ then follow by continuity from the analogous properties on $\H^{\fin}$, which are due to Propositions~\ref{prop:alpha_bar_bar=alpha}, \ref{prop:inner_product_bar}, and~\ref{prop:bar_antilinear}.
	\end{proof}
	
	From now on, whenever we use complex conjugation on $\H$, we understand that it is defined through Proposition~\ref{prop:bar_extends}. Given a linear subspace $\E \subseteq \H$, denote
	\begin{align*}
		\E_{\R} = \{v \in \E : \overline{v} = v\}.
	\end{align*}
	This is an $\R$-linear subspace of $\E$. For each $v \in \H$, there is a unique way to write $v = \re v + i \im v$ for some $\re v, \im v \in \H_{\R}$, namely by taking
	\begin{align*}
		\re v = \frac{1}{2}(v+\overline{v})
		\qquad \text{and} \qquad
		\im v = \frac{1}{2i}(v-\overline{v}).
	\end{align*}
	The existence and uniqueness of this decomposition into real and imaginary parts means that $\H = \H_{\R} \otimes_{\R} \C$ as vector spaces.
	
	An immediate corollary of \eqref{eqn:bar_extended_preserves_norm} in Proposition~\ref{prop:bar_extends} is
	
	\begin{cor} \label{cor:inner_product_real}
		Let $v,w \in \H_{\R}$. Then $\langle v,w \rangle_{\H} \in \R$.
	\end{cor}
	
	\begin{prop} \label{prop:norm_re_im_pythagoras}
		Let $v \in \H$. Then $\|v\|_{\H}^2 = \|\re v\|_{\H}^2 + \|\im v\|_{\H}^2$.
	\end{prop}
	
	\begin{proof}
		Expand $\|\re v + i \im v\|_{\H}^2$ and note that the cross terms vanish because of Corollary~\ref{cor:inner_product_real}.
	\end{proof}
	
	\begin{prop} \label{prop:conj_commutes_w_mult}
		Let $\alpha,\beta \in \H^{\fin}$. Then $\overline{\alpha\beta} = \overline{\alpha}\overline{\beta}$. In particular if $\alpha,\beta \in \H_{\R}^{\fin}$, then $\alpha\beta \in \H_{\R}^{\infty}$.
	\end{prop}
	
	\begin{proof}
		Let $\gamma \in \H^{\fin}$ be arbitrary. Then by \eqref{eqn:bar_extended_preserves_norm} in Proposition~\ref{prop:bar_extends},
		\begin{align*}
			\langle \overline{\alpha\beta}, \gamma \rangle_{\H}
			= \overline{\langle \alpha\beta, \overline{\gamma} \rangle_{\H}}
			= \overline{\langle \beta, \overline{\alpha} \, \overline{\gamma} \rangle_{\H}}
			= \overline{\langle \beta\gamma, \overline{\alpha} \rangle_{\H}}
			= \overline{\langle \gamma, \overline{\alpha} \overline{\beta} \rangle_{\H}}
			= \langle \overline{\alpha} \overline{\beta}, \gamma \rangle_{\H}.
		\end{align*}
		Since $\gamma$ was arbitrary, it follows by density that $\overline{\alpha\beta} = \overline{\alpha} \overline{\beta}$.
	\end{proof}
	
	
	\begin{prop} \label{prop:bar_equivariant}
		The $G$-action respects complex conjugation, i.e., $\overline{gv} = g\overline{v}$ for all $g \in G$ and $v \in \H$. Consequently $\overline{Xv} = \overline{X} \overline{v}$ for all $X \in \g$ and $v \in \H^{\infty}$.
	\end{prop}
	
	\begin{proof}
		We prove these two assertions in the reverse order, so we start with $\overline{Xv} = \overline{X} \overline{v}$. By density, it suffices to show that $\overline{X\alpha} = \overline{X}\overline{\alpha}$ for all $\alpha \in \H^{\fin}$. Let $\beta \in \H^{\fin}$ be arbitrary. Since $\mathbf{1}$ is $G$-invariant, we have $\overline{X}\mathbf{1} = 0$, so using \eqref{eqn:Lie_alg_unitarity} and the product rule, we can write
		\begin{align*}
			0 = -\langle \alpha\beta, \overline{X}\mathbf{1} \rangle_{\H}
			= \langle X(\alpha\beta), \mathbf{1} \rangle_{\H}
			&= \langle (X\alpha)\beta, \mathbf{1} \rangle_{\H} + \langle \alpha(X\beta), \mathbf{1} \rangle_{\H}
			\\&= \langle \beta, \overline{X\alpha} \rangle_{\H} + \langle X\beta, \overline{\alpha} \rangle_{\H}
			\\&= \langle \beta,\overline{X\alpha} \rangle_{\H} - \langle \beta, \overline{X}\overline{\alpha} \rangle_{\H}.
		\end{align*}
		Since $\beta$ was arbitrary, it follows by density that $\overline{X\alpha} = \overline{X}\overline{\alpha}$, as desired.
		
		We now show that $\overline{gv} = g\overline{v}$ for all $g \in G$ and $v \in \H$. By density, it suffices to prove this for $v$ a real analytic vector, i.e., such that $g \mapsto gv$ is a real analytic function from $G$ to $\H$. So fix $v \in \H$ real analytic. Consider $\Omega = \{g \in G : \overline{gv} = g\overline{v}\}$. This is evidently a closed subset of $G$. We claim that it is also open. To see this, let $g_0 \in \Omega$, and denote $w = g_0v$. We must show $g \in \Omega$ for $g$ sufficiently close to $g_0$. Write $g = \exp(X)g_0$ with $X \in \g_{\R}$ sufficiently small. Then
		\begin{align*}
			g \in \Omega
			\qquad \iff \qquad
			\overline{\exp(X)w}
			= \exp(X) \overline{w}.
		\end{align*}
		Since $v$ is real analytic and $w$ is a translate of $v$, we have that $w$ is also real analytic, so for $X$ sufficiently small, we can Taylor expand both sides to get
		\begin{align*}
			g \in \Omega
			\qquad \iff \qquad
			\sum_{n=0}^{\infty} \frac{1}{n!} \overline{X^nw}
			= \sum_{n=0}^{\infty} \frac{1}{n!} X^n\overline{w}.
		\end{align*}
		Noting that $X \in \g_{\R}$, this is true by the previous paragraph. Thus $\Omega$ is indeed open. Since $G$ is connected, we conclude that $\Omega = G$, so $\overline{gv} = g\overline{v}$ for all $g \in G$.
	\end{proof}
	
	
	\section{$L^{\infty}$ and $L^4$ norms on $\H^{\fin}$}
	\label{sec:L^infty_L^4}
	
	For each $\alpha \in \H^{\fin}$, let $M_{\alpha}$ denote the densely defined operator on $\H$, with domain $\H^{\fin}$, given by multiplication by $\alpha$. Define $\|\alpha\|_{L^{\infty}}$ to be the operator norm of $M_{\alpha}$, that is, the smallest constant $C \in [0,\infty]$ such that
	\begin{align*}
		\|\alpha\beta\|_{\H}
		\leq C\|\beta\|_{\H}
		\qquad \text{for all} \qquad
		\beta \in \H^{\fin}.
	\end{align*}
	The motivation for this definition is that whenever $f$ is a function on a measure space, the $L^{\infty}$ norm of $f$ can be characterized as the operator norm of multiplication-by-$f$ acting on $L^2$ (we already used this idea in the proof of Lemma~\ref{lem:Phi_preserves_L^infty}).
	Indeed, in the case $\H = L^2(\Gamma \backslash G)$, the $L^{\infty}$ norm on $\H^{\fin}$ coincides with the usual $L^{\infty}$ norm on functions on $\Gamma \backslash G$.
	In general, it is not at all clear that $\|\alpha\|_{L^{\infty}} < \infty$ for $\alpha \in \H^{\fin}$. This finiteness is the $n=0$ case of Theorem~\ref{thm:high_deriv_L^infty_bd}.
	
	Define the $L^4$ norm of $\alpha \in \H^{\fin}$ by
	\begin{align*}
		\|\alpha\|_{L^4}
		= \|\alpha^2\|_{\H}^{\frac{1}{2}}.
	\end{align*}
	This evidently agrees with the usual $L^4$ norm when $\H = L^2(\Gamma \backslash G)$.
	In general, for $\alpha \in \H^{\fin}$, we have $\alpha^2 \in \H^{\infty} \subseteq \H$, so $\|\alpha\|_{L^4} < \infty$. By crossing symmetry (the last axiom in Definition~\ref{def:mult_rep}),
	\begin{align} \label{eqn:|alpha|^2_H_norm}
		\|\alpha\|_{L^4}^4
		= \|\alpha^2\|_{\H}^2
		= \||\alpha|^2\|_{\H}^2.
	\end{align}
	
	Since complex conjugation commutes with multiplication and preserves the Hilbert space norm,
	\begin{align*}
		\|\overline{\alpha}\|_{L^{\infty}}
		= \|\alpha\|_{L^{\infty}}
		\qquad \text{and} \qquad
		\|\overline{\alpha}\|_{L^4}
		= \|\alpha\|_{L^4}
	\end{align*}
	for all $\alpha \in \H^{\fin}$. Since $M_{\alpha+\beta} = M_{\alpha} + M_{\beta}$, the $L^{\infty}$ norm satisfies the triangle inequality
	\begin{align*}
		\|\alpha+\beta\|_{L^{\infty}}
		\leq \|\alpha\|_{L^{\infty}} + \|\beta\|_{L^{\infty}}
	\end{align*}
	for all $\alpha,\beta \in \H^{\fin}$. The triangle inequality for the $L^4$ norm is Proposition~\ref{prop:L^4_triangle_ineq} below.
	
	\begin{prop}[$L^4$-Cauchy--Schwarz] \label{prop:L^4_C-S}
		Let $\alpha,\beta \in \H^{\fin}$. Then
		\begin{align*}
			\|\alpha\beta\|_{\H}
			\leq \|\alpha\|_{L^4} \|\beta\|_{L^4}.
		\end{align*}
	\end{prop}
	
	\begin{proof}
		By crossing symmetry, Cauchy--Schwarz, and \eqref{eqn:|alpha|^2_H_norm},
		\begin{align*}
			&\|\alpha\beta\|_{\H}^2
			= \langle |\alpha|^2, |\beta|^2 \rangle_{\H}
			\leq \||\alpha|^2\|_{\H} \||\beta|^2\|_{\H}
			= \|\alpha\|_{L^4}^2 \|\beta\|_{L^4}^2.
			\qedhere
		\end{align*}
	\end{proof}
	
	\begin{prop}[Triangle inequality for the $L^4$ norm] \label{prop:L^4_triangle_ineq}
		Let $\alpha,\beta \in \H^{\fin}$. Then
		\begin{align*}
			\|\alpha+\beta\|_{L^4}
			\leq \|\alpha\|_{L^4} + \|\beta\|_{L^4}.
		\end{align*}
	\end{prop}
	
	\begin{proof}
		Write
		\begin{align*}
			\|\alpha+\beta\|_{L^4}^2
			= \|(\alpha+\beta)^2\|_{\H}
			&\leq \|\alpha^2\|_{\H} + 2\|\alpha\beta\|_{\H} + \|\beta^2\|_{\H}
			\\&\leq \|\alpha\|_{L^4}^2 + 2\|\alpha\|_{L^4}\|\beta\|_{L^4} + \|\beta\|_{L^4}^2
			= (\|\alpha\|_{L^4} + \|\beta\|_{L^4})^2,
		\end{align*}
		where the inequality between the first and second lines is by $L^4$-Cauchy--Schwarz.
	\end{proof}
	
	\begin{prop}[Comparison of norms] \label{prop:L^2_leq_L^4_leq_L^infty}
		Let $\alpha \in \H^{\fin}$. Then
		\begin{align*}
			\|\alpha\|_{\H} \leq \|\alpha\|_{L^4} \leq \|\alpha\|_{L^{\infty}}.
		\end{align*}
	\end{prop}
	
	\begin{proof}
		By $L^4$-Cauchy--Schwarz and the normalization axiom in Definition~\ref{def:mult_rep},
		\begin{align*}
			\|\alpha\|_{\H}
			= \|\alpha\mathbf{1}\|_{\H}
			\leq \|\alpha\|_{L^4} \|\mathbf{1}\|_{L^4}
			= \|\alpha\|_{L^4}.
		\end{align*}
		This gives the first inequality. For the second, write
		\begin{align*}
			&\|\alpha\|_{L^4}^2
			= \|\alpha^2\|_{\H}
			\leq \|\alpha\|_{L^{\infty}} \|\alpha\|_{\H}
			= \|\alpha\|_{L^{\infty}} \|\alpha\mathbf{1}\|_{\H}
			\leq \|\alpha\|_{L^{\infty}}^2 \|\mathbf{1}\|_{\H}
			= \|\alpha\|_{L^{\infty}}^2.
			\qedhere
		\end{align*}
	\end{proof}
	
	
	\section{Outline of proof of the converse theorem for multiplicative representations} \label{sec:outline_existence}
	
	%
	Theorem~\ref{thm:eq_Gelfand_duality} says that if $\H$ is nontrivial, then $\H \simeq L^2(\Gamma \backslash G)$ for some cocompact lattice $\Gamma$ in $G$.
	In this section, we outline the proof of this theorem.
	
	
	We will use the notation for the functional calculus for the Casimir from Subsection~\ref{subsec:sl_2:func_calc}, as well as the $L^{\infty}$ and $L^4$ norms on $\H^{\fin}$ defined in Section~\ref{sec:L^infty_L^4}.
	
	\subsection{Roadmap}
	\label{subsec:outline:roadmap}
	
	The purpose of this subsection is to describe the structure of the proof and to state the most important intermediate results. For each of these intermediate results, we will either give some intuition as to why it is true, or in the next subsection we will explain why it is true in the case $\H = L^2(\Gamma \backslash G)$. Of course, by Theorem~\ref{thm:eq_Gelfand_duality}, this is the only nontrivial case, but it would be circular to use this. The proofs of these intermediate results in the general case are outlined in Subsections~\ref{subsec:outline:embedding_homog}--\ref{subsec:outline:bulk_tail_general}.
	
	We break the proof of Theorem~\ref{thm:eq_Gelfand_duality} into five steps.
	The first two are quantitative, and the last three are qualitative.
	
	\begin{step} \label{step:bulk_tail}
		Given $\lambda \geq 0$ and approximate Casimir eigenvectors $\alpha,\beta \in \H^K$ with approximate eigenvalue $\lambda$ (meaning that $\alpha,\beta \in \H_{\Delta \in I}^K$ for some small interval $I$ around $\lambda$), estimate the product $\alpha\beta$. More precisely, bound the Hilbert space norm of both the ``bulk" $\1_{\Delta \lesssim \lambda}(\alpha\beta)$ and the ``tail" $\1_{\Delta \gg \lambda}(\alpha\beta)$. In addition, given a lowest weight vector $f$, bound the tail of $|f|^2$.
		
		These three bounds correspond to Theorems~\ref{thm:L^4_quasi-Sobolev}, \ref{thm:exp_decay_quasimode}, and \ref{thm:exp_decay_form}, respectively. They are proved in Sections~\ref{sec:bulk_tail_poly} and \ref{sec:bulk_tail_general}.
	\end{step}
	
	\begin{step} \label{step:L^infty}
		From the results of Step~\ref{step:bulk_tail}, deduce two estimates on $L^{\infty}$ norms of automorphic vectors, the first with explicit dependence on the weight, and the second with explicit dependence on the Casimir eigenvalue.
		
		These two estimates correspond to Theorems~\ref{thm:high_deriv_L^infty_bd} and \ref{thm:L^infty_quasi-Sobolev}, and they are proved in Section~\ref{sec:reduce_to_bulk_tail}. Importantly, both of these theorems apply not just to automorphic vectors, but more generally to elements of $\H^{\fin}$. The case $n=0$ of Theorem~\ref{thm:high_deriv_L^infty_bd} says that the $L^{\infty}$ norm is finite-valued on $\H^{\fin}$. Thus for each $\alpha \in \H^{\fin}$, the multiplication operator $M_{\alpha}$ extends from $\H^{\fin}$ to a bounded linear operator on $\H$. We thereby view $M_{\alpha} \in \B(\H)$.
	\end{step}
	
	\begin{step} \label{step:C*}
		Let $\A \subseteq \B(\H)$ be the closed subalgebra generated by $\{M_{\alpha} : \alpha \in \H^{\fin}\}$.
		Show that $\A$ is closed under taking adjoints, so $\A$ is a C*-algebra. Show that $\A$ is commutative, unital, and separable. The unitary group $\U(\H)$ acts on $\B(\H)$ by conjugation; pull this action back along the representation $G \to U(\H)$ to obtain an action of $G$ on $\B(\H)$. 
		Use Theorem~\ref{thm:high_deriv_L^infty_bd} from Step~\ref{step:L^infty} to show that this $G$-action preserves $\A$, and that the action map $G \times \A \to \A$ is continuous.
		
		This is all done in Subsection~\ref{subsec:reduce_to_L^infty:C*-alg}.
	\end{step}
	
	\begin{step} \label{step:spectrum}
		By Gelfand duality, $\A \simeq C(X)$ for a unique compact Hausdorff space $X$, and the $G$-action on $\A$ induces a $G$-action on $X$. Since $\A$ is separable, $X$ is metrizable. Using that $G \times \A \to \A$ is continuous, show that $G \times X \to X$ is continuous. 
		Consider the GNS state on $\A \simeq C(X)$ associated to $\mathbf{1} \in \H$.
		By the Riesz representation theorem, this state is given by integration against a probability measure $\mu$ on $X$. Show that $\mu$ is $G$-invariant and has full support. Show that $\H$ identifies with $L^2(X,\mu)$ as unitary $G$-representations, in such a way that multiplication on $\H^{\fin}$ is given by pointwise multiplication on $X$. Translate known properties of $\H$ as a multiplicative representation into properties of $(X,\mu)$. In particular, translate Corollary~\ref{cor:qual_Sob_emb} of Theorem~\ref{thm:L^infty_quasi-Sobolev} from Step~\ref{step:L^infty} into the \emph{qualitative Sobolev embedding property} (Definition~\ref{defn:qualitative_Sob_emb}) for $(X,\mu)$.
		
		This is all done in Subsection~\ref{subsec:reduce_to_L^infty:spectrum}.
	\end{step}
	
	\begin{step} \label{step:embedding_G/Gamma}
		Deduce from the qualitative Sobolev embedding property that $X \simeq \Gamma \backslash G$ as $G$-spaces for some cocompact lattice $\Gamma$, and that $\mu$ is Haar measure.
		
		This deduction is Theorem~\ref{thm:embedding_G/Gamma}, proved in Section~\ref{sec:embedding_homog}.
	\end{step}
	
	
	As can be seen from the section numbering, we present these five steps almost backward, in the order \ref{step:embedding_G/Gamma}, \ref{step:C*}, \ref{step:spectrum}, \ref{step:L^infty}, \ref{step:bulk_tail}.
	This way, the end goal is always in sight.
	
	We now indicate the main results that make up each step (in the above order). The key definition for Step~\ref{step:embedding_G/Gamma} is

	%
	
	\begin{defn}[Qualitative Sobolev embedding property] \label{defn:qualitative_Sob_emb}
		Let $X$ be a $G$-space equipped with a $G$-invariant probability measure $\mu$. We say that $(X,\mu)$ has the \emph{qualitative Sobolev embedding property} if there exist $1 \leq p < q \leq \infty$ and a nonnegative continuous function $\varphi \in L^1(G)$, not identically zero, such that for all $f \in L^p(X,\mu)$, the convolution
		\begin{align} \label{eqn:phi_star_f_def}
			\varphi \ast f(x)
			= \int_G \varphi(g) f(g^{-1}x) \, dg
		\end{align}
		lies in $L^q(X,\mu)$.
	\end{defn}
	
	One should think of the operator $f \mapsto \varphi \ast f$ as ``smoothing along $G$-orbits." Then the slogan for the qualitative Sobolev embedding property is that ``smoothing along $G$-orbits improves integrability." One might worry that the source of improvement of integrability is cancellation in the integral \eqref{eqn:phi_star_f_def},
	not smoothing. We require $\varphi$ to be nonnegative so that when $f$ is also nonnegative, there can be no cancellation in \eqref{eqn:phi_star_f_def}.
	We require $\varphi$ to be continuous (and not identically zero) so that it is not just positive somewhere, but positive on a nonempty open set.
	
	By the end of Step~\ref{step:spectrum}, we will have established the following weak version of the converse theorem for multiplicative representations.
	
	\begin{thm}[Weak converse theorem] \label{thm:weak_structure_thm}
		There exists a connected compact metrizable $G$-space $X$ together with a $G$-invariant probability measure $\mu$ on $X$ of full support, such that $(X,\mu)$ has the qualitative Sobolev embedding property, $\H$ identifies with $L^2(X,\mu)$ as unitary representations of $G$, and the multiplication $\H^{\fin} \times \H^{\fin} \to \H^{\infty}$ is given by pointwise multiplication of functions on $X$.
	\end{thm}
	
	Connectedness of $X$ will come from the ergodicity axiom in the definition of a multiplicative representation.
	The proof of Theorem~\ref{thm:weak_structure_thm} shows that $(X,\mu)$ obeys the qualitative Sobolev embedding property with $p=2$ and $q=\infty$.
	Given Theorem~\ref{thm:weak_structure_thm}, it essentially remains to prove
	
	\begin{thm}[Sobolev embedding implies $X \simeq \Gamma \backslash G$] \label{thm:embedding_G/Gamma}
		As usual, let $G = \PSL_2(\R)$.
		Let $X$ be a connected Hausdorff $G$-space equipped with a $G$-invariant probability measure $\mu$ of full support. Assume $(X,\mu)$ has the qualitative Sobolev embedding property. Then either $X$ is a point, or $X \simeq \Gamma \backslash G$ as $G$-spaces for some cocompact lattice $\Gamma$ in $G$.
	\end{thm}
	
	The case where $X$ is a point corresponds to the trivial multiplicative representation. The nontriviality assumption in Theorem~\ref{thm:eq_Gelfand_duality} is used to eliminate this case. It would suffice to prove Theorem~\ref{thm:embedding_G/Gamma} for $X$ connected compact metrizable, but the proof only requires that $X$ is connected and Hausdorff. On the other hand, Theorem~\ref{thm:embedding_G/Gamma} does use something specific about $G$, namely that $G = \PSL_2(\R)$ is a noncompact simple Lie group. What is true in general is
	
	\begin{thm}[Sobolev embedding implies homogeneity] \label{thm:embedding_homog}
		Let $G$ be a connected Lie group, and let $X$ be a connected Hausdorff $G$-space equipped with a $G$-invariant probability measure $\mu$ of full support. Assume $(X,\mu)$ has the qualitative Sobolev embedding property. Then $X \simeq H \backslash G$ as $G$-spaces for some closed cocompact subgroup $H \subseteq G$.
	\end{thm}
	
	This combined with the following known proposition implies Theorem~\ref{thm:embedding_G/Gamma}.
	
	\begin{prop} \label{prop:discrete_stabilizers}
		Let $G$ be a noncompact simple Lie group (e.g., $G = \PSL_2(\R)$). Let $H \subseteq G$ be a closed cocompact subgroup such that $H \backslash G$ admits a $G$-invariant probability measure. Then either $H$ is discrete or $H = G$.
	\end{prop}
	
	For completeness, we provide a short self-contained proof of this proposition in Subsection~\ref{subsec:embedding_homog:discrete}.
	
	To summarize, Step~\ref{step:embedding_G/Gamma} reduces to proving Theorem~\ref{thm:embedding_homog}.
	The intuition for this theorem is as follows.
	Recall the slogan for the qualitative Sobolev embedding property is that ``smoothing along $G$-orbits improves integrability." If $X$ is not locally homogeneous, then smoothing along $G$-orbits amounts to smoothing along some (but not all) directions in $X$. It should then be possible to build a function $f$ which is smooth in these directions but which blows up in a transverse direction. Convolving $f$ with $\varphi$ will not qualitatively affect $f$, because $f$ is already smooth along $G$-orbits. This means that the integrability of $f$ will not improve, contradicting the qualitative Sobolev embedding property. We conclude, at least heuristically, that $X$ is locally homogeneous. A connectedness argument then allows us to upgrade local homogeneity to global homogeneity. An example of this heuristic at work is given in Subsection~\ref{subsec:outline:embedding_homog}.
	
	At this point, let $G = \PSL_2(\R)$ once more, and consider Steps~\ref{step:C*} and \ref{step:spectrum}. Although these two steps have a lot of components, almost all are reasonably straightforward. The only difficult thing to show is that the $G$-action on $\B(\H)$ preserves $\A$ (this is Theorem~\ref{thm:A_G-inv}).
	In other words, the restriction of the $G$-action on $\B(\H)$ to $\A$ is a map $G \times \A \to \B(\H)$, and we wish to show that the image of this map is $\A$.
	The reason this is difficult is that $\A$ is generated by (the multiplication operators corresponding to elements of) $\H^{\fin}$, and $\H^{\fin}$ is closed under the $\g$-action but not the $G$-action.
	To overcome this issue, we show that $G \times \A \to \B(\H)$ is obtained by exponentiating the $\g$-action on $\H^{\fin}$. We thus need the Taylor series for the appropriate exponential map to converge.
	In general, the Taylor series at $0$ for a function $f$ on the real line has a positive radius of convergence if and only if $f$ obeys the derivative bounds
	\begin{align*}
		|f^{(n)}(x)|
		\lesssim_f O_f(n)^n
	\end{align*}
	for $x$ in an interval around $0$, where the interval is independent of $n$. With this in mind, it is natural that we obtain convergence of the exponential Taylor series from the similar looking derivative bounds in Theorem~\ref{thm:high_deriv_L^infty_bd} below, which is the first of the two main results of Step~\ref{step:L^infty}. We defer a more detailed explanation to Subsection~\ref{subsec:outline:reduce_to_L^infty}.
	
	We now come to Step~\ref{step:L^infty}, which consists of the following two theorems.
	To state Theorem~\ref{thm:high_deriv_L^infty_bd}, fix once and for all a norm $\|\cdot\|_{\g}$ on the Lie algebra $\g$. Up to constants, it doesn't matter which one we choose.
	
	\begin{thm}[Derivative bounds in $L^{\infty}$] \label{thm:high_deriv_L^infty_bd}
		Let $\alpha \in \H^{\fin}$, let $X \in \g$, and let $n \in \Z_{\geq 0}$. Then
		\begin{align*} 
			\|X^n\alpha\|_{L^{\infty}}
			\lesssim_{\alpha} O(n\|X\|_{\g})^n.
		\end{align*}
	\end{thm}
	
	We emphasize that the $O$-constant, which is raised to the $n$th power, is independent of $\alpha$.
	
	\begin{thm}[Quasi-Sobolev embedding into $L^{\infty}$] \label{thm:L^infty_quasi-Sobolev}
		Let $\alpha \in \H^K \cap \H^{\fin}$. Then
		\begin{align} \label{eqn:L^infty_quasi-Sobolev}
			\|\alpha\|_{L^{\infty}}
			\leq \|\exp(O(\log_+^2\Delta))\alpha\|_{\H}.
		\end{align}
	\end{thm}
	
	We add the prefix ``quasi-" because the right hand side involves $f(\Delta)$ with $f \colon [0,\infty) \to \R$ a function of quasipolynomial rather than polynomial growth.
	As we will see in the next subsection, when $\H = L^2(\Gamma \backslash G)$, it is possible to replace $\exp(O(\log_+^2\Delta))$ with $(\Delta+1)^{O(1)}$ in \eqref{eqn:L^infty_quasi-Sobolev}. I do not know how to do this for a general multiplicative representation without going through Theorem~\ref{thm:eq_Gelfand_duality}. Fortunately, most of the time we will be dealing in exponentials, so the difference between polynomial and quasipolynomial will be irrelevant.
	
	%
		%
		%
		%
	
	If $\alpha$ is an automorphic vector of nonnegative weight and $X = E$ is the raising operator, then the $L^{\infty}$ bounds on the automorphic vectors $E^n\alpha$ given by Theorem~\ref{thm:high_deriv_L^infty_bd} have explicit dependence on the weight $n$, because the implicit and $O$-constants are independent of $n$.
	On the other hand, if $\alpha$ is an automorphic vector to which Theorem~\ref{thm:L^infty_quasi-Sobolev} applies, then $\alpha \in \H^K$ has weight $0$, and the bound on $\|\alpha\|_{L^{\infty}}$ given by Theorem~\ref{thm:L^infty_quasi-Sobolev} depends explicitly on the Casimir eigenvalue of $\alpha$. This justifies the description of these two estimates in the statement of Step~\ref{step:L^infty} above.
	The proofs of Theorems~\ref{thm:high_deriv_L^infty_bd} and \ref{thm:L^infty_quasi-Sobolev} are outlined in Subsection~\ref{subsec:outline:reduce_to_bulk_tail}.
	
	We have explained that Theorem~\ref{thm:high_deriv_L^infty_bd} is used to exponentiate the $\g$-action on $\H^{\fin}$ to a $G$-action on $\A$, and in Step~\ref{step:spectrum} it is stated that a corollary of Theorem~\ref{thm:L^infty_quasi-Sobolev} is used to prove the qualitative Sobolev embedding property for $(X,\mu)$. Theorem~\ref{thm:high_deriv_L^infty_bd} and this corollary are the only results from Steps~\ref{step:bulk_tail} and \ref{step:L^infty} which are used in Steps~\ref{step:C*}, \ref{step:spectrum}, and \ref{step:embedding_G/Gamma}. This corollary is the following. In the statement, $\mathcal{C}$ is the dense subspace of $\H$ consisting of vectors which are limits in $\H$ of $L^{\infty}$-Cauchy sequences in $\H^{\fin}$. So if $\H = L^2(\Gamma \backslash G)$, then $\mathcal{C} = C(\Gamma \backslash G)$.
	
	\begin{cor}[Qualitative Sobolev embedding into $L^{\infty}$] \label{cor:qual_Sob_emb}
		There exists a nonnegative continuous function $\varphi \in L^1(G)$, not identically zero, such that for all $v \in \H$, the convolution
		\begin{align*}
			\varphi \ast v
			= \int_G \varphi(g) gv \, dg
		\end{align*}
		lies in $\mathcal{C}$.
	\end{cor}
	
	\begin{proof}[Proof of Corollary~\ref{cor:qual_Sob_emb} assuming Theorem~\ref{thm:L^infty_quasi-Sobolev}]
		Let $\varphi$ be as in Lemma~\ref{lem:positive_super-smooth}. Let $v \in \H$ be arbitrary, and denote $w = \varphi \ast v$.
		By Lemma~\ref{lem:positive_super-smooth}, $w$ is super-smooth.
		We wish to show that $w \in \mathcal{C}$, i.e., that there is a sequence $\alpha_n \in \H^{\fin}$ which is Cauchy with respect to the $L^{\infty}$ norm and which converges to $w$ in $\H$. Let $\alpha_n = \1_{\Delta \leq n}w$. Clearly $\alpha_n \to w$ in $\H$. Super-smooth vectors are $K$-invariant, so $w \in \H^K$. Thus $\alpha_n \in \H_{\Delta \leq n}^K$, and hence $\alpha_n \in \H^K \cap \H^{\fin}$. Therefore we can apply Theorem~\ref{thm:L^infty_quasi-Sobolev} to see that for $m \leq n$,
		\begin{align*}
			\|\alpha_n-\alpha_m\|_{L^{\infty}}
			\leq \|\exp(O(\log_+^2\Delta))(\alpha_n-\alpha_m)\|_{\H}
			= \|\1_{\Delta \in (m,n]} \exp(O(\log_+^2\Delta)) w\|_{\H}.
		\end{align*}
		The right hand side goes to zero as $m,n \to \infty$ because $w$ is super-smooth.
		Thus the $\alpha_n$ are indeed Cauchy in $L^{\infty}$, and we conclude that $w \in \mathcal{C}$ as required.
	\end{proof}

	This completes our discussion of Step~\ref{step:L^infty}.
	The three main results of Step~\ref{step:bulk_tail} are
	
	\begin{thm}[Quasi-Sobolev embedding into $L^4$] \label{thm:L^4_quasi-Sobolev}
		Let $\alpha \in \H^K \cap \H^{\fin}$. Then
		\begin{align*}
			\|\alpha\|_{L^4}
			\leq \|\exp(O(\log_+^2 \Delta)) \alpha\|_{\H}.
		\end{align*}
	\end{thm}
	
	\begin{thm}[Exponential decay in the tail, part I] \label{thm:exp_decay_form}
		There is a positive constant $c \gtrsim 1$, such that whenever $f \in \H^{\fin}$ is a lowest weight vector,
		\begin{align} \label{eqn:exp_decay_form}
			\|\exp(c\sqrt{\Delta})(|f|^2)\|_{\H} < \infty.
		\end{align}
	\end{thm}
	
	\begin{thm}[Exponential decay in the tail, part II] \label{thm:exp_decay_quasimode}
		There are positive constants $C \lesssim 1$ and $c \gtrsim 1$, a partition of $[0,\infty)$ into intervals $I_i$ of length $\lesssim 1$, and points $\lambda_i \in I_i$ for each $i$, such that the partition has polynomial growth in the sense that
		\begin{align} \label{eqn:partition_growth}
			\#\{i : \lambda_i \leq X\} \lesssim X^{O(1)}
			\qquad \text{for} \qquad
			X \geq 1,
		\end{align}
		and such that for all $i$ and all $\alpha,\beta \in \H_{\Delta \in I_i}^K$,
		\begin{align} \label{eqn:exp_decay_quasimodes}
			\|\1_{\Delta \geq C\lambda_i} \exp(c\sqrt{\Delta})(\alpha\beta)\|_{\H}
			\leq \|\alpha\|_{\H} \|\beta\|_{\H}.
		\end{align}
	\end{thm}
	
	\begin{rem}[Comparison with Theorem~\ref{thm:triple_bds}]
		When $\H = L^2(\Gamma \backslash G)$, the holomorphic case of Theorem~\ref{thm:triple_bds} is much stronger than Theorem~\ref{thm:exp_decay_form}. For example, it implies that \eqref{eqn:exp_decay_form} holds for any $c < \frac{\pi}{2}$.
		Since Theorem~\ref{thm:triple_bds} is proved using the hyperbolic bootstrap equations, it seems at first glance that one should be able to generalize Theorem~\ref{thm:triple_bds} to arbitrary multiplicative representations (without appealing to Theorem~\ref{thm:eq_Gelfand_duality}).
		However, the proof of Theorem~\ref{thm:triple_bds} uses an extra bit of analytic input which is not available for general multiplicative representations; for the holomorphic case of Theorem~\ref{thm:triple_bds}, this input is Lemma~\ref{lem:bulk_bd_for_Weyl}.
	\end{rem}
	
	The statement of Theorem~\ref{thm:exp_decay_quasimode} is more complicated than that of Theorem~\ref{thm:exp_decay_form}, but it has a simple corollary which is more closely analogous to Theorem~\ref{thm:exp_decay_form}:
	
	\begin{cor} \label{cor:exp_decay_quasimode_qual}
		Let $c$ and $I_i$ be as in Theorem~\ref{thm:exp_decay_quasimode}. Then for all $i$ and all $\alpha,\beta \in \H_{\Delta \in I_i}^K$,
		\begin{align*} 
			\|\exp(c\sqrt{\Delta})(\alpha\beta)\|_{\H}
			< \infty.
		\end{align*}
		In particular, this holds whenever $\alpha,\beta \in \H^K$ are Casimir eigenvectors with the same eigenvalue.
	\end{cor}
	
	\begin{proof}[Proof of Corollary~\ref{cor:exp_decay_quasimode_qual} assuming Theorem~\ref{thm:exp_decay_quasimode}]
		Let $C$ and $\lambda_i$ also be as in Theorem~\ref{thm:exp_decay_quasimode}. Splitting $\1 = \1_{\Delta \leq C\lambda_i} + \1_{\Delta > C\lambda_i}$ and using \eqref{eqn:exp_decay_quasimodes},
		\begin{align*}
			\|\exp(c\sqrt{\Delta})(\alpha\beta)\|_{\H}
			\leq \exp(c\sqrt{C\lambda_i}) \|\alpha\beta\|_{\H} + \|\alpha\|_{\H} \|\beta\|_{\H},
		\end{align*}
		and the right hand side is finite.
	\end{proof}
	
	Theorem~\ref{thm:exp_decay_quasimode} clearly matches its description in the statement of Step~\ref{step:bulk_tail} as a bound on the ``tail" $\1_{\Delta \gg \lambda_i}(\alpha\beta)$. Taking into account the similarity between Theorem~\ref{thm:exp_decay_form} and Corollary~\ref{cor:exp_decay_quasimode_qual}, Theorem~\ref{thm:exp_decay_form} matches its description as well. It is perhaps less clear how Theorem~\ref{thm:L^4_quasi-Sobolev} fits into the picture. The connection is as follows. According to the statement of Step~\ref{step:bulk_tail}, given $\lambda \geq 0$ and approximate eigenvectors $\alpha,\beta \in \H^K$ with approximate eigenvalue $\lambda$, Theorem~\ref{thm:L^4_quasi-Sobolev} should give a bound on the ``bulk" $\1_{\Delta \lesssim \lambda}(\alpha\beta)$. By Theorem~\ref{thm:exp_decay_quasimode}, the tail $\1_{\Delta \gg \lambda}(\alpha\beta)$ is negligible, so bounding the bulk is equivalent to bounding $\alpha\beta$ itself. Of course, even without Theorem~\ref{thm:exp_decay_quasimode}, any upper bound for $\|\alpha\beta\|_{\H}$ is trivially also an upper bound for $\|\1_{\Delta \lesssim \lambda}(\alpha\beta)\|_{\H}$. Now, by $L^4$-Cauchy--Schwarz (Proposition~\ref{prop:L^4_C-S}), Theorem~\ref{thm:L^4_quasi-Sobolev}, and the fact that $\alpha,\beta \in \H_{\Delta \leq \lambda+O(1)}^K$ (which follows from $\alpha,\beta$ being approximate $\lambda$-eigenvectors),
	\begin{align*}
		\|\alpha\beta\|_{\H}
		\leq \|\alpha\|_{L^4} \|\beta\|_{L^4}
		\leq \|\exp(O(\log_+^2\Delta)) \alpha\|_{\H} \|\exp(O(\log_+^2\Delta)) \beta\|_{\H}
		\lesssim \exp(O(\log_+^2\lambda)) \|\alpha\|_{\H} \|\beta\|_{\H}.
	\end{align*}
	Thus indeed Theorem~\ref{thm:L^4_quasi-Sobolev} yields a bound on the bulk $\1_{\Delta \lesssim \lambda}(\alpha\beta)$.
	This sort of combination of Theorem~\ref{thm:L^4_quasi-Sobolev} with $L^4$-Cauchy--Schwarz will be used repeatedly.
	
	
	The proofs of Theorems~\ref{thm:L^4_quasi-Sobolev}, \ref{thm:exp_decay_form}, and \ref{thm:exp_decay_quasimode} are outlined in Subsection~\ref{subsec:outline:bulk_tail_poly} in the case where $\H$ obeys a polynomial Weyl law, and in Subsection~\ref{subsec:outline:bulk_tail_general} in the general case.
	
	
	\subsection{Proofs/heuristics for $L^{\infty}$, bulk, and tail bounds when $\H = L^2(\Gamma \backslash G)$}
	\label{subsec:outline:bds_G/Gamma}
	
	To get a feel for the $L^{\infty}$ bounds in Theorems~\ref{thm:high_deriv_L^infty_bd} and \ref{thm:L^infty_quasi-Sobolev} and the bulk and tail bounds in Theorems~\ref{thm:L^4_quasi-Sobolev}, \ref{thm:exp_decay_form}, and \ref{thm:exp_decay_quasimode}, it is helpful to see why they are true when $\H = L^2(\Gamma \backslash G)$. This also serves as a sanity check. Technically, nothing from this subsection is used in the proof of Theorem~\ref{thm:eq_Gelfand_duality}, but it may provide some intuition.
	
	In this subsection, let $\Gamma$ be a cocompact lattice in $G$, and let $\H$ be the multiplicative representation $L^2(\Gamma \backslash G)$. In this setting, we have access to tools from analysis on manifolds which make the above theorems much easier. The two most useful such tools are (crude) Sobolev inequalities and elliptic regularity. We will use elliptic regularity both in the smooth category and in the real analytic category. A precise form of the latter is
	
	\begin{lem}[Elliptic regularity in the real analytic category \cite{Boutet}] \label{lem:ell_reg_real_analytic}
		Let $M$ be a closed real analytic manifold, and embed $M$ as a totally real submanifold of a complex manifold $M_{\C}$. Given a neighborhood $U$ of $M$ in $M_{\C}$, let $\Hol(U)|_M$ denote the space of functions on $M$ obtained as restrictions of holomorphic functions on $U$. Let $\L$ be a second order elliptic linear differential operator on $M$ with real analytic coefficients. Assume $\L$ is self-adjoint and positive semidefinite on $L^2(M)$ (where the $L^2$ inner product is defined with respect to a real analytic density on $M$). Then for each $\varepsilon>0$, the functional calculus defines a bounded linear operator $\exp(-\varepsilon\sqrt{\L})$ on $L^2(M)$. These operators obey the following properties (i) and (ii).
		\begin{enumerate} \itemsep = 0.5em
			\item[(i)] For each $\varepsilon>0$, there exists a neighborhood $U$ of $M$ in $M_{\C}$ such that the image of $\exp(-\varepsilon\sqrt{\L})$ is contained in $\Hol(U)|_M$.
			
			\item[(ii)] For each neighborhood $U$ of $M$ in $M_{\C}$, there exists $\varepsilon > 0$ such that the image of $\exp(-\varepsilon\sqrt{\L})$ contains $\Hol(U)|_M$.
		\end{enumerate}
		In particular, by (i), there exists a neighborhood $U$ of $M$ in $M_{\C}$ (depending only on $M,M_{\C},\L$), such that every eigenfunction of $\L$ is contained in $\Hol(U)|_M$.
	\end{lem}
	
	A neighborhood $U$ as in the last sentence of the lemma is called a \emph{Grauert tube} (see, e.g., \cite{Zelditch}).
	
	It is instructive to verify Lemma~\ref{lem:ell_reg_real_analytic} using Fourier analysis in the case where $M = (\R/\Z)^d$, $M_{\C} = (\C/\Z)^d$, and $\L$ is the Laplacian. The only misleading feature of this example is that the Grauert tube can be taken to be all of $M_{\C}$. In general, this is usually not possible.
	
	The sharpest form of Lemma~\ref{lem:ell_reg_real_analytic} is an application of the theory of Fourier integral operators with complex phase \cite{Boutet}.
	
	We now prove or give heuristic arguments for the five theorems above in the present setting where $\H = L^2(\Gamma \backslash G)$.
	All the arguments are soft in the sense that the key analytic tools, like Lemma~\ref{lem:ell_reg_real_analytic}, use nothing about the structure of $\Gamma \backslash G$ except that it is compact and real analytic.
	
	\begin{proof}[Proof of Theorem~\ref{thm:high_deriv_L^infty_bd} when $\H = L^2(\Gamma \backslash G)$]
		Let $M$ denote the closed real analytic manifold $\Gamma \backslash G$, and embed $M$ as a totally real submanifold of a complex manifold $M_{\C}$ (e.g., $M_{\C} = \Gamma \backslash \PSL_2(\C)$). View $\P$, defined by \eqref{eqn:P_def}, as a differential operator on $M$. Then $\P$ satisfies the hypotheses of Lemma~\ref{lem:ell_reg_real_analytic}, and all automorphic forms are eigenfunctions for $\P$. Thus by Lemma~\ref{lem:ell_reg_real_analytic}, there is a neighborhood $U$ of $M$ in $M_{\C}$ such that all automorphic forms on $M$ extend to holomorphic functions on $U$.
		
		Let $\alpha \in \H^{\fin}$. By Proposition~\ref{prop:Maass_span_H^fin}, $\alpha$ is a finite linear combination of automorphic forms, so $\alpha$ extends holomorphically to $U$.
		There is a neighborhood $\Omega$ of the origin in $\g$ such that the map $\g_{\R} \times M \to M$ by
		\begin{align} \label{eqn:translation_by_exp}
			(X,\Gamma g) \mapsto \Gamma g \exp(X)
		\end{align}
		extends holomorphically in $X$ to a map $\Omega \times M \to U$.
		We denote this extension also by \eqref{eqn:translation_by_exp}.
		Define $h \colon \Omega \to L^{\infty}(M)$ by
		\begin{align*}
			h(X)(\Gamma g) = \alpha(\Gamma g \exp(X)).
		\end{align*}
		Then $h$ is holomorphic on $\Omega$ (as a function valued in a Banach space).
		It follows from Cauchy's estimate for derivatives of holomorphic functions that for $n \in \Z_{\geq 0}$ and $X \in \g$,
		\begin{align*}
			\Big\|\frac{d^n}{dt^n}\Big|_{t=0} h(tX)\Big\|_{L^{\infty}(M)}
			\lesssim_{\alpha} O(n\|X\|_{\g})^n
		\end{align*}
		(the $O$-constant depends only on $\Omega$, which ultimately depends only on $\H$). The $n$th derivative at $t=0$ on the left hand side is just $X^n\alpha$. Thus we have the desired bound.
	\end{proof}
	
	\begin{proof}[Proof of Theorem~\ref{thm:L^infty_quasi-Sobolev} when $\H = L^2(\Gamma \backslash G)$]
		The condition that $\alpha \in \H^K$ means that $\alpha$ is a function on $\Gamma \backslash G / K = \Gamma \backslash \mathbf{H}$.
		Let $s \gg 1$ be a sufficiently large constant.
		Then by Sobolev embedding and elliptic regularity,
		\begin{align*}
			\|\alpha\|_{L^{\infty}}
			= \|\alpha\|_{L^{\infty}(\Gamma \backslash \mathbf{H})}
			\lesssim \|\alpha\|_{H^s(\Gamma \backslash \mathbf{H})}
			\sim_s \|(\Delta+1)^{\frac{s}{2}} \alpha\|_{L^2(\Gamma \backslash \mathbf{H})}
			= \|(\Delta+1)^{\frac{s}{2}} \alpha\|_{\H}.
		\end{align*}
		Thus in particular
		\begin{align*}
			\|\alpha\|_{L^{\infty}}
			\lesssim \|\exp(O(\log_+^2\Delta)) \alpha\|_{\H},
		\end{align*}
		and $\lesssim$ can be replaced by $\leq$ at the cost of increasing the $O$-constant.
	\end{proof}
	
	As referred to earlier, this argument gives the improved estimate $\|\alpha\|_{L^{\infty}} \lesssim \|(\Delta+1)^{O(1)}\alpha\|_{\H}$, with polynomial rather than quasipolynomial dependence on $\Delta$ in the right hand side.
	
	\begin{proof}[Proof of Theorem~\ref{thm:L^4_quasi-Sobolev} when $\H = L^2(\Gamma \backslash G)$]
		The $L^4$ norm is bounded by the $L^{\infty}$ norm, so Theorem~\ref{thm:L^4_quasi-Sobolev} follows from Theorem~\ref{thm:L^infty_quasi-Sobolev}, which we proved above in the case $\H = L^2(\Gamma \backslash G)$.
	\end{proof}
	
	\begin{proof}[Proof of Theorem~\ref{thm:exp_decay_form} when $\H = L^2(\Gamma \backslash G)$]
		Let $M,M_{\C},U$ be as in the proof of Theorem~\ref{thm:high_deriv_L^infty_bd} above, so $U$ is a neighborhood of $M$ in the complexification $M_{\C}$ such that all automorphic forms on $\Gamma \backslash G$ extend holomorphically to $U$. Let $f \in \H^{\fin}$ be a lowest weight vector. Then $f,\overline{f}$ are both automorphic forms, so they both extend holomorphically to $U$. Thus $|f|^2 \in \Hol(U)|_M$. By (ii) in Lemma~\ref{lem:ell_reg_real_analytic}, there exists $c>0$ depending only on $U$, such that $|f|^2$ is in the image of $\exp(-c\sqrt{\P})$ (in the lemma $c$ is called $\varepsilon$). This means that
		\begin{align} \label{eqn:exp(csqrt{P})(|f|^2)_finite}
			\|\exp(c\sqrt{\P})(|f|^2)\|_{\H} < \infty.
		\end{align}
		Since $|f|^2$ has weight $0$ and $\P|_{\H^K} = \Delta|_{\H^K}$, we have $\exp(c\sqrt{\P})(|f|^2) = \exp(c\sqrt{\Delta})(|f|^2)$. Thus \eqref{eqn:exp(csqrt{P})(|f|^2)_finite} gives the desired finiteness. It remains to check that $c \gtrsim 1$. Indeed, $c$ depends only on $U$ and hence only on $\H$.
	\end{proof}
	
	This argument shows more generally that if $F \in C^{\infty}(\Gamma \backslash G/K)$ is any polynomial combination of automorphic forms on $\Gamma \backslash G$, then
	\begin{align} \label{eqn:Sarnak's_bd}
		\|\exp(c\sqrt{\Delta})F\|_{L^2} < \infty.
	\end{align}
	A very similar argument shows that if $F$ is any polynomial combination of Laplace eigenfunctions on any closed real analytic Riemannian manifold, then \eqref{eqn:Sarnak's_bd} holds for some $c>0$, with $\Delta$ interpreted as the Laplacian.
	Sarnak showed that when the manifold is a hyperbolic surface, any $c<\frac{\pi}{2}$ works \cite{Sarnak_94}.
	This crucially uses constant curvature.
	
	\begin{proof}[Heuristic for Theorem~\ref{thm:exp_decay_quasimode} when $\H = L^2(\Gamma \backslash G)$]
		It suffices to prove the stronger statement that there are positive constants $C \lesssim 1$ and $c \gtrsim 1$, such that for all $\lambda \geq 1$ and $\alpha,\beta \in \H_{\Delta\leq\lambda}^K$,
		\begin{align*}
			\|\1_{\Delta \geq C\lambda} \exp(c\sqrt{\Delta})(\alpha\beta)\|_{\H}
			\leq \|\alpha\|_{\H} \|\beta\|_{\H}.
		\end{align*}
		Normalize $\|\alpha\|_{\H} = \|\beta\|_{\H} = 1$, so we want to show that the left hand side is at most $1$. By the same argument as in the proof of Theorem~\ref{thm:exp_decay_form} above, we see that $\|\exp(c\sqrt{\Delta})(\alpha\beta)\|_{\H} < \infty$ for $c$ sufficiently small. By definition, this finiteness means that
		\begin{align*}
			\sum_{\mu \geq 0} \exp(2c\sqrt{\mu}) \|\1_{\Delta=\mu}(\alpha\beta)\|_{\H}^2
			< \infty.
		\end{align*}
		Therefore $\|\1_{\Delta=\mu}(\alpha\beta)\|_{\H}$ decays exponentially in $\sqrt{\mu}$ as $\mu \to \infty$. However, the rate of decay will not be uniform in $\alpha,\beta$.
		One should imagine that there is some threshold $\mu_0$ depending on $\alpha,\beta$, such that $\1_{\Delta=\mu}(\alpha\beta)$ is small for $\mu \gg \mu_0$, but not necessarily for $\mu \lesssim \mu_0$.
		The content of Theorem~\ref{thm:exp_decay_quasimode} is that $\mu_0$ can be taken to be $\lambda$. This can be understood heuristically as follows. The condition that $\alpha,\beta \in \H_{\Delta\leq\lambda}^K$ means that $\alpha,\beta$ are linear combinations of Laplace eigenfunctions on $\Gamma \backslash G/K = \Gamma \backslash \mathbf{H}$ with eigenvalues $\leq \lambda$. In general, a Laplace eigenfunction of eigenvalue $\mu$ on a fixed closed Riemannian manifold is smooth at scales $\ll \frac{1}{\sqrt{\mu}}$ and oscillatory at scales $\gg \frac{1}{\sqrt{\mu}}$. For example, when the manifold is a flat torus, this can be seen explicitly by Fourier analysis. In our setting, it follows that $\alpha,\beta$ are smooth at scales $\ll \frac{1}{\sqrt{\lambda}}$. Therefore the product $\alpha\beta$ is also smooth at scales $\ll \frac{1}{\sqrt{\lambda}}$. Now $\1_{\Delta=\mu}(\alpha\beta)$ is a Laplace eigenfunction of eigenvalue $\mu$, so $\1_{\Delta=\mu}(\alpha\beta)$ should be oscillatory at scales $\gg \frac{1}{\sqrt{\mu}}$. Suppose $\mu \gg \lambda$. Then let $\nu$ be the geometric mean of $\lambda,\mu$, so $\lambda \ll \nu \ll \mu$ and $\frac{1}{\sqrt{\mu}} \ll \frac{1}{\sqrt{\nu}} \ll \frac{1}{\sqrt{\lambda}}$. At the intermediate scale $\frac{1}{\sqrt{\nu}}$, we have that $\alpha\beta$ is smooth and $\1_{\Delta=\mu}(\alpha\beta)$ is oscillatory. Thus $\alpha\beta$ and $\1_{\Delta=\mu}(\alpha\beta)$ are almost orthogonal in $L^2(\Gamma \backslash G) = \H$. Consequently
		\begin{align*}
			\|\1_{\Delta=\mu}(\alpha\beta)\|_{\H}^2
			= \langle \alpha\beta, \1_{\Delta=\mu}(\alpha\beta) \rangle_{\H}
		\end{align*}
		is small. In summary, we have argued that if $\mu \gg \mu_0 := \lambda$, then $\|\1_{\Delta=\mu}(\alpha\beta)\|_{\H}$ is small. As discussed above, this is morally the content of Theorem~\ref{thm:exp_decay_quasimode}.
	\end{proof}
	
	
	It is notable that the bounds in Theorems~\ref{thm:high_deriv_L^infty_bd}, \ref{thm:exp_decay_form}, and \ref{thm:exp_decay_quasimode} for general multiplicative representations are of the same strength as those given by Lemma~\ref{lem:ell_reg_real_analytic} for $L^2(\Gamma \backslash G)$.
	
	\subsection{Main ideas in Section~\ref{sec:embedding_homog}} \label{subsec:outline:embedding_homog}
	
	Section~\ref{sec:embedding_homog} carries out Step~\ref{step:embedding_G/Gamma}. As discussed in the roadmap above (Subsection~\ref{subsec:outline:roadmap}), Step~\ref{step:embedding_G/Gamma} boils down to proving Theorem~\ref{thm:embedding_homog}. So in this subsection, allow $G$ to be an arbitrary connected Lie group. We prove Theorem~\ref{thm:embedding_homog} in the contrapositive, so (in the notation of the theorem) we show that if $X$ is not a compact homogeneous $G$-space, then $(X,\mu)$ does not have the qualitative Sobolev embedding property. Let us illustrate this in an example.
	
	\begin{exmp}[Irrational flow on a 2-torus] \label{exmp:irrational_flow}
		Let $G = \R$. Let $X = (\R/\Z)^2$ with Lebesgue measure $\mu$. Fix a point $\omega \in S^1 \subseteq \R^2$ on the unit circle, and let $G$ act on $X$ by $t \cdot x = x+t\omega$ for $t \in \R = G$. If the line $L_{\omega} \subseteq \R^2$ through the origin and $\omega$ has irrational slope, then this action is uniquely ergodic with invariant measure $\mu$, but no matter what $\omega$ is, $X$ is never homogeneous as a $G$-space. We must therefore construct a counterexample to qualitative Sobolev embedding for $(X,\mu)$. Let $1 \leq p < q \leq \infty$, and let $\varphi \in L^1(G)$ be a nonnegative continuous function which is not identically zero. Then a counterexample is, by definition, a function $f \in L^p(X,\mu)$ with $\varphi \ast f \not\in L^q(X,\mu)$.
		We construct such a function $f$ explicitly below.
		The geometry underlying the construction is depicted in Figure~\ref{fig:torus}.
		
		Let $W \subseteq X$ be the projection to $(\R/\Z)^2$ of the disc of radius $0.2$ around the origin in $\R^2$. Then $W$ is foliated by line segments parallel to $L_{\omega}$. We think of these line segments as ``local $G$-orbits." Let $\omega^{\perp} \in S^1$ be a point orthogonal to $\omega$, and let $\{I_i\}$ be an infinite collection of disjoint nonempty open subintervals of $(-0.1,0.1)$. Then let
		\begin{align*}
			R_i = \{t\omega + u\omega^{\perp} : t \in (-0.1,0.1) \text{ and } u \in I_i\} \mod \Z^2.
		\end{align*}
		These $R_i$ are open rectangles in $W$, and they are disjoint because each local orbit intersects at most one $R_i$. In addition, each $R_i$ has length $0.2$ in the direction $\omega$. Let $f = \sum_i c_i \1_{R_i}$, where the $c_i$ are nonnegative coefficients such that $f \in L^p(X,\mu)$ but $f \not\in L^q(X,\mu)$.
		
		It remains to check that $\varphi \ast f \not\in L^q(X,\mu)$.
		By translation symmetry, we may assume $\varphi(0) > 0$. Then since $\varphi$ is nonnegative and continuous, $\varphi(t) \geq \varepsilon\1_{|t| \leq \varepsilon}$ for some small $\varepsilon > 0$. Since the $R_i$ have uniform length in the direction $\omega$, in particular length $\geq \varepsilon$, we have $\varphi \ast \1_{R_i} \geq \varepsilon^2\1_{R_i}$ pointwise.
		Therefore $\varphi \ast f \geq \varepsilon^2 f$, and hence $\varphi \ast f \not\in L^q(X,\mu)$. This completes Example~\ref{exmp:irrational_flow}.
	\end{exmp}
	
	\begin{figure}
		\centering
		\includegraphics[width=0.75\textwidth]{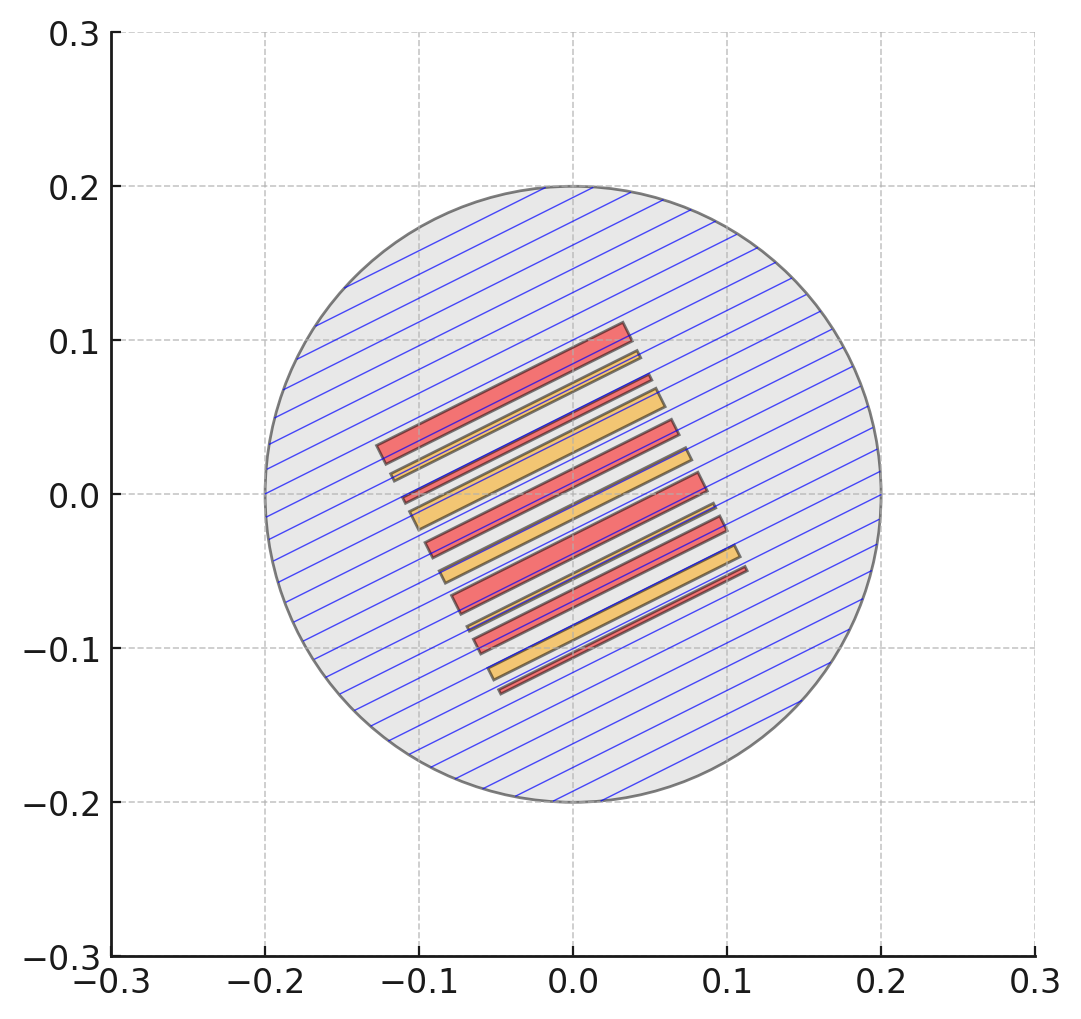}
		\caption{This is an illustration of Example~\ref{exmp:irrational_flow}. The disc is $W$. The blue lines are the local orbits.
			The red and orange rectangles are the $R_i$. In reality, there are of course infinitely many local orbits and rectangles $R_i$.}
		\label{fig:torus}
	\end{figure}
	
	The two key ingredients in Example~\ref{exmp:irrational_flow} are
	\begin{enumerate} \itemsep = 0.5em
		\item[(1)] There is a nonempty open subset $W \subseteq X$ foliated by local orbits; moreover, this foliation is non-pathological in the sense that the space of local orbits is Hausdorff.
		
		\item[(2)] There are infinitely many disjoint open sets $R_i$ in $X$ which have ``uniform length in the direction of the $G$-action."
	\end{enumerate}
	The purpose of (1) is to show that the $R_i$ in (2) are disjoint. Then (2) is used to build $f$.
	
	To prove Theorem~\ref{thm:embedding_homog} in the general case, $W$ is constructed in Lemma~\ref{lem:partial_orbits}, and the $R_i$ can be taken to be the product sets $UV_i$ appearing in Proposition~\ref{prop:top_characterization}. For technical reasons, we take $f$ to be of the form $\sum_i c_i\1_U \ast \1_{V_i}$ rather than $\sum_i c_i \1_{UV_i}$. Besides this, the proof in the general case proceeds in exactly the same way given these two ingredients.
	
	\subsection{Main ideas in Section~\ref{sec:reduce_to_L^infty}}
	\label{subsec:outline:reduce_to_L^infty}
	
	Let $G = \PSL_2(\R)$ again.
	Section~\ref{sec:reduce_to_L^infty} is split into two subsections which perform Steps~\ref{step:C*} and \ref{step:spectrum}, respectively.
	As discussed in the roadmap, the only difficult part of either of these two steps is the construction of the $G$-action on the C*-algebra $\A$. Recall that $\A$ is the closed subalgebra of $\B(\H)$ generated by $\{M_{\alpha} : \alpha \in \H^{\fin}\}$. We want $g \in G$ to act on $\A$ by $A \mapsto gAg^{-1}$, where this product is interpreted as a product of operators on $\H$. This will give a well-defined action as soon as we can show that $gAg^{-1} \in \A$; \textit{a priori} $gAg^{-1}$ is just an arbitrary bounded operator on $\H$.
	Since $G$ is connected, it is generated by any neighborhood of the identity, so it suffices to show that for $g$ sufficiently small, one has $gAg^{-1} \in \A$ for all $A \in \A$. Furthermore, since $\A$ is generated by the $M_{\alpha}$, it suffices to show that $gM_{\alpha}g^{-1} \in \A$ for all $\alpha \in \H^{\fin}$. If $\H = L^2(\Gamma \backslash G)$, then one can check that $gM_{\alpha}g^{-1} = M_{g\alpha}$, where $M_{g\alpha}$ is the operator on $L^2(\Gamma \backslash G)$ given by pointwise multiplication by $g\alpha$.
	Writing $g = \exp(X)$ for some $X \in \g_{\R}$ sufficiently small, and formally Taylor expanding in $X$,
	\begin{align*}
		gM_{\alpha}g^{-1}
		= M_{g\alpha}
		= M_{\exp(X)\alpha}
		= \sum_{n=0}^{\infty} \frac{1}{n!} M_{X^n\alpha}.
	\end{align*}
	For a general multiplicative representation, $M_{g\alpha}$ is not defined because $g\alpha \not\in \H^{\fin}$, but both the left and right hand sides of the above equation are well-defined. We show in Proposition~\ref{prop:conj_M_alpha} that equality of the left and right hand sides holds generally, with absolute convergence in operator norm on the right hand side by Theorem~\ref{thm:high_deriv_L^infty_bd} from Step~\ref{step:L^infty} (we use that $X$ is small to get convergence). This exhibits $gM_{\alpha}g^{-1}$ as a limit of elements of $\A$, so $gM_{\alpha}g^{-1} \in \A$ as desired.
	
	\subsection{Main ideas in Section~\ref{sec:reduce_to_bulk_tail}}
	\label{subsec:outline:reduce_to_bulk_tail}
	
	Section~\ref{sec:reduce_to_bulk_tail} carries out Step~\ref{step:L^infty}, which consists of the $L^{\infty}$ bounds in Theorems~\ref{thm:high_deriv_L^infty_bd} and \ref{thm:L^infty_quasi-Sobolev}. We prove Theorem~\ref{thm:L^infty_quasi-Sobolev} first in Subsection~\ref{subsec:L^infty_quasi-Sobolev}, and then prove Theorem~\ref{thm:high_deriv_L^infty_bd} in Subsection~\ref{subsec:deriv_bds_in_L^infty} using some algebraic identities from Subsection~\ref{subsec:basic_identities}.
	
	Let us begin by sketching the proof of Theorem~\ref{thm:L^infty_quasi-Sobolev}, the quasi-Sobolev embedding theorem. Let $I_i,\lambda_i$ be as in Theorem~\ref{thm:exp_decay_quasimode} from Step~\ref{step:bulk_tail}, so $\{I_i\}$ is a partition of $[0,\infty)$ into intervals of length $\lesssim 1$, and $\lambda_i \in I_i$. Then a short reduction using the triangle inequality shows that to prove Theorem~\ref{thm:L^infty_quasi-Sobolev}, it is enough to prove \eqref{eqn:L^infty_quasi-Sobolev} when $\alpha \in \H_{\Delta \in I_i}^K$ for some $i$. In this case, \eqref{eqn:L^infty_quasi-Sobolev} simplifies to
	\begin{align} \label{eqn:outline:L^infty_quasimode_bd}
		\|\alpha\|_{L^{\infty}}
		\lesssim \exp(O(\log_+^2\lambda_i)) \|\alpha\|_{\H}.
	\end{align}
	This simplified estimate is Proposition~\ref{prop:L^infty_quasi-Sobolev}, and this is the technical core of Subsection~\ref{subsec:L^infty_quasi-Sobolev}. If the $L^{\infty}$ norm were replaced by the $L^4$ norm, then we would already have \eqref{eqn:outline:L^infty_quasimode_bd} by Theorem~\ref{thm:L^4_quasi-Sobolev} from Step~\ref{step:bulk_tail}. With this in mind, the idea of the proof is to obtain $L^{\infty}$ control from $L^4$ control by an iterative process where at each step, $L^p$ control is boosted to $L^{2p}$ control. To make this precise, we need a good notion of $L^p$ norm on $\H^{\fin}$ for $p>4$ a power of $2$. The naive guess, extrapolating from the definition of the $L^4$ norm, would be to define $\|\alpha\|_{L^p} = \|\alpha^{\frac{p}{2}}\|_{\H}^{\frac{2}{p}}$. This unfortunately doesn't make sense, because $\alpha^{\frac{p}{2}}$ is a product of at least three elements of $\H^{\fin}$ when $p>4$, which typically is not defined. Indeed, as discussed above \eqref{eqn:associativity_vs_crossing}, the product of two elements of $\H^{\fin}$ is typically not in $\H^{\fin}$, and so cannot be multiplied by a third element. Nevertheless, let us pretend for now that we do have $L^p$ norms not just on $\H^{\fin}$ but on $\H^{\infty}$, and that these $L^p$ norms satisfy the usual properties. Then $\|\alpha\|_{L^{\infty}} = \lim_{p \to \infty} \|\alpha\|_{L^p}$, and for each $p$,
	\begin{align} \label{eqn:outline_L^2p_bdd_by_L^p}
		\|\alpha\|_{L^{2p}}^2
		= \||\alpha|^2\|_{L^p}
		\leq \|\1_{\Delta \gg \lambda_i}(|\alpha|^2)\|_{L^p} + \sum_{j \colon \lambda_j \lesssim \lambda_i} \|\1_{\Delta \in I_j}(|\alpha|^2)\|_{L^p},
	\end{align}
	where the inequality is by the triangle inequality. Assume by induction on $p$ that we already know \eqref{eqn:outline:L^infty_quasimode_bd} with $L^{\infty}$ replaced by $L^p$ (with the implicit and $O$-constants independent of $p$) for all $i$ and all $\alpha \in \H_{\Delta \in I_i}^K$. Then in \eqref{eqn:outline_L^2p_bdd_by_L^p}, the tail $\1_{\Delta \gg \lambda_i}(|\alpha|^2)$ can be discarded by induction and Theorem~\ref{thm:exp_decay_quasimode}, and the bulk terms $\1_{\Delta \in I_j}(|\alpha|^2)$ for $\lambda_j \lesssim \lambda_i$ can be estimated by induction and Theorem~\ref{thm:L^4_quasi-Sobolev}. This recovers \eqref{eqn:outline:L^infty_quasimode_bd} with $L^{\infty}$ replaced by $L^{2p}$ (with the same implicit and $O$-constants as for $L^p$). Taking $p \to \infty$, and using the independence of the constants from $p$, yields \eqref{eqn:outline:L^infty_quasimode_bd}. In practice, we will be able to run this argument using certain \textit{ad hoc} ``modified $L^p$ norms" on $\H^{\fin}$. 
	The reason we wanted the $L^p$ norms above to be defined on all of $\H^{\infty}$ instead of just $\H^{\fin}$ is so that we could make sense of $\||\alpha|^2\|_{L^p}$ for $\alpha \in \H^{\fin}$. We get around this by interpreting $\||\alpha|^2\|_{L^p}$ as $\liminf_{\Lambda \to \infty} \|\1_{|\Delta| \leq \Lambda}(|\alpha|^2)\|_{L^p}$, using that $\1_{\Delta \leq \Lambda}(|\alpha|^2) \in \H^{\fin}$.
	We will see that the modified $L^p$ norms share enough of the usual properties of $L^p$ norms to make the above argument go through more or less unchanged.
	
	Let us now turn to Theorem~\ref{thm:high_deriv_L^infty_bd}. The $n=0$ case of Theorem~\ref{thm:high_deriv_L^infty_bd} says that $\|\alpha\|_{L^{\infty}} < \infty$ for all $\alpha \in \H^{\fin}$, and the proof gives some information on how $\|\alpha\|_{L^{\infty}}$ depends on $\alpha$. In particular, applying the $n=0$ case of Theorem~\ref{thm:high_deriv_L^infty_bd} to $X^n\alpha$ for $n \in \Z_{\geq 0}$, we get $\|X^n\alpha\|_{L^{\infty}} < \infty$, and the proof gives a bound on $\|X^n\alpha\|_{L^{\infty}}$ in terms of $n$ and $\alpha$. This bound is the general case of Theorem~\ref{thm:high_deriv_L^infty_bd}. Therefore, to understand the general case, it essentially suffices to understand the $n=0$ case. So let $\alpha \in \H^{\fin}$, and let us explain why $\|\alpha\|_{L^{\infty}} < \infty$ (see the proof of Proposition~\ref{prop:bdd_mult_qual} for details). By Proposition~\ref{prop:Maass_span_H^fin} and the triangle inequality, we may assume $\alpha$ is an automorphic vector, and by complex conjugation symmetry, we may assume $\alpha$ has nonnegative weight. Then Proposition~\ref{prop:Maass_is_raised} gives two alternatives for the structure of $\alpha$, and both are treated in the same way. Suppose for concreteness that the first alternative holds, so $\alpha = E^n\varphi$ for some automorphic vector $\varphi \in \H^K$ and some $n \in \Z_{\geq 0}$. Splitting $\varphi$ into real and imaginary parts, we may assume $\varphi \in \H_{\R}$. If $n=0$, then $\alpha = \varphi \in \H^K$ and Theorem~\ref{thm:L^infty_quasi-Sobolev} gives $\|\alpha\|_{L^{\infty}} < \infty$, so the first new case is $n=1$, i.e., $\alpha = E\varphi$.
	Let us explain this case first.
	
	In general, we will show that for any $\beta \in \H^{\fin}$, one has
	\begin{align} \label{eqn:TT*_L^infty}
		\|\beta\|_{L^{\infty}}^2 \leq \||\beta|^2\|_{L^{\infty}}
	\end{align}
	(again, technically $\||\beta|^2\|_{L^{\infty}}$ is not defined because $|\beta|^2 \not\in \H^{\fin}$, but we interpret $\||\beta|^2\|_{L^{\infty}}$ as $\liminf_{\Lambda \to \infty} \|\1_{|\Delta| \leq \Lambda}(|\beta|^2)\|_{L^{\infty}}$).
	Of course, one expects equality in \eqref{eqn:TT*_L^infty}, but $\leq$ suffices.
	
	Now, we want to use \eqref{eqn:TT*_L^infty} to show that $\alpha = E\varphi$ has finite $L^{\infty}$ norm. So we want to show $\||E\varphi|^2\|_{L^{\infty}} < \infty$ (with the left hand side interpreted as above). Since $|E\varphi|^2$ has weight zero, it suffices by quasi-Sobolev embedding (Theorem~\ref{thm:L^infty_quasi-Sobolev}) to check that
	\begin{align} \label{eqn:outline:after_quasi-Sobolev}
		\|\exp(O(\log_+^2\Delta))(|E\varphi|^2)\|_{\H}
		< \infty.
	\end{align}
	If $|E\varphi|^2$ were replaced by $\varphi^2$, then this finiteness would follow from Corollary~\ref{cor:exp_decay_quasimode_qual} from Step~\ref{step:bulk_tail}. More generally, if $|E\varphi|^2$ were replaced by $\Delta^m(\varphi^2)$ for any $m \in \Z_{\geq 0}$, then we would still have finiteness, because $\Delta^m$ could be absorbed into $\exp(O(\log_+^2\Delta))$. The key observation which gives \eqref{eqn:outline:after_quasi-Sobolev} is that $|E\varphi|^2$ is a finite linear combination of the $\Delta^m(\varphi^2)$ for $m \in \Z_{\geq 0}$. Indeed, by the formula \eqref{eqn:Delta|_H^K} for the Casimir restricted to $\H^K$, and by the product rule,
	\begin{align*}
		\Delta(\varphi^2)
		= -\overline{E}E(\varphi^2)
		= -2\overline{E}(\varphi E\varphi)
		= -2E\varphi \overline{E}\varphi - 2\varphi \overline{E}E\varphi
		= -2|E\varphi|^2 + 2\varphi\Delta\varphi
		= -2|E\varphi|^2 + 2\lambda \varphi^2,
	\end{align*}
	where $\lambda$ is the Casimir eigenvalue of $\varphi$ (we used $\varphi \in \H_{\R}$ to write $E\varphi \overline{E}\varphi = |E\varphi|^2$). Rearranging,
	\begin{align} \label{eqn:outline:|Ephi|^2}
		|E\varphi|^2
		= \Big(\lambda - \frac{1}{2}\Delta\Big)(\varphi^2).
	\end{align}
	From this we obtain \eqref{eqn:outline:after_quasi-Sobolev} and hence $\|E\varphi\|_{L^{\infty}} < \infty$ as explained above.
	
	Recall that earlier we had $\alpha = E^n\varphi$ for some $n \in \Z_{\geq 0}$.
	The above handles the case $n=1$.
	To get $\|E^n\varphi\|_{L^{\infty}} < \infty$ for all $n$, the same method works, except that at the end we need to express $|E^n\varphi|^2$ as a linear combination of the $\Delta^m(\varphi^2)$ for $m \in \Z_{\geq 0}$. This is done in Subsection~\ref{subsec:basic_identities}, where we write down explicit polynomials $p_n$ in two variables with real coefficients such that
	\begin{align} \label{eqn:outline:p_n_def}
		|E^n\varphi|^2
		= p_n(\lambda,\Delta)(\varphi^2)
	\end{align}
	(see Proposition~\ref{prop:p_n_def}).
	For example, $p_0(\lambda,\mu) = 1$ because $|E^0\varphi|^2 = |\varphi|^2 = \varphi^2$ (remember $\varphi$ is real), and $p_1(\lambda,\mu) = \lambda - \frac{1}{2}\mu$ corresponding to \eqref{eqn:outline:|Ephi|^2}.
	This concludes our discussion of Theorem~\ref{thm:high_deriv_L^infty_bd}.
	
	\subsection{Main ideas in Section~\ref{sec:bulk_tail_poly}}
	\label{subsec:outline:bulk_tail_poly}
	
	Sections~\ref{sec:bulk_tail_poly} and \ref{sec:bulk_tail_general} do Step~\ref{step:bulk_tail}, which consists of the bulk and tail estimates in Theorems~\ref{thm:L^4_quasi-Sobolev}, \ref{thm:exp_decay_form}, and \ref{thm:exp_decay_quasimode}.
	The proofs of these three theorems are quite technical, especially Theorems~\ref{thm:L^4_quasi-Sobolev} and \ref{thm:exp_decay_quasimode}. Therefore, as a warmup, Section~\ref{sec:bulk_tail_poly} gives proofs which at various points make the simplifying assumption that $\H$ obeys a polynomial Weyl law (Definition~\ref{def:poly_Weyl_law}). We note, however, that the proof of Theorem~\ref{thm:exp_decay_form} in Section~\ref{sec:bulk_tail_poly} is unconditional.
	
	In this subsection, assume $\H$ obeys a polynomial Weyl law. The proof of Theorem~\ref{thm:exp_decay_form} is similar and easier than that of Theorem~\ref{thm:exp_decay_quasimode}, so let us focus on Theorems~\ref{thm:L^4_quasi-Sobolev} and \ref{thm:exp_decay_quasimode}. A major benefit of the polynomial Weyl law is that there exists a partition of $[0,\infty)$ into intervals $I_i$ of length $\lesssim 1$, which has polynomial growth in the sense of \eqref{eqn:partition_growth}, and such that each $I_i$ contains at most one Casimir eigenvalue (not counting multiplicity). Take the $I_i$ in Theorem~\ref{thm:exp_decay_quasimode} to be such a partition. Then Theorem~\ref{thm:exp_decay_quasimode} simplifies to the statement that there exists $c \gtrsim 1$ such that for all $\lambda \geq 0$, all $\alpha,\beta \in \H_{\Delta=\lambda}^K$, and any $M \gg \lambda+1$,
	\begin{align*}
		\|\1_{\Delta \geq M} \exp(c\sqrt{\Delta})(\alpha\beta)\|_{\H}
		\leq \|\alpha\|_{\H} \|\beta\|_{\H}.
	\end{align*}
	By splitting $\alpha,\beta$ into real and imaginary parts, and then applying the polarization identity to express $\alpha\beta$ as a linear combination of squares, it suffices to show that
	\begin{align} \label{eqn:outline:exp_decay_mode_new_H}
		\|\1_{\Delta \geq M} \exp(c\sqrt{\Delta})(\varphi^2)\|_{\H}
		\leq \|\varphi\|_{\H}^2
	\end{align}
	for all $\varphi \in \H_{\Delta=\lambda}^K \cap \H_{\R}$. This is Proposition~\ref{prop:exp_decay_mode_after_dyadic}.
	An immediate corollary of \eqref{eqn:outline:exp_decay_mode_new_H} would be
	\begin{align} \label{eqn:outline:exp_decay_mode_old_H}
		\|\1_{\Delta \geq M}(\varphi^2)\|_{\H}
		\leq \exp(-c\sqrt{M}) \|\varphi\|_{\H}^2.
	\end{align}
	Conversely, a dyadic decomposition argument shows that if one knows \eqref{eqn:outline:exp_decay_mode_old_H} for all $M \gg \lambda+1$, then one obtains \eqref{eqn:outline:exp_decay_mode_new_H} for $M \gg \lambda+1$ (after increasing the implicit constant in $M \gg \lambda+1$ and decreasing $c$). Thus it suffices to show \eqref{eqn:outline:exp_decay_mode_old_H}. This is \eqref{eqn:exp_decay_mode_old_L^2} in Proposition~\ref{prop:exp_decay_mode_old}. It is easier to show the same inequality with $\|\varphi\|_{L^4}$ on the right hand side instead of $\|\varphi\|_{\H}$, i.e.,
	\begin{align} \label{eqn:outline:exp_decay_mode_old_L^4}
		\|\1_{\Delta \geq M}(\varphi^2)\|_{\H}
		\leq \exp(-c\sqrt{M}) \|\varphi\|_{L^4}^2,
	\end{align}
	and in fact we will prove this without the polynomial Weyl law assumption.
	Once we have established Theorem~\ref{thm:L^4_quasi-Sobolev} (quasi-Sobolev embedding into $L^4$), it will be easy to deduce \eqref{eqn:outline:exp_decay_mode_old_H} from \eqref{eqn:outline:exp_decay_mode_old_L^4}, again after increasing the implicit constant in $M \gg \lambda+1$ and decreasing $c$. The inequality \eqref{eqn:outline:exp_decay_mode_old_L^4} is \eqref{eqn:exp_decay_mode_old_L^4} in Proposition~\ref{prop:exp_decay_mode_old}. To illustrate the method of proof for this inequality, let us derive the weaker estimate
	\begin{align} \label{eqn:outline:poly_decay_mode}
		\|\1_{\Delta \geq M}(\varphi^2)\|_{\H}
		\leq M^{-m} \|\varphi\|_{L^4}^2
		\qquad \text{for all} \qquad
		M \gg_{m,\varepsilon} \lambda^{1+\varepsilon} + 1,
	\end{align}
	for any fixed $m \geq 0$ and $\varepsilon>0$. Let $n$ be a large positive integer to be chosen later. By crossing symmetry and \eqref{eqn:outline:p_n_def}, we can write
	\begin{align}
		0 \leq \|E^{n+1}\varphi \overline{E}^n\varphi\|_{\H}^2
		&= \langle E^{n+1}\varphi \overline{E^n\varphi}, E^{n+1}\varphi \overline{E^n\varphi} \rangle_{\H}
		\notag
		\\&= \langle |E^n\varphi|^2, |E^{n+1}\varphi|^2 \rangle_{\H}
		= \langle p_n(\lambda,\Delta)(\varphi^2), p_{n+1}(\lambda,\Delta)(\varphi^2) \rangle_{\H}.
		\label{eqn:outline:derivation_of_basic_ineq_1}
	\end{align}
	Negating both sides, and using that $\Delta$ is self-adjoint and the $p_n$ have real coefficients, we get
	\begin{align} \label{eqn:outline:basic_ineq_1}
		-\langle p_n p_{n+1}(\lambda,\Delta)(\varphi^2), \varphi^2 \rangle_{\H} \leq 0.
	\end{align}
	The inequality \eqref{eqn:outline:basic_ineq_1} is Proposition~\ref{prop:basic_ineq_1}.
	Isolating the tail by moving the bulk to the right hand side in \eqref{eqn:outline:basic_ineq_1},
	\begin{align} \label{eqn:outline:basic_ineq_1_rearranged}
		-\langle \1_{\Delta \gg_n \lambda+1} \, p_np_{n+1}(\lambda,\Delta)(\varphi^2), \varphi^2 \rangle_{\H}
		\leq \langle \1_{\Delta \lesssim_n \lambda+1} \, p_np_{n+1}(\lambda,\Delta)(\varphi^2), \varphi^2 \rangle_{\H}
	\end{align}
	(where the spectral cutoffs on both sides are taken at the same point). It follows from the construction of $p_n$ that $p_n(\lambda,\mu)$ is a polynomial in $\lambda,\mu$ of total degree $n$, and for $n \geq 1$, when viewed as a polynomial in $\mu$ with $\lambda$ fixed, $p_n(\lambda,\mu)$ has degree $n$ and leading coefficient $(-1)^n\frac{1}{2}$. We can see this already for $n=0,1$, because as explained after \eqref{eqn:outline:p_n_def}, we have $p_0(\lambda,\mu) = 1$ and $p_1(\lambda,\mu) = \lambda - \frac{1}{2}\mu$. Consequently, for $n \geq 1$, the polynomial $-p_np_{n+1}(\lambda,\mu)$ has total degree $2n+1$ in $\lambda,\mu$, and has degree $2n+1$ with leading coefficient $\frac{1}{4}$ in $\mu$. Thus the leading term will dominate when $\mu \gg_n \lambda+1$. To be precise, for $\mu \gg_n \lambda+1$, we have $-p_np_{n+1}(\lambda,\mu) \geq 0$ and in fact $-p_np_{n+1}(\lambda,\mu) \sim \mu^{2n+1}$. Consequently,
	\begin{align*}
		M^{2n+1} \1_{\mu \geq M}
		\lesssim -\1_{\mu \gg_n \lambda+1} \, p_np_{n+1}(\lambda,\mu)
		\qquad \text{for all} \qquad
		M \gg_n \lambda+1 \text{ and } \mu \geq 0.
	\end{align*}
	Thus for $M \gg_n \lambda+1$,
	\begin{align*}
		M^{2n+1} \|\1_{\Delta \geq M}(\varphi^2)\|_{\H}^2
		= \langle M^{2n+1} \1_{\Delta \geq M}(\varphi^2), \varphi^2 \rangle_{\H}
		\lesssim -\langle \1_{\Delta \gg_n \lambda+1} \, p_np_{n+1}(\lambda,\Delta)(\varphi^2), \varphi^2 \rangle_{\H}.
	\end{align*}
	Applying \eqref{eqn:outline:basic_ineq_1_rearranged},
	\begin{align*}
		M^{2n+1} \|\1_{\Delta \geq M}(\varphi^2)\|_{\H}^2
		\lesssim \langle \1_{\Delta \lesssim_n \lambda+1} \, p_np_{n+1}(\lambda,\Delta)(\varphi^2), \varphi^2 \rangle_{\H}.
	\end{align*}
	Since $p_n$ has total degree $n$, we have $|p_n(\lambda,\mu)| \lesssim_n (\lambda+1)^n$ for $\mu \lesssim_n \lambda+1$. Therefore
	\begin{align*}
		M^{2n+1} \|\1_{\Delta \geq M}(\varphi^2)\|_{\H}^2
		\lesssim_n (\lambda+1)^{2n+1} \langle \varphi^2, \varphi^2 \rangle_{\H}
		= (\lambda+1)^{2n+1} \|\varphi\|_{L^4}^4.
	\end{align*}
	Dividing both sides by $M^{2n+1}$ and taking square roots,
	\begin{align*}
		\|\1_{\Delta \geq M}(\varphi^2)\|_{\H}
		\lesssim_n \Big(\frac{\lambda+1}{M}\Big)^{n+\frac{1}{2}} \|\varphi\|_{L^4}^2.
	\end{align*}
	Restricting to $M \gg_{m,\varepsilon} \lambda^{1+\varepsilon}+1$ and fixing $n$ sufficiently large depending on $m,\varepsilon$, we obtain
	\begin{align*}
		\|\1_{\Delta \geq M}(\varphi^2)\|_{\H}
		\lesssim_{m,\varepsilon} M^{-m} \|\varphi\|_{L^4}^2.
	\end{align*}
	This is almost the desired bound \eqref{eqn:outline:poly_decay_mode}, except we have $\lesssim_{m,\varepsilon}$ instead of $\leq$. This is fixed by applying the same bound with $m+1$ in place of $m$, and then increasing $M$ to absorb the implicit constant. We have now established \eqref{eqn:outline:poly_decay_mode}. To get \eqref{eqn:outline:exp_decay_mode_old_L^4} we squeeze a little more out of the same argument, analyzing $p_n$ more carefully and then optimizing the choice of $n$ in terms of $\lambda$ and $M$. This finishes our outline of the proof of Theorem~\ref{thm:exp_decay_quasimode}, and by extension the proof of Theorem~\ref{thm:exp_decay_form} which is similar.
	
	Let us now move on to Theorem~\ref{thm:L^4_quasi-Sobolev}, the qualitative Sobolev embedding theorem. Recall that Theorem~\ref{thm:L^4_quasi-Sobolev} is used in the proof of Theorem~\ref{thm:exp_decay_quasimode} to pass from \eqref{eqn:outline:exp_decay_mode_old_L^4} to \eqref{eqn:outline:exp_decay_mode_old_H}, so in Section~\ref{sec:bulk_tail_poly} we will actually prove Theorem~\ref{thm:L^4_quasi-Sobolev} first. A special case of Theorem~\ref{thm:L^4_quasi-Sobolev} is that when $\alpha = \varphi \in \H^K \cap \H^{\fin}$ is an automorphic vector,
	\begin{align} \label{eqn:outline:L^4_efn_bd}
		\|\varphi\|_{L^4}
		\leq \exp(O(\log_+^2\lambda)) \|\varphi\|_{\H},
	\end{align}
	where $\lambda \geq 0$ is the Casimir eigenvalue of $\varphi$. Conversely, if one knows \eqref{eqn:outline:L^4_efn_bd} for all $K$-invariant automorphic vectors $\varphi$, then one can recover Theorem~\ref{thm:L^4_quasi-Sobolev} by writing
	\begin{align} \label{eqn:outline:L^4_spectral_decomp_triangle_ineq}
		\|\alpha\|_{L^4}
		\leq \sum_{\lambda \geq 0} \|\1_{\Delta=\lambda} \alpha\|_{L^4}
	\end{align}
	and then estimating the right hand side by applying \eqref{eqn:outline:L^4_efn_bd} with $\varphi = \1_{\Delta=\lambda}\alpha$. In order for the application of the triangle inequality in \eqref{eqn:outline:L^4_spectral_decomp_triangle_ineq} not to be too lossy, it is crucial that we are assuming $\H$ obeys a polynomial Weyl law, or else the sum on the right hand side of \eqref{eqn:outline:L^4_spectral_decomp_triangle_ineq} would contain too many terms. The reduction of Theorem~\ref{thm:L^4_quasi-Sobolev} to \eqref{eqn:outline:L^4_efn_bd} takes place after the statement of Proposition~\ref{prop:C_s(Lambda)_inductive_bd}.
	It remains to prove \eqref{eqn:outline:L^4_efn_bd}.
	By splitting $\varphi$ into real and imaginary parts, we may assume $\varphi$ is real.
	We then prove \eqref{eqn:outline:L^4_efn_bd} by a sort of induction on $\lambda$. The base case $\lambda \lesssim 1$ is easy (see Lemma~\ref{lem:C(Lambda)_finite}), so let $\lambda \gg 1$. Then make the inductive assumption that \eqref{eqn:outline:L^4_efn_bd} holds for all automorphic vectors $\tilde{\varphi} \in \H_{\R}^K \cap \H^{\fin}$ with Casimir eigenvalue $\tilde{\lambda} \leq \frac{1}{2}\lambda$, and let us sketch the proof of \eqref{eqn:outline:L^4_efn_bd} for $\varphi,\lambda$. One of course has to be careful about implicit constants in inductive arguments, but in this outline we will be deliberately vague. Split
	\begin{align} \label{eqn:outline:L^4_split}
		\|\varphi\|_{L^4}^2
		= \|\varphi^2\|_{\H}
		\leq \|\1_{\Delta \leq \frac{1}{2}\lambda}(\varphi^2)\|_{\H} + \|\1_{\Delta > \frac{1}{2}\lambda}(\varphi^2)\|_{\H}.
	\end{align}
	Although it is not immediately apparent, we will be able to deal with the first term $\|\1_{\Delta \leq \frac{1}{2}\lambda}(\varphi^2)\|_{\H}$ by induction and a little extra massaging. Let us focus on the second term $\|\1_{\Delta > \frac{1}{2}\lambda}(\varphi^2)\|_{\H}$.
	Above, we estimated $\|\1_{\Delta \geq M}(\varphi^2)\|_{\H}$ for $M \gg \lambda$ (our final estimate was conditional on Theorem~\ref{thm:L^4_quasi-Sobolev}, but \eqref{eqn:outline:exp_decay_mode_old_L^4} was not). If we could push $M$ down to $\frac{1}{2}\lambda$, we would win. This is not possible if we ask for the same quality bound as in \eqref{eqn:outline:exp_decay_mode_old_L^4}, but we will be able to get a weaker (but sufficient) bound by a similar method. Above, the key inequality was \eqref{eqn:outline:basic_ineq_1}, which after rearranging allowed us to relate the tail to the bulk via \eqref{eqn:outline:basic_ineq_1_rearranged}. The first step in the derivation of \eqref{eqn:outline:basic_ineq_1} is the first inequality in \eqref{eqn:outline:derivation_of_basic_ineq_1}, which trivially lower bounds $\|E^{n+1}\varphi \overline{E}^n\varphi\|_{\H}^2$ by $0$. In the second half of Subsection~\ref{subsec:basic_ineqs} we find a nontrivial lower bound, which leads to the following modification of \eqref{eqn:outline:basic_ineq_1}:
	\begin{align*}
		\langle r_n(\lambda,\Delta)(\varphi^2), \varphi^2 \rangle_{\H}
		\leq 0,
	\end{align*}
	where $r_n(\lambda,\mu)$ is an explicit real-valued function of two variables defined by \eqref{eqn:r_n_def}. We then define a new function $R(\lambda,\mu)$ which is a positive linear combination of the $r_n$. By positivity,
	\begin{align*}
		\langle R(\lambda,\Delta)(\varphi^2), \varphi^2 \rangle_{\H} \leq 0.
	\end{align*}
	Rearranging as in \eqref{eqn:outline:basic_ineq_1_rearranged},
	\begin{align*}
		\langle \1_{\Delta > \frac{1}{2}\lambda} R(\lambda,\Delta)(\varphi^2), \varphi^2 \rangle_{\H}
		\leq -\langle \1_{\Delta \leq \frac{1}{2}\lambda} R(\lambda,\Delta)(\varphi^2), \varphi^2 \rangle_{\H}.
	\end{align*}
	The key property of $R(\lambda,\mu)$ is that for $\mu > \frac{1}{2}\lambda$, one has $R(\lambda,\mu) \gtrsim 1$ (see Lemma~\ref{lem:R_def}). One also has the trivial bound $|R(\lambda,\mu)| \lesssim \lambda^{O(1)}$ for $\mu \lesssim \lambda$. Therefore,
	\begin{align*}
		\|\1_{\Delta > \frac{1}{2}\lambda}(\varphi^2)\|_{\H}^2
		= \langle \1_{\Delta > \frac{1}{2}\lambda}(\varphi^2), \varphi^2 \rangle_{\H}
		&\lesssim \langle \1_{\Delta > \frac{1}{2}\lambda} R(\lambda,\Delta)(\varphi^2), \varphi^2 \rangle_{\H}
		\\&\leq -\langle \1_{\Delta \leq \frac{1}{2}\lambda} R(\lambda,\Delta)(\varphi^2), \varphi^2 \rangle_{\H}
		\lesssim \lambda^{O(1)} \|\1_{\Delta \leq \frac{1}{2}\lambda}(\varphi^2)\|_{\H}.
	\end{align*}
	This says that the second term in \eqref{eqn:outline:L^4_split} is bounded by $\lambda^{O(1)}$ times the first term. As mentioned after \eqref{eqn:outline:L^4_split}, the first term can be treated by induction with a little work.
	
	At each step of the induction, we lose a factor of $\lambda^{O(1)}$. In general, if $C(\lambda)$ is a function of $\lambda \geq 0$ satisfying $C(\lambda) \lesssim 1$ for $\lambda \lesssim 1$ and $C(\lambda) \lesssim \lambda^{O(1)} C(\frac{1}{2}\lambda)$ for $\lambda \gg 1$, then
	\begin{align*}
		C(\lambda) \leq \exp(O(\log_+^2\lambda)).
	\end{align*}
	This explains the quasipolynomial dependence on $\lambda$ in Theorem~\ref{thm:L^4_quasi-Sobolev}.
	
	\subsection{Main ideas in Section~\ref{sec:bulk_tail_general}}
	\label{subsec:outline:bulk_tail_general}
	
	Section~\ref{sec:bulk_tail_general} proves Theorems~\ref{thm:L^4_quasi-Sobolev} and \ref{thm:exp_decay_quasimode} unconditionally, completing Step~\ref{step:bulk_tail} (recall that the proof of Theorem~\ref{thm:exp_decay_form} in Section~\ref{sec:bulk_tail_poly} is already unconditional). The overall structure of the proofs of Theorems~\ref{thm:L^4_quasi-Sobolev} and \ref{thm:exp_decay_quasimode} is the same as in Section~\ref{sec:bulk_tail_poly}, but there are additional technical difficulties when $\H$ is not required to obey a polynomial Weyl law. We describe these difficulties and hint at how they are resolved below.
	
	We often want bounds on some $v \in \H^K$. If $\H$ obeys a polynomial Weyl law, then a typical strategy is to decompose
	\begin{align*}
		v = \sum_{\lambda \geq 0} \1_{\Delta=\lambda}v,
	\end{align*}
	estimate $\1_{\Delta=\lambda}v$ using that it is a $K$-invariant Casimir eigenvector, and then use the triangle inequality to recover a bound on $v$.
	Without the polynomial Weyl law, this application of the triangle inequality is unacceptably lossy.
	Therefore, in general, we instead decompose
	\begin{align*}
		v = \sum_i \1_{\Delta \in I_i}v
	\end{align*}
	as a sum of $K$-invariant approximate Casimir eigenvectors, where the $I_i$ are as in Theorem~\ref{thm:exp_decay_quasimode}. The condition in Theorem~\ref{thm:exp_decay_quasimode} that the partition $\{I_i\}$ of $[0,\infty)$ has polynomial growth serves as a substitute for the polynomial Weyl law. However, to estimate $\1_{\Delta \in I_i}v$, all of our machinery from Subsection~\ref{subsec:basic_identities} and Section~\ref{sec:bulk_tail_poly} for controlling $K$-invariant Casimir eigenvectors must be extended to approximate eigenvectors. In particular, the identity \eqref{eqn:outline:p_n_def} was crucial, and we need a generalization of \eqref{eqn:outline:p_n_def} for approximate eigenvectors $\alpha \in \H_{\R}^K \cap \H^{\fin}$.
	In fact, we need a formula not just for $|E^n\alpha|^2$, but for the real part of $E^n\alpha \overline{E}^n\beta$ whenever $\alpha,\beta \in \H_{\R}^K \cap \H^{\fin}$ are approximate eigenvectors with the same approximate eigenvalue. Let $\lambda \geq 0$ be a common approximate eigenvalue of $\alpha,\beta$. In Proposition~\ref{prop:p_n,i,j_def}, we will show that
	\begin{align} \label{eqn:outline:correction_terms}
		\re E^n\alpha \overline{E}^n\beta
		= p_n(\lambda,\Delta)(\alpha\beta) + \text{correction terms},
	\end{align}
	where $p_n$ is as in \eqref{eqn:outline:p_n_def}, and where the correction terms are $\R[\lambda,\Delta]$-linear combinations of the products $(\Delta-\lambda)^i\alpha (\Delta-\lambda)^j\beta$ for $i+j > 0$. This identity is true for all $\alpha,\beta \in \H_{\R}^K \cap \H^{\fin}$ and all $\lambda \in \R$, but the correction terms are only small when $\alpha,\beta$ are approximate $\lambda$-eigenvectors. When $\alpha,\beta$ are exact $\lambda$-eigenvectors, the correction terms vanish, and when in addition $\alpha=\beta$, the above identity reduces to \eqref{eqn:outline:p_n_def}.
	
	We have not yet explained where the discrete spectrum condition in Definition~\ref{def:mult_rep} is used; by Remark~\ref{rem:discrete_vs_pure_pt}, it must be used somewhere.
	It is used superficially in several places, but the first time it is used in an unavoidable way is in Section~\ref{sec:bulk_tail_general}. This is perhaps surprising, because discrete spectrum is a qualitative hypothesis, while Section~\ref{sec:bulk_tail_general} is quantitative in nature.
	Recall from Remark~\ref{rem:using_lambda_r_to_infty} that we will translate qualitative information into quantitative information via certain ``self-improving" inequalities.
	The toy example given in Remark~\ref{rem:using_lambda_r_to_infty} is the inequality $C \leq 1+\frac{1}{2}C$ for $C \in [0,\infty]$, which combined with the qualitative hypothesis $C<\infty$ implies the quantitative bound $C \leq 2$.
	In Section~\ref{sec:bulk_tail_general}, we will prove such inequalities with $C$ essentially a comparison constant between norms on the tensor product of two approximate eigenspaces in $\H^K$. For example, suppose given $\lambda \geq 1$ and $\|\cdot\|'$ a norm of interest on $\H_{|\Delta-\lambda| \leq \lambda^{-5}}^K \otimes \H_{|\Delta-\lambda| \leq \lambda^{-5}}^K$. Consider the smallest $C \in [0,\infty]$ such that
	\begin{align*}
		\|\alpha \otimes \beta\|' \leq C\|\alpha \otimes \beta\|_{\H \otimes \H}
		= \|\alpha\|_{\H} \|\beta\|_{\H}
	\end{align*}
	for all $\alpha,\beta \in \H_{|\Delta-\lambda| \leq \lambda^{-5}}^K$.
	In practice, $\|\alpha \otimes \beta\|'$ will depend only on $\alpha\beta$, so an upper bound on $C$ will mean quantitative control on the multiplication map $\H^{\fin} \times \H^{\fin} \to \H^{\infty}$.
	Suppose we wish to prove $C \lesssim 1$.
	Before using discrete spectrum, we will typically be able to prove something like
	\begin{align} \label{eqn:C_bd_before_disc_spectrum}
		C \lesssim 1 + \lambda^{-1}C,
	\end{align}
	where the term $\lambda^{-1} C$ arises when estimating the correction terms from \eqref{eqn:outline:correction_terms} in $\|\cdot\|'$, e.g., writing
	\begin{align} \label{eqn:correction_terms_bd}
		\|\text{correction terms}\|'
		\leq C\|\text{correction terms}\|_{\H \otimes \H}
		\leq \lambda^{-1}C
	\end{align}
	with the second inequality capturing the fact that $\H_{|\Delta-\lambda| \leq \lambda^{-5}}^K$ is an approximate eigenspace and hence the correction terms are small.
	Technically, \eqref{eqn:correction_terms_bd} is ill-defined as written because the correction terms live in $\H$ rather than $\H \otimes \H$, but they are $\R[\lambda,\Delta]$-linear combinations of images of elements of $\H \otimes \H$ under multiplication, and it is these elements which we estimate in $\|\cdot\|'$. 
	The fact that the term $\lambda^{-1}C$ comes from the correction terms in \eqref{eqn:outline:correction_terms} explains why this term does not appear in Section~\ref{sec:bulk_tail_poly}.
	Now, since $\H$ has discrete spectrum, $\H_{|\Delta-\lambda| \leq \lambda^{-5}}^K$ is finite-dimensional, and any two norms on a finite-dimensional space are equivalent, so $C<\infty$. Thus when $\lambda \gg 1$, we can move the term $\lambda^{-1}C$ to the left hand side in \eqref{eqn:C_bd_before_disc_spectrum} to get $C \lesssim 1$, as desired.
	
	\begin{rem} \label{rem:discrete_spectrum_first_use}
		The first place where we make this sort of argument is in the proof of Theorem~\ref{thm:L^4_quasi-Sobolev} in Subsection~\ref{subsec:bulk_general}.
		There \eqref{eqn:C_bd_before_disc_spectrum_true} plays the role of \eqref{eqn:C_bd_before_disc_spectrum} (c.f. the derivation of \eqref{eqn:C_bd_after_disc_spectrum_true} from \eqref{eqn:C_bd_before_disc_spectrum_true}).
		It is crucial that discrete spectrum is used there, because Theorem~\ref{thm:L^4_quasi-Sobolev} fails for the example described in Remark~\ref{rem:discrete_vs_pure_pt} with countable spectrum.
	\end{rem}
	
	In order to use \eqref{eqn:outline:correction_terms} to prove Theorem~\ref{thm:exp_decay_quasimode}, we need to make sure the main term $p_n(\lambda,\Delta)(\alpha\beta)$ in \eqref{eqn:outline:correction_terms} dominates the correction terms uniformly in the relevant ranges of parameters. In particular, we need a lower bound on $|p_n(\lambda,\mu)|$ for $\mu \gg \lambda+n^2$ which is sharp up to a multiplicative constant independent of $n,\lambda,\mu$. For now, let us suppress the dependence on $\lambda,\mu$, and focus on getting a lower bound which is sharp uniformly in $n$. Uniformity in $\lambda,\mu$ will come for free. The polynomials $p_n$ are defined by a second order linear recurrence in $n$ with non-constant coefficients (see Proposition~\ref{prop:p_n_def}). Equivalently,
	\begin{align*}
		\begin{pmatrix}
			p_{n+1} \\
			p_n
		\end{pmatrix}
		= A_n
		\begin{pmatrix}
			p_n \\
			p_{n-1}
		\end{pmatrix}
	\end{align*}
	for some explicit $2\times 2$ matrix $A_n$ depending on $n$ (and also on $\lambda,\mu$). Iterating this,
	\begin{align*}
		\begin{pmatrix}
			p_{n+1} \\
			p_n
		\end{pmatrix}
		= A_n \cdots A_1
		\begin{pmatrix}
			p_1 \\
			p_0
		\end{pmatrix}.
	\end{align*}
	Thus to estimate $p_n$, it suffices to estimate $A_n \cdots A_1$. When $\mu \gg \lambda+n^2$, the matrices $A_m$ for $m=1,\dots,n$ are diagonalizable, so we can write $A_m = Q_m D_m Q_m^{-1}$ with $D_m$ diagonal. 
	If $A_1,\dots,A_n$ were commuting matrices, then they would be simultaneously diagonalizable, so we could take $Q_m = Q$ independent of $m$ and write
	\begin{align*}
		A_n \cdots A_1
		= (QD_nQ^{-1}) \cdots (QD_1Q^{-1})
		= Q D_n \cdots D_1 Q^{-1}.
	\end{align*}
	Products of diagonal matrices are easily understood, so this would allow us to get a sharp estimate for $A_n \cdots A_1$. Unfortunately, $A_1,\dots,A_n$ do not commute. Remarkably, we can still approximate
	\begin{align*}
		A_n \cdots A_1 \approx Q_n D_n \cdots D_1 Q_1^{-1},
	\end{align*}
	and this is good enough. Conditions under which this approximation holds are given in Lemma~\ref{lem:matrix_product}, which is the key technical tool in Subsection~\ref{subsec:p_n,i,j_asymptotics}. The resulting lower bound for $|p_n|$ is Proposition~\ref{prop:p_n_lower_bd_by_Ds}.
	Both Lemma~\ref{lem:matrix_product} and Proposition~\ref{prop:p_n_lower_bd_by_Ds} are steps toward the main result of Subsection~\ref{subsec:p_n,i,j_asymptotics}, which is Proposition~\ref{prop:p_n,i,j_bds}. It is this last proposition which shows that the main term in \eqref{eqn:outline:correction_terms} dominates the correction terms in the relevant regimes.
	
	This concludes our discussion of Section~\ref{sec:bulk_tail_general}, and thus finishes our outline of Theorem~\ref{thm:eq_Gelfand_duality}.
	
	\section{Weak converse theorem $\implies$ Converse theorem} \label{sec:embedding_homog}
	
	In Sections~\ref{sec:reduce_to_L^infty}--\ref{sec:bulk_tail_general}, we will prove Theorem~\ref{thm:weak_structure_thm}, the weak converse theorem for multiplicative representations. In this section, we prove Theorem~\ref{thm:eq_Gelfand_duality} given Theorem~\ref{thm:weak_structure_thm}. As discussed in Subsection~\ref{subsec:outline:roadmap}, the main result needed to deduce Theorem~\ref{thm:eq_Gelfand_duality} from Theorem~\ref{thm:weak_structure_thm} is Theorem~\ref{thm:embedding_G/Gamma}. Theorem~\ref{thm:embedding_G/Gamma} is in turn an immediate consequence of Theorem~\ref{thm:embedding_homog} and Proposition~\ref{prop:discrete_stabilizers}.
	We prove Theorem~\ref{thm:embedding_homog} in Subsections~\ref{subsec:embedding_homog:reduction} and \ref{subsec:embedding_homog:proof}, and Proposition~\ref{prop:discrete_stabilizers} in Subsection~\ref{subsec:embedding_homog:discrete}. Finally, in Subsection~\ref{subsec:embedding_homog:finishing}, we finish the reduction of Theorem~\ref{thm:eq_Gelfand_duality} to Theorem~\ref{thm:weak_structure_thm}.
	
	\subsection{Reduction to a topological criterion for homogeneity} \label{subsec:embedding_homog:reduction}
	
	In the contrapositive, the following proposition gives a criterion for a $G$-space $X$ to be compact and homogeneous.
	
	\begin{prop}[Topological criterion for homogeneity] \label{prop:top_characterization}
		Let $G$ be a connected Lie group and $X$ a connected Hausdorff $G$-space. 
		Suppose $X$ is not of the form $H \backslash G$ for any closed cocompact subgroup $H \subseteq G$. Then there is a nonempty open set $U \subseteq G$, and there are infinitely many nonempty open sets $V_i \subseteq X$, such that the product sets $UV_i$ are disjoint.
	\end{prop}
	
	We prove this proposition in the next subsection. Here, we show how it implies
	Theorem~\ref{thm:embedding_homog}, which is restated below for convenience.
	
	\begin{thm*}[Restatement of Theorem~\ref{thm:embedding_homog}]
		Let $G$ be a connected Lie group, and let $X$ be a connected Hausdorff $G$-space equipped with a $G$-invariant probability measure $\mu$ of full support. Assume $(X,\mu)$ has the qualitative Sobolev embedding property. Then $X \simeq H \backslash G$ as $G$-spaces for some closed cocompact subgroup $H \subseteq G$.
	\end{thm*}
	
	\begin{proof}[Proof of Theorem~\ref{thm:embedding_homog} assuming Proposition~\ref{prop:top_characterization}]
		Suppose for a contradiction that $X$ is not of the form $H \backslash G$ for any closed cocompact subgroup $H \subseteq G$. Then let $U,V_i$ be as in Proposition~\ref{prop:top_characterization}. By shrinking $U$, take $U$ to be precompact.	
		Let $p,q,\varphi$ be as in the definition of the qualitative Sobolev embedding property (Definition~\ref{defn:qualitative_Sob_emb}). Translate $\varphi$ so that it is nonzero at the identity. Then $\varphi \ast \1_U$ (the convolution of two functions on $G$) is positive on $\overline{U}$, so by compactness, it is bounded below on $U$ by some $\delta > 0$. It is also trivially nonnegative everywhere, so we have $\varphi \ast \1_U \geq \delta \1_U$.
		
		The convolution $\1_U \ast \1_{V_i}$ is strictly positive in $UV_i$ and vanishes outside $UV_i$.
		The $UV_i$ have positive measure because $\mu$ has full support.
		At the same time, the $UV_i$ are disjoint subsets of a probability space, so their measures go to zero. Because of this, there exist coefficients $c_i \geq 0$ such that $f := \sum c_i \1_U \ast \1_{V_i}$ is in $L^p(X,\mu)$ but not in $L^q(X,\mu)$. By the qualitative Sobolev embedding property, we have $\varphi \ast f \in L^q(X,\mu)$. But
		\begin{align*}
			\varphi \ast f
			= \sum_i c_i \varphi \ast (\1_U \ast \1_{V_i})
			= \sum_i c_i (\varphi \ast \1_U) \ast \1_{V_i}
			\geq \delta \sum_i c_i\1_U \ast \1_{V_i}
			= \delta f
		\end{align*}
		pointwise, and the nonnegative function $f$ is not in $L^q(X,\mu)$ by construction. This is a contradiction.
	\end{proof}
	
	\subsection{Proof of the topological criterion for homogeneity} \label{subsec:embedding_homog:proof}
	
	Let the notation be as in Proposition~\ref{prop:top_characterization}.
	The proofs in this subsection use nets instead of sequences because we are assuming only that $X$ is Hausdorff (and connected). In practice, however, $X$ will be metrizable, so little would be lost by replacing the word ``net" with ``sequence" everywhere. Assuming metrizability would not shorten the proofs.
	
	For each $x \in X$, the stabilizer $\Stab_x$ is a closed subgroup of $G$, hence a Lie subgroup. 
	
	\begin{prop} \label{prop:dim_stab_usc}
		The function $x \mapsto \dim\Stab_x$ from $X$ to $\Z_{\geq 0}$ is upper semicontinuous.
	\end{prop}
	
	\begin{proof}
		Let $x_n \to x$ be a convergent net in $X$. We must show that
		\begin{align} \label{eqn:dim_stab_jumps_up}
			\limsup_n \dim\Stab_{x_n}
			\leq \dim\Stab_x.
		\end{align}
		Denote the $\limsup$ on the left by $d$. After passing to a subnet, we may assume $\dim\Stab_{x_n} = d$ for all $n$. Then the Lie algebra $\Lie\Stab_{x_n}$ is a $d$-dimensional subspace of $\g_{\R}$, or equivalently a point in the Grassmannian $\Gr_d(\g_{\R})$. This Grassmannian is compact, so after passing to a further subnet, we may assume $\Lie\Stab_{x_n}$ converges to some $L \in \Gr_d(\g_{\R})$. To prove \eqref{eqn:dim_stab_jumps_up}, it suffices to show that $\exp(L) \subseteq \Stab_x$. So let $X \in L$. Since $\Lie\Stab_{x_n} \to L$, there exist $X_n \in \Lie\Stab_{x_n}$ with $X_n \to X$ in $\g_{\R}$. Then $\exp(X_n) \in \Stab_{x_n}$, and $\exp(X_n) \to \exp(X)$ in $G$. It follows by continuity that $\exp(X) \in \Stab_x$. Thus indeed $\exp(L) \subseteq \Stab_x$, and \eqref{eqn:dim_stab_jumps_up} holds.
	\end{proof}
	
	Let $d$ be the minimal possible dimension of the stabilizer of a point in $X$, and let
	\begin{align*}
		X^{\min}
		= \{x \in X : \dim\Stab_x = d\}.
	\end{align*}
	Then $X^{\min}$ is nonempty by the definition of $d$, and by Proposition~\ref{prop:dim_stab_usc}, $X^{\min}$ is open in $X$.
	
	\begin{prop} \label{prop:x_to_LieStab_x_cts}
		The function $x \mapsto \Lie\Stab_x$ from $X^{\min}$ to $\Gr_d(\g_{\R})$ is continuous.
	\end{prop}
	
	\begin{proof}
		It suffices to show that every convergent net $x_n \to x$ in $X^{\min}$ has a subnet along which $\Lie\Stab_{x_n} \to \Lie\Stab_x$. Since $\Gr_d(\g_{\R})$ is compact, we may assume after passing to a subnet that $\Lie\Stab_{x_n}$ converges to some $L \in \Gr_d(\g_{\R})$. By the same argument as in the proof of Proposition~\ref{prop:dim_stab_usc}, we have $\exp(L) \subseteq \Stab_x$, and so $L \subseteq \Lie\Stab_x$. Since $x \in X^{\min}$, we also have $\dim\Stab_x = d$. Therefore $L$ and $\Lie\Stab_x$ have the same dimension, and hence $L = \Lie\Stab_x$. Thus $\Lie\Stab_{x_n} \to \Lie\Stab_x$ as desired.
	\end{proof}
	
	An example where $X^{\min} \neq X$ and where $x \mapsto \Lie\Stab_x$ is nowhere locally constant is $G = \PSL_2(\mathbf{R})$ acting on the one-point compactification of the upper half plane.
	
	\begin{lem}[Local foliation by local orbits] \label{lem:partial_orbits}
		Fix $x \in X^{\min}$. Let $K \subseteq G$ be a symmetric compact neighborhood of the identity which is sufficiently small depending on $x$. Let $W$ be an open neighborhood of $x$ which is sufficiently small depending on $K$. Define a relation $\sim$ on $W$ by $y \sim z$ when $z = ky$ for some $k \in K$. Then $\sim$ is an equivalence relation, the quotient map $W \to W/\sim$ is open, and the quotient space $W/\sim$ is Hausdorff.
	\end{lem}
	
	Note that the notation here is different from other sections: here $K$ is a symmetric compact neighborhood of the identity in $G$, rather than a maximal compact subgroup of $G$.
	
	In general, the space $X/G$ of $G$-orbits need not be Hausdorff. The point of Lemma~\ref{lem:partial_orbits} is that we can work with the ``local orbits" determined by $K$ instead, and that the ``local orbit space" $W/\sim$ is guaranteed to be Hausdorff.
	
	\begin{proof}[Proof of Lemma~\ref{lem:partial_orbits}]
		For notational convenience, fix a metric on $G$ inducing the topology on $G$. Then choose parameters, including $K$ and $W$, in the following order.
		\begin{enumerate} \itemsep = 0.5em
			\item Let $B$ be a bounded neighborhood of the origin in $\Lie\Stab_x$.
			
			\item Let $K$ as in the statement of Lemma~\ref{lem:partial_orbits} be small enough that $K^2 \cap \Stab_x \subseteq \exp(B)$.
			
			\item Choose $\varepsilon>0$ in terms of $K$, such that whenever $g \in K^2$ and $h \in G$ with $\dist(g,h) < 2\varepsilon$, one has $gh^{-1} \in K$.
			
			\item Let $W$ as in the statement of Lemma~\ref{lem:partial_orbits} be small enough that the following properties (i),(ii) hold. After stating (i),(ii), we will explain why they hold for $W$ sufficiently small.
			\\[-0.5em]
			\begin{enumerate} \itemsep = 0.5em
				\item[(i)] For each $Z \in B$ and $y \in W$, there exists $Y \in \Lie\Stab_y$ with $\dist(\exp(Y),\exp(Z)) < \varepsilon$.
				
				\item[(ii)] Whenever $g \in K^2$ with $gW \cap W$ nonempty, one has $\dist(g,K^2 \cap \Stab_x) < \varepsilon$.
			\end{enumerate}
		\end{enumerate}
		
		Let us first explain why (i) holds for $W$ small.
		From Proposition~\ref{prop:x_to_LieStab_x_cts}, we see that by taking $W$ sufficiently small, $\Lie\Stab_y$ can be made arbitrarily close to $\Lie\Stab_x$ uniformly for $y \in W$. Then since $B$ is bounded, one can find $Y \in \Lie\Stab_y$ arbitrarily close to any $Z \in B$, uniformly in $Z$. In particular, one can arrange $\dist(\exp(Y),\exp(Z)) < \varepsilon$, giving (i).
		
		We now explain why (ii) holds for $W$ small. Suppose for a contradiction that (ii) fails for $W$ arbitrarily small. Then there exists a net $W_n$ of neighborhoods of $x$ eventually contained in any given neighborhood of $x$, together with elements $g_n \in K^2$ for each $n$, such that $g_nW_n \cap W_n \neq \emptyset$ and $\dist(g_n, K^2 \cap \Stab_x) \geq \varepsilon$. Since $K^2$ is compact, we may assume after passing to a subnet that $g_n$ converges to some $g \in K^2$. This $g$ must satisfy $\dist(g,K^2 \cap \Stab_x) \geq \varepsilon$. Since $g \in K^2$, it follows that $g \not\in \Stab_x$. On the other hand, by continuity, the condition $g_nW_n \cap W_n \neq \emptyset$ yields $gx = x$ in the limit. This is a contradiction. Thus (ii) must hold for $W$ sufficiently small.
		
		With $B,K,\varepsilon,W$ as in (1)-(4), we finally verify the claimed properties of $\sim$. The relation $\sim$ is reflexive because $K$ contains the identity, and it is symmetric because $K$ is symmetric. For transitivity, suppose $y,z,w \in W$ satisfy $y \sim z$ and $z \sim w$. Then $w = gy$ for some $g \in K^2$. In particular, $gW \cap W$ is nonempty. Therefore by (ii), we have $\dist(g,K^2 \cap \Stab_x) < \varepsilon$. It follows from (2) that there exists $Z \in B$ with $\dist(g,\exp(Z)) < \varepsilon$. By (i), there exists $Y \in \Lie\Stab_y$ with $\dist(\exp(Y), \exp(Z)) < \varepsilon$. Let $h = \exp(Y) \in \Stab_y$. Then by the triangle inequality,
		\begin{align*}
			\dist(g,h)
			\leq \dist(g,\exp(Z)) + \dist(\exp(Z),h)
			< 2\varepsilon.
		\end{align*}
		Thus by (3), we have $gh^{-1} \in K$. Since $h \in \Stab_y$, we can write $w = gy = gh^{-1}y \in Ky$. Hence $y \sim w$, and $\sim$ is transitive. We conclude that $\sim$ is an equivalence relation.
		
		The quotient map $W \to W/\sim$ is open because the preimage of the image of an open set $V \subseteq W$ is $KV \cap W$, which is open in $W$. It follows from the compactness of $K$ that $\sim$ is closed (i.e.,
		\begin{align*}
			\{(y,z) \in W \times W : y \sim z\}
		\end{align*}
		is a closed subset of $W \times W$). Therefore Lemma~\ref{lem:hsdf_criterion} below tells us that $W/\sim$ is Hausdorff.
	\end{proof}
	
	\begin{lem}[Criterion for Hausdorff quotients] \label{lem:hsdf_criterion}
		Let $Y$ be a topological space, and $\sim$ a closed equivalence relation on $Y$ such that the quotient map $Y \to Y/\sim$ is open. Then $Y/\sim$ is Hausdorff.
	\end{lem}
	
	This is a standard fact in point-set topology.
	Taking $\sim$ to be $=$, it specializes to the more famous fact that $Y$ is Hausdorff if (and only if) the diagonal is closed in $Y \times Y$.
	
	We write the quotient map $Y \to Y/\sim$ as $y \mapsto \overline{y}$.
	
	\begin{proof}[Proof of Lemma~\ref{lem:hsdf_criterion}]
		Denote $S = \{(x,y) : x \sim y\}$. This is closed in $Y \times Y$ by assumption. Let $x,y \in Y$ with $\overline{x} \neq \overline{y}$ in $Y/\sim$. Our goal is to find disjoint neighborhoods of $\overline{x}$ and $\overline{y}$. Since $(x,y)$ lies in the open set $Y \times Y \setminus S$, it must be contained in an open subset of $Y \times Y \setminus S$ of the form $U \times V$ (by the definition of the product topology on $Y \times Y$). The images of $U$ and $V$ in $Y/\sim$ are open neighborhoods of $\overline{x}$ and $\overline{y}$, respectively, because the quotient map is open. These neighborhoods are disjoint because $U \times V$ does not intersect $S$.
	\end{proof}
	
	We now use Lemma~\ref{lem:partial_orbits} to prove Proposition~\ref{prop:top_characterization}.
	
	\begin{proof}[Proof of Proposition~\ref{prop:top_characterization}]
		Choose $x,K,W$ as in Lemma~\ref{lem:partial_orbits}.
		Then there are two cases.
		
		\noindent
		\emph{Case 1.} $\overline{x}$ is a limit point in $W/\sim$.
		
		Let $U \subseteq K$ be a small open neighborhood of the identity and $V \subseteq W$ a small open neighborhood of $x$ such that $UV \subseteq W$. Write $\overline{V}$ for the image of $V$ in $W/\sim$ (so $\overline{V}$ does not mean the closure of $V$). Since $W \to W/\sim$ is open, $\overline{V}$ is open. Using that the Hausdorff space $\overline{V}$ contains a limit point (namely $\overline{x}$), we can find an infinite sequence of disjoint nonempty open sets $\overline{V_i} \subseteq \overline{V}$. Let $V_i$ be the intersection of $V$ with the preimage of $\overline{V_i}$. Then the $V_i$ are open and nonempty in $X$, and the $UV_i$ are disjoint because their projections to $W/\sim$ are the $\overline{V_i}$, which are disjoint.
		
		\noindent
		\emph{Case 2.} $\overline{x}$ is isolated in $W/\sim$.
		
		That $\overline{x}$ is isolated means $Kx \cap W$ is open in $X$. Let $K/\Stab_x$ denote the image of $K$ in $G/\Stab_x$. Then the natural map $K/\Stab_x \to Kx$ is a continuous bijection from a compact space to a Hausdorff space, so it is a homeomorphism. In particular, a neighborhood of the identity in $K/\Stab_x$ maps homeomorphically to $Kx \cap W$, which is open in $X$. This means that the map $G/\Stab_x \to X$ given by acting on $x$ is a homeomorphism from a neighborhood of the identity to a neighborhood of $x$. By translation symmetry, it follows that $Gx$ is open in $X$, and $G/\Stab_x \to Gx$ is a local homeomorphism. Since $G/\Stab_x \to Gx$ is bijective, we conclude that it is in fact a (global) homeomorphism. Note that $G/\Stab_x$ is a smooth manifold, so this shows that $Gx$ is about as nice a topological space as possible --- in particular, $Gx$ is locally compact Hausdorff (LCH).
		
		If $Gx$ is non-compact, then it is easy to find nonempty opens $U \subseteq G$ and $V_i \subseteq Gx$ with the $UV_i$ disjoint --- just take the $V_i$ to accumulate rapidly at infinity in the LCH space $Gx$. Thus we may assume $Gx$ is compact. Then $Gx$ is both open and closed in $X$, so $Gx = X$ because $X$ is connected. Hence $X \simeq G/\Stab_x \simeq \Stab_x \backslash G$, and $\Stab_x$ is cocompact because $X = Gx$ is compact. This contradicts the assumption in the statement of Proposition~\ref{prop:top_characterization} that $X$ is not of the form $H \backslash G$ for any closed cocompact subgroup $H \subseteq G$.
	\end{proof}
	
	Now that Proposition~\ref{prop:top_characterization} is proved, Theorem~\ref{thm:embedding_homog} is fully proved as well.
	
	\subsection{Discreteness of the stabilizer subgroup} \label{subsec:embedding_homog:discrete}
	
	In this subsection, we prove Proposition~\ref{prop:discrete_stabilizers}, restated below.
	The proof given here is an adaptation of the proof of Theorem 4.5.1 in \cite{Morris}.
	
	\begin{prop*}[Restatement of Proposition~\ref{prop:discrete_stabilizers}]
		Let $G$ be a noncompact simple Lie group (e.g., $G = \PSL_2(\R)$). Let $H \subseteq G$ be a closed cocompact subgroup such that $H \backslash G$ admits a $G$-invariant probability measure. Then either $H$ is discrete or $H = G$.
	\end{prop*}
	
	\begin{proof}
		Since $G$ is non-compact and simple, it contains a unipotent element (this is the only place where non-compactness is used). The closure of the subgroup generated by unipotent elements is a nonempty, closed, connected, normal subgroup of $G$, so it is equal to $G$ because $G$ is simple.
		
		Since $G$ is simple, it is unimodular. The condition that $H \backslash G$ admits an invariant probability measure then implies that $H$ is unimodular. This means that $\Ad(H)$ acts trivially on $\wedge^d \mathfrak{h}$, where $d = \dim\mathfrak{h}$. Pick some nonzero element $\omega \in \wedge^d\mathfrak{h}$. We claim that $\Ad(G)$ fixes $\omega$ inside $\wedge^d\mathfrak{g}$. By the previous paragraph, it suffices to check that $\Ad(u)\omega = \omega$ for all unipotents $u \in G$. Because $u$ is unipotent, $n \mapsto \Ad(u^n)\omega$ is a polynomial function from the integers to $\wedge^d\mathfrak{g}$. This function lands in $\Ad(G)\omega$, which is compact because $\omega$ is stabilized by $\Ad(H)$, and $H$ is cocompact. But bounded polynomials are constant, so $\Ad(u^n)\omega$ must be independent of $n$. Equating $n=1$ with $n=0$, we get $\Ad(u)\omega = \omega$, proving the claim.
		
		We now know that $\Ad(G)$ acts trivially on $\wedge^d\mathfrak{h} \subseteq \wedge^d\mathfrak{g}$. Consequently $\Ad(G)$ preserves $\mathfrak{h}$. It follows by simplicity that either $\mathfrak{h} = 0$ or $\mathfrak{h} = \mathfrak{g}$, and hence either $H$ is discrete or $H=G$.
	\end{proof}
	
	Having proved Theorem~\ref{thm:embedding_homog} and Proposition~\ref{prop:discrete_stabilizers}, we immediately obtain Theorem~\ref{thm:embedding_G/Gamma}.
	
	\subsection{Finishing the reduction to the weak converse theorem} \label{subsec:embedding_homog:finishing}
	
	Finally, we prove Theorem~\ref{thm:eq_Gelfand_duality} conditioned on Theorem~\ref{thm:weak_structure_thm}.
	So let $G = \PSL_2(\R)$ and $K = \PSO_2(\R)$ once more.
	Recall the statement of Theorem~\ref{thm:eq_Gelfand_duality}:
	
	\begin{thm*}[Restatement of Theorem~\ref{thm:eq_Gelfand_duality}]
		Let $\H$ be a nontrivial multiplicative representation. Then $\H \simeq L^2(\Gamma \backslash G)$, as multiplicative representations, for some cocompact lattice $\Gamma$ in $G$.
	\end{thm*}
	
	\begin{proof}[Proof of Theorem~\ref{thm:eq_Gelfand_duality} assuming Theorem~\ref{thm:weak_structure_thm}]
		Let $X$ and $\mu$ be as in Theorem~\ref{thm:weak_structure_thm}, so in particular the multiplicative representation $\H$ identifies with $L^2(X,\mu)$ as unitary representations of $G$. Then by Theorem~\ref{thm:embedding_G/Gamma}, either $X$ is a point, or $X \simeq \Gamma \backslash G$ as $G$-spaces for some cocompact lattice $\Gamma$. If $X$ is a point, then $\H \simeq L^2(X,\mu)$ is one-dimensional, contradicting that $\H$ is nontrivial. Therefore $X \simeq \Gamma \backslash G$. The only $G$-invariant probability measure on $\Gamma \backslash G$ is Haar, so $\mu$ is the Haar probability measure. Thus $\H \simeq L^2(\Gamma \backslash G)$ as unitary representations. By Theorem~\ref{thm:weak_structure_thm}, this isomorphism can be chosen so that multiplication on $\H^{\fin}$ is given by pointwise multiplication on $\Gamma \backslash G$. We conclude that $\H \simeq L^2(\Gamma \backslash G)$ as multiplicative representations.
	\end{proof}
	
	\section{$L^{\infty}$ bounds $\implies$ Weak converse theorem} \label{sec:reduce_to_L^infty}
	
	Returning to our usual conventions, let $G = \PSL_2(\R)$, let $K = \PSO_2(\R)$, and let $\H$ be a multiplicative representation.
	
	Section~\ref{sec:embedding_homog} reduced Theorem~\ref{thm:eq_Gelfand_duality} to Theorem~\ref{thm:weak_structure_thm}.
	In this section we reduce Theorem~\ref{thm:weak_structure_thm} to Theorem~\ref{thm:high_deriv_L^infty_bd} and Corollary~\ref{cor:qual_Sob_emb} (a consequence of Theorem~\ref{thm:L^infty_quasi-Sobolev}).
	Theorems~\ref{thm:high_deriv_L^infty_bd} and \ref{thm:L^infty_quasi-Sobolev} will be proved in Sections~\ref{sec:reduce_to_bulk_tail}, \ref{sec:bulk_tail_poly}, and \ref{sec:bulk_tail_general}. For now, we take Theorem~\ref{thm:high_deriv_L^infty_bd} and Corollary~\ref{cor:qual_Sob_emb} for granted.
	
	\subsection{The commutative C*-algebra generated by $\H^{\fin}$} \label{subsec:reduce_to_L^infty:C*-alg}
	
	For all $\alpha \in \H^{\fin}$, by the $n=0$ case of Theorem~\ref{thm:high_deriv_L^infty_bd}, $M_{\alpha}$ extends from $\H^{\fin}$ to a bounded linear operator on $\H$. Thereby view $M_{\alpha} \in \B(\H)$.
	
	
	\begin{prop} \label{prop:self-adjoint}
		Let $\alpha \in \H^{\fin}$. Then $M_{\alpha}^{\ast} = M_{\overline{\alpha}}$. In particular, if $\alpha \in \H_{\R}^{\fin}$, then $M_{\alpha}$ is self-adjoint.
	\end{prop}
	
	\begin{proof}
		Since $\H^{\fin}$ is dense in $\H$, it suffices to check that
		\begin{align*}
			\langle M_{\alpha}\beta, \gamma \rangle_{\H}
			= \langle \beta, M_{\overline{\alpha}}\gamma \rangle_{\H}
		\end{align*}
		for all $\beta,\gamma \in \H^{\fin}$. By definition, $M_{\alpha}\beta = \alpha\beta$ and $M_{\overline{\alpha}}\gamma = \overline{\alpha}\gamma$, so this is immediate from the definition of $\overline{\alpha}$.
	\end{proof}
	
	\begin{prop} \label{prop:commute}
		Let $\alpha,\beta \in \H^{\fin}$. Then $M_{\alpha}, M_{\beta}$ commute.
	\end{prop}
	
	\begin{proof}
		Since $\H^{\fin}$ is dense in $\H$, it suffices to check that
		\begin{align*}
			\langle M_{\alpha} M_{\beta} \gamma, \delta \rangle_{\H}
			= \langle M_{\beta} M_{\alpha} \gamma, \delta \rangle_{\H}
		\end{align*}
		for all $\gamma,\delta \in \H^{\fin}$. Indeed, by Proposition~\ref{prop:self-adjoint} together with crossing symmetry,
		\begin{align*}
			&\langle M_{\alpha} M_{\beta} \gamma, \delta \rangle_{\H}
			= \langle M_{\beta} \gamma, M_{\overline{\alpha}} \delta \rangle_{\H}
			= \langle \beta\gamma, \overline{\alpha}\delta \rangle_{\H}
			= \langle \alpha\gamma, \overline{\beta}\delta \rangle_{\H}
			= \langle M_{\alpha} \gamma, M_{\overline{\beta}} \delta \rangle_{\H}
			= \langle M_{\beta} M_{\alpha} \gamma, \delta \rangle_{\H}.
			\qedhere
		\end{align*}
	\end{proof}
	
	Let $\A \subseteq \B(\H)$ be the closure (with respect to the operator norm topology) of the subalgebra generated by $\{M_{\alpha} : \alpha \in \H^{\fin}\}$. Propositions~\ref{prop:self-adjoint} and \ref{prop:commute} imply that $\A$ is a commutative C*-subalgebra of $\B(\H)$. Since $M_{\mathbf{1}}$ is the identity operator, $\A$ is unital, and since $\H^{\fin}$ has countable dimension (Proposition~\ref{prop:Maass_span_H^fin}), $\A$ is separable.
	
	The group $U(\H)$ of unitary operators on $\H$ acts by conjugation on $\B(\H)$, and for each element of $U(\H)$, its conjugation action is a C*-automorphism of $\B(\H)$. The representation of $G$ on $\H$ is a map $G \to U(\H)$. Pulling the $U(\H)$-action back to $G$, we obtain an action of $G$ on $\B(\H)$ by C*-automorphisms. Our next goal is to show that this $G$-action preserves $\A$ (Theorem~\ref{thm:A_G-inv}), and that the action map is continuous (Theorem~\ref{thm:A_G-action_cts}).
	
	\begin{thm} \label{thm:A_G-inv}
		The C*-subalgebra $\A \subseteq \B(\H)$ is $G$-invariant.
	\end{thm}
	
	\begin{thm} \label{thm:A_G-action_cts}
		The action map $G \times \A \to \A$, given by $(g,A) \mapsto gAg^{-1}$, is continuous (where $G$ has the Euclidean topology, $\A$ has the operator norm topology, and $G \times \A$ has the product topology).
	\end{thm}
	
	To prove these theorems, we will need a few preliminary results.
	
	\begin{lem} \label{lem:weak_prod_rule}
		Let $X \in \g_{\R}$, let $\alpha \in \H^{\fin}$, and let $v,w \in \H^{\infty}$. Then
		\begin{align*}
			\langle M_{X\alpha} v, w \rangle_{\H}
			= -\langle M_{\alpha} Xv, w \rangle_{\H} - \langle M_{\alpha} v, Xw \rangle_{\H}.
		\end{align*}
	\end{lem}
	
	\begin{proof}
		Since $\H^{\fin}$ is dense in $\H^{\infty}$, it suffices to check that
		\begin{align*}
			\langle M_{X\alpha} \beta, \gamma \rangle_{\H}
			= -\langle M_{\alpha} X\beta, \gamma \rangle_{\H} - \langle M_{\alpha} \beta, X\gamma \rangle_{\H}
		\end{align*}
		for all $\beta,\gamma \in \H^{\fin}$. Rearranging, we want to show that
		\begin{align*}
			-\langle \alpha\beta, X\gamma \rangle_{\H}
			= \langle (X\alpha)\beta, \gamma \rangle_{\H} + \langle \alpha X\beta, \gamma \rangle_{\H}.
		\end{align*}
		Using \eqref{eqn:Lie_alg_unitarity} to move the $X$ on the left hand side to the first argument in the inner product, we see that this is equivalent to
		\begin{align*}
			\langle \overline{X}(\alpha\beta), \gamma \rangle_{\H}
			= \langle (X\alpha)\beta, \gamma \rangle_{\H} + \langle \alpha X\beta, \gamma \rangle_{\H}.
		\end{align*}
		Now $X \in \g_{\R}$ by assumption, so $\overline{X} = X$, and the above is immediate from the product rule.
	\end{proof}
	
	\begin{prop} \label{prop:conj_M_alpha}
		Let $X \in \g_{\R}$ with $\|X\|_{\g} \ll 1$, and set $g = \exp X$. Let $\alpha \in \H^{\fin}$. Then
		\begin{align*}
			gM_{\alpha}g^{-1}
			= \sum_{n=0}^{\infty} \frac{1}{n!} M_{X^n\alpha}
		\end{align*}
		(because $\|X\|_{\g} \ll 1$, Theorem~\ref{thm:high_deriv_L^infty_bd} implies that the sum converges absolutely in operator norm). Hence $gM_{\alpha}g^{-1} \in \A$.
	\end{prop}
	
	\begin{proof}
		Since $\H^{\infty}$ is dense in $\H$, it suffices to check that
		\begin{align} \label{eqn:conj_M_alpha_test_v,w}
			\langle gM_{\alpha}g^{-1} v, w \rangle_{\H}
			= \sum_{n=0}^{\infty} \frac{1}{n!} \langle M_{X^n\alpha} v, w \rangle_{\H}
		\end{align}
		for all $v,w \in \H^{\infty}$. For $t \in [0,1]$, denote $g_t = \exp(tX)$. For fixed $v,w \in \H^{\infty}$ and for each $\beta \in \H^{\fin}$, let $F_{\beta} \colon [0,1] \to \C$ be the function
		\begin{align*}
			F_{\beta}(t)
			= \langle g_t M_{\beta} g_t^{-1} v, w \rangle_{\H}.
		\end{align*}
		With this notation, \eqref{eqn:conj_M_alpha_test_v,w} can be rewritten as
		\begin{align} \label{eqn:F_alpha_Taylor_goal}
			F_{\alpha}(1)
			= \sum_{n=0}^{\infty} \frac{1}{n!} F_{X^n\alpha}(0).
		\end{align}
		In general, by unitarity,
		\begin{align*}
			F_{\beta}(t)
			= \langle M_{\beta} g_t^{-1} v, g_t^{-1} w \rangle_{\H}.
		\end{align*}
		Since $v,w \in \H^{\infty}$, it follows from this formula that $F_{\beta} \in C^{\infty}([0,1])$, and furthermore its derivatives can be computed by the product rule. The first derivative is
		\begin{align*}
			F_{\beta}'(t)
			= -\langle M_{\beta} X g_t^{-1} v, g_t^{-1} w \rangle_{\H} - \langle M_{\beta} g_t^{-1} v, X g_t^{-1} w \rangle_{\H}.
		\end{align*}
		Since $g_t^{-1}v, g_t^{-1}w \in \H^{\infty}$, Lemma~\ref{lem:weak_prod_rule} gives
		\begin{align*}
			F_{\beta}'(t) = \langle M_{X\beta} g_t^{-1} v, g_t^{-1} w \rangle_{\H}
			= F_{X\beta}(t).
		\end{align*}
		This is valid for all $\beta \in \H^{\fin}$, so we can iterate it to see that the $n$th derivative of $F_{\alpha}$ is $F_{\alpha}^{(n)} = F_{X^n\alpha}$. By the trivial bound
		\begin{align*}
			|F_{\beta}(t)|
			\leq \|\beta\|_{L^{\infty}} \|v\|_{\H} \|w\|_{\H}
		\end{align*}
		and by Theorem~\ref{thm:high_deriv_L^infty_bd},
		\begin{align*}
			|F_{\alpha}^{(n)}(t)|
			\lesssim_{\alpha} O(n\|X\|_{\g})^n \|v\|_{\H} \|w\|_{\H}
		\end{align*}
		for all $n \in \Z_{\geq 0}$ and $t \in [0,1]$. Since $\|X\|_{\g} \ll 1$, it follows that the Taylor series for $F_{\alpha}$ at $t=0$ converges absolutely to $F_{\alpha}$ on all of $[0,1]$. Therefore
		\begin{align*}
			F_{\alpha}(1)
			= \sum_{n=0}^{\infty} \frac{1}{n!} F_{\alpha}^{(n)}(0)
			= \sum_{n=0}^{\infty} \frac{1}{n!} F_{X^n\alpha}(0).
		\end{align*}
		This is the desired equation \eqref{eqn:F_alpha_Taylor_goal}.
	\end{proof}
	
	From Proposition~\ref{prop:conj_M_alpha}, we deduce Theorem~\ref{thm:A_G-inv} with little additional effort. Theorem~\ref{thm:A_G-action_cts} will require one more intermediate result.
	
	\begin{proof}[Proof of Theorem~\ref{thm:A_G-inv}]
		By Proposition~\ref{prop:conj_M_alpha}, there is a neighborhood $U$ of the identity in $G$ such that for all $g \in U$ and $\alpha \in \H^{\fin}$, one has $gM_{\alpha}g^{-1} \in \A$. Since conjugation by $g$ is a C*-automorphism, the closed subalgebra $\A^g$ of $\B(\H)$ generated by $\{gM_{\alpha}g^{-1} : \alpha \in \H^{\fin}\}$ is the $g$-conjugate of the closed subalgebra generated by $\{M_{\alpha} : \alpha \in \H^{\fin}\}$, i.e., $\A^g = g\A g^{-1}$. From the definition of $U$, we know that $\A^g \subseteq \A$ for $g \in U$, and hence $g\A g^{-1} \subseteq \A$ for $g \in U$. Since $G$ is connected, it is generated by any neighborhood of the identity. Thus the $g$-invariance of $\A$ for $g \in U$ implies that $\A$ is $g$-invariant for all $g \in G$.
	\end{proof}
	
	\begin{prop} \label{prop:A_G-action_cts}
		For each $A \in \A$, the map $G \to \A$ given by $g \mapsto gAg^{-1}$ is continuous.
	\end{prop}
	
	\begin{proof}
		By translation symmetry, it suffices to show that there is a neighborhood $U$ of the identity in $G$ such that for each $A \in \A$, the map $\Phi_A \colon U \to \A$ given by $\Phi_A(g) = gAg^{-1}$ for $g \in U$ is continuous. Let
		\begin{align*}
			U = \{\exp X : X \in \g_{\R} \text{ and } \|X\|_{\g} \ll 1\},
		\end{align*}
		where the condition $\|X\|_{\g} \ll 1$ ensures that $X$ satisfies the hypothesis of Proposition~\ref{prop:conj_M_alpha}. Then the Taylor expansion in Proposition~\ref{prop:conj_M_alpha} implies that $\Phi_A$ is continuous when $A = M_{\alpha}$ for some $\alpha \in \H^{\fin}$. By the definition of $\A$, a general $A \in \A$ can be approximated by sums of products of $M_{\alpha}$'s. Since conjugation by $g$ is a C*-automorphism, such an approximation translates to an approximation of $\Phi_A$ by sums of products of $\Phi_{M_{\alpha}}$'s, where the latter approximation takes place in the Banach space of functions $\Psi \colon U \to \A$ with the uniform norm $\|\Psi\| = \sup_{g \in U} \|\Psi(g)\|_{\op}$. Therefore $\Phi_A$ is a uniform limit of continuous functions, and hence is itself continuous.
	\end{proof}
	
	We will only ever need Proposition~\ref{prop:A_G-action_cts} as opposed to the full strength of Theorem~\ref{thm:A_G-action_cts}, but Theorem~\ref{thm:A_G-action_cts} quickly follows from Proposition~\ref{prop:A_G-action_cts}, so we give the proof just to make clear that the $G$-action on $\A$ is well-behaved.
	
	\begin{proof}[Proof of Theorem~\ref{thm:A_G-action_cts}]
		Let $g \in G$ and $A \in \A$. Given sequences $g_n \to g$ and $A_n \to A$, we must show that $g_n A_n g_n^{-1} \to gAg^{-1}$. By the triangle inequality,
		\begin{align*}
			\|g_nA_ng_n^{-1} - gAg^{-1}\|_{\op}
			\leq \|g_nA_ng_n^{-1} - g_nAg_n^{-1}\|_{\op} + \|g_nAg_n^{-1} - gAg^{-1}\|_{\op}.
		\end{align*}
		The first term on the right hand side is equal to $\|A_n - A\|_{\op}$, so it goes to zero. The second term goes to zero by Proposition~\ref{prop:A_G-action_cts}.
	\end{proof}
	
	
	Now, let $\ev_{\mathbf{1}} \colon \A \to \H$ be given by evaluation at $\mathbf{1}$. This is a bounded linear map, and it is $G$-equivariant because
	\begin{align*}
		\ev_{\mathbf{1}}(gAg^{-1})
		= gAg^{-1}\mathbf{1}
		= gA\mathbf{1}
		= g\ev_{\mathbf{1}}(A)
	\end{align*}
	for all $g \in G$ and $A \in \A$.
	
	\begin{prop} \label{prop:ev_1_injective}
		The map $\ev_{\mathbf{1}} \colon \A \to \H$ is injective.
	\end{prop}
	
	\begin{proof}
		Let $A \in \A$ be nonzero. We wish to show that $\ev_{\mathbf{1}}(A)$ is nonzero. Since $\A \subseteq \B(\H)$, nonzeroness of $A$ means that there exists $\alpha \in \H$ with $A\alpha \neq 0$. By density, $\alpha$ may be taken to be in $\H^{\fin}$. Then using that $\A$ is commutative, we can write
		\begin{align*}
			0 \neq A\alpha
			= AM_{\alpha} \mathbf{1}
			= M_{\alpha} A\mathbf{1}
			= M_{\alpha}\ev_{\mathbf{1}}(A).
		\end{align*}
		This forces $\ev_{\mathbf{1}}(A) \neq 0$, as desired.
	\end{proof}
	
	Finally, we note that Corollary~\ref{cor:qual_Sob_emb} has the following easy consequence.
	
	\begin{cor} \label{cor:qual_Sob_emb_into_A}
		There exists a nonnegative continuous function $\varphi \in L^1(G)$, not identically zero, such that for all $v \in \H$, the convolution $\varphi \ast v$ lies in the image of $\ev_{\mathbf{1}}$.
	\end{cor}
	
	\begin{proof}
		Let $\varphi$ be as in Corollary~\ref{cor:qual_Sob_emb}, and let $v \in \H$. Then $\varphi \ast v \in \mathcal{C}$, which means that there is an $L^{\infty}$-Cauchy sequence $\alpha_n \in \H^{\fin}$ with $\alpha_n \to \varphi \ast v$ in $\H$. The fact that $\alpha_n$ is a Cauchy sequence in $L^{\infty}$ is equivalent to $M_{\alpha_n}$ being a Cauchy sequence in $\A$. Thus the $M_{\alpha_n}$ converge to a limit $A \in \A$. Since $\ev_{\mathbf{1}}$ is bounded,
		\begin{align*}
			\ev_{\mathbf{1}}(A)
			= \lim_{n \to \infty} \ev_{\mathbf{1}}(M_{\alpha_n})
			= \lim_{n \to \infty} \alpha_n
			= \varphi \ast v.
		\end{align*}
		Thus $\varphi \ast v$ is indeed in the image of $\ev_{\mathbf{1}}$.
	\end{proof}
	
	\subsection{The spectrum of the C*-algebra} \label{subsec:reduce_to_L^infty:spectrum}
	
	By Gelfand duality (Theorem~\ref{thm:Gelfand_duality}), $\A$ is isomorphic as a C*-algebra to $C(X)$ for a compact Hausdorff space $X$ unique up to homeomorphism. Fix an isomorphism $\Phi \colon C(X) \to \A$. Since $\A$ is separable, $X$ is metrizable. The linear functional
	\begin{align*}
		f \mapsto \langle \Phi(f)\mathbf{1},\mathbf{1} \rangle_{\H}
	\end{align*}
	on $C(X)$ is positive because for $f \in C(X)$,
	\begin{align} \label{eqn:fnl_positive}
		|f|^2 \mapsto
		\langle \Phi(|f|^2) \mathbf{1}, \mathbf{1} \rangle_{\H}
		= \langle \Phi(f)^{\ast} \Phi(f) \mathbf{1}, \mathbf{1} \rangle_{\H}
		= \|\Phi(f)\mathbf{1}\|_{\H}^2
		\geq 0
	\end{align}
	(indeed, this functional is the pullback to $C(X)$ of the GNS state on $\A$ associated to $\mathbf{1} \in \H$).
	Thus by the Riesz representation theorem, this functional is given by integration against a unique positive measure $\mu$ on $X$. The measure $\mu$ is a probability measure because
	\begin{align*}
		\mu(X)
		= \int_X 1 \, d\mu
		= \langle \Phi(1)\mathbf{1}, \mathbf{1} \rangle_{\H}
		= \langle \mathbf{1}, \mathbf{1} \rangle_{\H}
		= 1,
	\end{align*}
	where the last equality is the normalization axiom in the definition of a multiplicative representation.
	By \eqref{eqn:fnl_positive},
	\begin{align*}
		\|f\|_{L^2(X,\mu)}
		= \|\Phi(f)\mathbf{1}\|_{\H}
		= \|\ev_{\mathbf{1}} \circ \Phi(f)\|_{\H}
	\end{align*}
	for all $f \in C(X)$, so $\ev_{\mathbf{1}} \circ \Phi$ extends to $L^2(X,\mu)$ and gives an isometry $\Psi \colon L^2(X,\mu) \to \H$.
	Since $\Phi$ is an isomorphism onto $\A$, the image of $\ev_{\mathbf{1}} \circ \Phi$ is the same as the image of $\ev_{\mathbf{1}}$, which contains $\H^{\fin}$ and hence is dense in $\H$. Thus the image of $\Psi$ is dense in $\H$, and since $\Psi$ is an isometry, it follows that $\Psi$ is an isomorphism. In summary, we have a diagram
	$$
	\begin{tikzcd}
		C(X) \arrow[hookrightarrow]{r} \arrow{d}{\Phi} & L^2(X,\mu) \arrow{d}{\Psi} \\
		\A \arrow[hookrightarrow]{r}{\ev_{\mathbf{1}}} & \H
	\end{tikzcd}
	$$
	where the vertical maps are isomorphisms. The top arrow is the natural map taking a continuous function in $C(X)$ to its $\mu$-a.e. equivalence class in $L^2(X,\mu)$. It is not obvious \textit{a priori} that this top arrow is injective, but in fact it is because the bottom arrow is injective by Proposition~\ref{prop:ev_1_injective}.
	
	\begin{prop} \label{prop:mu_full_support}
		The measure $\mu$ has full support on $X$.
	\end{prop}
	
	\begin{proof}
		Suppose not. Then there is a nonempty open subset $U \subseteq X$ with $\mu(U) = 0$. By Urysohn's lemma, there exists a nonzero function $f \in C(X)$ which vanishes outside $U$. Then $f$ vanishes $\mu$-a.e., and hence vanishes as an element of $L^2(X,\mu)$. This contradicts that $C(X) \to L^2(X,\mu)$ is injective.
	\end{proof}
	
	Because $C(X) \to L^2(X,\mu)$ is injective, we view $C(X)$ as a subspace of $L^2(X,\mu)$ from now on.
	
	\begin{lem} \label{lem:Im_ev_1_in_C(X)}
		Let $f \in L^2(X,\mu)$ with $\Psi(f)$ in the image of $\ev_{\mathbf{1}}$. Then $f \in C(X)$.
	\end{lem}
	
	\begin{proof}
		This is because the vertical arrows in the diagram above are isomorphisms.
	\end{proof}

	By
	Gelfand duality, the $G$-action on $\A$ by conjugation induces a $G$-action on $X$ by homeomorphisms. \textit{A priori}, the action map $G \times X \to X$ is only continuous when $G$ is given the discrete topology, but Proposition~\ref{prop:joint_cty} below says that it is actually continuous when $G$ has the usual topology.
	The $G$-action on $X$ induces a $G$-action on $C(X)$ by $gf(x) = f(g^{-1}x)$. The isomorphism $\Phi \colon C(X) \to \A$ is tautologically $G$-equivariant --- the $G$-action on $X$ is defined to make $\Phi$ equivariant.
	
	\begin{prop} \label{prop:joint_cty}
		The action map $G \times X \to X$ is continuous.
	\end{prop}
	
	\begin{proof}
		Suppose for a contradiction that $G \times X \to X$ is not continuous. Then there are convergent sequences $g_n \to g$ and $x_n \to x$, and a neighborhood $V$ of $gx$, such that $g_nx_n \not\in V$ for all $n$. By Urysohn's lemma, there exists $f \in C(X)$ with $f(gx) = 1$ and $f|_{X \setminus V} = 0$. Then
		\begin{align*}
			|f(g_nx_n) - f(gx)| = 1
		\end{align*}
		for all $n$, so
		\begin{align*}
			1 = \lim_{n \to \infty} |f(g_nx_n) - f(gx)|
			&\leq \limsup_{n \to \infty} |f(g_nx_n) - f(gx_n)| + \limsup_{n \to \infty} |f(gx_n) - f(gx)|
			\\&\leq \limsup_{n \to \infty} \|g_n^{-1}f - g^{-1}f\|_{C(X)} + \limsup_{n \to \infty} |g^{-1}f(x_n) - g^{-1}f(x)|
			\\&= \limsup_{n \to \infty} \|g_n^{-1} \Phi(f) g_n - g^{-1} \Phi(f) g\|_{\op} + \limsup_{n \to \infty} |g^{-1}f(x_n) - g^{-1}f(x)|.
		\end{align*}
		The first $\limsup$ vanishes by Theorem~\ref{thm:A_G-action_cts} (in fact Proposition~\ref{prop:A_G-action_cts} suffices), and the second $\limsup$ vanishes because $g^{-1}f$ is a continuous function on $X$. This is a contradiction.
	\end{proof}
	
	\begin{prop} \label{prop:mu_G-invt}
		The measure $\mu$ is preserved by the $G$-action on $X$.
	\end{prop}
	
	\begin{proof}
		We must show that for all $f \in C(X)$ and $g \in G$, both $f$ and $gf$ have the same integral with respect to $\mu$. By the definition of $\mu$, equivariance of $\Phi$, and unitarity,
		\begin{align*}
			\int_X gf \, d\mu
			= \langle \Phi(gf) \mathbf{1}, \mathbf{1} \rangle_{\H}
			= \langle g\Phi(f)g^{-1} \mathbf{1}, \mathbf{1} \rangle_{\H}
			= \langle \Phi(f) g^{-1} \mathbf{1}, g^{-1} \mathbf{1} \rangle_{\H}
			= \langle \Phi(f) \mathbf{1}, \mathbf{1} \rangle_{\H}
			= \int_X f \, d\mu,
		\end{align*}
		as desired.
	\end{proof}
	
	By Proposition~\ref{prop:mu_G-invt}, the group $G$ acts unitarily on $L^2(X,\mu)$ by $gf(x) = f(g^{-1}x)$.
	
	\begin{prop} \label{prop:Psi_equivariant}
		The Hilbert space isomorphism $\Psi \colon L^2(X,\mu) \to \H$ is $G$-equivariant.
	\end{prop}
	
	\begin{proof}
		By density, it suffices to show that $\Psi|_{C(X)} = \ev_{\mathbf{1}} \circ \Phi$ is equivariant. This is clear, because both $\ev_{\mathbf{1}}$ and $\Phi$ are equivariant.
	\end{proof}
	
	\begin{prop} \label{prop:X_connected}
		The space $X$ is connected.
	\end{prop}
	
	\begin{proof}
		Since $G$ is connected and $G \times X \to X$ is continuous (Proposition~\ref{prop:joint_cty}), $G$ preserves each connected component of $X$. The indicator functions of the components are therefore $G$-invariant elements of $L^2(X,\mu)$, and they are linearly independent (in particular nonzero) in $L^2(X,\mu)$ because $\mu$ has full support (Proposition~\ref{prop:mu_full_support}). Thus the number of components is at most $\dim L^2(X,\mu)^G = \dim \H^G$. By the ergodicity assumption in the definition of a multiplicative representation, $\H^G = \C\mathbf{1}$ is one-dimensional, so $X$ must be connected.
	\end{proof}
	
	Combining Corollary~\ref{cor:qual_Sob_emb_into_A} with what we now know about $X,\mu$ gives
	
	\begin{cor} \label{cor:qual_Sob_emb_into_C(X)}
		The $G$-space $X$ with invariant measure $\mu$ has the qualitative Sobolev embedding property (with $p=2$ and $q=\infty$ in Definition~\ref{defn:qualitative_Sob_emb}).
	\end{cor}
	
	\begin{proof}
		Let $\varphi$ be as in Corollary~\ref{cor:qual_Sob_emb_into_A}, and let $f \in L^2(X,\mu)$. We wish to show that the convolution $\varphi \ast f$ is in $L^{\infty}(X,\mu)$. Since $X$ is compact, it suffices to show that $\varphi \ast f \in C(X)$.
		Corollary~\ref{cor:qual_Sob_emb_into_A} says that $\varphi \ast \Psi(f)$ is in the image of $\ev_{\mathbf{1}}$. By Proposition~\ref{prop:Psi_equivariant}, $\Psi$ is equivariant, so $\varphi \ast \Psi(f) = \Psi(\varphi \ast f)$. Thus by Lemma~\ref{lem:Im_ev_1_in_C(X)}, we indeed have $\varphi \ast f \in C(X)$.
	\end{proof}
	
	Proposition~\ref{prop:Psi_equivariant} says that $\Psi \colon L^2(X,\mu) \to \H$ is an isomorphism of unitary representations of $G$. Therefore $L^2(X,\mu)$ has discrete spectrum, so it makes sense to consider $L^2(X,\mu)^{\fin}$, and $\Psi$ restricts to an isomorphism $L^2(X,\mu)^{\fin} \to \H^{\fin}$. We can now extend the diagram above to
	$$
	\begin{tikzcd}
		C(X) \arrow[hookrightarrow]{r} \arrow{d}{\Phi} & L^2(X,\mu) \arrow{d}{\Psi} \arrow[hookleftarrow]{r} & L^2(X,\mu)^{\fin} \arrow{d}{\Psi} \\
		\A \arrow[hookrightarrow]{r}{\ev_{\mathbf{1}}} & \H \arrow[hookleftarrow]{r} & \H^{\fin}
	\end{tikzcd}
	$$
	again with all vertical arrows isomorphisms. The inclusion $\H^{\fin} \hookrightarrow \H$ in the bottom row factors as
	$$
	\begin{tikzcd}
		\H^{\fin} \arrow[hookrightarrow]{r}{\alpha \mapsto M_{\alpha}} & \A \arrow[hookrightarrow]{r}{\ev_{\mathbf{1}}} & \H.
	\end{tikzcd}
	$$
	Since the top row is isomorphic to the bottom row, there must be a matching factorization of the inclusion $L^2(X,\mu)^{\fin} \hookrightarrow L^2(X,\mu)$. This means that there is a map $L^2(X,\mu)^{\fin} \hookrightarrow C(X)$ making the diagram
	\begin{equation} \label{eqn:embedding_homog:full_diagram}
		\begin{tikzcd}
			L^2(X,\mu)^{\fin} \arrow[hookrightarrow]{r} \arrow{d}{\Psi} & C(X) \arrow[hookrightarrow]{r} \arrow{d}{\Phi} & L^2(X,\mu) \arrow{d}{\Psi} \\
			\H^{\fin} \arrow[hookrightarrow]{r}{\alpha \mapsto M_{\alpha}} & \A \arrow[hookrightarrow]{r}{\ev_{\mathbf{1}}} & \H
		\end{tikzcd}
	\end{equation}
	commute, with the compositions of the top and bottom rows being the natural inclusions $L^2(X,\mu)^{\fin} \hookrightarrow L^2(X,\mu)$ and $\H^{\fin} \hookrightarrow \H$. Using that both the second map in the top row and the composition of the two maps in the top row are the natural inclusions, we immediately obtain
	
	\begin{prop} \label{prop:L^2^fin_in_C(X)}
		One has $L^2(X,\mu)^{\fin} \subseteq C(X)$, and the map $L^2(X,\mu)^{\fin} \hookrightarrow C(X)$ in the diagram \eqref{eqn:embedding_homog:full_diagram} is the inclusion.
	\end{prop}
	

	
	
	
	The next proposition says that the pointwise multiplication map $L^2(X,\mu)^{\fin} \times L^2(X,\mu)^{\fin} \to L^2(X,\mu)$ identifies via $\Psi$ with the multiplication map $\H^{\fin} \times \H^{\fin} \to \H$.
	
	\begin{prop} \label{prop:mult_is_ptwise}
		Let $f,h \in L^2(X,\mu)^{\fin}$. Then $f,h \in C(X)$, and $\Psi(fh) = \Psi(f) \Psi(h)$.
	\end{prop}
	
	\begin{proof}
		By Proposition~\ref{prop:L^2^fin_in_C(X)}, we have $f,h \in C(X)$ and hence $fh \in C(X)$. Since the right square in the diagram \eqref{eqn:embedding_homog:full_diagram} commutes, we can write $\Psi(fh) = \ev_{\mathbf{1}}\Phi(fh)$ and $\Psi(h) = \ev_{\mathbf{1}}\Phi(h)$. Therefore
		\begin{align*}
			\Psi(fh)
			= \ev_{\mathbf{1}}\Phi(fh)
			= \ev_{\mathbf{1}} \Phi(f) \Phi(h)
			= \Phi(f) \Phi(h) \mathbf{1}
			= \Phi(f) \ev_{\mathbf{1}} \Phi(h)
			= \Phi(f) \Psi(h).
		\end{align*}
		Since the left square in \eqref{eqn:embedding_homog:full_diagram} commutes, $\Phi(f) = M_{\Psi(f)}$. Continuing the above calculation,
		\begin{gather*}
			\Psi(fh)
			= \Phi(f)\Psi(h)
			= M_{\Psi(f)} \Psi(h)
			= \Psi(f) \Psi(h).
			\qedhere
		\end{gather*}
	\end{proof}
	
	We end this section by observing that we have now proved Theorem~\ref{thm:weak_structure_thm} (assuming Theorem~\ref{thm:high_deriv_L^infty_bd} and Corollary~\ref{cor:qual_Sob_emb} of course). Recall the statement:
	
	\begin{thm*}[Restatement of Theorem~\ref{thm:weak_structure_thm}]
		There exists a connected compact metrizable $G$-space $X$ together with a $G$-invariant probability measure $\mu$ on $X$ of full support, such that $(X,\mu)$ has the qualitative Sobolev embedding property, $\H$ identifies with $L^2(X,\mu)$ as unitary representations of $G$, and the multiplication $\H^{\fin} \times \H^{\fin} \to \H^{\infty}$ is given by pointwise multiplication of functions on $X$.
	\end{thm*}
	
	\begin{proof}
		Let $X,\mu$ be as above. Then all required properties are explicitly stated and proved above.
	\end{proof}
	
	\section{Bulk and tail bounds $\implies$ $L^{\infty}$ bounds} \label{sec:reduce_to_bulk_tail}
	
	
	By the results of Sections~\ref{sec:embedding_homog} and \ref{sec:reduce_to_L^infty}, Theorem~\ref{thm:eq_Gelfand_duality} follows from Theorem~\ref{thm:high_deriv_L^infty_bd} and Corollary~\ref{cor:qual_Sob_emb}.
	We saw in Section~\ref{sec:outline_existence} that Corollary~\ref{cor:qual_Sob_emb} is a consequence of Theorem~\ref{thm:L^infty_quasi-Sobolev}. It therefore remains to prove Theorems~\ref{thm:high_deriv_L^infty_bd} and \ref{thm:L^infty_quasi-Sobolev}. In this section, we prove these two theorems assuming Theorems~\ref{thm:L^4_quasi-Sobolev}, \ref{thm:exp_decay_form}, and \ref{thm:exp_decay_quasimode}. These last three theorems are proved in Sections~\ref{sec:bulk_tail_poly} and \ref{sec:bulk_tail_general}.
	
	\subsection{Sobolev embedding into $L^{\infty}$}
	\label{subsec:L^infty_quasi-Sobolev}
	
	We will deduce Theorem~\ref{thm:L^infty_quasi-Sobolev} from the following proposition. Let $I_i,\lambda_i$ be as in Theorem~\ref{thm:exp_decay_quasimode}.
	
	\begin{prop} \label{prop:L^infty_quasi-Sobolev}
		Let $\alpha \in \H_{\Delta \in I_i}^K$ for some $i$. Then
		\begin{align} \label{eqn:L^infty_quasimode_bd}
			\|\alpha\|_{L^{\infty}}
			\lesssim \exp(O(\log_+^2\lambda_i)) \|\alpha\|_{\H}.
		\end{align}
	\end{prop}
	
	Both to prove Proposition~\ref{prop:L^infty_quasi-Sobolev} and to deduce Theorem~\ref{thm:L^infty_quasi-Sobolev} from Proposition~\ref{prop:L^infty_quasi-Sobolev}, we will need the following elementary estimate.
	
	\begin{lem} \label{lem:sum_i_bdd_by_f(Delta)}
		Let $A \geq 0$ and $v \in \H^K$. Then
		\begin{align*}
			\sum_i \exp(A\log_+^2\lambda_i) \|\1_{\Delta \in I_i} v\|_{\H}
			\lesssim_A \|\exp((A+1) \log_+^2 \Delta) v\|_{\H}.
		\end{align*}
	\end{lem}
	
	\begin{proof}
		On the left hand side, write
		\begin{align*}
			\exp(A\log_+^2\lambda_i) = \exp(-\log_+^2\lambda_i)\exp((A+1)\log_+^2\lambda_i),
		\end{align*}
		and then apply Cauchy--Schwarz to get
		\begin{align*}
			\sum_i \exp(A\log_+^2\lambda_i) \|\1_{\Delta \in I_i} v\|_{\H}
			\leq \Big(\sum_i \exp(-2\log_+^2\lambda_i)\Big)^{\frac{1}{2}} \Big(\sum_i \exp(2(A+1)\log_+^2\lambda_i) \|\1_{\Delta \in I_i} v\|_{\H}^2\Big)^{\frac{1}{2}}.
		\end{align*}
		The first sum on the right is finite because of the polynomial growth property \eqref{eqn:partition_growth}. Therefore,
		\begin{align*}
			\sum_i \exp(A\log_+^2\lambda_i) \|\1_{\Delta \in I_i} v\|_{\H}
			\lesssim \Big(\sum_i \exp(2(A+1)\log_+^2\lambda_i) \|\1_{\Delta \in I_i} v\|_{\H}^2\Big)^{\frac{1}{2}}.
		\end{align*}
		Since $I_i$ has length $\lesssim 1$, we can replace $\lambda_i$ with $\min I_i$ at the cost of a factor of $O_A(1)$, so
		\begin{align*}
			\sum_i \exp(A\log_+^2\lambda_i) \|\1_{\Delta \in I_i} v\|_{\H}
			\lesssim_A \Big(\sum_i \exp(2(A+1)\log_+^2\min I_i) \|\1_{\Delta \in I_i} v\|_{\H}^2\Big)^{\frac{1}{2}}.
		\end{align*}
		Rewriting this using orthogonality,
		\begin{align*}
			\sum_i \exp(A\log_+^2\lambda_i) \|\1_{\Delta \in I_i} v\|_{\H}
			\lesssim_A \Big\| \sum_i \exp((A+1)\log_+^2 \min I_i) \1_{\Delta \in I_i}(v) \Big\|_{\H}.
		\end{align*}
		For $\mu \geq 0$, we have the inequality
		\begin{align*}
			\sum_i \exp((A+1) \log_+^2\min I_i) \1_{\mu \in I_i}
			\leq \exp((A+1) \log_+^2 \mu).
		\end{align*}
		Inserting this above yields the desired estimate.
	\end{proof}
	
	Let us now see how Proposition~\ref{prop:L^infty_quasi-Sobolev} implies Theorem~\ref{thm:L^infty_quasi-Sobolev}. Recall the statement of the theorem:
	
	\begin{thm*}[Restatement of Theorem~\ref{thm:L^infty_quasi-Sobolev}]
		Let $\alpha \in \H^K \cap \H^{\fin}$. Then
		\begin{align*}
			\|\alpha\|_{L^{\infty}}
			\leq \|\exp(O(\log_+^2\Delta))\alpha\|_{\H}.
		\end{align*}
	\end{thm*}
	
	\begin{proof}[Proof of Theorem~\ref{thm:L^infty_quasi-Sobolev} assuming Proposition~\ref{prop:L^infty_quasi-Sobolev}]
		By the triangle inequality and Proposition~\ref{prop:L^infty_quasi-Sobolev},
		\begin{align*}
			\|\alpha\|_{L^{\infty}}
			\leq \sum_i \|\1_{\Delta \in I_i} \alpha\|_{L^{\infty}}
			\lesssim \sum_i \exp(O(\log_+^2\lambda_i)) \|\1_{\Delta \in I_i}\alpha\|_{\H}
		\end{align*}
		(note the triangle inequality was valid because $\alpha \in \H^{\fin}$, so $\1_{\Delta \in I_i}\alpha = 0$ for all but finitely many $i$).
		Thus by Lemma~\ref{lem:sum_i_bdd_by_f(Delta)},
		\begin{align*}
			\|\alpha\|_{L^{\infty}}
			\lesssim \|\exp(O(\log_+^2\Delta)) \alpha\|_{\H}.
		\end{align*}
		By increasing the $O$-constant, $\lesssim$ can be replaced with $\leq$.
	\end{proof}
	
	The remainder of this subsection is dedicated to the proof of Proposition~\ref{prop:L^infty_quasi-Sobolev}. The main tool is the ``modified $L^p$ norm" $\|\cdot\|_{\widetilde{L}^p}$ on $\H^{\fin}$, defined as follows for $4 \leq p \leq \infty$: let $\|\alpha\|_{\widetilde{L}^p}$ be the smallest constant $C \in [0,\infty]$ such that
	\begin{align*}
		\|\alpha\beta\|_{\H}
		\leq C\|\beta\|_{L^4}^{\frac{4}{p}} \|\beta\|_{\H}^{1-\frac{4}{p}}
		\qquad \text{for all} \qquad
		\beta \in \H^{\fin}.
	\end{align*}
	To motivate this definition, note that whenever $f,h$ are functions on a measure space and $4 \leq p \leq \infty$, one has
	\begin{align*}
		\|fh\|_{L^2}
		\leq \|f\|_{L^p} \|h\|_{L^{2p/(p-2)}}
		\leq \|f\|_{L^p} \|h\|_{L^4}^{\frac{4}{p}} \|h\|_{L^2}^{1-\frac{4}{p}}
	\end{align*}
	by H\"older and log-convexity of $L^p$ norms.
	
	The triangle inequality for $\|\cdot\|_{\widetilde{L}^p}$ is clear. For $\beta \in \H^{\fin}$, the first inequality in Proposition~\ref{prop:L^2_leq_L^4_leq_L^infty} says that $\|\beta\|_{\H} \leq \|\beta\|_{L^4}$. It follows that $\|\cdot\|_{\widetilde{L}^p}$ is increasing in $p$. It is trivial to see that the limit as $p \to \infty$ (not including $p=\infty$) is
	\begin{align} \label{eqn:modified_norm_limit}
		\lim_{p \to \infty}\|\cdot\|_{\widetilde{L}^p}
		= \|\cdot\|_{\widetilde{L}^{\infty}}
		= \|\cdot\|_{L^{\infty}}.
	\end{align}
	In addition to $\|\cdot\|_{\widetilde{L}^{\infty}}$ coinciding with the usual $L^{\infty}$ norm, $\|\cdot\|_{{\widetilde{L}}^4}$ coincides with the usual $L^4$ norm:
	
	\begin{prop} \label{prop:modified_L^4_is_L^4}
		Let $\alpha \in \H^{\fin}$. Then $\|\alpha\|_{\widetilde{L}^4} = \|\alpha\|_{L^4}$.
	\end{prop}
	
	\begin{proof}
		For any $\beta \in \H^{\fin}$, we have
		\begin{align} \label{eqn:modified_L^4_vs_L^4}
			\|\alpha\beta\|_{\H}
			\leq \|\alpha\|_{L^4} \|\beta\|_{L^4}
		\end{align}
		by $L^4$-Cauchy--Schwarz (Proposition~\ref{prop:L^4_C-S}), so $\|\alpha\|_{\widetilde{L}^4} \leq \|\alpha\|_{L^4}$. The reverse inequality is due to the fact that equality holds in \eqref{eqn:modified_L^4_vs_L^4} when $\beta = \alpha$.
	\end{proof}
	
	The next lemma will allow us to induct on $p$ when working with $\widetilde{L}^p$ norms. It is analogous to the $\leq$ direction of the equality $\|f\|_{L^{2p}}^2 = \||f|^2\|_{L^p}$, which holds for any function $f$ on a measure space.
	
	\begin{lem} \label{lem:modified_L^p_vs_L^2p}
		Let $\alpha \in \H^{\fin}$ and $4 \leq p \leq \infty$. Then
		\begin{align*}
			\|\alpha\|_{\widetilde{L}^{2p}}^2
			\leq \liminf_{\Lambda \to \infty} \|\1_{|\Delta| \leq \Lambda}(|\alpha|^2)\|_{\widetilde{L}^p}.
		\end{align*}
	\end{lem}
	
	As mentioned in Subsection~\ref{subsec:outline:bulk_tail_poly}, the spectral cutoff $\1_{|\Delta| \leq \Lambda}$ is needed because $|\alpha|^2 \not\in \H^{\fin}$ but $\1_{|\Delta| \leq \Lambda}(|\alpha|^2) \in \H^{\fin}$. We will almost always apply Lemma~\ref{lem:modified_L^p_vs_L^2p} with $\alpha$ a weight vector, in which case $|\alpha|^2 \in \H^K$ and $\1_{|\Delta| \leq \Lambda}(|\alpha|^2) = \1_{\Delta \leq \Lambda}(|\alpha|^2)$.
	
	\begin{proof}[Proof of Lemma~\ref{lem:modified_L^p_vs_L^2p}]
		For $\beta \in \H^{\fin}$, crossing symmetry allows us to write
		\begin{align*}
			\|\alpha\beta\|_{\H}^2
			= \langle |\alpha|^2, |\beta|^2 \rangle_{\H}
			= \lim_{\Lambda \to \infty} \langle \1_{|\Delta| \leq \Lambda}(|\alpha|^2), |\beta|^2 \rangle_{\H}
			= \lim_{\Lambda \to \infty} \langle \1_{|\Delta| \leq \Lambda}(|\alpha|^2) \beta, \beta \rangle_{\H}.
		\end{align*}
		By Cauchy--Schwarz and the definition of the $\widetilde{L}^p$ norm, we get
		\begin{align*}
			\|\alpha\beta\|_{\H}^2
			\leq \liminf_{\Lambda \to \infty} \|\1_{|\Delta| \leq \Lambda}(|\alpha|^2)\beta\|_{\H} \|\beta\|_{\H}
			\leq \liminf_{\Lambda \to \infty} \|\1_{|\Delta| \leq \Lambda}(|\alpha|^2)\|_{\widetilde{L}^p} \|\beta\|_{L^4}^{\frac{4}{p}} \|\beta\|_{\H}^{2-\frac{4}{p}}.
		\end{align*}
		Taking square roots,
		\begin{align*}
			\|\alpha\beta\|_{\H}
			\leq \liminf_{\Lambda \to \infty} \|\1_{|\Delta| \leq \Lambda}(|\alpha|^2)\|_{\widetilde{L}^p}^{\frac{1}{2}} \|\beta\|_{L^4}^{\frac{4}{2p}} \|\beta\|_{\H}^{1-\frac{4}{2p}}.
		\end{align*}
		Thus
		\begin{align*}
			\|\alpha\|_{\widetilde{L}^{2p}}
			\leq \liminf_{\Lambda \to \infty} \|\1_{|\Delta| \leq \Lambda}(|\alpha|^2)\|_{\widetilde{L}^p}^{\frac{1}{2}},
		\end{align*}
		and squaring both sides gives the result.
	\end{proof}
	
	We are now ready to prove Proposition~\ref{prop:L^infty_quasi-Sobolev}.
	
	\begin{proof}[Proof of Proposition~\ref{prop:L^infty_quasi-Sobolev}]
		Fix $A \gg 1$ and $B \gg_A 1$. We will show by induction on $p$ that for any $p \geq 4$ a power of $2$, any $i$, and any $\alpha \in \H_{\Delta \in I_i}^K$,
		\begin{align} \label{eqn:modified_L^p_ind_hypo}
			\|\alpha\|_{\widetilde{L}^p}
			\leq B \exp(A\log_+^2\lambda_i) \|\alpha\|_{\H}.
		\end{align}
		Then since $A,B$ are fixed constants independent of $p$, sending $p \to \infty$ and using \eqref{eqn:modified_norm_limit} will give the desired bound \eqref{eqn:L^infty_quasimode_bd}.
		
		To begin the proof of the base case $p = 4$ of \eqref{eqn:modified_L^p_ind_hypo}, we use Proposition~\ref{prop:modified_L^4_is_L^4}, Theorem~\ref{thm:L^4_quasi-Sobolev}, and the fact that $\alpha \in \H_{\Delta \in I_i}^K$ to write
		\begin{align*}
			\|\alpha\|_{\widetilde{L}^4}
			= \|\alpha\|_{L^4}
			\leq \|\exp(O(\log_+^2\Delta))\alpha\|_{\H}
			\leq \exp(O(\log_+^2\max I_i)) \|\alpha\|_{\H}.
		\end{align*}
		Since $I_i$ has length $\lesssim 1$, we can replace $\max I_i$ with $\lambda_i$ at the cost of increasing the $O$-constant, so
		\begin{align*}
			\|\alpha\|_{\widetilde{L}^4}
			\leq \exp(O(\log_+^2\lambda_i)) \|\alpha\|_{\H}.
		\end{align*}
		Taking $A$ to be at least the $O$-constant and $B$ to be at least $1$, we get \eqref{eqn:modified_L^p_ind_hypo} for $p=4$.
		
		Now make the induction hypothesis that \eqref{eqn:modified_L^p_ind_hypo} holds for some $4 \leq p < \infty$, and let us prove it for $2p$. By Lemma~\ref{lem:modified_L^p_vs_L^2p}, the triangle inequality, and the induction hypothesis,
		\begin{align*}
			\|\alpha\|_{\widetilde{L}^{2p}}^2
			\leq \liminf_{\Lambda \to \infty} \|\1_{\Delta\leq\Lambda}(|\alpha|^2)\|_{\widetilde{L}^p}
			\leq \sum_j \|\1_{\Delta \in I_j}(|\alpha|^2)\|_{\widetilde{L}^p}
			\leq B \sum_j \exp(A\log_+^2\lambda_j) \|\1_{\Delta \in I_j}(|\alpha|^2)\|_{\H},
		\end{align*}
		so by Lemma~\ref{lem:sum_i_bdd_by_f(Delta)},
		\begin{align*}
			\|\alpha\|_{\widetilde{L}^{2p}}^2
			\lesssim_A B \|\exp((A+1)\log_+^2\Delta)(|\alpha|^2)\|_{\H}.
		\end{align*}
		Let $C,c$ be as in Theorem~\ref{thm:exp_decay_quasimode}. Then splitting $\1 = \1_{\Delta < C\lambda_i} + \1_{\Delta \geq C\lambda_i}$,
		\begin{align*}
			\|\alpha\|_{\widetilde{L}^{2p}}^2
			\lesssim_A B \|\1_{\Delta < C\lambda_i}\exp((A+1)\log_+^2\Delta)(|\alpha|^2)\|_{\H} + B \|\1_{\Delta \geq C\lambda_i}\exp((A+1)\log_+^2\Delta)(|\alpha|^2)\|_{\H}.
		\end{align*}
		Estimating the first term trivially and the second term using that $(A+1)\log_+^2\mu \leq c\sqrt{\mu} + O_A(1)$ for $\mu \geq 0$ (note the $O$-constant can be taken independent of $c$ because $c \gtrsim 1$),
		\begin{align*}
			\|\alpha\|_{\widetilde{L}^{2p}}^2
			\lesssim_A B \exp((A+1)\log_+^2(C\lambda_i)) \||\alpha|^2\|_{\H} + B \|\1_{\Delta \geq C\lambda_i} \exp(c\sqrt{\Delta})(|\alpha|^2)\|_{\H}.
		\end{align*}
		Rewriting the first term using \eqref{eqn:|alpha|^2_H_norm} and estimating the second term by Theorem~\ref{thm:exp_decay_quasimode},
		\begin{align*}
			\|\alpha\|_{\widetilde{L}^{2p}}^2
			\lesssim_A B \exp((A+1)\log_+^2(C\lambda_i)) \|\alpha\|_{L^4}^2 + B\|\alpha\|_{\H}^2.
		\end{align*}
		By Theorem~\ref{thm:L^4_quasi-Sobolev}, the bound $C \lesssim 1$, and the fact that $I_i$ has length $\lesssim 1$,
		\begin{align*}
			\|\alpha\|_{\widetilde{L}^{2p}}^2
			\lesssim_A B \exp((A+O(1))\log_+^2\lambda_i) \|\alpha\|_{\H}^2.
		\end{align*}
		Since $A \gg 1$ and $B \gg_A 1$, we may assume $A$ is greater than the $O(1)$ in the exponential, and $B$ is greater than the implicit constant (importantly, these conditions are independent of $p$). Then
		\begin{align*}
			\|\alpha\|_{\widetilde{L}^{2p}}^2
			\leq B^2 \exp(2A\log_+^2\lambda_i) \|\alpha\|_{\H}^2.
		\end{align*}
		Taking square roots yields \eqref{eqn:modified_L^p_ind_hypo} and completes the induction.
	\end{proof}
	
	\subsection{Identities from the product rule} \label{subsec:basic_identities}
	
	To proceed, we need two elementary algebraic identities: Propositions~\ref{prop:p_n_def} and \ref{prop:q_k,n_def} below. These identities are crucial, and will be used many times throughout the rest of the paper.
	The proofs of these propositions are reminiscent of the proofs of Corollary~\ref{cor:num_recursion_specialized}, Proposition~\ref{prop:C_+n_-n_l_size}, and Corollary~\ref{cor:C_+n_-n_formula_explicit}.
	
	\begin{prop} \label{prop:p_n_def}
		Let $\varphi \in \H_{\R}^K \cap \H^{\fin}$ be an automorphic vector with Casimir eigenvalue $\lambda \geq 0$. Then for $n \in \Z_{\geq 0}$,
		\begin{align} \label{eqn:|E^n phi|^2_formula}
			|E^n\varphi|^2
			= p_n(\lambda,\Delta)(\varphi^2),
		\end{align}
		where $p_n$ is the polynomial in two variables with real coefficients given by the second order recurrence
		\begin{align} \label{eqn:p_n_recurrence}
			p_{n+1}(\lambda,\mu)
			= (2\lambda-\mu+2n^2) p_n(\lambda,\mu) - (\lambda+n(n-1))^2 p_{n-1}(\lambda,\mu)
		\end{align}
		for $n \geq 1$, with initial conditions
		\begin{align*}
			p_0(\lambda,\mu) = 1
			\qquad \text{and} \qquad
			p_1(\lambda,\mu) = \lambda - \tfrac{1}{2}\mu.
		\end{align*}
	\end{prop}
	
	The cases $n=0$ and $n=1$ were worked out in Subsection~\ref{subsec:outline:reduce_to_bulk_tail} (see \eqref{eqn:outline:|Ephi|^2}, \eqref{eqn:outline:p_n_def}, and the sentence after \eqref{eqn:outline:p_n_def}). These computations are very short, so we redo them here for convenience.
	
	\begin{proof}
		Since $\varphi$ is real, $|\varphi|^2 = \varphi^2$, so \eqref{eqn:|E^n phi|^2_formula} holds for $n=0$. By \eqref{eqn:Delta|_H^K} and the product rule,
		\begin{align*}
			\Delta(\varphi^2)
			= -\overline{E}E(\varphi^2)
			= -2 \overline{E}(\varphi E\varphi)
			= 2\lambda \varphi^2 - 2|E\varphi|^2.
		\end{align*}
		Rearranging,
		\begin{align*}
			|E\varphi|^2
			= \Big(\lambda - \frac{1}{2}\Delta\Big) (\varphi^2).
		\end{align*}
		This is \eqref{eqn:|E^n phi|^2_formula} for $n = 1$. By \eqref{eqn:Delta|_H^K}, the product rule, and \eqref{eqn:EEbar}, we compute for $n \geq 1$
		\begin{align*}
			\Delta(|E^n\varphi|^2)
			&= -\overline{E}E(E^n\varphi \overline{E}^n \varphi)
			\\&= -|E^{n+1}\varphi|^2 - |\overline{E}E^n\varphi|^2 - \overline{E}E^{n+1} \varphi \overline{E}^n \varphi - E^n\varphi \overline{E} E \overline{E}^n \varphi
			\\&= -|E^{n+1}\varphi|^2 - |(\Delta+H^2+H) E^{n-1}\varphi|^2 + (\Delta+H^2+H) E^n \varphi \overline{E}^n \varphi + E^n\varphi (\Delta+H^2+H) \overline{E}^n\varphi
			\\&= -|E^{n+1}\varphi|^2 - (\lambda+n(n-1))^2 |E^{n-1}\varphi|^2 + 2(\lambda+n^2) |E^n\varphi|^2.
		\end{align*}
		Rearranging,
		\begin{align*}
			|E^{n+1}\varphi|^2
			= (2\lambda - \Delta + 2n^2) (|E^n\varphi|^2) - (\lambda+n(n-1))^2 |E^{n-1}\varphi|^2.
		\end{align*}
		Comparing this with \eqref{eqn:p_n_recurrence}, we conclude by induction that \eqref{eqn:|E^n phi|^2_formula} holds for all $n$.
	\end{proof}
	
	\begin{prop} \label{prop:q_k,n_def}
		Let $f \in \H^{\fin}$ be a lowest weight vector of weight $k \geq 1$. Then for $n \in \Z_{\geq 0}$,
		\begin{align} \label{eqn:|E^n f|^2_formula}
			|E^nf|^2
			= q_{k,n}(\Delta) (|f|^2),
		\end{align}
		where $q_{k,n}$ is the polynomial in one variable with real coefficients given by the second order recurrence
		\begin{align} \label{eqn:q_k,n_recurrence}
			q_{k,n+1}(\mu) = (2(n+k)^2 - 2k(k-1) - \mu) q_{k,n}(\mu) - ((n+k)(n+k-1) - k(k-1))^2 q_{k,n-1}(\mu)
		\end{align}
		for $n \geq 1$, with initial conditions
		\begin{align*}
			q_{k,0}(\mu) = 1
			\qquad \text{and} \qquad
			q_{k,1}(\mu) = 2k-\mu.
		\end{align*}
	\end{prop}
	
	This proposition is analogous to Corollary~\ref{cor:C_+n_-n_formula_explicit}.
	
	\begin{proof}
		The case $n=0$ of \eqref{eqn:|E^n f|^2_formula} is trivial. Computing as in the proof of Proposition~\ref{prop:p_n_def} above,
		\begin{align*}
			\Delta(|f|^2)
			= -\overline{E}E(f\overline{f})
			= -|Ef|^2 - |\overline{E}f|^2 - (\overline{E}Ef)\overline{f} - f\overline{E}E\overline{f}.
		\end{align*}
		Since $f$ is a lowest weight vector, $\overline{E}f = 0$ and consequently $E\overline{f} = \overline{\overline{E}f} = 0$. Therefore
		\begin{align*}
			\Delta(|f|^2)
			= -|Ef|^2 - (\overline{E}Ef)\overline{f}
			= -|Ef|^2 + (\Delta+H^2+H)f \, \overline{f}
			= -|Ef|^2 + 2k|f|^2,
		\end{align*}
		where in the last equality we have used that $f$ has Casimir eigenvalue $-k(k-1)$ and weight $k$. Rearranging,
		\begin{align*}
			|Ef|^2
			= (2k-\Delta)(|f|^2).
		\end{align*}
		This is \eqref{eqn:|E^n f|^2_formula} for $n=1$. Now for $n \geq 1$, write
		\begin{align*}
			\Delta(|E^nf|^2)
			&= -\overline{E}E(E^nf \overline{E}^n\overline{f})
			\\&= -|E^{n+1}f|^2 - |\overline{E}E^nf|^2 - \overline{E}E^{n+1} f \overline{E}^n\overline{f} - E^nf \overline{E}E\overline{E}^n\overline{f}
			\\&= -|E^{n+1}f|^2 - |(\Delta+H^2+H)E^{n-1}f|^2 + (\Delta+H^2+H)E^nf \overline{E}^n\overline{f} + E^nf (\Delta+H^2+H) \overline{E}^n\overline{f}
			\\&= -|E^{n+1}f|^2 - ((n+k)(n+k-1) - k(k-1))^2 |E^{n-1}f|^2 + 2((n+k)^2-k(k-1))|E^nf|^2.
		\end{align*}
		Rearranging,
		\begin{align*}
			|E^{n+1}f|^2
			= (2(n+k)^2 - 2k(k-1) - \Delta)(|E^nf|^2) - ((n+k)(n+k-1) - k(k-1))^2 |E^{n-1}f|^2.
		\end{align*}
		Comparing with \eqref{eqn:q_k,n_recurrence} yields \eqref{eqn:|E^n f|^2_formula} by induction.
	\end{proof}
	
	The recurrences \eqref{eqn:p_n_recurrence} and \eqref{eqn:q_k,n_recurrence} are each analogous to a special case of Corollary~\ref{cor:num_recursion_specialized}.
	The coefficients in \eqref{eqn:p_n_recurrence} become the coefficients in \eqref{eqn:q_k,n_recurrence} after replacing $\lambda$ with $-k(k-1)$ and $n$ with $n+k$.
	Comparing the proofs of Propositions~\ref{prop:p_n_def} and \ref{prop:q_k,n_def} makes this clear.
	
	The polynomials $p_n$ and $q_{k,n}$ obey the following trivial bounds.
	
	\begin{lem}[Trivial bound for $p_n$] \label{lem:p_n_triv_bd}
		For $n \in \Z_{\geq 0}$ and $\lambda,\mu \geq 0$,
		\begin{align*}
			|p_n(\lambda,\mu)|
			\leq O(\lambda+\mu+n^2)^n.
		\end{align*}
	\end{lem}
	
	\begin{lem}[Trivial bound for $q_{k,n}$] \label{lem:q_k,n_triv_bd}
		For $n \in \Z_{\geq 0}$, $k \geq 1$, and $\mu \geq 0$,
		\begin{align*}
			|q_{k,n}(\mu)|
			\leq O(k^2+\mu+n^2)^n.
		\end{align*}
	\end{lem}
	
	\begin{proof}[Proof of Lemmas~\ref{lem:p_n_triv_bd} and \ref{lem:q_k,n_triv_bd}]
		By induction on $n$, both bounds follow directly from the recursive definitions of $p_n$ and $q_{k,n}$.
	\end{proof}
	
	\subsection{Derivative bounds in $L^{\infty}$} \label{subsec:deriv_bds_in_L^infty}
	
	In this subsection, we use Theorem~\ref{thm:L^infty_quasi-Sobolev} to prove Theorem~\ref{thm:high_deriv_L^infty_bd} (still assuming Theorems~\ref{thm:L^4_quasi-Sobolev}, \ref{thm:exp_decay_form}, and \ref{thm:exp_decay_quasimode}).
	
	From Lemma~\ref{lem:modified_L^p_vs_L^2p} and Theorem~\ref{thm:L^infty_quasi-Sobolev} we obtain the following corollary.
	
	\begin{cor} \label{cor:weight_vector_L^infty_bd}
		Let $\alpha \in \H^{\fin}$ be a weight vector. Then
		\begin{align*}
			\|\alpha\|_{L^{\infty}}^2
			\leq \|\exp(O(\log_+^2 \Delta))(|\alpha|^2)\|_{\H}.
		\end{align*}
	\end{cor}
	
	\begin{proof}
		Since the $L^{\infty}$ and $\widetilde{L}^{\infty}$ norms coincide, the $p=\infty$ case of Lemma~\ref{lem:modified_L^p_vs_L^2p} says that
		\begin{align*}
			\|\alpha\|_{L^{\infty}}^2
			\leq \liminf_{\Lambda \to \infty} \|\1_{\Delta \leq \Lambda}(|\alpha|^2)\|_{L^{\infty}}.
		\end{align*}
		Since $\alpha$ is a weight vector, $|\alpha|^2 \in \H^K$, and so $\1_{\Delta \leq \Lambda}(|\alpha|^2) \in \H^K \cap \H^{\fin}$. Thus we can apply Theorem~\ref{thm:L^infty_quasi-Sobolev} to conclude.
	\end{proof}
	
	\begin{prop} \label{prop:bdd_mult_qual}
		Let $\alpha \in \H^{\fin}$. Then $\|\alpha\|_{L^{\infty}} < \infty$.
	\end{prop}
	
	\begin{proof}
		By Proposition~\ref{prop:Maass_span_H^fin}, we know that $\alpha$ is a linear combination of automorphic vectors, so by the triangle inequality, it suffices to prove the proposition for $\alpha$ itself an automorphic vector. Replacing $\alpha$ by its complex conjugate if necessary, we may assume $\alpha$ has nonnegative weight $n \geq 0$. Then by Proposition~\ref{prop:Maass_is_raised}, either $\alpha = E^n\varphi$ for some automorphic vector $\varphi \in \H^K \cap \H^{\fin}$ with Casimir eigenvalue $\lambda \geq 0$, or $\alpha = E^{n-k} f$ for some lowest weight vector $f \in \H^{\fin}$ of weight $k$ between $1$ and $n$. If $\alpha = E^n\varphi$, then by splitting $\varphi$ into real and imaginary parts and using the triangle inequality, we may assume $\varphi \in \H_{\R}$. It follows from Corollary~\ref{cor:weight_vector_L^infty_bd}, together with Proposition~\ref{prop:p_n_def} if $\alpha = E^n\varphi$ or Proposition~\ref{prop:q_k,n_def} if $\alpha = E^{n-k}f$, that
		\begin{align*}
			\|\alpha\|_{L^{\infty}}^2
			\leq \|\exp(O(\log_+^2\Delta)) p_n(\lambda,\Delta) (\varphi^2)\|_{\H}
			\qquad \text{or} \qquad
			\|\alpha\|_{L^{\infty}}^2
			\leq \|\exp(O(\log_+^2\Delta)) q_{k,n-k}(\Delta) (|f|^2)\|_{\H}.
		\end{align*}
		By Lemma~\ref{lem:p_n_triv_bd} and Corollary~\ref{cor:exp_decay_quasimode_qual} if $\alpha = E^n\varphi$, or by Lemma~\ref{lem:q_k,n_triv_bd} and Theorem~\ref{thm:exp_decay_form} if $\alpha = E^{n-k}f$, we conclude that $\|\alpha\|_{L^{\infty}}$ is finite.
	\end{proof}
	
	We now quantify the second half of the proof of Proposition~\ref{prop:bdd_mult_qual}, keeping track of the dependence on the weight $n$.
	
	\begin{prop} \label{prop:E^n_phi_L^infty_bd}
		Let $\varphi \in \H^K \cap \H^{\fin}$ be an automorphic vector. Then for $n \in \Z_{\geq 0}$,
		\begin{align} \label{eqn:E^n_phi_L^infty_bd}
			\|E^n\varphi\|_{L^{\infty}}
			\lesssim_{\varphi} O(n)^n.
		\end{align}
	\end{prop}
	
	\begin{proof}
		Splitting $\varphi$ into real and imaginary parts and using the triangle inequality, we may assume without loss of generality that $\varphi \in \H_{\R}$.
		Let $\lambda$ denote the Casimir eigenvalue of $\varphi$. Then by Corollary~\ref{cor:weight_vector_L^infty_bd} and Proposition~\ref{prop:p_n_def} (as in the proof of Proposition~\ref{prop:bdd_mult_qual}),
		\begin{align*} 
			\|E^n\varphi\|_{L^{\infty}}^2
			\leq \|\exp(O(\log_+^2\Delta)) p_n(\lambda,\Delta) (\varphi^2)\|_{\H}.
		\end{align*}
		Inserting the trivial bound from Lemma~\ref{lem:p_n_triv_bd},
		\begin{align*}
			\|E^n\varphi\|_{L^{\infty}}^2
			\leq \|\exp(O(\log_+^2\Delta)) O(\Delta+\lambda+n^2)^n(\varphi^2)\|_{\H}.
		\end{align*}
		For $n \lesssim_{\varphi} 1$, Proposition~\ref{prop:bdd_mult_qual} trivially gives the desired bound \eqref{eqn:E^n_phi_L^infty_bd}, because the implicit constant in \eqref{eqn:E^n_phi_L^infty_bd} is allowed to depend on $\varphi$ (though the $O$-constant is not). So assume $n \gg_{\varphi} 1$.
		Take this to mean $n^2 \geq \lambda+1$.
		Then splitting $\1 = \1_{\Delta \leq n^2} + \1_{\Delta > n^2}$,
		\begin{align*}
			\|E^n\varphi\|_{L^{\infty}}^2
			\leq \|\1_{\Delta \leq n^2} \exp(O(\log_+^2\Delta)) O(n^2)^n(\varphi^2)\|_{\H}
			+ \|\1_{\Delta > n^2} \exp(O(\log_+^2\Delta)) O(\Delta)^n(\varphi^2)\|_{\H}.
		\end{align*}
		Estimating the first term by ignoring $\1_{\Delta \leq n^2}$ and using Corollary~\ref{cor:exp_decay_quasimode_qual},
		\begin{align*}
			\|E^n\varphi\|_{L^{\infty}}^2
			\lesssim_{\varphi} O(n)^{2n}
			+ \|\1_{\Delta > n^2} \exp(O(\log_+^2\Delta)) O(\Delta)^n(\varphi^2)\|_{\H}.
		\end{align*}
		Dyadically decomposing,
		\begin{align*}
			\|E^n\varphi\|_{L^{\infty}}^2
			\lesssim_{\varphi} O(n)^{2n}
			+ \sum_{\substack{M \in 2^{\mathbf{N}} \\ M \geq \frac{1}{2}n^2}} \|\1_{\Delta \in [M,2M]} \exp(O(\log_+^2\Delta)) O(\Delta)^n(\varphi^2)\|_{\H}.
		\end{align*}
		Let $c \gtrsim 1$ be as in Corollary~\ref{cor:exp_decay_quasimode_qual}.
		The summand above is at most
		\begin{align*}
			\exp(O(\log_+^2 M)) O(M)^n \exp(-c\sqrt{M}) \|\1_{\Delta \in [M,2M]} \exp(c\sqrt{\Delta})(\varphi^2)\|_{\H}.
		\end{align*}
		Ignoring $\1_{\Delta \in [M,2M]}$ and applying Corollary~\ref{cor:exp_decay_quasimode_qual}, the norm in this product is $\lesssim_{\varphi} 1$.
		Thus
		\begin{align*}
			\|E^n\varphi\|_{L^{\infty}}^2
			\lesssim_{\varphi} O(n)^{2n}
			+ \sum_{\substack{M \in 2^{\mathbf{N}} \\ M \geq \frac{1}{2}n^2}} \exp(O(\log_+^2 M)) O(M)^n \exp(-c\sqrt{M}).
		\end{align*}
		When $M$ is multiplied by $2$, the summand is multiplied by
		\begin{align*}
			M^{O(1)} O(1)^n \exp(-c'\sqrt{M}),
		\end{align*}
		where $c' = (\sqrt{2}-1)c \gtrsim 1$. So once $M$ is bigger than an absolute constant times $n^2$, the sum decays (faster than) geometrically. Thus the first $O(1)$ many terms in the sum dominate. Each of these $O(1)$ many terms is bounded by $O(n)^{2n}$. Hence
		\begin{align*}
			\|E^n\varphi\|_{L^{\infty}}^2
			\lesssim_{\varphi} O(n)^{2n},
		\end{align*}
		and taking square roots gives \eqref{eqn:E^n_phi_L^infty_bd}, as desired.
	\end{proof}
	
	\begin{prop} \label{prop:E^n_f_L^infty_bd}
		Let $f \in \H^{\fin}$ be a lowest weight vector. Then for $n \in \Z_{\geq 0}$,
		\begin{align*}
			\|E^n f\|_{L^{\infty}}
			\lesssim_f O(n)^n.
		\end{align*}
	\end{prop}
	
	\begin{proof}
		This can be shown in the same way as Proposition~\ref{prop:E^n_phi_L^infty_bd}, using Proposition~\ref{prop:q_k,n_def} in place of Proposition~\ref{prop:p_n_def}, and using Theorem~\ref{thm:exp_decay_form} in place of Corollary~\ref{cor:exp_decay_quasimode_qual}.
	\end{proof}
	
	Finally, we are ready to prove Theorem~\ref{thm:high_deriv_L^infty_bd}. Recall the statement:
	
	\begin{thm*}[Restatement of Theorem~\ref{thm:high_deriv_L^infty_bd}]
		Let $\alpha \in \H^{\fin}$, let $X \in \g$, and let $n \in \Z_{\geq 0}$. Then
		\begin{align} \label{eqn:X^n_alpha_L^infty_bd}
			\|X^n\alpha\|_{L^{\infty}}
			\lesssim_{\alpha} O(n\|X\|_{\g})^n.
		\end{align}
	\end{thm*}
	
	\begin{proof}
		For notational convenience, normalize $\|X\|_{\g} = 1$. Then we want to show $\|X^n\alpha\|_{L^{\infty}} \lesssim_{\alpha} O(n)^n$. As in the proof of Proposition~\ref{prop:bdd_mult_qual}, by Proposition~\ref{prop:Maass_span_H^fin} and symmetry under complex conjugation, we may assume without loss of generality that $\alpha$ is an automorphic vector of nonnegative weight $m \geq 0$. Then by Proposition~\ref{prop:Maass_is_raised}, either $\alpha = E^m\varphi$ for some automorphic vector $\varphi \in \H^K \cap \H^{\fin}$ with Casimir eigenvalue $\lambda \geq 0$, or $\alpha = E^{m-k}f$ for some lowest weight vector $f$ of weight $k$ between $1$ and $m$.
		If $\alpha = E^m\varphi$, then by splitting $\varphi$ into real and imaginary parts as usual, we may further assume $\varphi \in \H_{\R}$.
		For $r \in \Z_{\geq 0}$, denote
		\begin{align*}
			\varphi_r = E^r\varphi
			\text{ and }
			\varphi_{-r} = \overline{E}^r\varphi
			\qquad \text{or} \qquad
			f_r = E^rf.
		\end{align*}
		Then by Proposition~\ref{prop:E^n_phi_L^infty_bd} and complex conjugation symmetry, or by Proposition~\ref{prop:E^n_f_L^infty_bd},
		\begin{align} \label{eqn:phi_r_f_r_L^infty_bds}
			\|\varphi_r\|_{L^{\infty}}
			\lesssim_{\alpha} O(|r|)^{|r|} \text{ for } r \in \Z
			\qquad \text{or} \qquad
			\|f_r\|_{L^{\infty}}
			\lesssim_{\alpha} O(r)^r \text{ for } r \in \Z_{\geq 0}.
		\end{align}
		For $r \in \Z_{>0}$, by the commutation relations \eqref{eqn:comm_rlns} for $H,E,\overline{E}$,
		\begin{align} \label{eqn:phi_r_g-action}
			H\varphi_r = r\varphi_r,
			\qquad
			E\varphi_r = \varphi_{r+1},
			\qquad
			\overline{E} \varphi_r = -(\lambda+r(r-1)) \varphi_{r-1}.
		\end{align}
		Applying complex conjugation, we have again for $r \in \Z_{>0}$ that
		\begin{align} \label{eqn:phi_-r_g-action}
			H\varphi_{-r} = -r\varphi_{-r},
			\qquad
			\overline{E}\varphi_{-r} = \varphi_{-r-1},
			\qquad
			E\varphi_{-r} = -(\lambda+r(r-1)) \varphi_{-r+1}.
		\end{align}
		For $r=0$,
		\begin{align} \label{eqn:phi_0_g-action}
			H\varphi_0 = 0,
			\qquad
			E\varphi_0 = \varphi_1,
			\qquad
			\overline{E}\varphi_0 = \varphi_{-1}.
		\end{align}
		Similarly, for $r \in \Z_{\geq 0}$,
		\begin{align} \label{eqn:f_r_g-action}
			Hf_r = (r+k)f_r,
			\qquad
			Ef_r = f_{r+1},
			\qquad
			\overline{E}f_r = (k(k-1) - (r+k)(r+k-1)) f_{r-1}
		\end{align}
		(when $r=0$, the third equation should be interpreted as $\overline{E}f_r = 0$).
		Write $X$ as a linear combination of $H,E,\overline{E}$. Since $\|X\|_{\g} = 1$, the coefficients in this linear combination are $\lesssim 1$.
		By induction on $n$, using \eqref{eqn:phi_r_g-action}, \eqref{eqn:phi_-r_g-action}, and \eqref{eqn:phi_0_g-action} in case $\alpha = E^m\varphi$, or using $\eqref{eqn:f_r_g-action}$ in case $\alpha = E^{m-k}f$,
		\begin{align} \label{eqn:X^n_alpha_wt_decomp}
			X^n\alpha
			= \sum_{\substack{r \in \Z \\ |r-m| \leq n}} a_{n,r} \varphi_r
			\qquad \text{or} \qquad
			X^n\alpha
			= \sum_{\substack{r \in \Z_{\geq 0} \\ |r-m+k| \leq n}} b_{n,r} f_r
		\end{align}
		for some coefficients $a_{n,r}$ or $b_{n,r}$ with
		\begin{align} \label{eqn:a,b,_inductive_bds}
			|a_{n,r}| \leq O(1)^n (n+m+\sqrt{\lambda})^{n+m-|r|}
			\qquad \text{or} \qquad
			|b_{n,r}| \leq O(1)^n (n+m+k)^{n+m-k-r}.
		\end{align}
		Now, as in the proof of Proposition~\ref{prop:E^n_phi_L^infty_bd}, we may assume $n \gg_{\alpha} 1$, because the desired bound \eqref{eqn:X^n_alpha_L^infty_bd} trivially follows from Proposition~\ref{prop:bdd_mult_qual} for $n \lesssim_{\alpha} 1$.
		In particular, we can assume $n \gg_{m,\lambda} 1$ or $n \gg_{m,k} 1$. Then \eqref{eqn:a,b,_inductive_bds} reduces to
		\begin{align} \label{eqn:a,b_large_n_bds}
			|a_{n,r}| \leq O(1)^n n^{n-|r|}
			\qquad \text{or} \qquad
			|b_{n,r}| \leq O(1)^n n^{n-r}
		\end{align}
		after increasing the $O$-constant.
		Taking $L^{\infty}$ norms in \eqref{eqn:X^n_alpha_wt_decomp} and inserting \eqref{eqn:phi_r_f_r_L^infty_bds} and \eqref{eqn:a,b_large_n_bds}, we obtain \eqref{eqn:X^n_alpha_L^infty_bd}, as desired (recall we normalized $\|X\|_{\g} = 1$).
	\end{proof}
	
	\section{Proof of bulk and tail bounds assuming a polynomial Weyl law} \label{sec:bulk_tail_poly}
	
	At this point, we have reduced Theorem~\ref{thm:eq_Gelfand_duality} to Theorems~\ref{thm:L^4_quasi-Sobolev}, \ref{thm:exp_decay_form}, and \ref{thm:exp_decay_quasimode}. As mentioned at the beginning of Subsection~\ref{subsec:outline:bulk_tail_poly}, Theorems~\ref{thm:L^4_quasi-Sobolev} and \ref{thm:exp_decay_quasimode} are particularly technical. In this section, to warm up, we prove these three theorems under the assumption that $\H$ obeys a polynomial Weyl law (Definition~\ref{def:poly_Weyl_law}). We will be clear about when this assumption is used, and we do not assume it by default. In fact, the proof of Theorem~\ref{thm:exp_decay_form} does not use this assumption.

	\subsection{Bootstrap inequalities}
	\label{subsec:basic_ineqs}
	
	Combining the results of Section~\ref{subsec:basic_identities} with crossing symmetry, we can prove some exact inequalities which will provide the main analytic input for Theorems~\ref{thm:L^4_quasi-Sobolev}, \ref{thm:exp_decay_form}, and \ref{thm:exp_decay_quasimode}. These inequalities are Propositions~\ref{prop:basic_ineq_1}, \ref{prop:basic_ineq_2}, and \ref{prop:basic_ineq_3}.
	
	Let $p_n$ and $q_{k,n}$ be as in Propositions~\ref{prop:p_n_def} and \ref{prop:q_k,n_def}, respectively.
	
	\begin{prop} \label{prop:basic_ineq_1}
		Let $\varphi \in \H_{\R}^K \cap \H^{\fin}$ be an automorphic vector with Casimir eigenvalue $\lambda \geq 0$. For $n \in \Z_{\geq 0}$,
		\begin{align*}
			-\langle p_np_{n+1}(\lambda,\Delta)(\varphi^2), \varphi^2 \rangle_{\H}
			\leq 0.
		\end{align*}
	\end{prop}
	
	This inequality is \eqref{eqn:outline:basic_ineq_1}. We recap the proof, which is very short.
	
	\begin{proof}
		By crossing symmetry, Proposition~\ref{prop:p_n_def}, self-adjointness of $\Delta$, and the fact that the $p_n$ have real coefficients,
		\begin{align*}
			0 \leq \|E^{n+1}\varphi \overline{E}^n\varphi\|_{\H}^2
			= \langle |E^n\varphi|^2, |E^{n+1}\varphi|^2 \rangle_{\H}
			&= \langle p_n(\lambda,\Delta)(\varphi^2), p_{n+1}(\lambda,\Delta)(\varphi^2) \rangle_{\H}
			\\&= \langle p_np_{n+1}(\lambda,\Delta)(\varphi^2), \varphi^2 \rangle_{\H}.
		\end{align*}
		Negating both sides gives the result.
	\end{proof}
	
	The analog of Proposition~\ref{prop:basic_ineq_1} for lowest weight vectors is
	
	\begin{prop} \label{prop:basic_ineq_2}
		Let $f \in \H^{\fin}$ be a lowest weight vector of weight $k \geq 1$. Then for $n \in \Z_{\geq 0}$,
		\begin{align*}
			-\langle q_{k,n} q_{k,n+1}(\Delta)(|f|^2), |f|^2 \rangle_{\H} \leq 0.
		\end{align*}
	\end{prop}
	
	\begin{proof}
		By crossing symmetry, Proposition~\ref{prop:q_k,n_def}, self-adjointness of $\Delta$, and the fact that the $q_{k,n}$ have real coefficients,
		\begin{align*}
			0 \leq \|E^{n+1}f \overline{E}^n\overline{f}\|_{\H}^2
			= \langle |E^nf|^2, |E^{n+1}f|^2 \rangle_{\H}
			&= \langle q_{k,n}(\Delta)(|f|^2), q_{k,n+1}(\Delta)(|f|^2) \rangle_{\H}
			\\&= \langle q_{k,n}q_{k,n+1}(\Delta)(|f|^2), |f|^2 \rangle_{\H}.
		\end{align*}
		Negating both sides gives the result.
	\end{proof}
	
	Our next inequality, Proposition~\ref{prop:basic_ineq_3}, is sharper but more complicated than Proposition~\ref{prop:basic_ineq_1}. To state it, define
	\begin{align} \label{eqn:s_n_def}
		s_n(\lambda,\mu)
		= p_{n+1}(\lambda,\mu) - (\lambda+n(n+1)) p_n(\lambda,\mu)
	\end{align}
	and
	\begin{align} \label{eqn:r_n_def}
		r_n(\lambda,\mu) = \1_{\mu \geq 1} \mu^{-1} s_n(\lambda,\mu)^2 - p_np_{n+1}(\lambda,\mu).
	\end{align}
	The cutoff $\1_{\mu \geq 1}$ in the first term is somewhat arbitrary; we could use $\1_{\mu \geq \mu_0}$ instead for any fixed $\mu_0 > 0$ without changing much. Note that both $r_n$ and $s_n$ take real values on real inputs.
	
	\begin{prop} \label{prop:basic_ineq_3}
		Let $\varphi \in \H_{\R}^K \cap \H^{\fin}$ be an automorphic vector with Casimir eigenvalue $\lambda \geq 0$. For $n \in \Z_{\geq 0}$,
		\begin{align*}
			\langle r_n(\lambda,\Delta)(\varphi^2), \varphi^2 \rangle_{\H}
			\leq 0.
		\end{align*}
	\end{prop}
	
	This is stronger than Proposition~\ref{prop:basic_ineq_1} because $r_n \geq -p_np_{n+1}$ pointwise.
	The proof of Proposition~\ref{prop:basic_ineq_3} will use the following simple lemma.
	
	\begin{lem} \label{lem:wt_1_lower_bd}
		Let $v \in \H^{\infty}$ be a vector of weight $1$. Then
		\begin{align*}
			\|v\|_{\H}
			\geq \|\1_{\Delta \geq 1} \Delta^{-\frac{1}{2}} \overline{E}v\|_{\H}.
		\end{align*}
	\end{lem}
	
	\begin{proof}
		Rewrite the square of the right hand side using self-adjointness of $\Delta$, \eqref{eqn:Lie_alg_unitarity}, \eqref{eqn:EEbar}, and the fact that $v$ has weight $1$:
		\begin{align*}
			\|\1_{\Delta \geq 1} \Delta^{-\frac{1}{2}} \overline{E}v\|_{\H}^2
			&= -\langle \1_{\Delta \geq 1} \Delta^{-1} E\overline{E}v, v \rangle_{\H}
			\\&= \langle \1_{\Delta \geq 1} \Delta^{-1} (\Delta+H^2-H) v,v \rangle_{\H}
			= \langle \1_{\Delta \geq 1}v, v \rangle_{\H}
			= \|\1_{\Delta \geq 1}v\|_{\H}^2
			\leq \|v\|_{\H}^2.
			\qedhere
		\end{align*}
	\end{proof}
	
	\begin{proof}[Proof of Proposition~\ref{prop:basic_ineq_3}]
		By Lemma~\ref{lem:wt_1_lower_bd}, we can lower bound
		\begin{align*}
			\|E^{n+1}\varphi \overline{E}^n\varphi\|_{\H}^2
			\geq \|\1_{\Delta\geq 1} \Delta^{-\frac{1}{2}} \overline{E}(E^{n+1}\varphi \overline{E}^n\varphi)\|_{\H}^2. 
		\end{align*}
		By the product rule and \eqref{eqn:EEbar},
		\begin{align*}
			\|E^{n+1}\varphi \overline{E}^n\varphi\|_{\H}^2
			&\geq \|\1_{\Delta\geq 1} \Delta^{-\frac{1}{2}} (|E^{n+1}\varphi|^2 -(\Delta+H^2+H)E^n\varphi \overline{E}^n\varphi)\|_{\H}^2
			\\&= \|\1_{\Delta\geq 1} \Delta^{-\frac{1}{2}} (|E^{n+1}\varphi|^2 -(\lambda+n(n+1))|E^n\varphi|^2)\|_{\H}^2.
		\end{align*}
		Plugging in Proposition~\ref{prop:p_n_def}, and then writing the result in terms of $s_n$ to simplify notation,
		\begin{align*}
			\|E^{n+1}\varphi \overline{E}^n\varphi\|_{\H}^2
			\geq \|\1_{\Delta \geq 1} \Delta^{-\frac{1}{2}} s_n(\lambda,\Delta)(\varphi^2)\|_{\H}^2.
		\end{align*}
		By self-adjointness of $\Delta$ and real-valuedness of $s_n$, we can rewrite this as
		\begin{align*}
			\|E^{n+1}\varphi \overline{E}^n\varphi\|_{\H}^2
			\geq \langle \1_{\Delta \geq 1} \Delta^{-1} s_n(\lambda,\Delta)^2(\varphi^2), \varphi^2 \rangle_{\H}.
		\end{align*}
		On the other hand, by the proof of Proposition~\ref{prop:basic_ineq_1},
		\begin{align*}
			\|E^{n+1}\varphi \overline{E}^n\varphi\|_{\H}^2
			= \langle p_np_{n+1}(\lambda,\Delta)(\varphi^2), \varphi^2 \rangle_{\H}.
		\end{align*}
		Thus
		\begin{align*}
			\langle p_np_{n+1}(\lambda,\Delta)(\varphi^2), \varphi^2 \rangle_{\H}
			\geq \langle \1_{\Delta \geq 1} \Delta^{-1} s_n(\lambda,\Delta)^2(\varphi^2), \varphi^2 \rangle_{\H}.
		\end{align*}
		Subtracting the left hand side from the right hand side and writing the result in terms of $r_n$,
		\begin{align*}
			\langle r_n(\lambda,\Delta)(\varphi^2), \varphi^2 \rangle_{\H}
			\leq 0,
		\end{align*}
		as desired.
	\end{proof}
	
	\subsection{Bulk bound} \label{subsec:bulk_poly} Our aim in this subsection is to prove Theorem~\ref{thm:L^4_quasi-Sobolev} under the assumption that $\H$ obeys a polynomial Weyl law, i.e., to prove
	
	\begin{thm} \label{thm:L^4_quasi-Sobolev_poly_Weyl}
		Assume $\H$ obeys a polynomial Weyl law. Let $\alpha \in \H^K \cap \H^{\fin}$. Then
		\begin{align} \label{eqn:L^4_quasi-Sobolev_poly_Weyl}
			\|\alpha\|_{L^4}
			\leq \|\exp(O(\log_+^2\Delta))\alpha\|_{\H}.
		\end{align}
	\end{thm}
	
	For $\Lambda \geq 0$, let $C(\Lambda) \in [0,\infty]$ be the smallest nonnegative constant such that for all automorphic vectors $\varphi \in \H_{\R}^K \cap \H^{\fin}$ with Casimir eigenvalue $\leq \Lambda$,
	\begin{align*}
		\|\varphi\|_{L^4}
		\leq C(\Lambda) \|\varphi\|_{\H}.
	\end{align*}
	It is clear from the definition that $C(\Lambda)$ is increasing in $\Lambda$. We will reduce Theorem~\ref{thm:L^4_quasi-Sobolev_poly_Weyl} to Proposition~\ref{prop:C_s(Lambda)_inductive_bd} below, which is stated purely in terms of $C(\Lambda)$.
	First, we need a couple easy lemmas.
	
	
	
	
	
	\begin{lem} \label{lem:C(Lambda)_finite}
		For every $\Lambda \geq 0$, the constant $C(\Lambda)$ is finite.
	\end{lem}
	
	\begin{proof}
		Since $\H$ has discrete spectrum, $\H_{\Delta \leq \Lambda}^K$ is finite-dimensional. Any two norms on a finite-dimensional space are equivalent, so $\|\cdot\|_{L^4} \lesssim_{\Lambda} \|\cdot\|_{\H}$ on $\H_{\Delta \leq \Lambda}^K$. The implicit constant here is an upper bound for $C(\Lambda)$.
	\end{proof}
	
	\begin{lem} \label{lem:L^4_bdd_by_C_s(Lambda)}
		Assume $\H$ obeys a polynomial Weyl law. Let $\Lambda \geq 1$. Then for $\alpha \in \H_{\Delta \leq \Lambda}^K$,
		\begin{align*}
			\|\alpha\|_{L^4}
			\lesssim \Lambda^{O(1)} C(\Lambda) \|\alpha\|_{\H}.
		\end{align*}
	\end{lem}
	
	\begin{proof}
		Splitting $\alpha$ into real and imaginary parts, we may assume $\alpha \in \H_{\R}$. Then by the triangle inequality, the definition of $C(\Lambda)$, and the polynomial Weyl law,
		\begin{align*}
			\|\alpha\|_{L^4}
			\leq \sum_{\lambda \leq \Lambda} \|\1_{\Delta=\lambda}\alpha\|_{L^4}
			\leq C(\Lambda) \sum_{\lambda \leq \Lambda} \|\1_{\Delta=\lambda}\alpha\|_{\H}
			\lesssim
			\Lambda^{O(1)} C(\Lambda) \|\alpha\|_{\H}.
		\end{align*}
		Note the second inequality was valid because $\1_{\Delta=\lambda}\alpha$ is a real $K$-invariant automorphic vector with Casimir eigenvalue $\leq\Lambda$.
	\end{proof}
	
	\begin{prop} \label{prop:C_s(Lambda)_inductive_bd}
		Assume $\H$ obeys a polynomial Weyl law. Then there is an absolute constant $c \in (0,1)$ such that for $\Lambda \gg 1$,
		\begin{align} \label{eqn:C_s(Lambda)_inductive_bd}
			C(\Lambda)
			\lesssim \Lambda^{O(1)} C(c\Lambda).
		\end{align}
	\end{prop}
	
	This proposition is the technical heart of the proof of Theorem~\ref{thm:L^4_quasi-Sobolev_poly_Weyl}. We defer the proof of the proposition to the end of this subsection.
	
	\begin{proof}[Proof of Theorem~\ref{thm:L^4_quasi-Sobolev_poly_Weyl} assuming Proposition~\ref{prop:C_s(Lambda)_inductive_bd}]
		Let $c \in (0,1)$ be an absolute constant as in Proposition~\ref{prop:C_s(Lambda)_inductive_bd}. Then by induction on $\Lambda$ (using Lemma~\ref{lem:C(Lambda)_finite} for the base case), for $\Lambda \geq 0$ we have
		\begin{align} \label{eqn:C_s(Lambda)_quasipolynomial_bd}
			C(\Lambda)
			\lesssim \exp(O(\log_+^2\Lambda)).
		\end{align}
		Now let $\alpha \in \H^K \cap \H^{\fin}$. We wish to prove \eqref{eqn:L^4_quasi-Sobolev_poly_Weyl}.
		Denote $\alpha_m = \1_{\Delta \in [m-1,m)}\alpha$, so that $\alpha = \sum_{m=1}^{\infty} \alpha_m$. Since $\alpha \in \H^{\fin}$, this is a finite sum. Thus by the triangle inequality, Lemma~\ref{lem:L^4_bdd_by_C_s(Lambda)}, and \eqref{eqn:C_s(Lambda)_quasipolynomial_bd},
		\begin{align*}
			\|\alpha\|_{L^4}
			\leq \sum_{m=1}^{\infty} \|\alpha_m\|_{L^4}
			\lesssim \sum_{m=1}^{\infty} m^{O(1)} C(m) \|\alpha_m\|_{\H}
			\lesssim \sum_{m=1}^{\infty} \exp(O(\log_+^2 m)) \|\alpha_m\|_{\H}.
		\end{align*}
		Applying Cauchy--Schwarz similarly to how it is applied in Lemma~\ref{lem:sum_i_bdd_by_f(Delta)},
		\begin{align*}
			\|\alpha\|_{L^4}
			\lesssim \Big(\sum_{m=1}^{\infty} \exp(O(\log_+^2 m)) \|\alpha_m\|_{\H}^2\Big)^{\frac{1}{2}}
			\lesssim \|\exp(O(\log_+^2\Delta))\alpha\|_{\H}.
		\end{align*}
		We have almost proved \eqref{eqn:L^4_quasi-Sobolev_poly_Weyl}, except we have $\lesssim$ instead of $\leq$. This is easily remedied by increasing the $O$-constant.
	\end{proof}
	
	Recall the functions $p_n,q_{k,n},r_n,s_n$ defined in Proposition~\ref{prop:p_n_def}, Proposition~\ref{prop:q_k,n_def}, \eqref{eqn:r_n_def}, and \eqref{eqn:s_n_def}, respectively.
	
	\begin{lem} \label{lem:R_def}
		There exist absolute constants $N \in \Z_{\geq 0}$, $a_0,\dots,a_N \geq 0$, and $c \in (0,1)$, such that if we let
		\begin{align*}
			R(\lambda,\mu)
			= \sum_{n=0}^{N} a_n \lambda^{2(N-n)} r_n(\lambda,\mu),
		\end{align*}
		then for $\mu \geq c\lambda \gg 1$, we have $R(\lambda,\mu) \geq 0$ and in fact
		\begin{align} \label{eqn:R_asymp}
			R(\lambda,\mu) \sim \mu^{2N+1}.
		\end{align}
	\end{lem}
	
	We will see that the $c$ in Proposition~\ref{prop:C_s(Lambda)_inductive_bd} can be taken to be the same $c$ as in Lemma~\ref{lem:R_def}.
	
	\begin{proof}[Proof of Lemma~\ref{lem:R_def}]
		We claim that $N=1$, $a_0 = a_1 = 1$, and $c = \frac{1}{2}$ work. Let $\mu \geq \frac{1}{2}\lambda \gg 1$. Using the recursive definition of $p_n$ from Proposition~\ref{prop:p_n_def}, we compute
		\begin{align*}
			p_0(\lambda,\mu) = 1,
			\qquad
			p_1(\lambda,\mu) = \lambda - \frac{1}{2}\mu,
			\qquad
			p_2(\lambda,\mu) = \lambda^2 - 2\lambda\mu + \frac{1}{2}\mu^2 + 2\lambda - \mu.
		\end{align*}
		Plugging these into the definition of $s_n$ and then into the definition of $r_n$, and using that $\mu \geq 1$,
		\begin{align*}
			r_0(\lambda,\mu)
			= -\lambda + \frac{3}{4}\mu
			\qquad \text{and} \qquad
			r_1(\lambda,\mu) = -\lambda^3 + \frac{19}{4}\lambda^2\mu - 3\lambda\mu^2 + \frac{1}{2}\mu^3 - 2\lambda^2 + 2\lambda\mu - \frac{1}{2}\mu^2.
		\end{align*}
		Thus
		\begin{align*}
			R(\lambda,\mu) = \lambda^2 r_0(\lambda,\mu) + r_1(\lambda,\mu) = -2\lambda^3 + \frac{11}{2}\lambda^2\mu - 3\lambda\mu^2 + \frac{1}{2}\mu^3 - 2\lambda^2 + 2\lambda\mu - \frac{1}{2}\mu^2.
		\end{align*}
		Writing $\mu = t\lambda$ with $t \geq \frac{1}{2}$,
		\begin{align*}
			R(\lambda,\mu)
			= \lambda^3\Big[\frac{1}{2}t^3 - \Big(3+\frac{1}{2\lambda}\Big)t^2 + \Big(\frac{11}{2} + \frac{2}{\lambda}\Big)t - \Big(2 + \frac{2}{\lambda}\Big)\Big]
			= \lambda^3\Big[\frac{1}{2}t^3 - 3t^2 + \frac{11}{2}t - 2 + O(\lambda^{-1}t^2)\Big].
		\end{align*}
		The maximum (in fact the only) real root of the polynomial
		\begin{align*}
			\frac{1}{2}t^3 - 3t^2 + \frac{11}{2}t - 2
		\end{align*}
		is $\approx 0.48$, which is less than $c=\frac{1}{2}$. Since $\lambda \gg 1$ and $t \geq \frac{1}{2}$, it follows that $R(\lambda,\mu) \gtrsim \lambda^3t^3 = \mu^3$.
	\end{proof}
	
	\begin{rem}[Lemma~\ref{lem:R_def} is robust]
		The smaller $c$ is, the harder it is to find $N,a_0,\dots,a_N$ satisfying the property desired in Lemma~\ref{lem:R_def}.
		Nevertheless, crude numerical experiments show that one can take $c$ much smaller than $1$ and still find satisfactory $N,a_0,\dots,a_N$.
		For example, one can check that $c = \frac{1}{37}$ works, with
		\begin{align*}
			N=9,
			\qquad
			a_0=a_2=a_3=a_4=a_5=a_6=a_7=a_9=2,
			\qquad
			a_1 = 1,
			\qquad
			a_8 = 0.
		\end{align*}
	\end{rem}
	
	
	\begin{cor} \label{cor:mu^s_bdd_by_R}
		Let $N,a_0,\dots,a_N,c$ be absolute constants as in Lemma~\ref{lem:R_def}. Define $R(\lambda,\mu)$ as in Lemma~\ref{lem:R_def}. Then for $\lambda \gg 1$ and $\mu \geq 0$,
		\begin{align} \label{eqn:mu^s_bdd_by_R}
			(\mu+1)^{2N+1}
			\lesssim \lambda^{O(1)}\1_{\mu \leq c\lambda} + R(\lambda,\mu).
		\end{align}
		In particular, for $\lambda \gg 1$ and $\mu \geq 0$,
		\begin{align} \label{eqn:1_bdd_by_R}
			1 \lesssim \lambda^{O(1)}\1_{\mu \leq c\lambda} + R(\lambda,\mu).
		\end{align}
	\end{cor}
	
	In this section, we will only use \eqref{eqn:1_bdd_by_R}, but in the next section, we will need \eqref{eqn:mu^s_bdd_by_R}.
	
	\begin{proof}
		For $\mu \leq c\lambda$, we have the trivial bounds
		\begin{align*}
			(\mu+1)^{2N+1} \lesssim \lambda^{O(1)}
			\qquad \text{and} \qquad
			|R(\lambda,\mu)|
			\lesssim \lambda^{O(1)}.
		\end{align*}
		Thus \eqref{eqn:mu^s_bdd_by_R} holds when $\mu \leq c\lambda$. When $\mu > c\lambda$, \eqref{eqn:mu^s_bdd_by_R} follows from \eqref{eqn:R_asymp}. Lastly, \eqref{eqn:mu^s_bdd_by_R} obviously implies \eqref{eqn:1_bdd_by_R}.
	\end{proof}
	
	We are now ready to prove Proposition~\ref{prop:C_s(Lambda)_inductive_bd}.
	
	\begin{proof}[Proof of Proposition~\ref{prop:C_s(Lambda)_inductive_bd}]
		Fix absolute constants $N,a_0,\dots,a_N,c$ as in Lemma~\ref{lem:R_def}, and define $R(\lambda,\mu)$ as in Lemma~\ref{lem:R_def}. For example, as is shown in the proof of Lemma~\ref{lem:R_def}, we can take $N = 1$, $a_0 = a_1 = 1$, and $c = \frac{1}{2}$. Let $\Lambda \gg 1$ be arbitrary. Then we wish to show \eqref{eqn:C_s(Lambda)_inductive_bd}.
		
		By the definition of $C(\Lambda)$ and the fact that $C(\Lambda)$ is finite (Lemma~\ref{lem:C(Lambda)_finite}), there exists a nonzero automorphic vector $\varphi \in \H_{\R}^K \cap \H^{\fin}$ with Casimir eigenvalue $\lambda \leq \Lambda$, such that
		\begin{align*}
			\|\varphi\|_{L^4}
			\gtrsim C(\Lambda) \|\varphi\|_{\H}.
		\end{align*}
		We may assume $\lambda > c\Lambda$, or else \eqref{eqn:C_s(Lambda)_inductive_bd} follows immediately.
		Thus in particular $\lambda \gg 1$.
		By normalizing $\varphi$, we may additionally assume $\|\varphi\|_{\H} = 1$ and so $\|\varphi\|_{L^4} \gtrsim C(\Lambda)$.
		Then
		\begin{align*}
			C(\Lambda)^4
			\lesssim \|\varphi\|_{L^4}^4
			= \|\varphi^2\|_{\H}^2.
		\end{align*}
		Since $1 \ll \lambda \leq \Lambda$, we deduce from \eqref{eqn:1_bdd_by_R} in Corollary~\ref{cor:mu^s_bdd_by_R} that
		\begin{align*} 
			C(\Lambda)^4
			\lesssim \Lambda^{O(1)} \langle \1_{\Delta \leq c\Lambda}(\varphi^2), \varphi^2 \rangle_{\H}
			+ \langle R(\lambda,\Delta)(\varphi^2), \varphi^2 \rangle_{\H}.
		\end{align*}
		By Proposition~\ref{prop:basic_ineq_3} and the nonnegativity of the coefficients $a_n$ in the definition of $R$ in Lemma~\ref{lem:R_def},
		\begin{align*}
			\langle R(\lambda,\Delta)(\varphi^2), \varphi^2 \rangle_{\H} \leq 0,
		\end{align*}
		so
		\begin{align} \label{eqn:C_s^4_to_bd}
			C(\Lambda)^4
			\lesssim \Lambda^{O(1)} \langle \1_{\Delta \leq c\Lambda}(\varphi^2), \varphi^2 \rangle_{\H}.
		\end{align}
		By \eqref{eqn:3_term_crossing} and the fact that $\varphi \in \H_{\R}$, Cauchy--Schwarz, the normalization $\|\varphi\|_{\H} = 1$, and $L^4$-Cauchy--Schwarz (Proposition~\ref{prop:L^4_C-S}),
		\begin{align*}
			\langle \1_{\Delta \leq c\Lambda}(\varphi^2), \varphi^2 \rangle_{\H}
			= \langle \varphi \1_{\Delta \leq c\Lambda}(\varphi^2), \varphi \rangle_{\H}
			\leq \|\varphi \1_{\Delta \leq c\Lambda}(\varphi^2)\|_{\H}
			\leq \|\1_{\Delta \leq c\Lambda}(\varphi^2)\|_{L^4} \|\varphi\|_{L^4}.
		\end{align*}
		Using Lemma~\ref{lem:L^4_bdd_by_C_s(Lambda)} to control $\|\1_{\Delta \leq c\Lambda}(\varphi^2)\|_{L^4}$,
		\begin{align*}
			\langle \1_{\Delta \leq c\Lambda}(\varphi^2), \varphi^2 \rangle_{\H}
			\lesssim \Lambda^{O(1)} C(c\Lambda) \|\1_{\Delta \leq c\Lambda}(\varphi^2)\|_{\H} \|\varphi\|_{L^4}
			\leq \Lambda^{O(1)} C(c\Lambda) \|\varphi\|_{L^4}^3.
		\end{align*}
		Using the definition of $C(\Lambda)$ and the normalization $\|\varphi\|_{\H} = 1$ to estimate $\|\varphi\|_{L^4}$,
		\begin{align*} 
			\langle \1_{\Delta \leq c\Lambda}(\varphi^2), \varphi^2 \rangle_{\H}
			\lesssim \Lambda^{O(1)} C(c\Lambda) C(\Lambda)^3.
		\end{align*}
		Inserting this into \eqref{eqn:C_s^4_to_bd},
		\begin{align*}
			C(\Lambda)^4
			\lesssim \Lambda^{O(1)} C(c\Lambda) C(\Lambda)^3.
		\end{align*}
		Since $C(\Lambda)$ is finite by Lemma~\ref{lem:C(Lambda)_finite}, we can cancel factors of $C(\Lambda)^3$ on both sides to obtain the desired bound \eqref{eqn:C_s(Lambda)_inductive_bd}.
	\end{proof}
	
	The proof of Theorem~\ref{thm:L^4_quasi-Sobolev_poly_Weyl} is finally complete.
	
	\subsection{First tail bound} \label{subsec:bulk:first_tail_bd}
	
	In this subsection we prove Theorem~\ref{thm:exp_decay_form}. We recall the statement below, but we need some preliminary results first.
	
	\begin{lem} \label{lem:q_k,n_sign}
		Let $k \geq 1$, $n \in \Z_{\geq 0}$, and $\mu \gg k^2+n^2$. Then
		\begin{align} \label{eqn:q_k,n_sign}
			(-1)^n q_{k,n}(\mu)
			\geq (0.9\mu)^n.
		\end{align}
	\end{lem}
	
	\begin{proof}
		When $n=0$, we have \eqref{eqn:q_k,n_sign} because $q_{k,0} = 1$. For $n \geq 1$, using the definition of $q_{k,n}$ in Proposition~\ref{prop:q_k,n_def} and the fact that $\mu \gg k^2+n^2$, induction on $n$ gives
		\begin{align*}
			(-1)^n q_{k,n}(\mu)
			\geq 0.9 \mu (-1)^{n-1} q_{k,n-1}(\mu)
			\geq 0.
		\end{align*}
		Iterating this yields \eqref{eqn:q_k,n_sign}.
	\end{proof}
	
	For the remainder of this subsection, fix a positive constant $A \lesssim 1$ such that \eqref{eqn:q_k,n_sign} holds for all $k \geq 1$, $n \in \Z_{\geq 0}$, and $\mu \geq A(k^2+n^2)$.
	
	For $k \geq 1$, $n \in \Z_{\geq 0}$, and $M \geq 0$, let $\|\cdot\|_{k,n,M}$ be the seminorm on $\H^K$ given by
	\begin{align*}
		\|v\|_{k,n,M}
		= \|\1_{\Delta \geq M}|q_{k,n}q_{k,n+1}(\Delta)|^{\frac{1}{2}} v\|_{\H}.
	\end{align*}
	This is evidently decreasing in $M$.
	By Proposition~\ref{prop:smooth_iff_spectral_decay}, it is finite-valued on $\H^K \cap \H^{\infty}$.
	
	\begin{prop} \label{prop:|f|^2_k,n,M_norm}
		Let $f \in \H^{\fin}$ be a lowest weight vector of weight $k \geq 1$. Let $n \in \Z_{\geq 0}$ and $M \geq A(k^2+(n+1)^2)$. Then
		\begin{align*}
			\||f|^2\|_{k,n,M}
			\leq O(k^2+n^2)^{n+\frac{1}{2}} \|f\|_{L^4}^2.
		\end{align*}
	\end{prop}
	
	\begin{proof}
		Since $\|\cdot\|_{k,n,M}$ decreases with $M$, we may assume $M = A(k^2+(n+1)^2)$.
		By the definition of $\|\cdot\|_{k,n,M}$ and self-adjointness of $\Delta$,
		\begin{align*}
			\||f|^2\|_{k,n,M}^2
			= \langle \1_{\Delta \geq M} |q_{k,n}q_{k,n+1}(\Delta)|(|f|^2), |f|^2 \rangle_{\H}.
		\end{align*}
		In view of the definitions of $A$ and $M$, Lemma~\ref{lem:q_k,n_sign} implies that
		\begin{align*}
			-q_{k,n}q_{k,n+1}(\mu) \geq 0
		\end{align*}
		for $\mu \geq M$, so the above becomes
		\begin{align*}
			\||f|^2\|_{k,n,M}^2
			= -\langle \1_{\Delta \geq M} q_{k,n}q_{k,n+1}(\Delta)(|f|^2), |f|^2 \rangle_{\H}.
		\end{align*}
		Writing $-\1_{\Delta \geq M} = \1_{\Delta<M} - \1$,
		\begin{align*}
			\||f|^2\|_{k,n,M}^2
			= \langle \1_{\Delta < M} q_{k,n}q_{k,n+1}(\Delta)(|f|^2), |f|^2 \rangle_{\H}
			- \langle q_{k,n}q_{k,n+1}(\Delta)(|f|^2), |f|^2 \rangle_{\H}.
		\end{align*}
		By Proposition~\ref{prop:basic_ineq_2}, the second term on the right (including the minus sign) is non-positive, so
		\begin{align*}
			\||f|^2\|_{k,n,M}^2
			\leq \langle \1_{\Delta < M} q_{k,n}q_{k,n+1}(\Delta)(|f|^2), |f|^2 \rangle_{\H}.
		\end{align*}
		Estimating this by the trivial bound Lemma~\ref{lem:q_k,n_triv_bd} and using that $M \lesssim k^2+n^2$,
		\begin{align*}
			\||f|^2\|_{k,n,M}^2
			\leq O(k^2+n^2)^{2n+1} \||f|^2\|_{\H}^2.
		\end{align*}
		Since $\||f|^2\|_{\H}^2 = \|f\|_{L^4}^4$ by \eqref{eqn:|alpha|^2_H_norm}, taking square roots gives the desired bound.
	\end{proof}
	
	\begin{prop} \label{prop:exp_decay_form_old}
		There is a positive constant $c \gtrsim 1$, such that whenever $f \in \H^{\fin}$ is a lowest weight vector of weight $k \geq 1$, and whenever $M \gg k^2$,
		\begin{align*} 
			\|\1_{\Delta \geq M}(|f|^2)\|_{\H}
			\leq \exp(-c\sqrt{M}) \|f\|_{L^4}^2.
		\end{align*}
	\end{prop}
	
	\begin{proof}
		Let $A' \gg 1$, to be chosen later. Assume in particular $A' \geq A$. Let $n \in \Z_{\geq 0}$ be such that
		\begin{align*}
			n^2 \gtrsim_{A'} M \geq A'(k^2+(n+1)^2)
		\end{align*}
		(such an $n$ exists because $M \gg k^2$). Then by Lemma~\ref{lem:q_k,n_sign} and Proposition~\ref{prop:|f|^2_k,n,M_norm},
		\begin{align*}
			\|\1_{\Delta \geq M}(|f|^2)\|_{\H}
			\leq (0.9M)^{-n-\frac{1}{2}} \||f|^2\|_{k,n,M}
			\leq O\Big(\frac{k^2+n^2}{M}\Big)^{n+\frac{1}{2}} \|f\|_{L^4}^2
			\leq O(1/A')^{n+\frac{1}{2}} \|f\|_{L^4}^2.
		\end{align*}
		Take $A'$ to be twice the $O$-constant. Then
		\begin{align*}
			\|\1_{\Delta \geq M}(|f|^2)\|_{\H}
			\leq 2^{-n} \|f\|_{L^4}^2.
		\end{align*}
		By the definition of $n$, we have $n^2 \gtrsim_{A'} M$, and now that $A'$ is fixed, we can drop the dependence on $A'$.
		So we can let $c' \gtrsim 1$
		be such that $n^2 \geq c'M$. Then
		\begin{align*}
			\|\1_{\Delta \geq M}(|f|^2)\|_{\H}
			\leq \exp(-(\sqrt{c'} \log 2)\sqrt{M}) \|f\|_{L^4}^2.
		\end{align*}
		Thus the desired bound holds with $c = \sqrt{c'} \log 2$.
		%
		%
		%
	\end{proof}
	
	We now obtain Theorem~\ref{thm:exp_decay_form} by dyadic decomposition. Recall the statement of the theorem:
	
	\begin{thm*}[Restatement of Theorem~\ref{thm:exp_decay_form}]
		There is a positive constant $c \gtrsim 1$, such that whenever $f \in \H^{\fin}$ is a lowest weight vector,
		\begin{align*}
			\|\exp(c\sqrt{\Delta})(|f|^2)\|_{\H} < \infty.
		\end{align*}
	\end{thm*}
	
	
	\begin{proof}
		Let $c' \gtrsim 1$ be such that one can take $c=c'$ in Proposition~\ref{prop:exp_decay_form_old}. Fix $c = \tfrac{1}{2}c'$.
		Let $k \geq 1$ be the weight of $f$.
		Then by dyadic decomposition,
		\begin{align*}
			\|\exp(c\sqrt{\Delta})(|f|^2)\|_{\H}
			\leq \|\1_{\Delta \lesssim k^2} \exp(c\sqrt{\Delta})(|f|^2)\|_{\H}
			+ \sum_{\substack{M \in 2^{\mathbf{N}} \\ M \gg k^2}} \|\1_{\Delta \in [M,2M)} \exp(c\sqrt{\Delta})(|f|^2)\|_{\H}.
		\end{align*}
		The first term on the right hand side is finite because $\1_{\Delta \lesssim k^2} \exp(c\sqrt{\Delta})$ is a bounded operator, so it remains to show that the infinite series is finite. Indeed, by Proposition~\ref{prop:exp_decay_form_old},
		\begin{align}
			\sum_{\substack{M \in 2^{\mathbf{N}} \\ M \gg k^2}} \|\1_{\Delta \in [M,2M)} \exp(c\sqrt{\Delta})(|f|^2)\|_{\H}
			&\leq \sum_{\substack{M \in 2^{\mathbf{N}} \\ M \gg k^2}} \exp(c\sqrt{2M}) \|\1_{\Delta \geq M} (|f|^2)\|_{\H}
			\notag
			\\&\leq \|f\|_{L^4}^2 \sum_{\substack{M \in 2^{\mathbf{N}} \\ M \gg k^2}} \exp(c\sqrt{2M}-c'\sqrt{M}),
			\label{eqn:dyadic_exmp}
		\end{align}
		and this is finite because $c\sqrt{2} < c'$.
	\end{proof}
	
	
	
	\subsection{Second tail bound} Here we prove Theorem~\ref{thm:exp_decay_quasimode} assuming $\H$ obeys a polynomial Weyl law. The structure of the proof is very similar to that of Theorem~\ref{thm:exp_decay_form}, so it would be reasonable to leave some of the intermediate results to the reader. However, this would not save a huge amount of space, so we err on the side of giving full details.
	
	\begin{lem} \label{lem:p_n_sign}
		Let $n \in \Z_{\geq 0}$, let $\lambda \geq 0$, and let $\mu \gg \lambda+n^2$. Then
		\begin{align} \label{eqn:p_n_sign}
			(-1)^n p_n(\lambda,\mu)
			\geq \tfrac{1}{2}(0.9\mu)^n.
		\end{align}
	\end{lem}
	
	\begin{proof}
		The cases $n=0$ and $n=1$ of \eqref{eqn:p_n_sign} can be checked directly from the definition of $p_n$ in Proposition~\ref{prop:p_n_def}. For $n > 1$, using the recurrence \eqref{eqn:p_n_recurrence} and the fact that $\mu \gg \lambda+n^2$, induction on $n$ gives
		\begin{align*}
			(-1)^n p_n(\lambda,\mu)
			\geq 0.9\mu (-1)^{n-1} p_{n-1}(\lambda,\mu)
			\geq 0.
		\end{align*}
		Iterating this gives \eqref{eqn:p_n_sign}.
	\end{proof}
	
	For the remainder of this subsection, fix a positive constant $A \lesssim 1$ such that \eqref{eqn:p_n_sign} holds for all $n \in \Z_{\geq 0}$, $\lambda \geq 0$, and $\mu \geq A(\lambda+n^2)$.
	
	For $n \in \Z_{\geq 0}$ and $\lambda,M \geq 0$, let $\|\cdot\|_{n,\lambda,M}$ be the seminorm on $\H^K$ given by
	\begin{align} \label{eqn:norm_n,lambda,M_def}
		\|v\|_{n,\lambda,M} 
		= \|\1_{\Delta \geq M}|p_np_{n+1}(\lambda,\Delta)|^{\frac{1}{2}} v\|_{\H}.
	\end{align}
	This is decreasing in $M$ and finite-valued on $\H^K \cap \H^{\infty}$.
	
	\begin{prop} \label{prop:phi^2_n,lambda,M_norm}
		Let $\varphi \in \H_{\R}^K \cap \H^{\fin}$ be an automorphic vector with Casimir eigenvalue $\lambda \geq 0$. Let $n \in \Z_{\geq 0}$ and $M \geq A(\lambda+(n+1)^2)$. Then
		\begin{align*}
			\|\varphi^2\|_{n,\lambda,M}
			\leq O(\lambda+n^2+1)^{n+\frac{1}{2}} \|\varphi\|_{L^4}^2.
		\end{align*}
	\end{prop}
	
	\begin{proof}
		Since $\|\cdot\|_{n,\lambda,M}$ decreases with $M$, we may assume $M = A(\lambda+(n+1)^2)$.
		By the definition of $\|\cdot\|_{n,\lambda,M}$ and self-adjointness of $\Delta$,
		\begin{align*}
			\|\varphi^2\|_{n,\lambda,M}^2
			= \langle \1_{\Delta \geq M} |p_np_{n+1}(\lambda,\Delta)|(\varphi^2), \varphi^2 \rangle_{\H}.
		\end{align*}
		In view of the definitions of $A$ and $M$, Lemma~\ref{lem:p_n_sign} implies that
		\begin{align*}
			-p_np_{n+1}(\lambda,\mu) \geq 0
		\end{align*}
		for $\mu \geq M$, so the above becomes
		\begin{align*}
			\|\varphi^2\|_{n,\lambda,M}^2
			= -\langle \1_{\Delta \geq M} p_np_{n+1}(\lambda,\Delta)(\varphi^2), \varphi^2 \rangle_{\H}.
		\end{align*}
		Writing $-\1_{\Delta \geq M} = \1_{\Delta<M} - \1$,
		\begin{align*}
			\|\varphi^2\|_{n,\lambda,M}^2
			= \langle \1_{\Delta<M} p_np_{n+1}(\lambda,\Delta)(\varphi^2), \varphi^2 \rangle_{\H}
			- \langle p_np_{n+1}(\lambda,\Delta)(\varphi^2), \varphi^2 \rangle_{\H}.
		\end{align*}
		By Proposition~\ref{prop:basic_ineq_1}, the second term on the right (including the minus sign) is non-positive, so
		\begin{align*}
			\|\varphi^2\|_{n,\lambda,M}^2
			\leq \langle \1_{\Delta<M} p_np_{n+1}(\lambda,\Delta)(\varphi^2), \varphi^2 \rangle_{\H}.
		\end{align*}
		Estimating this by the trivial bound Lemma~\ref{lem:p_n_triv_bd} and using that $M \lesssim \lambda+n^2+1$,
		\begin{align*}
			\|\varphi^2\|_{n,\lambda,M}^2
			\leq O(\lambda+n^2+1)^{2n+1} \|\varphi^2\|_{\H}^2.
		\end{align*}
		Writing $\|\varphi^2\|_{\H}^2 = \|\varphi\|_{L^4}^4$ and taking square roots gives the desired bound.
	\end{proof}
	
	\begin{prop} \label{prop:exp_decay_mode_old}
		There is a positive constant $c \gtrsim 1$, such that whenever $\varphi \in \H_{\R}^K \cap \H^{\fin}$ is an automorphic vector with Casimir eigenvalue $\lambda \geq 0$, and whenever $M \gg \lambda+1$,
		\begin{align} \label{eqn:exp_decay_mode_old_L^4}
			\|\1_{\Delta \geq M}(\varphi^2)\|_{\H}
			\leq \exp(-c\sqrt{M}) \|\varphi\|_{L^4}^2.
		\end{align}
		Furthermore, if $\H$ obeys a polynomial Weyl law, then \eqref{eqn:exp_decay_mode_old_L^4} can be upgraded to
		\begin{align} \label{eqn:exp_decay_mode_old_L^2}
			\|\1_{\Delta \geq M}(\varphi^2)\|_{\H}
			\leq \exp(-c\sqrt{M}) \|\varphi\|_{\H}^2.
		\end{align}
	\end{prop}
	
	
	\begin{proof}
		Let $A' \gg 1$, to be chosen later. Assume in particular that $A' \geq A$.
		Let $n \in \Z_{\geq 0}$ such that
		\begin{align*}
			n^2 \gtrsim_{A'} M \geq A'(\lambda+(n+1)^2)
		\end{align*}
		(such an $n$ exists because $M \gg \lambda+1$). Then by Lemma~\ref{lem:p_n_sign} and Proposition~\ref{prop:phi^2_n,lambda,M_norm},
		\begin{align*}
			\|\1_{\Delta \geq M}(\varphi^2)\|_{\H}
			\leq 2(0.9M)^{-n-\frac{1}{2}} \|\varphi^2\|_{n,\lambda,M}
			\leq O\Big(\frac{\lambda+n^2+1}{M}\Big)^{n+\frac{1}{2}} \|\varphi\|_{L^4}^2
			\leq O(1/A')^{n+\frac{1}{2}} \|\varphi\|_{L^4}^2.
		\end{align*}
		Take $A'$ to be twice the $O$-constant. Then
		\begin{align*}
			\|\1_{\Delta \geq M}(\varphi^2)\|_{\H}
			\leq 2^{-n} \|\varphi\|_{L^4}^2.
		\end{align*}
		By the definition of $n$, we have $n^2 \gtrsim_{A'} M$, and now that $A'$ is fixed, we can drop the dependence on $A'$.
		So we can let $c' \gtrsim 1$
		be such that $n^2 \geq c'M$. Then
		\begin{align*}
			\|\1_{\Delta \geq M}(\varphi^2)\|_{\H}
			\leq \exp(-(\sqrt{c'} \log 2)\sqrt{M}) \|\varphi\|_{L^4}^2.
		\end{align*}
		Thus \eqref{eqn:exp_decay_mode_old_L^4} holds with $c = \sqrt{c'} \log 2$.
		
		Suppose now that $\H$ obeys a polynomial Weyl law. Then by \eqref{eqn:exp_decay_mode_old_L^4} and Theorem~\ref{thm:L^4_quasi-Sobolev_poly_Weyl},
		\begin{align*}
			\|\1_{\Delta \geq M}(\varphi^2)\|_{\H}
			\leq \exp(-c\sqrt{M}) \|\exp(O(\log_+^2\Delta)) \varphi\|_{\H}^2
			= \exp(-c\sqrt{M}+O(\log_+^2\lambda)) \|\varphi\|_{\H}^2.
		\end{align*}
		Since $c \gtrsim 1$ and $M \gg \lambda+1$, we may assume $M$ is large enough that the quantity inside the exponential satisfies
		\begin{align*}
			-c\sqrt{M} + O(\log_+^2\lambda)
			\leq -\tfrac{1}{2}c \sqrt{M}.
		\end{align*}
		Then \eqref{eqn:exp_decay_mode_old_L^2} holds after replacing $c$ with $\tfrac{1}{2}c$.
	\end{proof}
	
	
	\begin{prop} \label{prop:exp_decay_mode_after_dyadic}
		Assume $\H$ obeys a polynomial Weyl law. Then there is a positive constant $c \gtrsim 1$, such that whenever $\varphi \in \H_{\R}^K \cap \H^{\fin}$ is an automorphic vector with Casimir eigenvalue $\lambda \geq 0$, and whenever $M \gg \lambda+1$,
		\begin{align*}
			\|\1_{\Delta \geq M} \exp(c\sqrt{\Delta})(\varphi^2)\|_{\H}
			\leq \|\varphi\|_{\H}^2.
		\end{align*}
	\end{prop}
	
	\begin{proof}
		This follows from a dyadic decomposition argument along the lines of \eqref{eqn:dyadic_exmp}, using \eqref{eqn:exp_decay_mode_old_L^2} in Proposition~\ref{prop:exp_decay_mode_old} instead of Proposition~\ref{prop:exp_decay_form_old}.
	\end{proof}
	
	%
	
	We are finally ready to prove Theorem~\ref{thm:exp_decay_quasimode} assuming $\H$ obeys a polynomial Weyl law. With this assumption, Theorem~\ref{thm:exp_decay_quasimode} becomes
	
	\begin{thm}
		Assume $\H$ obeys a polynomial Weyl law. Then there are positive constants $C \lesssim 1$ and $c \gtrsim 1$, a partition of $[0,\infty)$ into intervals $I_i$ of length $\lesssim 1$, and points $\lambda_i \in I_i$ for each $i$, such that the partition has polynomial growth in the sense that
		\begin{align*} 
			\#\{i : \lambda_i \leq X\} \lesssim X^{O(1)}
			\qquad \text{for} \qquad
			X \geq 1,
		\end{align*}
		and such that for all $i$ and all $\alpha,\beta \in \H_{\Delta \in I_i}^K$,
		\begin{align} \label{eqn:exp_decay_quasimodes_poly_Weyl}
			\|\1_{\Delta \geq C\lambda_i} \exp(c\sqrt{\Delta})(\alpha\beta)\|_{\H}
			\leq \|\alpha\|_{\H} \|\beta\|_{\H}.
		\end{align}
	\end{thm}
	
	\begin{proof}
		Since $\H$ obeys a polynomial Weyl law, there exists a partition of $[0,\infty)$ into intervals $I_i$ of length $\lesssim 1$, such that the partition has polynomial growth, and such that each interval contains at most one eigenvalue of the Casimir on $\H^K$ (not counting multiplicity). One can ensure also that the interval containing $0$ is not the degenerate interval $\{0\}$. Fix such a partition. Let $\lambda_i$ be the midpoint of $I_i$ (for concreteness). If $I_{i_0}$ is the interval containing $0$, then $\lambda_i \geq \lambda_{i_0}$ for all $i$, and $\lambda_{i_0} > 0$ because $I_{i_0}$ is non-degenerate. Since $\lambda_{i_0}$ depends only on $\H$, this means that $\lambda_{i_0} \gtrsim 1$. Thus $\lambda_i \gtrsim 1$ for all $i$.
		
		Now let $I_i$ be an arbitrary interval in the partition.
		If $I_i$ does not contain a Casimir eigenvalue, then $\H_{\Delta \in I_i}^K = 0$ and \eqref{eqn:exp_decay_quasimodes_poly_Weyl} holds vacuously. The only interesting case, then, is that $I_i$ does contain a Casimir eigenvalue. Call this eigenvalue $\lambda_i'$. By construction, $\lambda_i'$ is unique, so $\H_{\Delta \in I_i}^K = \H_{\Delta=\lambda_i'}^K$.
		Let $c \gtrsim 1$ be as in Proposition~\ref{prop:exp_decay_mode_after_dyadic}.
		Since $\lambda_i \gtrsim 1$ and $I_i$ has length $\lesssim 1$, there exists $C \lesssim 1$ (independent of $i$) such that $C\lambda_i$ is a large enough multiple of $\lambda_i'+1$ for Proposition~\ref{prop:exp_decay_mode_after_dyadic} to give
		\begin{align} \label{eqn:after_exp_decay_mode_after_dyadic}
			\|\1_{\Delta \geq C\lambda_i} \exp(c\sqrt{\Delta})(\varphi^2)\|_{\H}
			\leq \|\varphi\|_{\H}^2
		\end{align}
		for all automorphic vectors $\varphi \in \H_{\R}^K \cap \H^{\fin}$ with Casimir eigenvalue $\lambda_i'$, or equivalently for all $\varphi \in \H_{\Delta \in I_i}^K \cap \H_{\R}$. Let $\alpha,\beta \in \H_{\Delta \in I_i}^K$. We claim that
		\begin{align} \label{eqn:claim_via_polarization}
			\|\1_{\Delta \geq C\lambda_i} \exp(c\sqrt{\Delta})(\alpha\beta)\|_{\H}
			\lesssim \|\alpha\|_{\H} \|\beta\|_{\H}.
		\end{align}
		By splitting $\alpha,\beta$ into real and imaginary parts, we may assume $\alpha,\beta \in \H_{\R}$. By normalizing, we may additionally assume $\|\alpha\|_{\H} = \|\beta\|_{\H} = 1$. Then the polarization identity $\alpha\beta = \frac{1}{2}(\alpha+\beta)^2 - \frac{1}{2}\alpha^2 - \frac{1}{2}\beta^2$ combined with \eqref{eqn:after_exp_decay_mode_after_dyadic} gives the claimed estimate \eqref{eqn:claim_via_polarization}. This is almost the same as \eqref{eqn:exp_decay_quasimodes_poly_Weyl}, except with $\lesssim$ instead of $\leq$. To replace $\lesssim$ with $\leq$, we simply increase $C$ and decrease $c$.
	\end{proof}
	
	\section{Proof of bulk and tail bounds in general} \label{sec:bulk_tail_general}
	
	Recall that Sections~\ref{sec:embedding_homog}, \ref{sec:reduce_to_L^infty}, and \ref{sec:reduce_to_bulk_tail} reduced Theorem~\ref{thm:eq_Gelfand_duality} to Theorems~\ref{thm:L^4_quasi-Sobolev}, \ref{thm:exp_decay_form}, and \ref{thm:exp_decay_quasimode}. Section~\ref{sec:bulk_tail_poly} proved Theorems~\ref{thm:L^4_quasi-Sobolev} and \ref{thm:exp_decay_quasimode} under the assumption that $\H$ obeys a polynomial law, and proved Theorem~\ref{thm:exp_decay_form} unconditionally. In this section, we prove Theorems~\ref{thm:L^4_quasi-Sobolev} and \ref{thm:exp_decay_quasimode} unconditionally as well, completing the proof of Theorem~\ref{thm:eq_Gelfand_duality}.
	
	Throughout this section, given $\gamma \in \H^{\fin}$, $i \in \Z_{\geq 0}$, and $\lambda \in \R$, denote $\gamma_i(\lambda) = (\Delta-\lambda)^i\gamma$. It will usually be clear from context what $\lambda$ is, in which case we just write $\gamma_i$. Note that $\gamma_0 = \gamma$, and if $\gamma$ is a Casimir eigenvector with eigenvalue $\lambda$, then $\gamma_i = 0$ for $i>0$. In practice, $\gamma$ will be an approximate Casimir eigenvector with approximate eigenvalue $\lambda$, so $\gamma_i \approx 0$ when $i>0$.
	
	The above notation is used, for example, to define $\alpha_i$ and $\beta_j$ in the right hand side of \eqref{eqn:Re E^n alpha F^n beta_formula} below, with $\lambda$ taken to be the same $\lambda$ as in the statement of Proposition~\ref{prop:p_n,i,j_def}.
	
	
	\subsection{A refined identity from the product rule} \label{subsec:refined_identity}
	
	The main result of this subsection is Proposition~\ref{prop:p_n,i,j_def}, which generalizes Proposition~\ref{prop:p_n_def}. The terms $(i,j) \neq (0,0)$ on the right hand side of \eqref{eqn:Re E^n alpha F^n beta_formula} are the correction terms referred to in Subsection~\ref{subsec:outline:bulk_tail_general}.
	
	\begin{prop} \label{prop:p_n,i,j_def}
		Let $\alpha,\beta \in \H_{\R}^K \cap \H^{\fin}$. Let $\lambda \in \R$. Then for $n \in \Z_{\geq 0}$,
		\begin{align} \label{eqn:Re E^n alpha F^n beta_formula}
			\re E^n\alpha \overline{E}^n\beta
			= \sum_{i,j} p_{n,i,j}(\lambda,\Delta)(\alpha_i\beta_j),
		\end{align}
		where $p_{n,i,j}$, defined for $n \in \Z_{\geq 0}$ and $i,j \in \Z$, is the polynomial in two variables with real coefficients given by the recurrence
		\begin{align} \label{eqn:p_n,i,j_recurrence}
			p_{n+1,i,j}(\lambda,\mu)
			&= [(2\lambda-\mu+2n^2) p_{n,i,j} - (\lambda+n(n-1))^2 p_{n-1,i,j}](\lambda,\mu)
			\\&\quad+ [p_{n,i-1,j} + p_{n,i,j-1}](\lambda,\mu) \notag
			\\&\quad- (\lambda+n(n-1)) [p_{n-1,i-1,j} + p_{n-1,i,j-1}](\lambda,\mu) \notag
			\\&\quad- p_{n-1,i-1,j-1}(\lambda,\mu) \notag
		\end{align}
		for $n \geq 1$, with initial conditions
		\begin{align} \label{eqn:p_0,i,j_def}
			p_{0,i,j}(\lambda,\mu) =
			\begin{cases}
				1 &\textnormal{if } (i,j) = (0,0), \\
				0 &\textnormal{otherwise},
			\end{cases}
		\end{align}
		and
		\begin{align} \label{eqn:p_1,i,j_def}
			p_{1,i,j}(\lambda,\mu) =
			\begin{cases}
				\lambda - \frac{1}{2}\mu &\textnormal{if } (i,j) = (0,0), \\
				\frac{1}{2} &\textnormal{if } (i,j) = (1,0) \textnormal{ or } (0,1), \\
				0 &\textnormal{otherwise.}
			\end{cases}
		\end{align}
	\end{prop}
	
	Note by induction on $n$ that
	\begin{align} \label{eqn:p_n,i,j_vanishing}
		p_{n,i,j} = 0
		\qquad \text{if} \qquad
		i < 0 \text{ or } j < 0 \text{ or } i+j>n.
	\end{align}
	Thus the right hand side of \eqref{eqn:Re E^n alpha F^n beta_formula} is a finite sum over terms indexed by nonnegative integers $i,j$; this nonnegativity means that it makes sense to write $\alpha_i,\beta_j$. The right hand side of \eqref{eqn:Re E^n alpha F^n beta_formula} is also real, so taking the real part on the left is necessary.
	
	Using \eqref{eqn:p_n,i,j_vanishing} and comparing the recurrences for $p_{n,i,j}$ and $p_n$ (where as in previous sections, $p_n$ is as in Proposition~\ref{prop:p_n_def}), we see that
	\begin{align*}
		p_{n,0,0} = p_n.
	\end{align*}
	It is then clear that Proposition~\ref{prop:p_n,i,j_def} is a direct generalization of Proposition~\ref{prop:p_n_def}.
	
	\begin{proof}[Proof of Proposition~\ref{prop:p_n,i,j_def}]
		Since $\alpha\beta$ is already real, the case $n=0$ of \eqref{eqn:Re E^n alpha F^n beta_formula} is immediate from the definition of $p_{0,i,j}$. By \eqref{eqn:Delta|_H^K} and the product rule,
		\begin{align*}
			\Delta(\alpha\beta)
			= -\overline{E}E(\alpha\beta)
			= -E\alpha \overline{E}\beta - \overline{E}\alpha E\beta - (\overline{E}E\alpha)\beta - \alpha \overline{E}E\beta
			= -2\re E\alpha \overline{E}\beta + (\Delta\alpha)\beta + \alpha(\Delta\beta).
		\end{align*}
		Rearranging,
		\begin{align*}
			\re E\alpha\overline{E}\beta
			= -\frac{1}{2}\Delta(\alpha\beta) + \frac{1}{2}(\Delta\alpha)\beta + \frac{1}{2}\alpha(\Delta\beta).
		\end{align*}
		In two of the three terms here, $\Delta$ falls on either $\alpha$ or $\beta$. Rewriting these two $\Delta$'s as $\lambda + (\Delta-\lambda)$,
		\begin{align*}
			\re E\alpha \overline{E}\beta
			= \Big(\lambda - \frac{1}{2}\Delta\Big) \alpha\beta + \frac{1}{2}(\alpha_1\beta + \alpha\beta_1).
		\end{align*}
		This is \eqref{eqn:Re E^n alpha F^n beta_formula} for $n=1$. For $n \geq 1$, use \eqref{eqn:Delta|_H^K}, \eqref{eqn:EEbar}, and the product rule to similarly compute
		\begin{align*}
			\Delta(E^n\alpha \overline{E}^n\beta)
			&= -\overline{E}E(E^n\alpha \overline{E}^n\beta)
			\\&= -E^{n+1}\alpha \overline{E}^{n+1}\beta - \overline{E}E^n\alpha E\overline{E}^n\beta - \overline{E}E^{n+1}\alpha \overline{E}^n\beta - E^n\alpha \overline{E}E\overline{E}^n\beta
			\\&= -E^{n+1}\alpha \overline{E}^{n+1}\beta - (\Delta+H^2+H) E^{n-1}\alpha (\Delta+H^2-H) \overline{E}^{n-1}\beta
			\\&\qquad\qquad\qquad\qquad + (\Delta+H^2+H) E^n\alpha\overline{E}^n\beta
			+ E^n\alpha (\Delta+H^2+H) \overline{E}^n\beta
			\\&= -E^{n+1}\alpha \overline{E}^{n+1}\beta - E^{n-1}(\Delta+n(n-1))\alpha \overline{E}^{n-1}(\Delta+n(n-1))\beta
			\\&\qquad\qquad\qquad\qquad + E^n\Delta\alpha \overline{E}^n\beta + E^n\alpha \overline{E}^n\Delta\beta
			+ 2n^2 E^n\alpha \overline{E}^n\beta.
		\end{align*}
		Rearranging,
		\begin{align*}
			E^{n+1}\alpha \overline{E}^{n+1}\beta
			= &-\Delta(E^n\alpha \overline{E}^n\beta) + E^n\Delta\alpha \overline{E}^n\beta + E^n\alpha \overline{E}^n\Delta\beta + 2n^2 E^n\alpha \overline{E}^n\beta
			\\&- E^{n-1}(\Delta + n(n-1))\alpha \overline{E}^{n-1}(\Delta + n(n-1))\beta.
		\end{align*}
		Replacing the $\Delta$'s that fall on $\alpha$ or $\beta$ with $\lambda + (\Delta-\lambda)$,
		\begin{align*}
			E^{n+1}\alpha \overline{E}^{n+1}\beta
			&= (2\lambda-\Delta+2n^2)(E^n\alpha \overline{E}^n\beta) - (\lambda+n(n-1))^2 E^{n-1}\alpha \overline{E}^{n-1}\beta
			\\&\quad+ E^n\alpha_1 \overline{E}^n\beta + E^n\alpha \overline{E}^n \beta_1
			\\&\quad- (\lambda+n(n-1))[E^{n-1}\alpha_1 \overline{E}^{n-1}\beta + E^{n-1}\alpha \overline{E}^{n-1}\beta_1]
			\\&\quad- E^{n-1}\alpha_1 \overline{E}^{n-1}\beta_1.
		\end{align*}
		Taking real parts and comparing with \eqref{eqn:p_n,i,j_recurrence}, we obtain \eqref{eqn:Re E^n alpha F^n beta_formula} by induction on $n$.
	\end{proof}
	
	\subsection{Refined bootstrap inequalities} \label{subsec:refined_ineqs}
	
	Recall $r_n$ and $s_n$, defined by \eqref{eqn:r_n_def} and \eqref{eqn:s_n_def}. Just as $p_{n,i,j}$ generalizes $p_n$, let $s_{n,i,j}$ be the following generalization of $s_n$. For $n \in \Z_{\geq 0}$ and $i,j \in \Z$, define
	\begin{align} \label{eqn:s_n,i,j_def}
		s_{n,i,j}(\lambda,\mu)
		= [p_{n+1,i,j} - (\lambda+n(n+1)) p_{n,i,j} - p_{n,i-1,j}](\lambda,\mu).
	\end{align}
	By \eqref{eqn:p_n,i,j_vanishing}, we have
	\begin{align} \label{eqn:s_n,i,j_vanishing}
		s_{n,i,j} = 0
		\qquad \text{if} \qquad
		i < 0 \text{ or } j < 0 \text{ or } i+j > n+1
	\end{align}
	and
	\begin{align*}
		s_{n,0,0} = s_n.
	\end{align*}
	We can now state the two ``refined bootstrap inequalities" which we will prove in this subsection. They are Propositions~\ref{prop:refined_ineq_1} and \ref{prop:refined_ineq_2}, and they generalize Propositions~\ref{prop:basic_ineq_1} and \ref{prop:basic_ineq_3}, respectively.
	
	\begin{prop} \label{prop:refined_ineq_1}
		Let $\alpha \in \H_{\R}^K \cap \H^{\fin}$. Let $\lambda \in \R$. Then for $n \in \Z_{\geq 0}$,
		\begin{align*}
			-\langle p_np_{n+1}(\lambda,\Delta)(\alpha^2), \alpha^2 \rangle_{\H}
			\leq \sum_{\substack{i,j,i',j' \\ \textnormal{not all }0}} \||p_{n,i,j}p_{n+1,i',j'}(\lambda,\Delta)|^{\frac{1}{2}}(\alpha_i\alpha_j)\|_{\H} \||p_{n,i,j}p_{n+1,i',j'}(\lambda,\Delta)|^{\frac{1}{2}}(\alpha_{i'}\alpha_{j'})\|_{\H}.
		\end{align*}
	\end{prop}
	
	\begin{prop} \label{prop:refined_ineq_2}
		Let $\alpha \in \H_{\R}^K \cap \H^{\fin}$. Let $\lambda \in \R$. Then for $n \in \Z_{\geq 0}$,
		\begin{align*}
			\langle r_n(\lambda,\Delta)(\alpha^2), \alpha^2 \rangle_{\H}
			\leq \sum_{\substack{i,j,i',j' \\ \textnormal{not all }0}} \Big(&\||p_{n,i,j}p_{n+1,i',j'}(\lambda,\Delta)|^{\frac{1}{2}}(\alpha_i\alpha_j)\|_{\H} \||p_{n,i,j}p_{n+1,i',j'}(\lambda,\Delta)|^{\frac{1}{2}}(\alpha_{i'}\alpha_{j'})\|_{\H}
			\\&+ \|\1_{\Delta\geq 1} \Delta^{-\frac{1}{2}} s_{n,i,j}(\lambda,\Delta)(\alpha_i\alpha_j)\|_{\H} \|\1_{\Delta\geq 1} \Delta^{-\frac{1}{2}} s_{n,i',j'}(\lambda,\Delta)(\alpha_{i'}\alpha_{j'})\|_{\H}\Big).
		\end{align*}
	\end{prop}
	
	Because of \eqref{eqn:p_n,i,j_vanishing} and \eqref{eqn:s_n,i,j_vanishing}, the right hand sides are finite sums.
	If $\Delta\alpha = \lambda\alpha$, then the right hand sides vanish.
	Thus Propositions~\ref{prop:refined_ineq_1} and \ref{prop:refined_ineq_2} indeed generalize Propositions~\ref{prop:basic_ineq_1} and \ref{prop:basic_ineq_3}.
	A schematic form of the proof of these inequalities is as follows.
	
	\begin{proof}[Proof of Propositions~\ref{prop:refined_ineq_1} and \ref{prop:refined_ineq_2}]
		As in the proofs of Propositions~\ref{prop:basic_ineq_1} and \ref{prop:basic_ineq_3}, consider the crossing equation
		\begin{align} \label{eqn:fund_crossing}
			\langle |E^n\alpha|^2, |E^{n+1}\alpha|^2 \rangle_{\H}
			= \|E^{n+1}\alpha \overline{E}^n\alpha\|_{\H}^2.
		\end{align}
		In Propositions~\ref{prop:LHS=MT+RT} and \ref{prop:RHS_geq_MT+RT} below, we will write
		\begin{align} \label{eqn:MT+RT}
			\LHS\eqref{eqn:fund_crossing}
			= \MT_{\LHS} + \RT_{\LHS}
			\qquad \text{and} \qquad
			\RHS\eqref{eqn:fund_crossing}
			\geq \MT_{\RHS} + \RT_{\RHS}
		\end{align}
		(here MT stands for ``main term" and RT stands for ``remainder term").
		In Propositions~\ref{prop:LHS_RT_leq_Error} and \ref{prop:RHS_RT_leq_Error} below, we will then bound
		\begin{align} \label{eqn:RT_leq_Error}
			|\RT_{\LHS}|
			\leq \Error_{\LHS}
			\qquad \text{and} \qquad
			|\RT_{\RHS}| \leq \Error_{\RHS}.
		\end{align}
		Since $\RHS\eqref{eqn:fund_crossing}$ is a square and hence nonnegative, we have $\LHS\eqref{eqn:fund_crossing} \geq 0$, so
		\begin{align} \label{eqn:schematic_fund_ineq_1}
			-\MT_{\LHS}
			\leq \RT_{\LHS}
			\leq \Error_{\LHS}.
		\end{align}
		Similarly, combining \eqref{eqn:fund_crossing}, \eqref{eqn:MT+RT}, and \eqref{eqn:RT_leq_Error},
		\begin{align} \label{eqn:schematic_fund_ineq_2}
			\MT_{\RHS}-\MT_{\LHS}
			\leq \RT_{\LHS} - \RT_{\RHS}
			\leq \Error_{\LHS} + \Error_{\RHS}.
		\end{align}
		After inserting the definitions of these quantities from the propositions mentioned above, \eqref{eqn:schematic_fund_ineq_1} and \eqref{eqn:schematic_fund_ineq_2} become the inequalities in Propositions~\ref{prop:refined_ineq_1} and \ref{prop:refined_ineq_2}, respectively.
	\end{proof}
	
	\begin{prop} \label{prop:LHS=MT+RT}
		Let $\alpha \in \H_{\R}^K \cap \H^{\fin}$. Let $\lambda \in \R$. Then for $n \in \Z_{\geq 0}$,
		\begin{align*}
			\LHS\eqref{eqn:fund_crossing}
			= \MT_{\LHS} + \RT_{\LHS},
		\end{align*}
		where
		\begin{align*}
			\MT_{\LHS}
			= \langle p_n p_{n+1}(\lambda,\Delta)(\alpha^2), \alpha^2 \rangle_{\H}
		\end{align*}
		and
		\begin{align*}
			\RT_{\LHS}
			= \sum_{\substack{i,j,i',j' \\ \textnormal{not all }0}} \langle p_{n,i,j}(\lambda,\Delta)(\alpha_i\alpha_j), p_{n+1,i',j'}(\lambda,\Delta)(\alpha_{i'} \alpha_{j'}) \rangle_{\H}.
		\end{align*}
	\end{prop}
	
	\begin{proof}
		Expanding $|E^n\alpha|^2$ and $|E^{n+1}\alpha|^2$ as in Proposition~\ref{prop:p_n,i,j_def},
		\begin{align*}
			\LHS\eqref{eqn:fund_crossing}
			= \sum_{i,j,i',j'} \langle p_{n,i,j}(\lambda,\Delta)(\alpha_i\alpha_j), p_{n+1,i',j'}(\lambda,\Delta)(\alpha_{i'} \alpha_{j'}) \rangle_{\H}.
		\end{align*}
		The term where $i=j=i'=j'=0$ is $\MT_{\LHS}$, and the remainder is $\RT_{\LHS}$.
	\end{proof}
	
	\begin{prop} \label{prop:RHS_geq_MT+RT}
		Let $\alpha \in \H_{\R}^K \cap \H^{\fin}$. Let $\lambda \in \R$. Then for $n \in \Z_{\geq 0}$,
		\begin{align} \label{eqn:RHS_geq_MT+RT}
			\RHS\eqref{eqn:fund_crossing}
			\geq \MT_{\RHS} + \RT_{\RHS},
		\end{align}
		where
		\begin{align*}
			\MT_{\RHS}
			= \langle \1_{\Delta \geq 1} \Delta^{-1} s_n(\lambda,\Delta)^2(\alpha^2), \alpha^2 \rangle_{\H}
		\end{align*}
		and
		\begin{align*}
			\RT_{\RHS}
			= \sum_{\substack{i,j,i',j' \\ \textnormal{not all }0}} \langle \1_{\Delta \geq 1} \Delta^{-\frac{1}{2}} s_{n,i,j}(\lambda,\Delta)(\alpha_i\alpha_j), \1_{\Delta \geq 1} \Delta^{-\frac{1}{2}} s_{n,i',j'}(\lambda,\Delta)(\alpha_{i'}\alpha_{j'}) \rangle_{\H}.
		\end{align*}
	\end{prop}
	
	\begin{proof}[Proof of Proposition~\ref{prop:RHS_geq_MT+RT}]
		By Lemma~\ref{lem:wt_1_lower_bd}, we can lower bound
		\begin{align*}
			\RHS\eqref{eqn:fund_crossing}
			\geq \|\1_{\Delta\geq 1} \Delta^{-\frac{1}{2}} \overline{E}(E^{n+1}\alpha \overline{E}^n\alpha)\|_{\H}^2.
		\end{align*}
		Applying the product rule and \eqref{eqn:EEbar},
		\begin{align*}
			\RHS\eqref{eqn:fund_crossing}
			\geq \|\1_{\Delta\geq 1} \Delta^{-\frac{1}{2}}(|E^{n+1}\alpha|^2 - (\Delta+H^2+H)E^n\alpha \overline{E}^n\alpha)\|_{\H}^2.
		\end{align*}
		Splitting $\Delta = \lambda + (\Delta-\lambda)$, and using that $E^n\alpha$ has weight $n$,
		\begin{align*}
			\RHS\eqref{eqn:fund_crossing}
			\geq \|\1_{\Delta \geq 1} \Delta^{-\frac{1}{2}}(|E^{n+1}\alpha|^2 - (\lambda+n(n+1))|E^n\alpha|^2 - E^n\alpha_1 \overline{E}^n\alpha)\|_{\H}^2.
		\end{align*}
		Applying the trivial inequality $\|v\|_{\H} \geq \|\re v\|_{\H}$ which holds for all $v \in \H$ by Proposition~\ref{prop:norm_re_im_pythagoras},
		\begin{align*}
			\RHS\eqref{eqn:fund_crossing}
			\geq \|\1_{\Delta \geq 1} \Delta^{-\frac{1}{2}}(|E^{n+1}\alpha|^2 - (\lambda+n(n+1))|E^n\alpha|^2 - \re E^n\alpha_1 \overline{E}^n\alpha)\|_{\H}^2.
		\end{align*}
		Expanding the three terms on the right hand side as in Proposition~\ref{prop:p_n,i,j_def}, and then writing the result in terms of the polynomials $s_{n,i,j}$ to simplify notation,
		\begin{align*}
			\RHS\eqref{eqn:fund_crossing}
			\geq \Big\|\1_{\Delta \geq 1}\Delta^{-\frac{1}{2}} \sum_{i,j} s_{n,i,j}(\lambda,\Delta)(\alpha_i\alpha_j)\Big\|_{\H}^2.
		\end{align*}
		Expanding the square,
		\begin{align*}
			\RHS\eqref{eqn:fund_crossing}
			\geq \sum_{i,j,i',j'} \langle \1_{\Delta \geq 1} \Delta^{-\frac{1}{2}} s_{n,i,j}(\lambda,\Delta)(\alpha_i\alpha_j), \1_{\Delta \geq 1} \Delta^{-\frac{1}{2}} s_{n,i',j'}(\lambda,\Delta)(\alpha_{i'}\alpha_{j'}) \rangle_{\H}.
		\end{align*}
		The term where $i=j=i'=j'=0$ is $\MT_{\RHS}$, and the remainder is $\RT_{\RHS}$.
	\end{proof}
	
	\begin{prop} \label{prop:LHS_RT_leq_Error}
		With notation as in Proposition~\ref{prop:LHS=MT+RT}, one has
		\begin{align*}
			|\RT_{\LHS}|
			\leq \Error_{\LHS},
		\end{align*}
		where
		\begin{align*}
			\Error_{\LHS}
			= \sum_{\substack{i,j,i',j' \\ \textnormal{not all }0}} \||p_{n,i,j}p_{n+1,i',j'}(\lambda,\Delta)|^{\frac{1}{2}}(\alpha_i\alpha_j)\|_{\H} \||p_{n,i,j}p_{n+1,i',j'}(\lambda,\Delta)|^{\frac{1}{2}}(\alpha_{i'}\alpha_{j'})\|_{\H}.
		\end{align*}
	\end{prop}
	
	\begin{proof}
		This is immediate from the definition of $\RT_{\LHS}$ and the general Lemma~\ref{lem:operator_C-S} below.
	\end{proof}
	
	\begin{lem} \label{lem:operator_C-S}
		Let $A,B$ commuting self-adjoint operators on a Hilbert space $\V$. Let $v,w \in \V$. Then
		\begin{align*}
			|\langle Av, Bw \rangle_{\V}|
			\leq \||AB|^{\frac{1}{2}}v\|_{\V} \||AB|^{\frac{1}{2}}w\|_{\V}.
		\end{align*}
	\end{lem}
	
	Here $|AB|^{\frac{1}{2}}$ is defined using the functional calculus for $AB$, which is self-adjoint because $A,B$ commute.
	We are not careful with issues of domain, because in practice $v,w$ will always be in the domains of the relevant operators.
	
	\begin{proof}
		Write
		\begin{align*}
			\langle Av, Bw \rangle_{\V}
			= \langle ABv,w \rangle_{\V}
			= \langle \sgn(AB)|AB|v,w \rangle_{\V}
			= \langle \sgn(AB)|AB|^{\frac{1}{2}}v, |AB|^{\frac{1}{2}}w \rangle_{\V}.
		\end{align*}
		The desired bound now follows from Cauchy--Schwarz and the fact that $\sgn(AB)$ is unitary.
	\end{proof}
	
	\begin{prop} \label{prop:RHS_RT_leq_Error}
		With notation as in Proposition~\ref{prop:RHS_geq_MT+RT}, one has
		\begin{align*}
			|\RT_{\RHS}|
			\leq \Error_{\RHS},
		\end{align*}
		where
		\begin{align*}
			\Error_{\RHS}
			= \sum_{\substack{i,j,i',j' \\ \textnormal{not all }0}} \|\1_{\Delta \geq 1} \Delta^{-\frac{1}{2}} s_{n,i,j}(\lambda,\Delta)(\alpha_i\alpha_j)\|_{\H} \|\1_{\Delta \geq 1} \Delta^{-\frac{1}{2}} s_{n,i',j'}(\lambda,\Delta)(\alpha_{i'}\alpha_{j'})\|_{\H}.
		\end{align*}
	\end{prop}
	
	\begin{proof}
		This is immediate from the definition of $\RT_{\RHS}$ and Cauchy--Schwarz.
	\end{proof}
	
	Propositions~\ref{prop:refined_ineq_1} and \ref{prop:refined_ineq_2} are now fully proved.
	
	\subsection{Bulk bound} \label{subsec:bulk_general}
	
	This subsection is somewhat similar to Subsection~\ref{subsec:bulk_poly}.
	Our goal is to prove Theorem~\ref{thm:L^4_quasi-Sobolev}.
	
	%
	
	For $s \in \R$ and $v \in \H^K$, define the Sobolev norm
	\begin{align*} 
		\|v\|_{H^s}
		= \|(\Delta+1)^{\frac{s}{2}}v\|_{\H}.
	\end{align*}
	In general this may be $+\infty$, but it is always finite when $v \in \H^K \cap \H^{\infty}$. The $H^0$ Sobolev norm is just the Hilbert space norm. For $p,s \geq 0$ and $\Lambda \geq 1$, let $C_{p,s}(\Lambda) \in [0,\infty]$ be the smallest nonnegative constant such that
	\begin{align*}
		\|\alpha\beta\|_{H^s}
		\leq C_{p,s}(\Lambda)^2 \|\alpha\|_{\H} \|\beta\|_{\H}
		\qquad \text{for all} \qquad
		\lambda \in [1,\Lambda]
		\text{ and }
		\alpha,\beta \in \H_{|\Delta-\lambda| \leq \lambda^{-p}}^K.
	\end{align*}
	It is clear from the definition that $C_{p,s}(\Lambda)$ is decreasing in $p$ and increasing in $s$ and $\Lambda$.
	We will reduce Theorem~\ref{thm:L^4_quasi-Sobolev} to Proposition~\ref{prop:C_p,s(Lambda)_inductive_bd} below, which is analogous to Proposition~\ref{prop:C_s(Lambda)_inductive_bd}.
	
	\begin{lem} \label{lem:C_p,s(Lambda)_finite}
		For all $p,s \geq 0$ and $\Lambda \geq 1$, the constant $C_{p,s}(\Lambda)$ is finite.
	\end{lem}
	
	\begin{proof}
		Since $\H$ has discrete spectrum, $\H_{\Delta \leq \Lambda+1}^K$ is finite-dimensional. Thus the multiplication map
		\begin{align*}
			\H_{\Delta \leq \Lambda+1}^K \otimes \H_{\Delta \leq \Lambda+1}^K \to \H^K \cap \H^{\infty} \hookrightarrow H^s
		\end{align*}
		is a linear map from a finite-dimensional Hilbert space to a Banach space, and hence is bounded. The operator norm of this map is an upper bound for $C_{p,s}(\Lambda)^2$.
	\end{proof}
	
	\begin{lem} \label{lem:L^4_bdd_by_C_p,s(Lambda)}
		Let $p,s \geq 0$ and $\Lambda \geq 1$. Then for all $\alpha \in \H_{\Delta\leq\Lambda}^K$,
		\begin{align*}
			\|\alpha\|_{L^4}
			\leq \Lambda^{O_p(1)} C_{p,s}(\Lambda) \|\alpha\|_{\H}.
		\end{align*}
	\end{lem}
	
	\begin{proof}
		Since $C_{p,s}(\Lambda)$ increases with $s$, it suffices to prove this for $s=0$. Choose $\Lambda^{O_p(1)}$ many points $\lambda_i \in [1,\Lambda]$ such that
		\begin{align*}
			[0,\Lambda] \subseteq \bigcup_i [\lambda_i - \lambda_i^{-p}, \lambda_i + \lambda_i^{-p}].
		\end{align*}
		Choose intervals $I_i \subseteq [\lambda_i-\lambda_i^{-p}, \lambda_i+\lambda_i^{-p}]$ centered at $\lambda_i$, such that the $I_i$ partition $[0,\Lambda]$. Denote $\alpha_i = \1_{\Delta \in I_i}\alpha$, so that $\alpha = \sum_i \alpha_i$ and $\alpha_i \in \H_{|\Delta-\lambda_i| \leq \lambda_i^{-p}}^K$. Then
		\begin{align*}
			&\|\alpha\|_{L^4}
			\leq \sum_i \|\alpha_i\|_{L^4}
			= \sum_i \|\alpha_i^2\|_{H^0}^{\frac{1}{2}}
			\leq C_{p,0}(\Lambda) \sum_i \|\alpha_i\|_{\H}
			\leq \Lambda^{O_p(1)} C_{p,0}(\Lambda) \|\alpha\|_{\H}.
			\qedhere
		\end{align*}
	\end{proof}
	
	\begin{prop} \label{prop:C_p,s(Lambda)_inductive_bd}
		There are absolute constants $p,s \geq 0$ and $c \in (0,1)$ such that for $\Lambda \gg 1$,
		\begin{align} \label{eqn:C_p,s(Lambda)_inductive_bd}
			C_{p,s}(\Lambda)
			\lesssim \Lambda^{O(1)} C_{p,s}(c\Lambda).
		\end{align}
	\end{prop}
	
	
	Before proving Proposition~\ref{prop:C_p,s(Lambda)_inductive_bd}, we first show how it implies Theorem~\ref{thm:L^4_quasi-Sobolev}. This is almost identical to the deduction of Theorem~\ref{thm:L^4_quasi-Sobolev_poly_Weyl} from Proposition~\ref{prop:C_s(Lambda)_inductive_bd}.
	Recall Theorem~\ref{thm:L^4_quasi-Sobolev}:
	
	\begin{thm*}[Restatement of Theorem~\ref{thm:L^4_quasi-Sobolev}]
		Let $\alpha \in \H^K \cap \H^{\fin}$. Then
		\begin{align} \label{eqn:L^4_quasi-Sobolev_restated}
			\|\alpha\|_{L^4}
			\leq \|\exp(O(\log_+^2 \Delta)) \alpha\|_{\H}.
		\end{align}
	\end{thm*}
	
	\begin{proof}[Proof of Theorem~\ref{thm:L^4_quasi-Sobolev} assuming Proposition~\ref{prop:C_p,s(Lambda)_inductive_bd}]
		Let $p,s,c$ be absolute constants as in Proposition~\ref{prop:C_p,s(Lambda)_inductive_bd}. Then by induction on $\Lambda$ (using Lemma~\ref{lem:C_p,s(Lambda)_finite} for the base case), for $\Lambda \geq 1$ we have
		\begin{align} \label{eqn:C_p,s(Lambda)_quasipolynomial_bd}
			C_{p,s}(\Lambda)
			\lesssim \exp(O(\log_+^2\Lambda)).
		\end{align}
		Now let $\alpha \in \H^K \cap \H^{\fin}$.
		Denote $\alpha_m = \1_{\Delta \in [m-1,m)}\alpha$, so that $\alpha = \sum_{m=1}^{\infty} \alpha_m$. Since $\alpha \in \H^{\fin}$, this is a finite sum. Thus by the triangle inequality, Lemma~\ref{lem:L^4_bdd_by_C_p,s(Lambda)}, and \eqref{eqn:C_p,s(Lambda)_quasipolynomial_bd},
		\begin{align*}
			\|\alpha\|_{L^4}
			\leq \sum_{m=1}^{\infty} \|\alpha_m\|_{L^4}
			\leq \sum_{m=1}^{\infty} m^{O(1)} C_{p,s}(m) \|\alpha_m\|_{\H}
			\lesssim \sum_{m=1}^{\infty} \exp(O(\log_+^2 m)) \|\alpha_m\|_{\H}.
		\end{align*}
		Applying Cauchy--Schwarz as in the proof of Lemma~\ref{lem:sum_i_bdd_by_f(Delta)}, we get
		\begin{align*}
			\|\alpha\|_{L^4}
			\lesssim \Big(\sum_{m=1}^{\infty} \exp(O(\log_+^2 m)) \|\alpha_m\|_{\H}^2\Big)^{\frac{1}{2}}
			\lesssim \|\exp(O(\log_+^2\Delta))\alpha\|_{\H}.
		\end{align*}
		We have almost proved \eqref{eqn:L^4_quasi-Sobolev_restated}, except we have $\lesssim$ instead of $\leq$. This is easily remedied by increasing the $O$-constant.
	\end{proof}
	
	
	In the proof of Proposition~\ref{prop:C_p,s(Lambda)_inductive_bd}, we will use the following crude bounds for $p_{n,i,j}$ and $s_{n,i,j}$.
	
	\begin{lem} \label{lem:p,q_crude_bds}
		Let $n \in \Z_{\geq 0}$ and $i,j \in \Z$. Then for $\lambda \geq 1$ and $\mu \geq 0$, one has
		\begin{align} \label{eqn:p,q_crude_bds}
			|p_{n,i,j}(\lambda,\mu)|
			\lesssim_{n,i,j} \lambda^n(\mu+1)^n
			\qquad \text{and} \qquad
			|s_{n,i,j}(\lambda,\mu)|
			\lesssim_{n,i,j} \lambda^{n+1}(\mu+1)^{n+1}.
		\end{align}
	\end{lem}
	
	\begin{proof}[Proof of Lemma~\ref{lem:p,q_crude_bds}]
		By induction on $n$, the polynomial $p_{n,i,j}$ has total degree at most $n$. From this and the definition \eqref{eqn:s_n,i,j_def} of $s_{n,i,j}$, it follows that $s_{n,i,j}$ has total degree at most $n+1$. Therefore, we obtain \eqref{eqn:p,q_crude_bds}.
	\end{proof}
	
	We are now ready to prove Proposition~\ref{prop:C_p,s(Lambda)_inductive_bd}.
	
	\begin{proof}[Proof of Proposition~\ref{prop:C_p,s(Lambda)_inductive_bd}]
		Fix absolute constants $N,a_0,\dots,a_N,c$ as in Lemma~\ref{lem:R_def}. For example, as is shown in the proof of Lemma~\ref{lem:R_def}, we can take $N = 1$, $a_0 = a_1 = 1$, and $c = \frac{1}{2}$. Set $s = 2N+1$. Let $p \geq 0$, to be chosen later. Let $\Lambda \gg 1$ be arbitrary. Then we will show that \eqref{eqn:C_p,s(Lambda)_inductive_bd} holds with these choices of parameters.
		
		By the definition of $C_{p,s}(\Lambda)$ and the fact that $C_{p,s}(\Lambda)$ is finite (Lemma~\ref{lem:C_p,s(Lambda)_finite}), there exist $\lambda \in [1,\Lambda]$ and $\alpha,\beta \in \H_{|\Delta-\lambda| \leq \lambda^{-p}}^K$ nonzero, such that
		\begin{align*}
			\|\alpha\beta\|_{H^s}
			\gtrsim C_{p,s}(\Lambda)^2 \|\alpha\|_{\H} \|\beta\|_{\H}.
		\end{align*}
		We may assume $\lambda > c\Lambda$, or else \eqref{eqn:C_p,s(Lambda)_inductive_bd} follows immediately. Thus in particular $\lambda \gg 1$. Now, after normalizing $\alpha,\beta$, we see that there exist $\alpha,\beta \in \H_{|\Delta-\lambda| \leq \lambda^{-p}}^K$ with
		\begin{align*}
			\|\alpha\|_{\H},\|\beta\|_{\H} \lesssim 1
			\qquad \text{and} \qquad
			\|\alpha\beta\|_{H^s}
			\gtrsim C_{p,s}(\Lambda)^2.
		\end{align*}
		Splitting into real and imaginary parts, we may assume in addition that $\alpha,\beta \in \H_{\R}$. Using the polarization identity $\alpha\beta = \frac{1}{2}(\alpha+\beta)^2 - \frac{1}{2}\alpha^2 - \frac{1}{2}\beta^2$, we may further assume that $\alpha = \beta$. In summary, we conclude that there exists $\alpha \in \H_{|\Delta-\lambda| \leq \lambda^{-p}}^K \cap \H_{\R}$ with
		\begin{align} \label{eqn:alpha_properties}
			\|\alpha\|_{\H} \lesssim 1
			\qquad \text{and} \qquad
			\|\alpha^2\|_{H^s} \gtrsim C_{p,s}(\Lambda)^2.
		\end{align}
		We can thus estimate
		\begin{align*}
			C_{p,s}(\Lambda)^4
			\lesssim \|\alpha^2\|_{H^s}^2
			= \langle (\Delta+1)^s (\alpha^2), \alpha^2 \rangle_{\H}.
		\end{align*}
		Since $s = 2N+1$ and $1 \ll \lambda \leq \Lambda$, we deduce from \eqref{eqn:mu^s_bdd_by_R} in Corollary~\ref{cor:mu^s_bdd_by_R} that
		\begin{align} \label{eqn:C_p,s^4_terms_to_bd}
			C_{p,s}(\Lambda)^4
			\lesssim \Lambda^{O(1)} \langle \1_{\Delta \leq c\Lambda}(\alpha^2), \alpha^2 \rangle_{\H}
			+ \langle R(\lambda,\Delta)(\alpha^2), \alpha^2 \rangle_{\H},
		\end{align}
		where $R$ is as in Lemma~\ref{lem:R_def}.
		We bound each of the two terms on the right hand side separately.
		
		Let us begin with the first term in $\RHS\eqref{eqn:C_p,s^4_terms_to_bd}$. By \eqref{eqn:3_term_crossing} and the fact that $\alpha$ is real, Cauchy--Schwarz, and $L^4$-Cauchy--Schwarz (Proposition~\ref{prop:L^4_C-S}),
		\begin{align*}
			\langle \1_{\Delta \leq c\Lambda}(\alpha^2), \alpha^2 \rangle_{\H}
			= \langle \alpha \1_{\Delta \leq c\Lambda}(\alpha^2), \alpha \rangle_{\H}
			\leq \|\alpha \1_{\Delta \leq c\Lambda}(\alpha^2)\|_{\H} \|\alpha\|_{\H}
			\leq \|\1_{\Delta \leq c\Lambda}(\alpha^2)\|_{L^4} \|\alpha\|_{L^4} \|\alpha\|_{\H}.
		\end{align*}
		Recalling from \eqref{eqn:alpha_properties} that $\|\alpha\|_{\H} \lesssim 1$, we can ignore factors of $\|\alpha\|_{\H}$ if we work up to constants. Then using Lemma~\ref{lem:L^4_bdd_by_C_p,s(Lambda)} to control $\|\1_{\Delta \leq c\Lambda}(\alpha^2)\|_{L^4}$, we get
		\begin{align*}
			\langle \1_{\Delta \leq c\Lambda}(\alpha^2), \alpha^2 \rangle_{\H}
			\lesssim \Lambda^{O_p(1)} C_{p,s}(c\Lambda) \|\1_{\Delta \leq c\Lambda}(\alpha^2)\|_{\H} \|\alpha\|_{L^4}
			\leq \Lambda^{O_p(1)} C_{p,s}(c\Lambda) \|\alpha\|_{L^4}^3.
		\end{align*}
		Writing $\|\alpha\|_{L^4} = \|\alpha^2\|_{\H}^{\frac{1}{2}} \leq \|\alpha^2\|_{H^s}^{\frac{1}{2}}$ and estimating this using the definition of $C_{p,s}(\Lambda)$,
		\begin{align} \label{eqn:C_p,s^4_first_term_bd}
			\langle \1_{\Delta \leq c\Lambda}(\alpha^2), \alpha^2 \rangle_{\H}
			\lesssim \Lambda^{O_p(1)} C_{p,s}(c\Lambda) C_{p,s}(\Lambda)^3.
		\end{align}
		This bound will suffice for the first term in $\RHS\eqref{eqn:C_p,s^4_terms_to_bd}$.
		
		For the second term in $\RHS\eqref{eqn:C_p,s^4_terms_to_bd}$, writing out the definition of $R$ from Lemma~\ref{lem:R_def}, recalling that $a_0,\dots,a_N \geq 0$, and applying Proposition~\ref{prop:refined_ineq_2} yields
		\begin{align*}
			&\langle R(\lambda,\Delta)(\alpha^2), \alpha^2 \rangle_{\H}
			\\&= \sum_{n=0}^{N} a_n \lambda^{2(N-n)} \langle r_n(\lambda,\Delta)(\alpha^2), \alpha^2 \rangle_{\H}
			\\&\leq \sum_{n=0}^{N} a_n \lambda^{2(N-n)} \sum_{\substack{i,j,i',j' \\ \textnormal{not all }0}}
			\Big(\||p_{n,i,j}p_{n+1,i',j'}(\lambda,\Delta)|^{\frac{1}{2}}(\alpha_i\alpha_j)\|_{\H} \||p_{n,i,j}p_{n+1,i',j'}(\lambda,\Delta)|^{\frac{1}{2}}(\alpha_{i'}\alpha_{j'})\|_{\H}
			\\&\qquad\qquad\qquad\qquad+ \|\1_{\Delta\geq 1} \Delta^{-\frac{1}{2}} s_{n,i,j}(\lambda,\Delta)(\alpha_i\alpha_j)\|_{\H} \|\1_{\Delta\geq 1} \Delta^{-\frac{1}{2}} s_{n,i',j'}(\lambda,\Delta)(\alpha_{i'}\alpha_{j'})\|_{\H}\Big),
		\end{align*}
		where we use the notation from Subsection~\ref{subsec:refined_ineqs}, so in particular $\alpha_i$ denotes the expression $(\Delta-\lambda)^i\alpha$.
		Crudely estimating the right hand side by Lemma~\ref{lem:p,q_crude_bds} (keeping in mind the vanishing conditions \eqref{eqn:p_n,i,j_vanishing} and \eqref{eqn:s_n,i,j_vanishing}, and the fact that $N$ is an absolute constant),
		\begin{align*}
			\langle R(\lambda,\Delta)(\alpha^2), \alpha^2 \rangle_{\H}
			\lesssim \Lambda^{2(N+1)} \sum_{\substack{i,j,i',j' \in \Z_{\geq 0} \\ i+j, i'+j' \leq N+1 \\ \text{not all 0}}} \|(\Delta+1)^{N+\frac{1}{2}}(\alpha_i\alpha_j)\|_{\H} \|(\Delta+1)^{N+\frac{1}{2}}(\alpha_{i'}\alpha_{j'})\|_{\H}.
		\end{align*}
		Since $s = 2N+1$, we can rewrite this as
		\begin{align*}
			\langle R(\lambda,\Delta)(\alpha^2), \alpha^2 \rangle_{\H}
			\lesssim \Lambda^{s+1} \sum_{\substack{i,j,i',j' \in \Z_{\geq 0} \\ i+j, i'+j' \leq N+1 \\ \text{not all 0}}} \|\alpha_i\alpha_j\|_{H^s} \|\alpha_{i'}\alpha_{j'}\|_{H^s}.
		\end{align*}
		Since $\alpha \in \H_{|\Delta-\lambda| \leq \lambda^{-p}}^K$ and $\lambda \in [1,\Lambda]$, we can bound the $H^s$ norms in terms of $C_{p,s}(\Lambda)$ to get
		\begin{align*}
			\langle R(\lambda,\Delta)(\alpha^2), \alpha^2 \rangle_{\H}
			\lesssim \Lambda^{s+1} C_{p,s}(\Lambda)^4 \sum_{\substack{i,j,i',j' \in \Z_{\geq 0} \\ i+j, i'+j' \leq N+1 \\ \text{not all 0}}} \|\alpha_i\|_{\H} \|\alpha_j\|_{\H} \|\alpha_{i'}\|_{\H} \|\alpha_{j'}\|_{\H}.
		\end{align*}
		By the definition $\alpha_i = (\Delta-\lambda)^i\alpha$ and the normalization $\|\alpha\|_{\H} \lesssim 1$ from \eqref{eqn:alpha_properties}, we see that $\|\alpha_i\|_{\H} \lesssim \lambda^{-ip}$. Combining this with the fact that $\lambda > c\Lambda$ and that not all $i,j,i',j'$ are zero,
		\begin{align} \label{eqn:C_p,s^4_second_term_bd}
			\langle R(\lambda,\Delta)(\alpha^2), \alpha^2 \rangle_{\H}
			\lesssim_p \Lambda^{s+1-p} C_{p,s}(\Lambda)^4.
		\end{align}
		This is the bound we need for the second term in $\RHS\eqref{eqn:C_p,s^4_terms_to_bd}$.
		
		Inserting \eqref{eqn:C_p,s^4_first_term_bd} and \eqref{eqn:C_p,s^4_second_term_bd} into $\RHS\eqref{eqn:C_p,s^4_terms_to_bd}$,
		\begin{align*}
			C_{p,s}(\Lambda)^4
			\lesssim_p \Lambda^{O_p(1)} C_{p,s}(c\Lambda) C_{p,s}(\Lambda)^3 + \Lambda^{s+1-p} C_{p,s}(\Lambda)^4.
		\end{align*}
		Take $p = s+2$ (anything bigger than $s+1$ would do), and let $A \lesssim 1$ denote the implicit constant. Then the above becomes
		\begin{align} \label{eqn:C_bd_before_disc_spectrum_true}
			C_{p,s}(\Lambda)^4
			\leq  A\Lambda^{O(1)} C_{p,s}(c\Lambda) C_{p,s}(\Lambda)^3 + A\Lambda^{-1} C_{p,s}(\Lambda)^4.
		\end{align}
		Since $\Lambda \gg 1$, we may assume $A\Lambda^{-1} \leq \frac{1}{2}$. Then since $C_{p,s}(\Lambda)$ is finite by Lemma~\ref{lem:C_p,s(Lambda)_finite}, we can subtract the second term on the right hand side from both sides to get
		\begin{align} \label{eqn:C_bd_after_disc_spectrum_true}
			\tfrac{1}{2}C_{p,s}(\Lambda)^4 \leq A\Lambda^{O(1)} C_{p,s}(c\Lambda) C_{p,s}(\Lambda)^3.
		\end{align}
		Using once again that $C_{p,s}(\Lambda)$ is finite, we can cancel factors of $C_{p,s}(\Lambda)^3$ on both sides to conclude the desired bound \eqref{eqn:C_p,s(Lambda)_inductive_bd}.
	\end{proof}
	
	The proof of Theorem~\ref{thm:L^4_quasi-Sobolev} is finally complete.
	
	\subsection{Asymptotics for $p_{n,i,j}$}
	\label{subsec:p_n,i,j_asymptotics}
	
	As above, let $p_n$ and $p_{n,i,j}$ be the polynomials in two variables defined in Propositions~\ref{prop:p_n_def} and \ref{prop:p_n,i,j_def}, respectively. Recall $p_{n,0,0} = p_n$. The main result of this subsection is Proposition~\ref{prop:p_n,i,j_bds}, which gives a good estimate for $p_{n,i,j}(\lambda,\mu)$ with explicit dependence on all parameters (unlike Lemma~\ref{lem:p,q_crude_bds}, where the dependence on $n$ is implicit).
	
	\begin{prop} \label{prop:p_n,i,j_bds}
		Let $n,i,j \in \Z_{\geq 0}$ with $i+j \leq n$. Let $\lambda,\mu \geq 0$. Then one has the trivial bound
		\begin{align} \label{eqn:p_n,i,j_triv_bd}
			|p_{n,i,j}(\lambda,\mu)| \leq O(1)^n (\lambda+\mu+n^2)^{n-i-j}.
		\end{align}
		If in addition $\mu \gg \lambda+n^2$, then
		\begin{align} \label{eqn:p_n,i,j_vs_p_n}
			|p_{n,i,j}(\lambda,\mu)| \leq O(\mu^{-1})^{i+j} \binom{n}{i+j} |p_n(\lambda,\mu)|.
		\end{align}
	\end{prop}
	
	Here $\binom{n}{i+j}$ is a binomial coefficient.
	
	The bound \eqref{eqn:p_n,i,j_triv_bd} follows directly from the recurrence \eqref{eqn:p_n,i,j_recurrence} and induction on $n$, where the base case for the induction comes from the initial conditions \eqref{eqn:p_0,i,j_def} and \eqref{eqn:p_1,i,j_def}. So \eqref{eqn:p_n,i,j_triv_bd} really is ``trivial." Nevertheless, \eqref{eqn:p_n,i,j_triv_bd} is much stronger than the crude bound given in Lemma~\ref{lem:p,q_crude_bds}. Indeed, when $i=j=0$, we see from Lemma~\ref{lem:p_n_sign} that \eqref{eqn:p_n,i,j_triv_bd} is sharp up to a factor of $O(1)^n$ in the regime $\mu \gg \lambda+n^2$. The estimate \eqref{eqn:p_n,i,j_vs_p_n} says that in this regime, each time $i$ or $j$ is increased by $1$, we save at least an extra factor of $n/\mu$.
	
	It remains to establish \eqref{eqn:p_n,i,j_vs_p_n}, so assume $\mu \gg \lambda+n^2$. Writing the recurrence \eqref{eqn:p_n,i,j_recurrence} for $p_{n,i,j}$ in matrix form,
	\begin{align*}
		\begin{pmatrix}
			p_{n+1,i,j} \\
			p_{n,i,j}
		\end{pmatrix}
		= A_n
		\begin{pmatrix}
			p_{n,i,j} \\
			p_{n-1,i,j}
		\end{pmatrix}
		+ B_n \Big[
		\begin{pmatrix}
			p_{n,i-1,j} \\
			p_{n-1,i-1,j}
		\end{pmatrix}
		+
		\begin{pmatrix}
			p_{n,i,j-1} \\
			p_{n-1,i,j-1}
		\end{pmatrix}
		\Big]
		+ C_n
		\begin{pmatrix}
			p_{n,i-1,j-1} \\
			p_{n-1,i-1,j-1}
		\end{pmatrix}
	\end{align*}
	for $n \geq 1$, where
	\begin{align*}
		A_n = \begin{pmatrix}
			-\mu + 2\lambda + 2n^2 & -(\lambda+n(n-1))^2 \\
			1 & 0
		\end{pmatrix},
		\quad
		B_n = \begin{pmatrix}
			1 & -(\lambda+n(n-1)) \\
			0 & 0
		\end{pmatrix},
		\quad
		C_n = \begin{pmatrix}
			0 & -1 \\
			0 & 0
		\end{pmatrix}.
	\end{align*}
	The initial conditions \eqref{eqn:p_0,i,j_def} and \eqref{eqn:p_1,i,j_def} can similarly be written in matrix form as
	\begin{align*}
		\begin{pmatrix}
			p_{1,i,j} \\
			p_{0,i,j}
		\end{pmatrix}
		=
		\begin{cases}
			\begin{pmatrix}
				-\frac{1}{2}\mu + \lambda \\
				1
			\end{pmatrix}
			&\text{if } (i,j) = (0,0),
			\\[1.5em]
			\begin{pmatrix}
				\frac{1}{2} \\
				0
			\end{pmatrix}
			&\text{if } (i,j) = (1,0) \text{ or } (0,1),
			\\[1.5em]
			\begin{pmatrix}
				0 \\
				0
			\end{pmatrix}
			&\text{otherwise}.
		\end{cases}
	\end{align*}
	Denote
	\begin{align} \label{eqn:tilde{p}_n,i,j_def}
		\tilde{p}_{n,i,j}
		= (-\mu)^{-(n-i-j)} p_{n,i,j}
		\qquad \text{and} \qquad
		\tilde{p}_n = \tilde{p}_{n,0,0} = (-\mu)^{-n} p_n.
	\end{align}
	With this normalization, the inequality \eqref{eqn:p_n_sign} in Lemma~\ref{lem:p_n_sign} becomes the lower bound
	\begin{align*}
		\tilde{p}_n(\lambda,\mu)
		\geq \tfrac{1}{2}(0.9)^n,
	\end{align*}
	and our goal, \eqref{eqn:p_n,i,j_vs_p_n}, becomes the slightly simpler estimate
	\begin{align} \label{eqn:normalized_goal}
		|\tilde{p}_{n,i,j}(\lambda,\mu)|
		\leq O(1)^{i+j} \binom{n}{i+j} \tilde{p}_n(\lambda,\mu).
	\end{align}
	The $\tilde{p}_{n,i,j}$ satisfy the modified recurrence
	\begin{align} \label{eqn:tilde{p}_recurrence}
		\begin{pmatrix}
			\tilde{p}_{n+1,i,j} \\
			\tilde{p}_{n,i,j}
		\end{pmatrix}
		= \tilde{A}_n
		\begin{pmatrix}
			\tilde{p}_{n,i,j} \\
			\tilde{p}_{n-1,i,j}
		\end{pmatrix}
		+ \tilde{B}_n \Big[
		\begin{pmatrix}
			\tilde{p}_{n,i-1,j} \\
			\tilde{p}_{n-1,i-1,j}
		\end{pmatrix}
		+
		\begin{pmatrix}
			\tilde{p}_{n,i,j-1} \\
			\tilde{p}_{n-1,i,j-1}
		\end{pmatrix}
		\Big]
		+ \tilde{C}_n
		\begin{pmatrix}
			\tilde{p}_{n,i-1,j-1} \\
			\tilde{p}_{n-1,i-1,j-1}
		\end{pmatrix}
	\end{align}
	for $n \geq 1$, with
	\begin{align*}
		\tilde{A}_n = \begin{pmatrix}
			1 - \frac{2\lambda+2n^2}{\mu} & -(\frac{\lambda+n(n-1)}{\mu})^2 \\[0.5em]
			1 & 0
		\end{pmatrix},
		\qquad
		\tilde{B}_n = \begin{pmatrix}
			1 & \frac{\lambda+n(n-1)}{\mu} \\[0.5em]
			0 & 0
		\end{pmatrix},
		\qquad
		\tilde{C}_n = \begin{pmatrix}
			0 & -1 \\
			0 & 0
		\end{pmatrix},
	\end{align*}
	and with initial conditions
	\begin{align*}
		\begin{pmatrix}
			\tilde{p}_{1,i,j} \\
			\tilde{p}_{0,i,j}
		\end{pmatrix}
		=
		\begin{cases}
			\begin{pmatrix}
				\frac{1}{2} - \frac{\lambda}{\mu} \\
				1
			\end{pmatrix}
			&\text{if } (i,j) = (0,0),
			\\[1.5em]
			\begin{pmatrix}
				\frac{1}{2} \\
				0
			\end{pmatrix}
			&\text{if } (i,j) = (1,0) \text{ or } (0,1),
			\\[1.5em]
			\begin{pmatrix}
				0 \\
				0
			\end{pmatrix}
			&\text{otherwise}.
		\end{cases}
	\end{align*}
	Denote
	\begin{align*}
		\delta_n
		= \frac{2\lambda+2n^2}{\mu}
		\qquad \text{and} \qquad
		\varepsilon_n
		= \frac{\lambda+n(n-1)}{\mu},
	\end{align*}
	so that
	\begin{align} \label{eqn:A,B,C_in_delta,eps}
		\tilde{A}_n
		=
		\begin{pmatrix}
			1 - \delta_n & -\varepsilon_n^2 \\
			1 & 0
		\end{pmatrix},
		\qquad
		\tilde{B}_n
		=
		\begin{pmatrix}
			1 & \varepsilon_n \\
			0 & 0
		\end{pmatrix},
		\qquad
		\tilde{C}_n
		=
		\begin{pmatrix}
			0 & -1 \\
			0 & 0
		\end{pmatrix}.
	\end{align}
	Since $\mu \gg \lambda+n^2$ by assumption, we have $0 \leq \delta_n \ll 1$ and $0 \leq \varepsilon_n \ll 1$. Note also that $\delta_n,\varepsilon_n$ are increasing in $n$, so the entries of $\tilde{A}_n$ are monotonic in $n$. The only properties of $\tilde{B}_n$ and $\tilde{C}_n$ that we will use are the trivial bounds
	\begin{align} \label{eqn:B,C_triv_bds}
		\|\tilde{B}_n\|_{\op} \lesssim 1
		\qquad \text{and} \qquad
		\|\tilde{C}_n\|_{\op} \lesssim 1.
	\end{align}
	
	If we compute
	$
	(
	\begin{smallmatrix}
		\tilde{p}_{n+1,i,j} \\
		\tilde{p}_{n,i,j}
	\end{smallmatrix}
	)
	$
	by iterating the recurrence \eqref{eqn:tilde{p}_recurrence} until either $i$ or $j$ drop, we find that
	\begin{align}
		\begin{pmatrix}
			\tilde{p}_{n+1,i,j} \\
			\tilde{p}_{n,i,j}
		\end{pmatrix}
		= &\,\tilde{A}_n \tilde{A}_{n-1} \cdots \tilde{A}_1
		\begin{pmatrix}
			\tilde{p}_{1,i,j} \\
			\tilde{p}_{0,i,j}
		\end{pmatrix}
		\notag
		\\&+
		\sum_{m=1}^n \tilde{A}_n \tilde{A}_{n-1} \cdots \tilde{A}_{m+1}
		\Big[
		\tilde{B}_m
		\begin{pmatrix}
			\tilde{p}_{m,i-1,j} \\
			\tilde{p}_{m-1,i-1,j}
		\end{pmatrix}
		+
		\tilde{B}_m
		\begin{pmatrix}
			\tilde{p}_{m,i,j-1} \\
			\tilde{p}_{m-1,i,j-1}
		\end{pmatrix}
		+
		\tilde{C}_m
		\begin{pmatrix}
			\tilde{p}_{m,i-1,j-1} \\
			\tilde{p}_{m-1,i-1,j-1}
		\end{pmatrix}
		\Big]
		\label{eqn:recurrence_iterated}
	\end{align}
	(here when $m=n$, the empty product $\tilde{A}_n \tilde{A}_{n-1} \cdots \tilde{A}_{m+1}$ is taken to be the identity matrix). In order to use this to estimate $\tilde{p}_{n,i,j}$, we need sharp bounds on products of consecutive $\tilde{A}$'s. These bounds will come from the following general linear algebra lemma.
	
	\begin{lem} \label{lem:matrix_product}
		Fix a nonzero vector $u \in \C^d$.
		Fix a compact subset $\mathcal{K} \subseteq \Mat_d(\C)$ of $d \times d$ matrices, such that each matrix in $\mathcal{K}$ has $d$ eigenvalues of distinct absolute value, and has eigenvectors all not orthogonal to $u$. Let $M_1,\dots,M_N \in \mathcal{K}$ be such that the real and imaginary parts of the entries of $M_n$ depend monotonically on $n$. Diagonalize
		\begin{align*}
			M_n
			= Q_n D_n Q_n^{-1},
		\end{align*}
		with $D_n$ the diagonal matrix of eigenvalues of $M_n$ in decreasing order of absolute value from top left to bottom right, and $Q_n$ the matrix whose columns are the eigenvectors of $M_n$ in the corresponding order, normalized so that each eigenvector has inner product $1$ with $u$. The assumptions on $\mathcal{K}$ ensure that $D_n$ and $Q_n$ exist and are uniquely defined. Then
		\begin{align} \label{eqn:matrix_product_asymp}
			M_N \cdots M_1
			= Q_N D_N \cdots D_1 Q_1^{-1}
			+ O_{\mathcal{K},u,d}(\|D_1\|_{\op} \cdots \|D_N\|_{\op} \|M_N-M_1\|_{\op}).
		\end{align}
		In particular,
		\begin{align} \label{eqn:matrix_product_norm_bd}
			\|M_N \cdots M_1\|_{\op}
			\lesssim_{\mathcal{K},u,d} \|D_1\|_{\op} \cdots \|D_N\|_{\op}.
		\end{align}
	\end{lem}
	
	We emphasize that the $O$- and implicit constants in \eqref{eqn:matrix_product_asymp} and \eqref{eqn:matrix_product_norm_bd} are independent of $N$.
	
	The hypotheses of Lemma~\ref{lem:matrix_product} can be weakened and modified in various ways, but this statement suffices for our purposes. In fact, we will only need to work with real $2\times 2$ matrices with real eigenvalues, but it causes no additional difficulties to prove the lemma for complex $d \times d$ matrices for any $d$.
	
	\begin{proof}
		Multiplying by $Q_N^{-1}$ on the left and $Q_1$ on the right, it suffices to show that
		\begin{align} \label{eqn:move_Q_N,Q_1}
			Q_N^{-1} M_N \cdots M_1 Q_1
			= D_N \cdots D_1 + O_{\mathcal{K},u,d}(\|D_1\|_{\op} \cdots \|D_N\|_{\op} \|M_N-M_1\|_{\op}).
		\end{align}
		Denoting $Q_{N+1} = Q_N$, the left hand side can be expressed as a telescoping product
		\begin{align*}
			\LHS\eqref{eqn:move_Q_N,Q_1}
			= \prod_{N \geq n \geq 1} Q_{n+1}^{-1} M_n Q_n.
		\end{align*}
		Write
		\begin{align*}
			Q_{n+1}^{-1} M_n Q_n
			= Q_{n+1}^{-1} Q_n D_n.
		\end{align*}
		The matrices $D_n$ and $Q_n$ are smooth functions of $M_n \in \mathcal{K}$. In particular, there is a Lipschitz function $Q \colon \mathcal{K} \to \GL_d(\C)$, depending only on $\mathcal{K},u,d$, such that $Q(M_n) = Q_n$. Therefore
		\begin{align*}
			Q_{n+1}^{-1} Q_n
			= I + O_{\mathcal{K},u,d}(\|Q_{n+1}-Q_n\|_{\op})
			= I + O_{\mathcal{K},u,d}(\|M_{n+1}-M_n\|_{\op}),
		\end{align*}
		where $I$ is the identity matrix
		(for $n=N$, denote $M_{N+1} = M_N$). We thus obtain
		\begin{align*}
			\LHS\eqref{eqn:move_Q_N,Q_1}
			= \prod_{N \geq n \geq 1} [D_n + O_{\mathcal{K},u,d}(\|D_n\|_{\op} \|M_{n+1}-M_n\|_{\op})].
		\end{align*}
		Expanding this, treating all terms except the first as error terms, and then refactoring these terms,
		\begin{align*}
			\LHS\eqref{eqn:move_Q_N,Q_1}
			= D_N \cdots D_1 + \Error,
		\end{align*}
		where
		\begin{align*}
			\|\Error\|_{\op}
			&\leq \|D_1\|_{\op} \cdots \|D_N\|_{\op} \Big[\prod_{n=1}^{N} (1+O_{\mathcal{K},u,d}(\|M_{n+1}-M_n\|_{\op})) - 1\Big]
			\\&\leq \|D_1\|_{\op} \cdots \|D_N\|_{\op} \Big[\exp\Big(O_{\mathcal{K},u,d}\Big(\sum_{n=1}^{N} \|M_{n+1}-M_n\|_{\op}\Big)\Big) - 1\Big].
		\end{align*}
		All norms on finite-dimensional spaces are equivalent, so $\|\cdot\|_{\op} \sim_d \|\cdot\|_1$, where $\|M\|_1$ is defined to be the sum of the absolute values of the real and imaginary parts of each entry of $M$. By assumption, the real and imaginary parts of the entries of $M_n$ are monotonic in $n$, so
		\begin{align*}
			\sum_{n=1}^{N} \|M_{n+1}-M_n\|_{\op}
			\sim_d \sum_{n=1}^{N} \|M_{n+1}-M_n\|_1
			= \|M_N - M_1\|_1
			\sim_d \|M_N-M_1\|_{\op}.
		\end{align*}
		Inserting this above,
		\begin{align*}
			\|\Error\|_{\op}
			&\leq \|D_1\|_{\op} \cdots \|D_N\|_{\op} [\exp(O_{\mathcal{K},u,d}(\|M_N-M_1\|_{\op})) - 1]
			\\&\lesssim_{\mathcal{K},u,d} \|D_1\|_{\op} \cdots \|D_N\|_{\op} \|M_N-M_1\|_{\op}.
		\end{align*}
		This establishes \eqref{eqn:move_Q_N,Q_1} and concludes the proof.
	\end{proof}
	
	Let us now apply Lemma~\ref{lem:matrix_product} in dimension $d=2$ with
	\begin{align*}
		u =
		\begin{pmatrix}
			0 \\
			1
		\end{pmatrix},
		\qquad
		\mathcal{K} = \Big\{
		\begin{pmatrix}
			1-s & -t \\
			1 & 0
		\end{pmatrix}
		: s,t \in [0,10^{-10}]
		\Big\},
		\qquad
		M_m = \tilde{A}_m \text{ for } m \leq n
	\end{align*}
	(since the variable $n$ is already in use in Proposition~\ref{prop:p_n,i,j_bds}, we index the $M$'s by $m$ here instead of $n$ as in Lemma~\ref{lem:matrix_product}).
	Using as always that $\mu \gg \lambda+n^2$, it is easy to check that the hypotheses of the lemma are satisfied. As in the lemma, diagonalize
	\begin{align*}
		\tilde{A}_m
		= Q_m D_m Q_m^{-1}.
	\end{align*}
	Then because $\tilde{A}_m \in \mathcal{K}$,
	\begin{align*}
		D_m \approx
		\begin{pmatrix}
			1 & 0 \\
			0 & 0
		\end{pmatrix}
		\qquad \text{and} \qquad
		Q_m \approx
		\begin{pmatrix}
			1 & 0 \\
			1 & 1
		\end{pmatrix}.
	\end{align*}
	To be precise, one has for example
	\begin{align} \label{eqn:D,Q_bds}
		D_m
		\in \Big\{
		\begin{pmatrix}
			1-s & 0 \\
			0 & t
		\end{pmatrix}
		: s,t \in [-10^{-5},10^{-5}]
		\Big\}
		\qquad \text{and} \qquad
		Q_m
		\in \Big\{
		\begin{pmatrix}
			1-s & t \\
			1 & 1
		\end{pmatrix}
		: s,t \in [-10^{-5},10^{-5}]
		\Big\}.
	\end{align}
	In particular,
	\begin{align} \label{eqn:D_m_op_size}
		\|D_m\|_{\op} \sim 1.
	\end{align}
	Now, using \eqref{eqn:matrix_product_norm_bd} and the trivial bounds \eqref{eqn:B,C_triv_bds} to upper bound the products of matrices in \eqref{eqn:recurrence_iterated},
	\begin{align}
		\Big|
		\begin{pmatrix}
			\tilde{p}_{n+1,i,j} \\
			\tilde{p}_{n,i,j}
		\end{pmatrix}
		\Big|
		\lesssim
		&\, \|D_1\|_{\op} \cdots \|D_n\|_{\op}
		\Big|
		\begin{pmatrix}
			\tilde{p}_{1,i,j} \\
			\tilde{p}_{0,i,j}
		\end{pmatrix}
		\Big|
		\notag
		\\&+ \sum_{m=1}^n \|D_{m+1}\|_{\op} \cdots \|D_n\|_{\op}
		\Bigg(
		\Big|
		\begin{pmatrix}
			\tilde{p}_{m,i-1,j} \\
			\tilde{p}_{m-1,i-1,j}
		\end{pmatrix}
		\Big|
		+
		\Big|
		\begin{pmatrix}
			\tilde{p}_{m,i,j-1} \\
			\tilde{p}_{m-1,i,j-1}
		\end{pmatrix}
		\Big|
		+
		\Big|
		\begin{pmatrix}
			\tilde{p}_{m,i-1,j-1} \\
			\tilde{p}_{m-1,i-1,j-1}
		\end{pmatrix}
		\Big|
		\Bigg).
		\label{eqn:p_n,i,j_recursive_bd}
	\end{align}
	Denote
	\begin{align} \label{eqn:P_n,i,j_def}
		P_{n,i,j}
		= |\tilde{p}_{n,i,j}| + |\tilde{p}_{n-1,i,j}|
	\end{align}
	(for the case $n=0$, set $\tilde{p}_{-1,i,j} = 0$).
	For $n=0$ and $n=1$, the initial conditions for $\tilde{p}_{n,i,j}$ imply
	\begin{align} \label{eqn:P_n,i,j_init_bds}
		P_{0,i,j}
		= \1_{i=j=0}
		\qquad \text{and} \qquad
		P_{1,i,j}
		\lesssim \1_{\substack{i,j \geq 0 \\ i+j \leq 1}}.
	\end{align}
	For $n \geq 0$, the recursive estimate \eqref{eqn:p_n,i,j_recursive_bd} becomes
	\begin{align}
		P_{n+1,i,j}
		\lesssim &\, \|D_1\|_{\op} \cdots \|D_n\|_{\op} \, \1_{\substack{i,j \geq 0 \\ i+j \leq 1}}
		\notag
		\\&+ \sum_{m=1}^{n} \|D_{m+1}\|_{\op} \cdots \|D_n\|_{\op} (P_{m,i-1,j} + P_{m,i,j-1} + P_{m,i-1,j-1}).
		\label{eqn:P_n,i,j_recursive_bd}
	\end{align}
	Also, it follows trivially from the definitions \eqref{eqn:tilde{p}_n,i,j_def} and \eqref{eqn:P_n,i,j_def} that the vanishing conditions \eqref{eqn:p_n,i,j_vanishing} hold without change for $P_{n,i,j}$:
	\begin{align} \label{eqn:P_n,i,j_vanishing}
		P_{n,i,j} = 0
		\qquad \text{if} \qquad
		i < 0
		\text{ or }
		j < 0
		\text{ or }
		i+j > n.
	\end{align}
	Next, we use \eqref{eqn:P_n,i,j_init_bds}, \eqref{eqn:P_n,i,j_recursive_bd}, and \eqref{eqn:P_n,i,j_vanishing} to prove
	
	\begin{prop} \label{prop:P_n,i,j_bd}
		Let $n,i,j \in \Z_{\geq 0}$. Let $\lambda,\mu \geq 0$ with $\mu \gg \lambda+n^2$. Then
		\begin{align} \label{eqn:P_n,i,j_bd_by_Ds}
			P_{n,i,j}
			\lesssim O(1)^{i+j} \binom{n}{i+j} \|D_1\|_{\op} \cdots \|D_n\|_{\op}
		\end{align}
		(note both the scalar $P_{n,i,j}$ and the matrices $D_m$ depend on $\lambda,\mu$, though as usual the implicit constants do not). 
	\end{prop}
	
	As in Proposition~\ref{prop:p_n,i,j_bds}, $\binom{n}{i+j}$ is a binomial coefficient rather than a column vector.
	
	\begin{proof}
		The case $n=0$ is trivial by \eqref{eqn:P_n,i,j_init_bds} and \eqref{eqn:P_n,i,j_vanishing}, so it suffices to consider $n \geq 1$.
		Fix a large constant $C \gg 1$. We will show by induction on $i+j$ that for all $n,i,j \in \Z_{\geq 0}$ with $n \geq 1$,
		\begin{align} \label{eqn:P_n,i,j_bd_to_induct}
			P_{n,i,j}
			\leq C^{i+j+1} \Big[\sum_{k=0}^{i+j} \binom{n}{k}\Big] \|D_1\|_{\op} \cdots \|D_n\|_{\op}.
		\end{align}
		First, let us check that this is sufficient to prove the proposition.
		If $i+j > n$, then $P_{n,i,j} = 0$ by \eqref{eqn:P_n,i,j_vanishing}, so without loss of generality, $i+j \leq n$.
		For $k \leq \frac{1}{2}n$, the binomial coefficient $\binom{n}{k}$ is increasing in $k$, so
		\begin{align} \label{eqn:binom_sum_bd_lo}
			\sum_{k=0}^{i+j} \binom{n}{k}
			\leq (i+j+1) \binom{n}{i+j}
			\qquad \text{for} \qquad
			i+j \leq \frac{1}{2}n.
		\end{align}
		In the complementary range,
		\begin{align} \label{eqn:binom_sum_bd_hi}
			\sum_{k=0}^{i+j} \binom{n}{k}
			\leq 2^n
			\leq 4^{i+j}
			\qquad \text{for} \qquad
			i+j \geq \frac{1}{2}n.
		\end{align}
		Putting \eqref{eqn:binom_sum_bd_lo} and \eqref{eqn:binom_sum_bd_hi} together,
		\begin{align*}
			\sum_{k=0}^{i+j} \binom{n}{k}
			\lesssim O(1)^{i+j} \binom{n}{i+j}
		\end{align*}
		(note here it's important that $i+j \leq n$, or else the binomial coefficient on the right would vanish).
		Thus \eqref{eqn:P_n,i,j_bd_to_induct} implies the desired bound \eqref{eqn:P_n,i,j_bd_by_Ds}.
		
		We now turn to proving \eqref{eqn:P_n,i,j_bd_to_induct}. When $i+j = 0$, so $i=j=0$, it follows from \eqref{eqn:P_n,i,j_recursive_bd} and \eqref{eqn:P_n,i,j_vanishing} that
		\begin{align*}
			P_{n,0,0}
			\lesssim \|D_1\|_{\op} \cdots \|D_{n-1}\|_{\op}.
		\end{align*}
		Since $\|D_n\|_{\op} \sim 1$ by \eqref{eqn:D_m_op_size}, this is equivalent to
		\begin{align*}
			P_{n,0,0}
			\lesssim \|D_1\|_{\op} \cdots \|D_n\|_{\op}.
		\end{align*}
		Since $C \gg 1$, we may assume $C$ is bigger than the implicit constant, so
		\begin{align*}
			P_{n,0,0} \leq C \|D_1\|_{\op} \cdots \|D_n\|_{\op}.
		\end{align*}
		This establishes \eqref{eqn:P_n,i,j_bd_to_induct} for $i+j=0$, and hence completes the base case of the induction.
		
		For the induction step, let $i+j \geq 1$. Again without loss of generality, $i+j \leq n$. Assume inductively that \eqref{eqn:P_n,i,j_bd_to_induct} holds for all $n',i',j' \in \Z_{\geq 0}$ with $n' \geq 1$ and $i'+j' < i+j$. Using this inductive assumption to estimate the right hand side of \eqref{eqn:P_n,i,j_recursive_bd} (with the index $n$ in \eqref{eqn:P_n,i,j_recursive_bd} shifted by $1$),
		\begin{align*}
			P_{n,i,j}
			\lesssim \|D_1\|_{\op} \cdots \|D_{n-1}\|_{\op} \Big(\1_{i+j=1} + \sum_{m=1}^{n-1} \Big[2C^{i+j} \sum_{k=0}^{i+j-1} \binom{m}{k} + C^{i+j-1} \sum_{k=0}^{i+j-2} \binom{m}{k} \Big]\Big).
		\end{align*}
		Multiplying the right hand side by $\|D_n\|_{\op} \sim 1$, switching the order of summation, and using the ``hockey stick identity" for sums over diagonals in Pascal's triangle,
		\begin{align*}
			P_{n,i,j}
			\lesssim C^{i+j} \|D_1\|_{\op} \cdots \|D_n\|_{\op} \sum_{k=0}^{i+j} \binom{n}{k}.
		\end{align*}
		Again since $C \gg 1$, we may assume $C$ is bigger than the implicit constant (which is independent of $n,i,j$). Then \eqref{eqn:P_n,i,j_bd_to_induct} holds, and the induction step is complete.
	\end{proof}
	
	Recall our goal is to prove \eqref{eqn:p_n,i,j_vs_p_n}, which we have seen is equivalent to \eqref{eqn:normalized_goal}. It in fact suffices to show \eqref{eqn:normalized_goal} with $\lesssim$ in place of $\leq$, i.e.,
	\begin{align} \label{eqn:normalized_goal_lesssim}
		|\tilde{p}_{n,i,j}| \lesssim O(1)^{i+j} \binom{n}{i+j} \tilde{p}_n.
	\end{align}
	This is sufficient because \eqref{eqn:normalized_goal} trivially holds when $i=j=0$ since then both sides are the same, and when $i+j > 0$, the implicit constant in $\lesssim$ can be absorbed into the $O$-constant.
	Now from Proposition~\ref{prop:P_n,i,j_bd} and the definition of $P_{n,i,j}$,
	\begin{align*}
		|\tilde{p}_{n,i,j}|
		\lesssim O(1)^{i+j} \binom{n}{i+j} \|D_1\|_{\op} \cdots \|D_n\|_{\op}.
	\end{align*}
	Thus to conclude \eqref{eqn:normalized_goal_lesssim} and hence \eqref{eqn:p_n,i,j_vs_p_n}, it only remains to prove
	
	\begin{prop} \label{prop:p_n_lower_bd_by_Ds}
		Let $n \in \Z_{\geq 0}$. Let $\lambda,\mu \geq 0$ with $\mu \gg \lambda+n^2$. Then
		\begin{align} \label{eqn:p_n_lower_bd_by_Ds}
			\tilde{p}_n
			\gtrsim \|D_1\|_{\op} \cdots \|D_n\|_{\op}.
		\end{align}
	\end{prop}
	
	\begin{proof}
		When $n=0$, both sides are equal to $1$. So assume $n \geq 1$.
		It follows from the definition $\tilde{p}_n = \tilde{p}_{n,0,0}$, the recursive formula \eqref{eqn:recurrence_iterated}, and the fact that $\tilde{p}_{n,i,j} = 0$ if $i < 0$ or $j < 0$, that
		\begin{align*}
			\tilde{p}_n
			=
			\begin{pmatrix}
				0 & 1
			\end{pmatrix}
			\tilde{A}_n \cdots \tilde{A}_1
			\begin{pmatrix}
				\tilde{p}_1 \\
				\tilde{p}_0
			\end{pmatrix}.
		\end{align*}
		Plugging in the initial conditions for $\tilde{p}_0, \tilde{p}_1$ and applying \eqref{eqn:matrix_product_asymp} from Lemma~\ref{lem:matrix_product},
		\begin{align} \label{eqn:p_n_matrix_product_asymp}
			\tilde{p}_n
			=
			\begin{pmatrix}
				0 & 1
			\end{pmatrix}
			Q_n D_n \cdots D_1 Q_1^{-1}
			\begin{pmatrix}
				\frac{1}{2} - \frac{\lambda}{\mu} \\
				1
			\end{pmatrix}
			+ O(\|D_1\|_{\op} \cdots \|D_n\|_{\op} \|\tilde{A}_n - \tilde{A}_1\|_{\op}).
		\end{align}
		Combining \eqref{eqn:D,Q_bds} with the fact that $\lambda/\mu \ll 1$, we can crudely estimate
		\begin{align*}
			Q_1^{-1}
			\begin{pmatrix}
				\frac{1}{2} - \frac{\lambda}{\mu} \\
				1
			\end{pmatrix}
			=
			\begin{pmatrix}
				a \\
				b
			\end{pmatrix}
			\qquad \text{with} \qquad
			a \geq 0.4
			\text{ and }
			|b| \leq 0.6.
		\end{align*}
		Then again by \eqref{eqn:D,Q_bds},
		\begin{align*}
			D_n \cdots D_1 Q_1^{-1}
			\begin{pmatrix}
				\frac{1}{2} - \frac{\lambda}{\mu} \\
				1
			\end{pmatrix}
			=
			\begin{pmatrix}
				a' \\
				b'
			\end{pmatrix}
		\end{align*}
		with
		\begin{align*}
			a' = \|D_1\|_{\op} \cdots \|D_n\|_{\op} \, a \geq 0.4(1-10^{-5})^n
			\qquad \text{ and } \qquad
			|b'| \leq 0.6 \times 10^{-5n}.
		\end{align*}
		Since $n \geq 1$, we have in particular
		\begin{align*}
			a' \gtrsim \|D_1\|_{\op} \cdots \|D_n\|_{\op}
			\qquad \text{and} \qquad
			|b'| \leq \frac{1}{2}a'.
		\end{align*}
		By \eqref{eqn:D,Q_bds} one last time,
		\begin{align*}
			\begin{pmatrix}
				0 & 1
			\end{pmatrix}
			Q_n
			=
			\begin{pmatrix}
				1 & 1
			\end{pmatrix}.
		\end{align*}
		Thus
		\begin{align*}
			\begin{pmatrix}
				0 & 1
			\end{pmatrix}
			Q_n D_n \cdots D_1 Q_1^{-1}
			\begin{pmatrix}
				\frac{1}{2} - \frac{\lambda}{\mu} \\
				1
			\end{pmatrix}
			=
			\begin{pmatrix}
				1 & 1
			\end{pmatrix}
			\begin{pmatrix}
				a' \\
				b'
			\end{pmatrix}
			= a'+b'
			\geq \frac{1}{2} a'
			\gtrsim \|D_1\|_{\op} \cdots \|D_n\|_{\op}.
		\end{align*}
		Inserting this into \eqref{eqn:p_n_matrix_product_asymp},
		\begin{align} \label{eqn:p_n_bd_A_n_error}
			\tilde{p}_n
			\gtrsim \|D_1\|_{\op} \cdots \|D_n\|_{\op} (1 - O(\|\tilde{A}_n - \tilde{A}_1\|_{\op})).
		\end{align}
		By \eqref{eqn:A,B,C_in_delta,eps},
		\begin{align*}
			\|\tilde{A}_n - \tilde{A}_1\|_{\op}
			\lesssim \delta_n + \varepsilon_n^2
			\ll 1.
		\end{align*}
		Thus we may assume that the term $O(\|\tilde{A}_n - \tilde{A}_1\|_{\op})$ in \eqref{eqn:p_n_bd_A_n_error} is $\leq \frac{1}{2}$. Then \eqref{eqn:p_n_bd_A_n_error} reduces to the desired lower bound \eqref{eqn:p_n_lower_bd_by_Ds}.
	\end{proof}
	
	The proof of Proposition~\ref{prop:p_n,i,j_bds} is now complete.
	
	\subsection{Second tail bound}
	
	In this subsection we prove Theorem~\ref{thm:exp_decay_quasimode}.

	Throughout this subsection, fix a positive constant $A \lesssim 1$ such that \eqref{eqn:p_n_sign} holds for all $n \in \Z_{\geq 0}$, $\lambda \geq 0$, and $\mu \geq A(\lambda+n^2)$, and such that \eqref{eqn:p_n,i,j_vs_p_n} holds for all $n,i,j \in \Z_{\geq 0}$ with $i+j \leq n$, all $\lambda \geq 0$, and all $\mu \geq A(\lambda+n^2)$.
	
	For $n \in \Z_{\geq 0}$ and $\lambda,M \geq 0$, recall the seminorm $\|\cdot\|_{n,\lambda,M}$ on $\H^K$ defined by \eqref{eqn:norm_n,lambda,M_def}.
	Let $C_{n,\lambda,M} \in [0,\infty]$ be the smallest constant such that
	\begin{align*}
		\|\alpha\beta\|_{n,\lambda,M}
		\leq C_{n,\lambda,M}^2 \|\alpha\|_{\H} \|\beta\|_{\H}
	\end{align*}
	for all $\alpha,\beta \in \H_{|\Delta-\lambda| \leq 1}^K$. Since $\|\cdot\|_{n,\lambda,M}$ decreases with $M$, so does $C_{n,\lambda,M}$.
	
	\begin{lem} \label{lem:C_n,lambda,M_finite}
		For all $n \in \Z_{\geq 0}$ and $\lambda,M \geq 0$, the constant $C_{n,\lambda,M}$ is finite.
	\end{lem}
	
	\begin{proof}
		The same argument used to prove Lemma~\ref{lem:C_p,s(Lambda)_finite} works here.
	\end{proof}
	
	The key technical result toward Theorem~\ref{thm:exp_decay_quasimode} is the following proposition, which is analogous to Proposition~\ref{prop:phi^2_n,lambda,M_norm} combined with Theorem~\ref{thm:L^4_quasi-Sobolev}.
	
	\begin{prop} \label{prop:C_n,lambda,M_bd}
		Let $n \in \Z_{\geq 0}$, let $\lambda \geq 0$, and let $M \gg \lambda+n^2+1$. Then
		\begin{align} \label{eqn:C_n,lambda,M_bd}
			C_{n,\lambda,M}^2
			\leq O(\lambda+n^2+1)^{n+\frac{1}{2}} \exp(O(\log_+^2\lambda)).
		\end{align}
	\end{prop}
	
	\begin{proof}
		Since $C_{n,\lambda,M}$ decreases with $M$, we may assume $M = A'(\lambda+(n+1)^2)$ for some fixed $A' \gg 1$ to be chosen later.
		Assume $A' \geq A$ and $A' \geq 1$.
		By the definition of $C_{n,\lambda,M}$ and the fact that $C_{n,\lambda,M}$ is finite (Lemma~\ref{lem:C_n,lambda,M_finite}), there exist $\alpha,\beta \in \H_{|\Delta-\lambda|\leq 1}^K$ nonzero, such that
		\begin{align*}
			\|\alpha\beta\|_{n,\lambda,M}
			\gtrsim C_{n,\lambda,M}^2 \|\alpha\|_{\H} \|\beta\|_{\H}.
		\end{align*}
		After normalizing $\alpha,\beta$, we see that there exist $\alpha,\beta \in \H_{|\Delta-\lambda| \leq 1}^K$ with
		\begin{align*}
			\|\alpha\|_{\H}, \|\beta\|_{\H} \lesssim 1
			\qquad \text{and} \qquad
			\|\alpha\beta\|_{n,\lambda,M}
			\gtrsim C_{n,\lambda,M}^2.
		\end{align*}
		Splitting into real and imaginary parts, we may assume in addition that $\alpha,\beta \in \H_{\R}$. Using the polarization identity $\alpha\beta = \frac{1}{2}(\alpha+\beta)^2 - \frac{1}{2}\alpha^2 - \frac{1}{2}\beta^2$, we may further assume $\alpha=\beta$. In summary, we conclude that there exists $\alpha \in \H_{|\Delta-\lambda| \leq 1}^K \cap \H_{\R}$ with
		\begin{align} \label{eqn:alpha_properties_2}
			\|\alpha\|_{\H} \lesssim 1
			\qquad \text{and} \qquad
			\|\alpha^2\|_{n,\lambda,M}
			\gtrsim C_{n,\lambda,M}^2.
		\end{align}
		We can thus estimate
		\begin{align} \label{eqn:C_n,lambda,M_taut_bd}
			C_{n,\lambda,M}^4
			\lesssim \|\alpha^2\|_{n,\lambda,M}^2
			= \langle \1_{\Delta \geq M} |p_np_{n+1}(\lambda,\Delta)|(\alpha^2), \alpha^2 \rangle_{\H}.
		\end{align}
		In view of the definitions of $A$ and $M$, Lemma~\ref{lem:p_n_sign} implies that
		\begin{align*}
			-p_np_{n+1}(\lambda,\mu) \geq 0
		\end{align*}
		for $\mu \geq M$, so
		\eqref{eqn:C_n,lambda,M_taut_bd} becomes
		\begin{align*}
			C_{n,\lambda,M}^4
			\lesssim -\langle \1_{\Delta \geq M} p_np_{n+1}(\lambda,\Delta)(\alpha^2), \alpha^2 \rangle_{\H}.
		\end{align*}
		Writing $-\1_{\Delta \geq M} = \1_{\Delta < M} - \1$,
		\begin{align*}
			C_{n,\lambda,M}^4
			\lesssim \langle \1_{\Delta < M} p_np_{n+1}(\lambda,\Delta)(\alpha^2), \alpha^2 \rangle_{\H} - \langle p_np_{n+1}(\lambda,\Delta)(\alpha^2), \alpha^2 \rangle_{\H}.
		\end{align*}
		Estimating the first term by the trivial bound Lemma~\ref{lem:p_n_triv_bd} and writing $\|\alpha^2\|_{\H}^2 = \|\alpha\|_{L^4}^4$, and estimating the second term by Proposition~\ref{prop:refined_ineq_1} together with \eqref{eqn:p_n,i,j_vanishing},
		\begin{align} \label{eqn:C_n,l,M_to_insert}
			C_{n,\lambda,M}^4
			\lesssim
			O(M)^{2n+1} \|\alpha\|_{L^4}^4
			+ \sum_{\substack{i+j \leq n \\ i'+j' \leq n+1 \\ \textnormal{not all }0}} \prod_{(\ell,m) \in \{(i,j),(i',j')\}} \||p_{n,i,j}p_{n+1,i',j'}(\lambda,\Delta)|^{\frac{1}{2}} (\alpha_{\ell}\alpha_m)\|_{\H},
		\end{align}
		where the sum is over nonnegative integers $i,j,i',j'$, and where we use the notation from Subsection~\ref{subsec:refined_ineqs}, so in particular $\alpha_{\ell}$ denotes the expression $(\Delta-\lambda)^{\ell}\alpha$.
		Splitting $\1 = \1_{\Delta<M} + \1_{\Delta\geq M}$, each term in the product is bounded by
		\begin{align} \label{eqn:muliplicand_bd}
			\|\1_{\Delta < M}|p_{n,i,j}p_{n+1,i',j'}(\lambda,\Delta)|^{\frac{1}{2}} (\alpha_{\ell}\alpha_m)\|_{\H}
			+ \|\1_{\Delta \geq M}|p_{n,i,j}p_{n+1,i',j'}(\lambda,\Delta)|^{\frac{1}{2}} (\alpha_{\ell}\alpha_m)\|_{\H}.
		\end{align}
		Estimating the first term by \eqref{eqn:p_n,i,j_triv_bd} and the second term by \eqref{eqn:p_n,i,j_vs_p_n} (the latter being valid because of the definitions of $A$ and $M$),
		\begin{align} \label{eqn:need_alpha_l*alpha_m_bd}
			\eqref{eqn:muliplicand_bd}
			\leq O(M)^{n+\frac{1}{2}} \|\alpha_{\ell}\alpha_m\|_{\H} + \Big[O(M^{-1})^{i+j+i'+j'} \binom{n}{i+j} \binom{n+1}{i'+j'}\Big]^{\frac{1}{2}} \|\alpha_{\ell}\alpha_m\|_{n,\lambda,M}
		\end{align}
		(by \eqref{eqn:p_n,i,j_triv_bd}, the exponent $n+\frac{1}{2}$ in the first term could be improved to $n+\frac{1}{2} - \frac{1}{2}(i+j+i'+j')$, but this would not improve the final result).
		By $L^4$-Cauchy--Schwarz (Proposition~\ref{prop:L^4_C-S}), Theorem~\ref{thm:L^4_quasi-Sobolev}, the fact that $\alpha \in \H_{|\Delta-\lambda| \leq 1}^K$, and the normalization condition $\|\alpha\|_{\H} \lesssim 1$ in \eqref{eqn:alpha_properties_2},
		\begin{align*}
			\|\alpha_{\ell}\alpha_m\|_{\H}
			\leq \|\alpha_{\ell}\|_{L^4} \|\alpha_m\|_{L^4}
			\leq \exp(O(\log_+^2\lambda)) \|\alpha_{\ell}\|_{\H} \|\alpha_m\|_{\H}
			\leq \exp(O(\log_+^2\lambda)) \|\alpha\|_{\H}^2
			\lesssim \exp(O(\log_+^2\lambda)).
		\end{align*}
		By the definition of $C_{n,\lambda,M}$, the fact that $\alpha \in \H_{|\Delta-\lambda| \leq 1}^K$, and the normalization $\|\alpha\|_{\H} \lesssim 1$,
		\begin{align*}
			\|\alpha_{\ell}\alpha_m\|_{n,\lambda,M}
			\leq C_{n,\lambda,M}^2 \|\alpha_{\ell}\|_{\H} \|\alpha_m\|_{\H}
			\leq C_{n,\lambda,M}^2 \|\alpha\|_{\H}^2
			\lesssim C_{n,\lambda,M}^2.
		\end{align*}
		Inserting these bounds into \eqref{eqn:need_alpha_l*alpha_m_bd},
		\begin{align*}
			\eqref{eqn:muliplicand_bd}
			\lesssim O(M)^{n+\frac{1}{2}} \exp(O(\log_+^2\lambda)) + \Big[O(M^{-1})^{i+j+i'+j'} \binom{n}{i+j} \binom{n+1}{i'+j'}\Big]^{\frac{1}{2}} C_{n,\lambda,M}^2.
		\end{align*}
		By the triangle inequality and the elementary inequality $(A+B)^2 \lesssim A^2+B^2$, we deduce that the product in $\RHS\eqref{eqn:C_n,l,M_to_insert}$ is bounded by
		\begin{align*}
			O(M)^{2n+1} \exp(O(\log_+^2\lambda))
			+ O(M^{-1})^{i+j+i'+j'} \binom{n}{i+j} \binom{n+1}{i'+j'} C_{n,\lambda,M}^4.
		\end{align*}
		Thus from \eqref{eqn:C_n,l,M_to_insert} we get
		\begin{align*} 
			C_{n,\lambda,M}^4
			\lesssim O(M)^{2n+1} \|\alpha\|_{L^4}^4
			&+ \sum_{\substack{i+j \leq n \\ i'+j' \leq n+1 \\ \textnormal{not all }0}} O(M)^{2n+1} \exp(O(\log_+^2\lambda))
			\notag
			\\&+ \sum_{\substack{i+j \leq n \\ i'+j' \leq n+1 \\ \textnormal{not all }0}} O(M^{-1})^{i+j+i'+j'} \binom{n}{i+j} \binom{n+1}{i'+j'} C_{n,\lambda,M}^4.
		\end{align*}
		To bound the first term on the right hand side, we estimate $\|\alpha\|_{L^4} \lesssim \exp(O(\log_+^2\lambda))$ by Theorem~\ref{thm:L^4_quasi-Sobolev} and the normalization $\|\alpha\|_{\H} \lesssim 1$.
		The sum in the second term on the right hand side consists of $(n+1)^{O(1)} \leq O(1)^n$ many identical terms, and this factor of $O(1)^n$ can be absorbed into the factor of $O(M)^{2n+1}$ in the summand. Thus
		\begin{align*}
			C_{n,\lambda,M}^4
			\lesssim O(M)^{2n+1} \exp(O(\log_+^2\lambda)) + \Big[\sum_{\substack{i+j \leq n \\ i'+j' \leq n+1 \\ \textnormal{not all }0}} O(M^{-1})^{i+j+i'+j'} \binom{n}{i+j} \binom{n+1}{i'+j'} \Big]C_{n,\lambda,M}^4.
		\end{align*}
		We wish to bound the sum in brackets. Denoting $k = i+j$ and $k' = i'+j'$, we can express this as a weighted sum over $k$ and $k'$, where each term is weighted by the number of ways to write $k = i+j$ and $k' = i'+j'$. This weight is equal to $(k+1)(k'+1) \leq O(1)^{k+k'}$, so it can be absorbed into the factor of $O(M^{-1})^{i+j+i'+j'}$. Therefore
		\begin{align*}
			C_{n,\lambda,M}^4
			\lesssim O(M)^{2n+1} \exp(O(\log_+^2\lambda)) + \Big[\sum_{\substack{k \leq n \\ k' \leq n+1 \\ \textnormal{not both }0}} O(M^{-1})^{k+k'} \binom{n}{k} \binom{n+1}{k'} \Big]C_{n,\lambda,M}^4.
		\end{align*}
		Using the identity
		\begin{align*}
			\sum_{\substack{k \leq n \\ k' \leq n+1}} X^{k+k'} \binom{n}{k} \binom{n+1}{k'}
			= (1+X)^n(1+X)^{n+1}
			= (1+X)^{2n+1}
		\end{align*}
		to simplify the sum in brackets,
		\begin{align*}
			C_{n,\lambda,M}^4
			\lesssim O(M)^{2n+1} \exp(O(\log_+^2\lambda)) + [(1+O(M^{-1}))^{2n+1} - 1] C_{n,\lambda,M}^4.
		\end{align*}
		Since $M = A'(\lambda+(n+1)^2)$ with $A' \geq 1$, the expression in brackets is $\lesssim \frac{n}{M} \leq M^{-\frac{1}{2}}$. Thus, absorbing implicit constants into factors of $O(1)$,
		\begin{align*}
			C_{n,\lambda,M}^4
			\leq O(M)^{2n+1} \exp(O(\log_+^2\lambda)) + O(M^{-\frac{1}{2}}) C_{n,\lambda,M}^4.
		\end{align*}
		Choose $A'$ large enough that the quantity $O(M^{-\frac{1}{2}})$ is $\leq \frac{1}{2}$. Then the second term on the right can be subtracted from both sides to give
		\begin{align} \label{eqn:half_C_n,l,M_bd}
			\tfrac{1}{2} C_{n,\lambda,M}^4 \leq O(M)^{2n+1} \exp(O(\log_+^2\lambda))
		\end{align}
		(note this subtraction is valid because $C_{n,\lambda,M}$ is finite by Lemma~\ref{lem:C_n,lambda,M_finite}). Now that $A'$ has been fixed, we have $M \lesssim \lambda+n^2+1$. Substituting this into \eqref{eqn:half_C_n,l,M_bd}, multiplying both sides by $2$, and taking square roots, we obtain the desired bound \eqref{eqn:C_n,lambda,M_bd}.
	\end{proof}
	
	\begin{prop} \label{prop:exp_decay_quasimode_old}
		There is a positive constant $c \gtrsim 1$, such that for any $\lambda \geq 0$, any $\alpha,\beta \in \H_{|\Delta-\lambda| \leq 1}^K$, and any $M \gg \lambda+1$,
		\begin{align*}
			\|\1_{\Delta \geq M}(\alpha\beta)\|_{\H}
			\leq \exp(-c\sqrt{M}) \|\alpha\|_{\H} \|\beta\|_{\H}.
		\end{align*}
	\end{prop}
	
	\begin{proof}
		The proof is almost identical to that of Proposition~\ref{prop:exp_decay_mode_old}, except with Proposition~\ref{prop:C_n,lambda,M_bd} replacing the combination of Proposition~\ref{prop:phi^2_n,lambda,M_norm} and Theorem~\ref{thm:L^4_quasi-Sobolev}.
	\end{proof}
	
	
	\begin{prop} \label{prop:exp_decay_quasimode_after_dyadic}
		There is a positive constant $c \gtrsim 1$, such that for any $\lambda \geq 0$, any $\alpha,\beta \in \H_{|\Delta-\lambda| \leq 1}^K$, and any $M \gg \lambda+1$,
		\begin{align} \label{eqn:exp_decay_quasimode_after_dyadic}
			\|\1_{\Delta \geq M} \exp(c\sqrt{\Delta}) (\alpha\beta)\|_{\H}
			\leq \|\alpha\|_{\H} \|\beta\|_{\H}.
		\end{align}
	\end{prop}
	
	\begin{proof}
		Similarly to Proposition~\ref{prop:exp_decay_mode_after_dyadic}, this follows from a dyadic decomposition argument along the lines of \eqref{eqn:dyadic_exmp}, except using Proposition~\ref{prop:exp_decay_quasimode_old} instead of Proposition~\ref{prop:exp_decay_form_old}.
	\end{proof}
	
	
	We are finally ready to prove Theorem~\ref{thm:exp_decay_quasimode}. Recall the statement:
	
	\begin{thm*} [Restatement of Theorem~\ref{thm:exp_decay_quasimode}]
		There are positive constants $C \lesssim 1$ and $c \gtrsim 1$, a partition of $[0,\infty)$ into intervals $I_i$ of length $\lesssim 1$, and points $\lambda_i \in I_i$ for each $i$, such that the partition has polynomial growth in the sense that
		\begin{align*} 
			\#\{i : \lambda_i \leq X\} \lesssim X^{O(1)}
			\qquad \text{for} \qquad
			X \geq 1,
		\end{align*}
		and such that for all $i$ and all $\alpha,\beta \in \H_{\Delta \in I_i}^K$,
		\begin{align} \label{eqn:exp_decay_quasimodes_restated}
			\|\1_{\Delta \geq C\lambda_i} \exp(c\sqrt{\Delta})(\alpha\beta)\|_{\H}
			\leq \|\alpha\|_{\H} \|\beta\|_{\H}.
		\end{align}
	\end{thm*}
	
	\begin{proof}
		For $i \in \Z_{\geq 0}$, let $I_i = [i,i+1)$, and let $\lambda_i$ be the midpoint of $I_i$. Then each $I_i$ has length $\leq 1$, and together the $I_i$ make up a partition of $[0,\infty)$ of polynomial growth.
		Let $c \gtrsim 1$ be as in Proposition~\ref{prop:exp_decay_quasimode_after_dyadic}. Since $\lambda_i \geq \frac{1}{2}$, the condition $M \gg \lambda_i+1$ is equivalent to $M \gg \lambda_i$. Thus by Proposition~\ref{prop:exp_decay_quasimode_after_dyadic}, there exists $C \lesssim 1$ such that for $M \geq C\lambda_i$, the estimate \eqref{eqn:exp_decay_quasimode_after_dyadic} holds whenever $\alpha,\beta \in \H_{\Delta \in I_i}^K$. In particular, taking $M = C\lambda_i$ gives \eqref{eqn:exp_decay_quasimodes_restated}.
	\end{proof}
	
	All implications shown in \eqref{eqn:pf_structure} have now been established, so Theorem~\ref{thm:eq_Gelfand_duality} is proved.
	
	\bibliography{bootstrap_converse_arxiv}
	
\end{document}